\documentclass[11pt,a4paper]{article}
\usepackage[applemac]{inputenc}

\usepackage{fullpage}
\usepackage{
 amsfonts}
 \usepackage{amsmath}
\usepackage{amsmath,esint, amssymb, amsthm,mathbbol,upgreek,exscale}
\usepackage{amsthm,latexsym,epsfig,graphicx,marvosym, mathrsfs,bigints,stmaryrd}
\usepackage{amsmath,stmaryrd}
\usepackage{relsize}
\usepackage{amsfonts}
\usepackage{amssymb}
\usepackage{pifont}
 \usepackage{mathtools}

\usepackage[colorlinks=true, pdfstartview=FitV, linkcolor=blue, citecolor=blue,
 urlcolor=blue]{hyperref}
\usepackage{color}

\def\N{{\mathbb N}}
\date{}

\newcommand{\ii }{{\rm i} }

\newcommand{\RR}{\mathbb{R}}  
\newcommand{\R}{\mathbb{R}} 
\newcommand{\ZZ}{\mathbb{Z}} 
\newcommand{\Z}{\mathbb{Z}} 
\newcommand{\T}{\mathbb{T}}

\theoremstyle{plain}

\numberwithin{equation}{section}
\newtheorem{theorem}{Theorem}[section]
\newtheorem{proposition}[theorem]{Proposition}
\newtheorem{lemma}[theorem]{Lemma}

\newtheorem*{remarks}{Remarks}
\newtheorem{coro}[theorem]{Corollary}

\usepackage[textmath,displaymath,floats,graphics,delayed]{preview}
\title{Rigorous derivation of the leapfrogging motion for planar Euler equations}
\author{ Zineb Hassainia  \and Taoufik Hmidi \and Nader Masmoudi}

\begin{document}
 \maketitle
 \begin{abstract}
 The main goal of  this paper is to explore the leapfrogging phenomenon in the  inviscid planar flows. We show for $2$d Euler equations that under suitable constraints, four concentrated vortex patches leapfrog  for all time. When observed from a translating frame of reference, the evolution of these vortex patches can be described as a  non-rigid time periodic motion. Our proof hinges upon two key components. First, we desingularize the symmetric four  point vortex configuration, which leapfrogs in accordance with Love's result \cite{Love}, by concentrated vortex patches. Second, we borrow some tools from   KAM  theory  to effectively  tackle the small divisor problem and deal with the degeneracy in the time direction. 
  Our approach is robust and flexible and  solves a long-standing problem that has remained unresolved for many decades.
 \end{abstract}

 \tableofcontents

 \section{Introduction}
 This paper is dedicated to the exploration of the leapfrogging phenomenon within the framework of the incompressible Euler equations given in  terms of a transport nonlinear equation on the vorticity as follows
 \begin{equation}\label{Euler-Eq}
	\left\lbrace\begin{array}{ll}
		\partial_{t}{\omega}+u\cdot\nabla \omega=0, 
		\\
		u=\nabla^{\perp}(\Delta)^{-1}{\omega},\quad\hbox{with}\quad \nabla^{\perp}=(-\partial_{2},\partial_{1}).
	\end{array}\right.
\end{equation}
Within the realm of 2D vortex dynamics in ideal flows, a plethora of intriguing phenomena is intricately linked to coherent structures. One informal reason to the emergence of huge variety of  long-lived structures   lies in the fact  that the equations of motion are Hamiltonian with infinitely many equilibria.  Actually,   their phase portrait reveals a multitude of surprising and intricate structures. One particularly significant facet of this field revolves around relative equilibria, which represent equilibrium states in a rigid body frame.
 Relative equilibria stand as a central and extensively researched aspect of this subject, garnering considerable theoretical and experimental examination. For a comprehensive exploration of the bifurcation of rigid time-periodic solutions stemming  from various  equilibria, we direct the reader's attention to an exhaustive list that outlines numerous pertinent questions and findings in this domain. To illustrate, the papers \cite{B,CCG4,DHHM,HFMV,gomez2019symmetry,Hass-Mass-Wheel,HM2,HM3,HMV} have demonstrated the existence of nontrivial rotating patch solutions nearby the disc, the ellipse and the annulus. Furthermore, the exploration of the general case, encompassing nonuniform vortices near monotonic radial shapes, has been recently investigated in detail in \cite{CCG2,GHM,GHS}.
 
 Another pertinent topic pertains to the generation of highly concentrated rotating solutions through the desingularization of symmetric configurations within the point vortex system. This latter model was derived by Helmholtz as a simplified representation of the two-dimensional incompressible Euler equations, wherein vorticity is confined to a discrete set of moving points. It is  important to note that these solutions exhibit extreme singularity and do not adhere to the governing Euler equations. As an illustrative example, this approach was applied to pairs of rotating vortex points in \cite{HH2,HM}, and subsequently extended to cover  Thomson and nested Thomson polygons,  K\'arm\'an vortex  streets as described in \cite{G-Kar,Garcia,HW}. An alternative approach based on variational arguments was developed before in \cite{T}.
 
 Besides these rigid  time periodic structures, the literature has extensively documented remarkable structures within the vortex point system, particularly in the context of leapfrogging motion, dating back to their initial discovery by Helmholtz. Leapfrogging refers to the mesmerizing interaction between vortices spinning whirlpools of fluid that play a pivotal role in shaping fluid flow patterns. These vortices can exhibit an astonishing behavior as they engage in a delicate dance  and motion through one another. For instance,  an intriguing   example was reported  by Love \cite{Love} and  Gr\"obli \cite{Grobli} in the late nineteenth century. They illustrated how a system consisting of four symmetrical vortices, each with equal strength, two with positive circulation and two with negative circulation, gives rise to a remarkable series of leapfrogging orbits. Initially, the inner pair moves swiftly, overtaking the outer pair. As time progresses, the inner pair decelerates and spreads apart, while the separation between the outer pair diminishes, causing them to accelerate. At the midpoint of the motion's period, the roles of the inner and outer pairs swap, and the motion cyclically repeats up to translation. For a comprehensive review of these observations and analytical findings about the stability/instability of Leapfrogging vortex pairs, we refer to  \cite{Behring, Goodman, Aref}. It is also worth pointing out that in case of  the 3D axisymmetric Euler equations without swirl the motion of two leapfrogging vortex rings was considered by  Helmholtz in his celebrated work \cite{Helmholtz}.   
  In particular, Helmholtz qualitatively described  the motion 
 where the rings move along a common symmetry axis, alternately passing through each other. A generalized analysis of Leapfrogging vortex rings have been also discussed in  \cite{Borisov,Dyson1,Dyson2,Hicks,Lamb} and the  references therein. In addition to the leapfrogging motion, many  periodic non-rigid configuration  of vortex point system set in a general domain were also discussed in  \cite{Bartsch,Bartsch2,Gustafsson,Newton}. Very recently, time quasi-periodic solutions have been explored for Euler equations around the ellipses \cite{BHM} or around the Rankine vortices in the disc \cite{HR22}. Similar studies have been performed for different active scalar equations, see for instance  \cite{BCFS,GIP23,HHM21,HR21}.
   
In this paper, our primary objective is to provide a thorough and meticulous derivation of the Leapfrogging phenomenon for Euler  equations \eqref{Euler-Eq}.  We will specifically concentrate on constructing concentrated vortex patches that will replicate  the dynamics exhibited by their core, which is described by a  four  point vortex system  subject to leapfrogging  interaction as described in \cite{Love}. This construction remains valid for all the  time, and produces solutions with  non-rigid periodic motion when observed from a moving frame of reference. This marks the first rigorous derivation of this phenomenon within the context of classical solutions to Euler equations.  
In order to articulate our findings effectively, it is necessary to introduce some basic  concepts related to point vortex system.
We will provide here a concise overview of the dynamics governing four symmetric  point vortices. For a more comprehensive exploration of this topic, we refer to Section \ref{sec:lep-vortex-pairs}. 

We shall  consider  four point vortices  in the complex plane, where the first pair comprises  two points  $z_1$ and $z_2$ located at the upper half plane with  the same circulation $\pi$. Conversely,   the second pair is located at the complex conjugate numbers  $\overline{z_1}$ and $\overline{z_2}$ with the same negative  \mbox{circulation $-\pi$.} Under this symmetry,  the point vortex system  reduces to  two coupled differential equations,
	\begin{equation}\label{4vort-leapIn}
	\begin{aligned}
	\dot{z}_k(t)=\frac{\ii}{2}\Big[\frac{1}{\overline{z_k}-\overline{z_{3-k}}}-\frac{1}{\overline{z_k}-{z_{k}}}-\frac{1}{\overline{z_k}-{z_{3-k}}}\Big], \quad k=1,2.
	\end{aligned}
	\end{equation}
			Denote by 
		\begin{equation*}
		\begin{aligned}
		 x_0(t)&\triangleq \textnormal{Re}\{z_1(t)+z_2(t)\},\qquad y_0(t)\triangleq \textnormal{Im}\{z_1(t)+z_2(t)\},\\
		{\eta(t)}&\triangleq \textnormal{Re}\{z_1(t)-z_2(t)\}, \qquad {\xi(t)}\triangleq \textnormal{Im}\{z_1(t)-z_2(t)\}. 
		\end{aligned}
\end{equation*}
Then,  $(\eta,\xi)$ solves the Hamiltonian system
\begin{eqnarray}  
       \left\{\begin{array}{ll}
          	{\dot{\eta}}&=\partial_{\xi} H(\eta,\xi), \\
\dot{\xi}&=-\partial_{\eta} H(\eta,\xi), 
       \end{array}\right.\quad\hbox{with}\qquad H(\eta,\xi)=-\tfrac12\log\Big(\frac{1}{y_0^2-\xi^2}-\frac{1}{y_0^2+\eta^2}\Big)
\end{eqnarray}	
and $(x_0,y_0)$ satisfies 
\begin{eqnarray}   \label{A04pts0L}
       \left\{\begin{array}{ll}
	\dot{x_0}&=y_0\big(\frac{1}{\eta^2+y_0^2}+\frac{1}{y_0^2-\xi^2}\big), 
	\\
	\dot{y_0}&=0.
	 \end{array}\right.
\end{eqnarray}
The initial  observation is that  from the last equation we get $y_0(t)=y_0>0$ and thus the two first equations are decoupled from the third one. Notice that the  equation \eqref{A04pts0L} describes the speed of the center of the mass of the upper pair $
z_G(t)\triangleq \tfrac12\big(x_0(t)+\ii y_0\big)$. Actually, this point  is translating  along the horizontal axis with a nonuniform speed, according to the formula
	$$
	z_G(t)=\tfrac12\big(x_0(t)+\ii y_0\big).
	$$ 
	Without loss of generality, we can assume that at $t=0$ all the vortices are located on the vertical axis $(Oy)$.
Love proved in \cite{Love}  that the system \eqref{4vort-leapIn} generates nonrigid periodic orbits with a period $T(\xi_0)$ in the translating frame centered at $z_G$  if and only if $\frac{\xi_0}{y_0}<\frac{\sqrt{2}}{2}$, with $\xi_0$ representing the distance between the two vortices $z_1$, $z_2$ when they are aligned vertically. Hence, the motion of these points within the initial frame can be characterized as a combination of nonuniform  translation and nonuniform rotation, which ultimately results in the leapfrogging phenomenon, as previously described. 

Our  main goal  is to desingularize this dynamics by concentrated vortex patches  that accurately emulate the movement of the core vortices via a leapfrogging motion. For this purpose, we shall restrict the exploration to the  solutions to \eqref{Euler-Eq} in the patch form 
\begin{equation}\label{omegatepMs}
\omega(t)=\tfrac{1}{\varepsilon^2}{\bf{1}}_{{D}_{t,1}^\varepsilon}+\tfrac{1}{\varepsilon^2}{\bf{1}}_{{D}_{t,2}^{\varepsilon}}-\tfrac{1}{\varepsilon^2}{\bf{1}}_{\overline{{D}_{t,1}^\varepsilon}}-\tfrac{1}{\varepsilon^2}{\bf{1}}_{\overline{{D}_{t,2}^{\varepsilon}}},
\end{equation}
where $\varepsilon\in(0,1)$ is a parameter, $D_{t,k}^{\varepsilon}$ are planar domains  given by  
\begin{equation}\label{domainsM}
{D}_{t,k}^\varepsilon \triangleq\,   \varepsilon  {O}_{t,k}^\varepsilon+z_k(t), \quad k=1,2,
\end{equation}
with   ${O}_{t,k}^\varepsilon$ being simply connected domains localized around the unit disc.  The domains $\overline{{D}_{t,k}^{\varepsilon}}$ are the complex  conjugate of ${D}_{t,k}^{\varepsilon}.$
A priori the contour dynamics equations are given by a coupling of two nonlinear equations governing the two boundaries. However, we will see  in \mbox{Section \ref{Bound-motion}} that using a symmetry reduction we find out that  a particular solution is given by 
\begin{align}\label{sym-YU}
\forall t\in\R,\quad {O}_{t,2}^\varepsilon={O}_{t+\frac{T(\xi_0)}{2},1}^\varepsilon.
\end{align}
 Our main result reads as follows,
\begin{theorem}\label{TH-Main1}
Given $y_0>0$ and  $0<\xi_*<\xi^*<\frac{y_0}{\sqrt{2}}$, there exists   $\varepsilon_0>0$  small enough such that for all $ \varepsilon\in(0, \varepsilon_0)$
there exists a Cantor like set $\mathcal{C}_\varepsilon\subset [\xi_*,\xi^*]$ with 
$$
\lim_{\varepsilon\to 0}|\mathcal{C}_\varepsilon|=\xi^*-\xi_*$$ such that
for any ${\xi_0\in \mathcal{C}_\varepsilon}$, the system \eqref{Euler-Eq} with the constraints \eqref{omegatepMs}, \eqref{domainsM} and \eqref{sym-YU} admits a solution satisfying 
$$\forall t\in\R,\quad  {O}_{t+T(\xi_0),1}^\varepsilon={O}_{t,1}^\varepsilon.$$
 Here  $T(\xi_0)$ denotes  the time period of the orbit of the four points vortices  $\{z_k, k=1,..,4\}$ in the translating frame. The quantity  $|\mathcal{C}_\varepsilon|$ stands for  the Lebesgue measure of the set $\mathcal{C}_\varepsilon.$
\end{theorem}
Before outlining the proof, it is essential to make few remarks.
\begin{remarks}
\begin{enumerate}
\item We emphasize that the domain ${O}_{t,1}^\varepsilon$  constructed in Theorem ${\rm \ref{TH-Main1}}$ is time periodic, but not rigidly rotating. Indeed,  we prove that the corresponding polar parametrization of its boundary $\partial {O}_{t,1}^\varepsilon$ has the asymptotics 
\begin{align*}
\theta\in\T\mapsto \Big(1+\varepsilon^2 \mathtt{r}_0(\omega_0 t,\theta)+O(\varepsilon^{2+\mu})\Big)e^{\ii\theta},
\end{align*}
where $\mu\in(0,1)$,  $\omega_0=\omega_0(\xi_0)=\frac{2\pi}{T(\xi_0)}$ is the frequency of the periodic orbit of the four points vortices  $\{z_k, k=1,..4\}$ in the translating frame,
$$
\forall (\varphi,\theta)\in\T^2,\quad \mathtt{r}_0(\varphi,\theta)
=\textnormal{Re}\Big\{\tfrac{e^{\ii2(\theta-\Theta(\varphi))}}{q(\varphi)}+\tfrac{e^{\ii 2\theta}}{( \sqrt{q(\varphi)}\sin( \Theta(\varphi))+ y_0)^2}-\tfrac{e^{\ii 2\theta}}{( \sqrt{q(\varphi)}\cos( \Theta(\varphi))+ \ii y_0)^2} \Big\}
$$
and $(\Theta,q)$ are defined through the identity 
$
z_1(t)- z_2(t)=\sqrt{q(\omega_0 t)} e^{\ii \Theta(\omega_0t)}.
$
\item The  ansatz \eqref{sym-YU} introduces a new behavior related to the time delay effects on the  contour dynamics equations.  Actually, this delay    is active  at the scale   $\varepsilon^2$ through  a smoothing operator  and it  has a few impacts on the construction of periodic solutions. The major one is related to the analysis of the  degenerate modes $\pm1$ where we should analyze the companion monodromy matrix related to a $4\times4$ resolvant  matrix with variable coefficients.  
\item   In the context of three dimensional axially symmetric case, the investigation conducted in  $\textnormal{  \cite{delpino}}$ reveals  the existence of a weak form of the  leapfrogging motion with  smooth ring solutions to the Euler equations. Actually, using a gluing scheme, initiated in $\textnormal{\cite{delpino3}}$,  the authors were able to decompose locally  in time the solution into concentrated smooth vortex rings that leapfrog  up to a small error. It seems extremely  challenging to generically eliminate this residual terms and get time periodic solutions due to  some instabilities mechanisms as outlined in $\textnormal{ \cite{Al,Yamada}.}$ Very recently, the gluing method was successfully  used in $\textnormal{\cite{delpino2}}$ to provide a long time description of the motion  of two vortex pairs solutions  traveling in opposite directions for  the 2D Euler equations.      
On the other hand, similar results on the leapfrogging phenomena   were achieved before  for the Gross-Pitaevskii equation in $\textnormal{\cite{smets,smets2}}$.
 In our Theorem ${\rm \ref{TH-Main1}}$ we are able to track carefully the leapfrogging  for all time in  the two-dimensional case. Actually, the solutions in the patch-like form leapfrog globally in time without any error. The desingularization is valid for essentially  {\it almost} all the periodic  point vortex orbits when the parameter $\xi_0$ belongs to the Cantor like set $\mathcal{C}_\varepsilon$.  \end{enumerate}
\end{remarks}
To the best of our knowledge, Theorem \ref{TH-Main1} represents the first construction of classical solutions for Euler equations featuring leapfrogging motion for any time. Our approach is using Nash-Moser scheme together with KAM tools to tackle  degenerate quasi-linear transport equations forced by time-space periodic coefficients arising  from the point vortex system. This approach   is highly robust and flexible offering various perspectives to explore numerous  important ordered structures in geophysical flows. As we shall see below,  in addition to the challenges posed by the small divisor problem stemming  from the time-space resonances, our proof encounters several substantial   difficulties succinctly  outlined as follows:

\begin{itemize}
\item As $\varepsilon$ approaches zero, the time direction degenerates as illustrated in \eqref{FF-eq1}. This compels us to operate within Cantor sets on the parameter $\xi_0$ subject to Diophantine conditions that degenerate with $\varepsilon$. Actually, as we shall see in \eqref{Cantor-set1} we need to  penalize these sets by  a degenerating coefficient of size $\varepsilon^{2+\delta}, \delta>0$ in order to guarantee  an almost full Lebesgue measure after the multiple excision steps implemented  along  Nash-Moser scheme.  This rate of degeneracy in $\varepsilon$  has severe impacts in the scheme. One of them is related to the construction of  a right inverse for the linearized operator. It  exhibits a loss of regularity and entails a   divergent control of order  $\varepsilon^{-2-\delta}$.

\item The equation governing radial deformation, stated in \eqref{FF-eq1}, lacks  equilibrium states around the four point vortex system. The natural  approximation given by small discs appears inadequate to   counterbalance  the divergence rate $\varepsilon^{-2-\delta}$ of the approximate inverse mentioned before. This becomes crucial when we will implement   the first step in  Nash-Moser scheme which gives an unbounded control in $\varepsilon$ for the next iteration. Hence, we devise a careful scheme, without any Cantor sets, to generate a suitable  approximation for the nonlinear equation. 

\item The invertibility of the linearized operator defined in \eqref{Lin-Op1} is not straightforward. This operator takes the form of a transport equation with variable coefficients perturbed by an operator of order zero (Hilbert transform) together with a smoothing (in space)  compact operator.   The goal is to conjugate this operator, exhibiting different scales in $\varepsilon$, to a constant coefficient operator up to some error terms. Notice that at the leading term, this operator is a perturbation of a time degenerating Fourier  multiplier that undergoes time-space resonances. However the perturbation scales at the same rate  of the time degeneracy ($\varepsilon^2$) rendering  the first step in  KAM reduction impossible. To overcome this issue,  an initial step on a change of coordinate system (without Cantor sets) is required  to reduce the perturbation size.

\item As we can observe from the structure \eqref{Lin-Op1}, the space  modes $\pm1$  degenerate  at the leading order. This trivial resonance has been previously observed in various  situations and addressed  differently according to the context. For instance, it was discussed in the paper  \cite{HM} devoted to  the construction of translating and rotating rigid pairs. To deal with this  degeneracy,  an  external parameter was used (the angular velocity or the speed) as a Lagrange multiplier in order to remove  these modes from the phase space. However, this approach does not fit  with our ansatz here as the period   is fixed from the outset  according to the point vortex period.  Another context related to this degeneracy  is quoted in the papers  \cite{BHM,HHM21,HR21} on time quasi-periodic patches close to Rankine vortices. There,  an extra parameter modeling a rigid rotating frame was used  allowing to shift  the spectrum  and eliminate the degeneracy.  This cannot be implemented here  because it will destroy the time periodic structure. To fix this issue, we use the non rigid rotating frame associated with the point vortices and  we proceed with a careful description at the order $\varepsilon^2$ of the asymptotic structure of the linearized operator in \eqref{Lin-Op1}. Then we project  the equation on the degenerate modes giving rise due to the delay terms to a system of $4\times4$  coupled differential equations that we solve using the monodromy matrix. Its analysis requires refined arguments related to  the analyticity dependence of the four point vortex  orbits. 
\end{itemize}
%
%
Now, let us  outline the fundamental  steps  of the proof of Theorem \ref{TH-Main1} and explore more the different technical  points raised before. \\
\ding{202}{\it \; Contour dynamics equations and linearization}. We start by expressing  the solution of \eqref{A04pts0L} in terms of  action-angle coordinates, 
$$
z_1(t)- z_2(t)=\sqrt{q(\omega_0 t)} e^{\ii \Theta(\omega_0t)}.
$$
In  the symmetric case \eqref{domainsM}, the motion simplifies the dynamics to a single  equation on   the boundary of $ {O}_{t,1}^\varepsilon$. As we will explore  in Section \ref{Bound-motion}, it is a fundamental fact  to express  its evolution in the rotating frame centered at the point $z_1(t)$ and with angular velocity $\Theta(\omega_0 t)$. Actually, this adjustment  shifts the spectrum  and allows later to tackle    the degeneracy of  the monodromy matrix that will be discussed in Section \ref{sec-monodromy}.  To provide a more precise description, we will parametrize the boundary $ \partial {O}_{t,1}^\varepsilon$ as follows
\begin{align*}
\theta\in\T\mapsto e^{\ii \Theta(\omega_0 t)}\sqrt{1+2\varepsilon r( \omega_0 t,\theta)}\,e^{\ii\theta}
\end{align*}
with $r:(\varphi,\theta)\in\T^2\mapsto r(\varphi,\theta)\in\R$ being a smooth periodic  function. Then from the contour dynamics equation, see  \eqref{Edc eq rkPsi}, we find  
the nonlinear equation with a time delay (here $\widetilde{G}=\tfrac1\varepsilon G$)
\begin{align}\label{FF-eq1}
 \widetilde{G}(r) (\varphi ,\theta)  &\triangleq\varepsilon^2\omega_0\partial_\varphi r(\varphi ,\theta)-\varepsilon^2\omega_0 {\dot\Theta(\varphi )} \partial_\theta r(\varphi ,\theta)+\partial_{\theta}\big[F(\varepsilon, {q}, r)\big]=0.
		\end{align}
Note that there is no  explicit  equilibrium state to this equation and one deduces  from \eqref{lem eq EDC r} that $\widetilde{G}(0)\neq 0$, with
\begin{align}\label{Initi-est}
\widetilde{G}(0)=O(\varepsilon).
\end{align}
By linearizing around a small state $r$ we get the asymptotic structure
\begin{align*}
\partial_{r }{\widetilde{G}}(r) [ h ] &=\varepsilon^2\omega_0\partial_\varphi  h +\partial_{\theta}\Big[\Big(\tfrac12- \tfrac{\varepsilon}{2} r-\varepsilon^2\omega_0\dot\Theta (\varphi)-\tfrac{\varepsilon^2}{2}{\mathtt{g}}+{\varepsilon^2}{V}_1^\varepsilon(r)+{\varepsilon^3}{V}_2^\varepsilon(r)\Big) h\Big]
\\ &\quad-\tfrac{1}{2}{\mathcal{H}[h]}-{\varepsilon^2}\partial_\theta {\mathcal{Q}_0[h]}+{\varepsilon^2}  \partial_\theta\mathcal{R}_1^\varepsilon(r)[h]+{\varepsilon^3} \partial_\theta\mathcal{R}_2^\varepsilon(r)[h],
		\end{align*} 
		where $\mathcal{H}$ denotes the toroidal  Hilbert transform given by \eqref{H-Def}. The function $\mathtt{g}$ takes the form 
\begin{equation*}
{\mathtt{g}(\varphi,\theta)=\textnormal{Re}\Big\{ \Big(\tfrac{1}{q(\varphi)}+\tfrac{e^{\ii 2\Theta(\varphi)}}{\big( \sqrt{q(\varphi)}\sin( \Theta(\varphi))+ y_0\big)^2}-\tfrac{e^{\ii 2\Theta(\varphi)}}{\big( \sqrt{q(\varphi)}\cos( \Theta(\varphi))+\ii y_0\big)^2}\Big)e^{\ii 2\theta} \Big\}.
}\end{equation*}
We point out that the real-valued function $r\mapsto V_1^\varepsilon(r)$ is quadratic while   $r\mapsto V_2^\varepsilon(r)$ is {affine}.
As to the   operator $\mathcal{R}_1^\varepsilon(r)$, it exhibits  spatial  smoothing  effects and quadratic dependence \mbox{ in $r$.}     In contrast,  
the  operator $\mathcal{R}_2^\varepsilon(r)$ features a  spatial smoothing effect in space but it has an {affine} dependence  \mbox{in $r$.}
This properties will be important in the asymptotics when the solution will be rescaled, see Section \ref{rescaled func}.
 Finally, the operator  $\mathcal{Q}_0$ has a specific structure and  localizes  on the spatial modes $\pm 1$ according to the form  
\begin{align*}
{\mathcal{Q}_0}[h](\varphi,\theta)&= \int_{\mathbb{T}}  h (\varphi+\pi,\eta)\textnormal{Re}\Big\{\tfrac{e^{\ii(\theta+\eta)}}{q(\varphi)}-\tfrac{e^{\ii(\theta-\eta)}}{( \sqrt{q(\varphi)}\cos( \Theta(\varphi))+\ii y_0)^2}\Big\} d\eta\\ &\quad-\int_{\mathbb{T}}h(\varphi,\eta)\textnormal{Re}\Big\{ \tfrac{e^{\ii(\theta-\eta)}}{( \sqrt{q(\varphi)}\sin( \Theta(\varphi))+ y_0)^2}\Big\} d\eta.
\end{align*}
It is worth noting  that this  property holds significant  importance  when we address  the degeneracy of the modes $\pm1.$
 In  the case   $r=0$, the linearized operator simplifies to
\begin{align*}
\partial_{r }{\widetilde{G}}(r) [ h ] &=\varepsilon^2\omega_0(\xi_0)\partial_\varphi  h +\partial_{\theta}\Big[\Big(\tfrac12-\varepsilon^2\omega_0\dot\Theta (\varphi)-\tfrac{\varepsilon^2}{2}{\mathtt{g}}+O({\varepsilon^3})\Big) h\Big]
-\tfrac{1}{2}{\mathcal{H}[h]}-{\varepsilon^2} \partial_\theta{\mathcal{Q}_0[h]}+O({\varepsilon^3}).
		\end{align*} 
		This operator assumes  variable coefficients but  at the main order it can be approximated by the Fourier multiplier operator 
		$$
		\mathcal{L}_{0,0}\triangleq\varepsilon^2\omega_0\partial_\varphi   +\tfrac12\big[\partial_{\theta}-\mathcal{H}\big].
				$$
Using Fourier series, we can express it as
$$
h(\varphi,\theta)=\sum_{(\ell,j)\in\mathbb{Z}^{2 }}h_{\ell,j}\,e^{\ii(\ell\cdot\varphi+j\theta)}\Longrightarrow
 \mathcal{L}_{0,0}h(\varphi,\theta)=\ii\sum_{(\ell,j)\in\mathbb{Z}^{2 }}\big[\varepsilon^2 \ell\omega_0(\xi_0)+\tfrac12(j-\textnormal{sign}(j))]\,h_{\ell,j}\,e^{\ii(\ell\cdot\varphi+j\theta)}.
 $$
At this stage, two primary challenges become evident for inverting this operator, which will persist in the general case. The first challenge is related to the degeneracy of the spatial \mbox{modes $j=\pm1$,} and its exploration requires to  include the  terms of order $\varepsilon^2$ from $\partial_{r}{\widetilde{G}}(r)$. The second challenge pertains to a small divisor problem arising from the time-space resonance, where the time direction degenerates as $\varepsilon$ tends to zero. To tackle this latter problem, we must work with Cantor sets on the parameter $\xi_0$ that are described by the Diophantine conditions,
\begin{align}\label{Cantor-set1}
\mathcal{A}=\bigcap_{(\ell,j)\in\mathbb{Z}^{2}\atop 2\leqslant |j|\leqslant N_{n}}\Big\lbrace \xi_0\in (a,b);\;\, \big|\varepsilon^2\omega(\xi_0)  \ell+\tfrac12(j-\textnormal{sign}(j))\big|\geqslant \varepsilon^{2+\delta}{| j|^{-\tau}}\Big\rbrace\quad\hbox{with}\quad  0<\delta<1.
\end{align}
Notice that it essential to have  $\delta>0$ in order to ensure that we cover an  almost full Lebesgue measure. In fact, we establish the following result,
$$
|\mathcal{A}|\geqslant b-a-C \varepsilon^\delta.
$$
Hence by restricting to this Cantor set we get a formal right inverse satisfying
$$
\|(\Pi^\perp\mathcal{L}_{0,0} \Pi^\perp)^{-1} h\|_{s}^{\textnormal{Lip}(\lambda)}\leqslant C\varepsilon^{-2-\delta}\|h\|_{s+2\tau}^{\textnormal{Lip}(\lambda)},
$$
where $\Pi$ denotes the projector on the spatial modes $\pm1$ and $\Pi^\perp=\hbox{Id}-\Pi$.
Now, in the context of  Nash-Moser scheme, when we initialize with $r=0$, the subsequent iteration will be associated with the term $(\Pi^\perp\mathcal{L}_{0,0} \Pi^\perp)^{-1}\widetilde{G}(0)$. However, considering \eqref{Initi-est} and the aforementioned estimate, it becomes evident that this quantity will diverge for small $\varepsilon$. 
Therefore, this approach makes the scheme unfeasible in its current form and prior to embarking on the invertibility of the linearized operator and  Nash-Moser scheme, it is imperative to construct a more suitable approximation for the nonlinear equation \eqref{FF-eq1}.\\
\ding{203} {\it \; Approximate solutions to \eqref{FF-eq1} and rescaling}. The construction of a good approximation  will be carefully explored  in Section \ref{Approxim-sol1} and  in particular in Lemma  
\ref{lem: construction of appx sol} where we will show the existence of a function $r_\varepsilon:(a,b)\times \T^2\to\R$ such that
$$
 \|r_\varepsilon\|_{s}^{\textnormal{Lip}(\lambda)}\lesssim 1\quad\hbox{and}\quad  \|\widetilde{G}( \varepsilon r_\varepsilon)\|_{s}^{\textnormal{Lip}(\lambda)}\lesssim \varepsilon^4.
 $$
 Unlike the intermediate Cantor sets, such as $\mathcal{A}$, where $\varepsilon$ exhibits some adverse effects, the time degeneracy will be an  advantage in constructing this approximate solution by inverting only spatial elliptic operators. This construction can be achieved straightforwardly by expanding $\widetilde{G}$ using Taylor series, in conjunction with the ansatz $r_\varepsilon=\varepsilon r_0+\varepsilon^2 r_1$. By identifying the leading terms with respect to $\varepsilon$, we arrive at an equation in the following form,
 $$
\big[\partial_\theta -\mathcal{H}\big] r_0+ \tfrac{2}{\varepsilon} \widetilde{G}(0)=0.
 $$
According to Lemma \ref{lem eq EDC r} the term $\widetilde{G}(0)=\tfrac1\varepsilon G(0)$ does not contain   the degenerate spatial \mbox{modes $\pm1$.} Consequently, we can solve this equation with a solution of magnitude $1$, by virtue of  \eqref{Initi-est}. Regarding $r_1$, we find a similar equation to that of  $r_0$ with a forcing term that assumes a complex form. Nonetheless, it possesses the advantage of not including the degenerate space \mbox{modes $\pm1$} at the leading term.
Finally, we will introduce a rescaling of the function $\widetilde{G}$ and proceed to work with the transformed functional equation
\begin{equation}\label{Scale-func}
\mathcal{F}(\rho)\triangleq\tfrac{1}{\varepsilon^{1+\mu}} G( \varepsilon r_\varepsilon+\varepsilon^{1+\mu}\rho)=0,
\end{equation}
where  $\mu\in(0,1)$ is a free parameter that will be adjusted during  Nash-Moser scheme and the measure of the  final Cantor set. \\
\ding{204} {\it \, Construction of an approximate right inverse.}
  To implement a modified Nash-Moser scheme that fits with  our purpose, we must first establish an approximate right inverse for the linearized operator of the functional $\mathcal{F}$ introduced in \eqref{Scale-func} for any small state $\rho$. This task will involve a series of well-defined steps. First, we  elucidate the asymptotic structure of the linearized operator for small $\varepsilon$. Following this, we will outline our strategy to conjugate  the transport part into an operator with constant coefficients. Subsequently, we will focus on reducing a truncated operator at a suitable order in $\varepsilon$ and addressing the challenge posed by the degeneracy of the spatial modes $\pm1$, employing the monodromy matrix as a tool.
Ultimately, we will construct an approximate right inverse through a perturbative argument. This comprehensive approach will enable us to effectively apply the modified Nash-Moser scheme to our problem.\\
\ding{228}{\it  \; Structure of the linearized operator.} We will see in  Proposition \ref{prop-size} that the linearized operator of $\mathcal{F}$ given by \eqref{Scale-func} 
  assumes the following  form   
\begin{align}\label{Lin-Op1}
\mathcal{L}_0[h] \triangleq\partial_\rho \mathcal{F}(\rho)[h] &=\varepsilon^2\omega_0\partial_\varphi  h +\partial_{\theta}\big[\mathcal{V}^\varepsilon(\rho)h \big]-\tfrac{1}{2}\mathcal{H}[h]-{\varepsilon^2}\mathcal{Q}_0[h](\varphi,\theta)+\varepsilon^3 \partial_\theta \mathcal{R}^\varepsilon_0(\rho)[h],
\end{align}
where the function $\mathcal{V}^\varepsilon(\rho)$ is given by
\begin{align*}
\nonumber\mathcal{V}^\varepsilon(\rho)(\varphi,\theta)&\triangleq \tfrac12-{\varepsilon^2}\omega_0\dot{\Theta}-{\varepsilon^2}\mathtt{g}(\varphi,\theta) - \tfrac{\varepsilon^{2+\mu}}{2} \rho(\varphi,\theta)+\varepsilon^{3}{V}^\varepsilon(\rho)(\varphi,\theta)
\end{align*}
and $ \mathcal{R}^\varepsilon_0$ is a smoothing operator in the space direction.\\
\ding{228}{\it \; Reduction of the transport part.} In section \ref{Sec-Red-Tran}, we will investigate the conjugation to constant coefficients the transport part of $\mathcal{L}_0$, that is 
 $$
 \mathcal{T}_0\triangleq \varepsilon^2\omega_0\partial_\varphi   +\partial_{\theta}\big[\mathcal{V}^\varepsilon(\rho)\cdot \big].
 $$  
 This  task will be carried out in two distinct stages, employing different transformations. The initial step is driven by the observation that the function $\mathcal{V}^\varepsilon$ in \eqref{def calV} contains terms of order~$\varepsilon^2$, which is not compatible with the standard  KAM scheme employed for transport part reduction, as documented in \cite{Baldi-berti,Baldi-Montalto21,BFM21,FGMP19}. This discrepancy arises due to the strong degeneracy rate of the Cantor sets described by \eqref{Cantor-set1}, which deteriorates at a rate of $\varepsilon^{2+\delta}$. In Proposition~\ref{QP-change0}, we will show that a smooth symplectic  change of coordinates system $\mathfrak{B}$ allows to get  
  \begin{align}\label{trans-M1}
\mathcal{T}_1\triangleq \mathfrak{B}^{-1}\mathcal{T}_0\mathfrak{B}=\varepsilon^2\omega_0\partial_\varphi +\partial_{\theta}\big[\big(\mathtt{c}_0+{\varepsilon^{6} }{\mathcal{V}}_0^\varepsilon({\rho}) \big)\cdot\big].
\end{align}
 Here,  $\mathtt{c}_0$ is constant with respect to the variables $\varphi,\theta$ and admits the following asymptotic behavior
 \begin{equation*}
\mathtt{c}_0= \tfrac12-\varepsilon^2\omega_0 +\varepsilon^{3}\mathtt{c}_1, \quad \| \mathtt{c}_1 \|^{\textnormal{Lip}(\lambda)}\lesssim 1.\end{equation*}
Once again, in this construction we take advantage of the time degeneracy in $\mathcal{T}_0$. We generate $\mathfrak{B}$ by solving  an elliptic equation that exclusively involves   a spatial operator, allowing to  sidestep the challenges posed by the small divisors problem stemming from the interaction between time and space.\\
The second step will be carried out in Section \ref{Section-Change2}. We will show  in Proposition \ref{QP-change} that for any given integer $n$, we can construct a symplectic change of coordinates $\mathscr{B}$ such that in the Cantor set
\begin{equation*}
{\mathcal{O}_{n}^{1}(\rho)}=\bigcap_{(\ell,j)\in\mathbb{Z}^{2}\atop 1\leqslant |j|\leqslant N_{n}}\Big\lbrace \xi_0\in (a,b);\;\, \big|\varepsilon^2\omega(\xi_0)  \ell+j \mathtt{c}(\xi_0,\rho)\big|\geqslant{\lambda}{| j|^{-\tau}}\Big\rbrace
\end{equation*}
we have 
\begin{align*}
\mathscr{B}^{-1}\mathcal{T}_1\mathscr{B}=\varepsilon^2\omega_0\partial_\varphi+ \mathtt{c}(\xi_0,\rho)\partial_{\theta}+\mathtt{E}_{n}
\end{align*}
with $\xi_0\mapsto\mathtt{c}(\xi_0,\rho)$ a function independent of the time-space variables and  conforms to  the following estimate
\begin{equation*}
\| \mathtt{c}-\mathtt{c}_0 \|^{\textnormal{Lip}(\lambda)}\lesssim {\varepsilon}^{{6}}.
\end{equation*}
Whereas $\mathtt{E}_{n}$ denotes a linear  operator satisfying the frequency decay estimate,
\begin{equation*}
\|\mathtt{E}_{n}h\|_{s_0}^{\textnormal{Lip}(\lambda)}\lesssim \varepsilon^{{6}} N_{0}^{\mu_{2}}N_{n+1}^{-\mu_{2}}\|h\|_{s_{0}+2}^{\textnormal{Lip}(\lambda)}.
\end{equation*}
The proof will be conducted  in a similar way to the recent papers \cite{BFM,FGMP19,HHM21,HR21} where we employ  KAM scheme to deal with the small divisors problem. Actually, it is a consequence of a more general statement that will be formulated in  Proposition \ref{Thm transport}.\\
\ding{228}{\it \; Invertibility of $\mathcal{L}_0$}. The main statement is described by  Proposition \ref{prop-inverse}. Now, let us   sketch some key  ideas of the proof. We define first the symplectic change of coordinates  as follows
$$
\mathscr{A}\triangleq \mathfrak{B}\mathscr{B}.
$$
Then by making use of  Proposition \ref{lemma-beta1}, we find  that $\mathscr{A}$ satisfies tame estimates and in the Cantor set ${\mathcal{O}_{n}^{1}(\rho)}$, 
\begin{align*}
\mathcal{L}_1\triangleq\mathscr{A}^{-1}\mathcal{L}_0\mathscr{A}&= \varepsilon^2\omega_0\partial_\varphi +\mathtt{c}(\xi_0,\rho)\partial_{\theta}-\tfrac{1}{2}\mathcal{H}-{\varepsilon^2 } \partial_\theta\mathcal{Q}_1 +O\big({\varepsilon^{6}\lambda ^{-1}}+\varepsilon^{2+\mu}\big)+\mathtt{E}_{n}
\end{align*}
where 
\begin{align*}
\mathcal{Q}_1[h](\varphi,\theta)&\triangleq \int_{\mathbb{T}}h(\varphi+\pi,\eta)\textnormal{Re}\Big\{ \tfrac{e^{\ii(\theta+\eta+2\varphi-2\Theta(\varphi))}}{q(\varphi)}-\tfrac{e^{\ii(\theta-\eta)}}{\big( \sqrt{q(\varphi)}\cos( \Theta(\varphi))+\ii y_0\big)^2}\Big\} d\eta \nonumber
\\
&-\tfrac12\int_{\mathbb{T}} h( \varphi, \eta) \textnormal{Re}\Big\{\tfrac{e^{\ii(\theta+\eta+2\varphi- 2 {\Theta}(\varphi))}}{q(\varphi)}+\tfrac{e^{\ii (\theta+\eta+2\varphi)}+2e^{\ii(\theta-\eta)}}{\big( \sqrt{q(\varphi)}\sin( \Theta(\varphi))+ y_0\big)^2}-\tfrac{e^{\ii (\theta+\eta+2\varphi)}}{\big( \sqrt{q(\varphi)}\cos( \Theta(\varphi))+\ii y_0\big)^2}\Big\}  d\eta .
\end{align*}
Notice that the operator $\mathcal{Q}_1$ exhibits  time variable coefficients but it has the advantage to  localize in Fourier side  on the spatial modes $\pm1.$ However, it involves a time delay operator that will have some strong  consequences on  its invertibility.
By performing a perturbative argument, inverting  $\mathcal{L}_0$ boils down to invert  the truncated operator
\begin{align*}
\mathbb{L}_1\triangleq& \varepsilon^2\omega_0\partial_\varphi +\mathtt{c}(\xi_0,\rho)\partial_{\theta}-\tfrac{1}{2}\mathcal{H}-{\varepsilon^2 } \partial_\theta\mathcal{Q}_1.
\end{align*}
However, due to the degeneracy of the spatial modes $\pm1$ we need to split the phase space into two parts as described in \eqref{Xs-space}, that is,
$$
\hbox{Lip}_\lambda(\mathcal{O},H_0^{s})=X^{s}_{\circ}\oplus X^s_{\perp},
$$
where $X^{s}_{\circ}$ is the set of functions whose Fourier decomposition  involves only the spatial \mbox{modes $\pm1.$ } As we shall investigate in Section \ref{sec-Norm76}, the restriction of $\mathbb{L}_1$ on the normal direction $X^s_{\perp}$ gives rise to an invertible operator with a loss of regularity.
The analysis of its restriction to the subspace $X^{s}_{\circ}$ will be carried out meticulously in Section \ref{sec-monodromy}. This analysis is intricately connected to certain properties of a companion monodromy matrix. We emphasize  that the analysis of the time delay  concerns only the spatial modes $\pm1$.  Its effect  on the remaining modes  scales at the size $\varepsilon^3$, which can be simply  included in the error terms  during the implementation of   the perturbative argument for inverting the linearized operator. The main  observation  driving  the perturbative argument   revolves around  the structure of the  Cantor sets which are penalized by a spatial frequency decay that can be managed with the spatial smoothing effects of  the remainders. This contrasts with  the quasi-periodic setting  where we penalize with time frequency decay. The advantage here is that we do not need to reduce the remainder, unlike the quasi-periodic case. 
\\
\ding{205} \; {\it Nash-Moser scheme and measure of the final Cantor set.}
The  construction of   solutions to  the nonlinear transport equation \eqref{Scale-func} will be carried out in Section \ref{N-M-S1}.  We basically  follow   a modified Nash-Moser scheme as in the papers \cite{Baldi-berti,BCP,BertiMontalto},  albeit with some slight simplifications in our framework. As proved  in Corollary \ref{prop-construction}, the solutions are built when the internal parameter $\xi_0$ belongs to a carefully chosen  Cantor set.\\
 We will establish a lower bound for the Lebesgue measure of the final Cantor  set, see \mbox{Lemma \ref{Lem-Cantor-measu},} where the main ingredient is the non degeneracy of the point vortex frequency $\xi_0\mapsto \omega_0(\xi_0)$ established in Lemma \ref{lem-period}. \\

  Let us  precise some notational conventions that will be frequently used in this paper.
 \paragraph{Notations.}
\ding{227} The sets of numbers that will be frequently used are denoted as follows  
		$$
		\mathbb{N}= \{0,1,2,\ldots\},\quad\mathbb{N}^*= \mathbb{N}\setminus\{0\},\quad\mathbb{Z}=\mathbb{N}\cup(-\mathbb{N}),\qquad\mathbb{Z}^*= \mathbb{Z}\setminus\{0\}, \qquad \mathbb{T}= \mathbb{R}/2\pi\mathbb{Z}.
		$$
\ding{227} The set of all $n\times n$ square  matrices with complex number entries is denoted $M_n(\mathbb{C})$.\\		
	\ding{227} We consider a list of real numbers with the constraints
	\begin{equation}\label{cond1}
 \lambda,\mu \in(0,1),\quad \tau>1, \quad S\geqslant s\geqslant s_0>3.
\end{equation}
\ding{227} We fix
	\begin{equation}\label{cond-interval}
y_0>0, \qquad 0<\xi_*<\xi^*<\tfrac{y_0}{\sqrt{2}}, \qquad \mathcal{O}= [\xi_*,\xi^*].
\end{equation}
\ding{227} For any function $h:\T\to\R$ we define its average by
		$$
		\langle h\rangle =\int_{\mathbb{T}}h(\eta)d\eta=\frac{1}{2\pi }\int_0^{2\pi}h(\eta)d\eta.
		$$ 
		\ding{227}  For any functional  $r\mapsto F(r)$  ($r$ belongs to a some set of functions), we define the difference
		$$
		\Delta_{12}F=F(r_1)-F(r_2).
		$$

	\section{Leafrogging quartets of  points vortices }\label{sec:lep-vortex-pairs}
	In this section, we aim to analyze the dynamics of four vortices with equal absolute circulation,  where the system has the real axis as a symmetry axis. While certain aspects of their dynamics have been previously explored several decades ago in \cite{Love}, our focus here will be on the conditions necessary to achieve leapfrogging. In this framework,  leapfrogging corresponds to a non-rigid time periodic  motion in a relative frame that translates. We will also delve into quantitative estimates of the period and its analytical dependence on the initial configuration that will be needed later.

We consider  four point vortices  in the complex plane, where the first pair comprises  two points  $z_1$ and $z_2$, initially  located at the upper half plane with  the same circulation $\pi$. The second pair is located at the complex conjugate numbers  $\overline{z_1}$ and $\overline{z_2}$ with the same negative  \mbox{circulation $-\pi$.}
The equations of motion of our point vortex system are
	\begin{equation}\label{Vort-leap}
	\dot{z}_k(t)=\frac{\ii}{2}\Big[\frac{1}{\overline{z_k}-\overline{z_{3-k}}}-\frac{1}{\overline{z_k}-{z_{k}}}-\frac{1}{\overline{z_k}-{z_{3-k}}}\Big], \quad k=1,2.
	\end{equation}
	Denote by 
		\begin{equation}\label{notation4pts}
		\begin{aligned}
		x_0(t)\triangleq \textnormal{Re}\{z_1(t)+z_2(t)\},\qquad y_0(t)\triangleq \textnormal{Im}\{z_1(t)+z_2(t)\},\\
		\eta(t)\triangleq \textnormal{Re}\{z_1(t)-z_2(t)\}, \qquad \xi(t)\triangleq \textnormal{Im}\{z_1(t)-z_2(t)\}. 
		\end{aligned}
		\end{equation}
Then, the system \eqref{Vort-leap} is equivalent to	
\begin{eqnarray}   \label{A04pts0}
       \left\{\begin{array}{ll}
	\dot{\eta}&=-\frac{\xi(y_0^2+\eta^2)}{(\xi^2+\eta^2)(y_0^2-\xi^2)},
	\\
	\dot{\xi}&=\frac{\eta(y_0^2-\xi^2)}{(\xi^2+\eta^2)(y_0^2+\eta^2)},
	\\
	\dot{x_0}&=y_0\big(\frac{1}{y_0^2+\eta^2}+\frac{1}{y_0^2-\xi^2}\big), 
	\\
	\dot{y_0}&=0.
	 \end{array}\right.
\end{eqnarray}
This can be expressed in terms of the Hamiltonian as
\begin{eqnarray}   \label{A04pts}
       \left\{\begin{array}{ll}
          	{\dot{\eta}}&=\partial_{\xi} H(\eta,\xi), \\
\dot{\xi}&=-\partial_{\eta} H(\eta,\xi), 
       \end{array}\right.\qquad H(\eta,\xi)=-\tfrac12\log\Big(\frac{1}{y_0^2-\xi^2}-\frac{1}{y_0^2+\eta^2}\Big).
\end{eqnarray}	
Now, assume that
$$
x_0(0)=0,\quad y_0(0)=y_0>0,\quad \eta(0)=0 \quad \textnormal{and}\quad\xi(0)=\xi_0>0.$$
 Then, from the identity  
$$
H(\eta,\xi)=H(0,\xi_0),
$$
we obtain
\begin{equation}\label{identity:Hamiltonian0}
\frac{(y_0^2-\xi^2)(y_0^2+\eta^2)}{\xi^2+\eta^2}=e^{2H(0,\xi_0)}=y_0^2\Big(\frac{y_0^2}{\xi_0^2}-1\Big)\triangleq h_0,
\end{equation}
which can also be written as
\begin{equation}\label{identity:Hamiltonian}
(\eta^2+y_0^2+h_0)(\xi^2-y_0^2+h_0)=h_0^2.
\end{equation}  
Therefore,  the orbits are contained in the set  
	\begin{equation}\label{Alg-var}
	\Big\{(\eta,\xi)\in\R^2,\quad \big(\eta^2+\tfrac{y_0^4}{\xi_0^2}\big)\big(\xi^2+\tfrac{y_0^4}{\xi_0^2}-2y_0^2\big)=y_0^4\big(\tfrac{y_0^2}{\xi_0^2}-1\big)^2\Big\},
	\end{equation}
which has a two fold structure.	
The motion is periodic provided  that the equation \eqref{identity:Hamiltonian0} admits a positive solution  $\eta$ when $\xi=0$, that is
\begin{equation}\label{eta-equation0}
\eta^2=\tfrac{\xi_0^2y_0^2}{y_0^2-2\xi_0^2}
\end{equation}
admits a positive solution. This is true if  
$$
0<\tfrac{\xi_0}{y_0}<\tfrac{\sqrt{2}}{2}.
$$
If $T=T(\xi_0)>0$ denotes the period, then at $T/4$ the point $z_1(t)-z_2(t)$ is located at the horizontal axis, that is 
	\begin{equation}\label{bound xi 0t4}
	\xi(T/4)=0\quad\hbox{and}\quad 0<\xi(t)\leqslant \xi_0, \quad\forall t\in[0,T/4).
	\end{equation}
Moreover, denoting
\begin{equation}\label{varia1}
 \alpha_0\triangleq \tfrac{\xi_0^2}{y_0^2}\quad\textnormal{and}\quad s\triangleq -\tfrac{\sqrt{1-2 \alpha_0}}{\xi_0}\eta,
\end{equation}
we obtain, from \eqref{identity:Hamiltonian0}, 
$$
\begin{aligned}
\xi^2
 &=\xi_0^2\frac{1-s^2}{1+\frac{ \alpha_0^2}{1-2 \alpha_0} s^2}\cdot
\end{aligned}
$$
It follows from the first equation of the system \eqref{A04pts0} that
\begin{align*}
\dot{s}=\frac{ (1-2 \alpha_0)}{\xi_0^2(1- \alpha_0)}\frac{\sqrt{(1-s^2)({1-2 \alpha_0}+{ \alpha_0^2}  s^2)^{3}}}{( {1-2 \alpha_0}+{ \alpha_0} s^2)^2},\qquad s(0)=0, \qquad s(T/4)=1.
\end{align*}
Integrating this equation leads to
\begin{align}\label{Period}
T( \xi_0)&=\frac{4\xi_0^2(1- \alpha_0)}{ (1-2 \alpha_0)}\bigintss_0^1\frac{( 1-2 \alpha_0+{ \alpha_0}  s^2)^2} {\sqrt{(1-s^2)(1-2 \alpha_0 +\alpha_0^2 s^2)^{3}}}ds.
\end{align}
In view if of the identity 
$$
\big( \tfrac{1-2 \alpha_0}{ \alpha_0}+  s^2\big)^2=\big(\tfrac{1-2 \alpha_0}{ \alpha_0^2}\big)^2( { \alpha_0}-1)^2+\big( \tfrac{1-2 \alpha_0}{ \alpha_0^2}+  s^2\big)^2+2\big(\tfrac{1-2 \alpha_0}{ \alpha_0^2}\big)( { \alpha_0}-1)\big(\tfrac{1-2 \alpha_0}{ \alpha_0^2}+  s^2\big),
$$
 the expression of $T( \xi_0)$ also writes 
 \begin{align*}
T( \xi_0)&=\frac{4\xi_0^2(1- \alpha_0)}{ \alpha_0(1-2 \alpha_0)}\Bigg[\bigintss_0^1\frac{\sqrt{ \frac{1-2 \alpha_0}{ \alpha_0^2}+  s^2}}{\sqrt{1-s^2}}ds
+(1-\alpha_0)^2\big(\tfrac{1-2 \alpha_0}{ \alpha_0^2}\big)^2\bigintss_0^1\frac{ ds}{\sqrt{(1-s^2)(\frac{1-2 \alpha_0}{ \alpha_0^2}+  s^2)^{3}}}\\
&\qquad\qquad\qquad -2(1-\alpha_0)\big(\tfrac{1-2 \alpha_0}{ \alpha_0^2}\big) \bigintss_0^1\frac{ ds} {\sqrt{(1-s^2)(\frac{1-2 \alpha_0}{ \alpha_0^2}+  s^2)}}\Bigg].
\end{align*}
Thus, by \cite[Formula 3 page 310]{Z2014}, \cite[Formula 5 page 288]{Z2014}, \cite[Formula 3 page 282]{Z2014} we obtain
\begin{align*}
T( \xi_0)
&=\tfrac{8\xi_0^2(1- \alpha_0)}{(1-2 \alpha_0)}\Big[\tfrac{(1- \alpha_0)^2}{ \alpha_0^2}\mathtt{E}\big(\tfrac{\alpha_0}{1-\alpha_0}\big)
-\tfrac{1-2 \alpha_0}{\alpha_0^2}\mathtt{K}\big(\tfrac{\alpha_0}{1-\alpha_0}\big) \Big],
\end{align*}
where $\alpha_0=\frac{\xi_0}{y_0}$, $\mathtt{E}$ and $\mathtt{K}$ are the complete elliptic integrals of the first
and second kind, respectively.  We point out that we break  here with prior convention and use the initial distance $\xi_0$ separating  $z_1(0)$ and $z_2(0)$  to parameterize the family of of periodic orbits, rather than using the ratio of the breadths of the vortex pairs, $\alpha=\frac{y_0-\xi_0}{y_0+\xi_0}$, as was done in \cite{Love,Aref} or the value $h_0$ of the Hamiltonian $H$ in \eqref{identity:Hamiltonian0} as was done in \cite{Behring}.
We shall denote by $\omega_0$ the frequency of the periodic orbit, that is, 
\begin{align}\label{Frequence}
	\omega_0\triangleq\,\omega(\xi_0)=\frac{2\pi}{T(\xi_0)}\cdot
		\end{align}
Next, we introduce the angle-action coordinates
	$$
	\eta(t)+\ii\xi(t)=\sqrt{\mathtt{I}(t)}e^{\ii \phi(t)}.
	$$
	In view of \eqref{A04pts0}, we get  the system
	\begin{eqnarray}   \label{A}
        \left\{\begin{array}{ll}
          	{\dot{\mathtt{I}}}+\Big[\frac{1}{y_0^2-\mathtt{I}\sin^2(\phi)}+\frac{1}{y_0^2+\mathtt{I}\cos^2(\phi)}\Big]\sin(2\phi)\mathtt{I}&=0, \\
\dot\phi-\frac{1}{\mathtt{I}}-\frac{\sin^2(\phi)}{y_0^2-\mathtt{I}\sin^2(\phi)}+\frac{\cos^2(\phi)}{y_0^2+\mathtt{I}\cos^2(\phi)}&=0,\\
\mathtt{I}(0)=\mathtt{I}_0>0,\qquad\qquad\qquad \phi(0)&=\frac\pi2.
       \end{array}\right.
\end{eqnarray}	
This system can be recast in terms of  a Hamiltonian one  as 
\begin{eqnarray}   \label{A0}
       \left\{\begin{array}{ll}
          	{\dot{\mathtt{I}}}&=\partial_{\phi} H_0(\mathtt{I},\phi), \\
\dot\phi&=-\partial_{\mathtt{I}} H_0(\mathtt{I},\phi), 
       \end{array}\right.\;\; H_0(\mathtt{I},\phi)=-\ln\big( \mathtt{I}\big)+\ln\big( y_0^2-\mathtt{I}\sin^2(\phi)\big)+\ln\big( y_0^2+\mathtt{I}\cos^2(\phi)\big).
\end{eqnarray}	
From this structure we easily derive the following  constant of motion
$$
H_0(\mathtt{I}(t),\phi(t))=H_0(\mathtt{I}_0,\tfrac\pi2).
$$
Notice that the periodicity for \eqref{A} can be read in the form
		\begin{align}\label{B-11}
		\forall t\in\R,\quad \mathtt{I}(t+T(\xi_0))=\mathtt{I}(t)\quad\hbox{and}\quad \phi(t+T(\xi_0))=\phi(t)+2\pi.
		\end{align}

The next result deals with some useful estimates.
\begin{lemma}\label{lem-period}
	Let $y_0>0$, then the following assertions hold true.
	\begin{enumerate} 
\item The two mappings $\xi_0\in(0,\frac{y_0}{\sqrt{2}})\mapsto T(\xi_0),\omega_0(\xi_0)$ are analytic. 	
\item	 The mapping $\xi_0\in(0,\frac{y_0}{\sqrt{2}})\mapsto \,\omega(\xi_0) $ is strictly decreasing. In addition, for any compact set $\mathfrak{C}\subset (0,\frac{y_0}{\sqrt{2}})$ 
$$
{\inf_{\xi_0\in \mathfrak{C}}|\omega^\prime(\xi_0)|>0.}
	$$
\item For any  $0\leqslant \alpha_0=\frac{\xi_0^2}{y_0^2}<\frac12$, we have
\begin{align*}
	2\pi \xi_0^2\leqslant T(\xi_0)
	&\leqslant 2\pi \xi_0^2 \tfrac{1}{1-2\alpha_0}\cdot
	\end{align*}
	
		\item The function $t\in\R\mapsto \phi(t)-\omega_0 t$ is $T(\xi_0)$-periodic. In addition, there exists $C>0$ such that 	$$
\tfrac{\xi_0^2}{y_0^2}<\tfrac12\Longrightarrow \sup_{t\in\R}\big|\tfrac{1}{\mathtt{I}(t)}-\tfrac{1}{\mathtt{I}_0}\big|+\sup_{t\in\R}|\dot\phi(t)-\omega_0 |\leqslant \tfrac{6 \pi}{y_0^2(1-2\alpha_0)}e^{ \tfrac{6\pi \alpha_0}{1-2\alpha_0}}.
$$
\item The functions  $\mathtt{I},\phi:(0,\frac{y_0}{\sqrt{2}})\times \R\mapsto \mathtt{I}(\xi_0,t),\phi(\xi_0,t)$ are real analytic.
	\end{enumerate}
	\end{lemma}
	\begin{proof}
		{\bf{1)}}
		Coming back to \eqref{Period}, we write  
			\begin{align}\label{periodformL1}
	T(\xi_0)		&=4 \xi_0^2 \bigintss_0^{1}\frac{K(\alpha_0,s^2)}{\sqrt{1-s^2}}ds,\quad K(\alpha_0,s)\triangleq \frac{(1- \alpha_0)( 1-2 \alpha_0+{ \alpha_0}  s)^2}{(1-2 \alpha_0)(1-2 \alpha_0 +\alpha_0^2 s)^{\frac32}}\cdot
	\end{align}
	
	Hence for any $\delta_0\in(0,1)$ small enough
\begin{align*}
\forall s\in[0,1],\quad\forall \delta_0< \alpha_0<\tfrac{1}{2},\quad  K(\alpha_0,s^2)&\geqslant \tfrac{8\delta_0}{25}.
\end{align*}
This allows  to get the analyticity of $\alpha_0\mapsto T(\xi_0)$. Invoking a composition law we deduce  that  $\xi_0\in(0,\frac{y_0}{\sqrt{2}})\mapsto T(\xi_0)$ is analytic.\\
	{\bf{2)}} From direct computations we infer
\begin{align*}
	&\partial_{ \alpha_0}  K({ \alpha_0},s^2)
	=\tfrac{( 1-2 \alpha_0+{ \alpha_0}  s^2)}{(1-2 \alpha_0)^2(1-2 \alpha_0 +\alpha_0^2 s^2)^{\frac52}}\bigg[( 1-2 \alpha_0+{ \alpha_0}  s^2)(1-2 \alpha_0 +\alpha_0^2 s^2)
	\\ &+(1-2 \alpha_0)(1- \alpha_0)\Big({2(s^2-2){(1-2 \alpha_0 +\alpha_0^2 s^2)+3(1-\alpha_0}s^2)( 1-2 \alpha_0+{ \alpha_0}  s^2)}\Big)\bigg].
	\end{align*}
That is
\begin{align*}
	\partial_{ \alpha_0}  K({ \alpha_0},s^2)
	&=\tfrac{( 1-2 \alpha_0+{ \alpha_0}  s^2)}{(1-2 \alpha_0)^2(1-2 \alpha_0 +\alpha_0^2 s^2)^{\frac52}}\bigg[2s^2( 1-2 \alpha_0)^3+( 1-2 \alpha_0)^2\Big((3\alpha_0-1)\alpha_0s^2+4\alpha_0\Big)
	\\ &+(1-2 \alpha_0)\alpha_0  s^2\Big(4(1-2\alpha_0)+\alpha_0(2-s^2)+\alpha_0^2s^2+4 \alpha_0^2 \Big)+\alpha_0^3 s^4\bigg].
	\end{align*}
Thus, we conclude that
	\begin{align}\label{mima-r1}
	\forall s\in(0,1), \quad \forall \alpha_0\in(0,\tfrac12),\quad \partial_{\alpha_0} K(\alpha_0,s^2)>0.
	\end{align}
	It follows,  from \eqref{periodformL1}, 
 the chain rule and \eqref{varia1}, that
	\begin{align}\label{diff-per2}
	\forall \xi_0\in\big(0,\tfrac{y_0}{\sqrt{2}}\big),\quad \partial_{\xi_0} T>0.
	\end{align}
	Hence for any compact $\mathfrak{C}\subset  (0,\frac{y_0}{\sqrt{2}})$, we have
$$
{\inf_{\xi_0\in \mathfrak{C}}\partial_{\xi_0} T(\xi_0)>0.}
	$$
Since $T(\xi_0)>0$, then
$$
{\inf_{\xi_0\in K}|\omega^\prime(\xi_0)|>0.}
	$$
{\bf{3)}}	From \eqref{mima-r1}, one has 
\begin{equation}\label{inqk low}
\forall s\in[0,1],\quad \forall \alpha_0\in(0,\tfrac{1}{2}), \quad 1\leqslant K(\alpha_0,s^2)
\end{equation}
On the other hand, differentiating  $K(\alpha_0,s)$ with respect to $s$ gives
\begin{align*}
\partial_{s} K(\alpha_0,s)&=\frac{\alpha_0(1- \alpha_0)( 1-2 \alpha_0+{ \alpha_0}  s)}{(1-2 \alpha_0)(1-2 \alpha_0 +\alpha_0^2 s)^{\frac32}}\Big[2-\tfrac32\alpha_0( 1-2 \alpha_0+{ \alpha_0}  s)(1-2 \alpha_0 +\alpha_0^2 s)\Big]\\ 
&\geqslant \frac{\alpha_0(1- \alpha_0)( 1-2 \alpha_0+{ \alpha_0}  s)}{(1-2 \alpha_0)(1-2 \alpha_0 +\alpha_0^2 s)^{\frac32}}\Big[2-\tfrac{45}{32}\Big]>0.
\end{align*}
Thus, for all  $\alpha_0\in(0,\tfrac{1}{2})$, the function   $s\in[0,1]\mapsto K(\alpha_0,s)$ is increasing  and one has
\begin{equation}\label{inqk high}
\forall s\in[0,1],\quad \forall \alpha_0\in(0,\tfrac{1}{2}), \quad K(\alpha_0,s^2)\leqslant \frac{1}{1-2 \alpha_0}\cdot
\end{equation}
In view of \eqref{inqk low}, \eqref{inqk high} and  \eqref{periodformL1} we find that
\begin{align}\label{est T}
\forall \alpha_0\in(0,\tfrac{1}{2}), \quad	2\pi \xi_0^2\leqslant T(\xi_0)\leqslant \frac{2\pi \xi_0^2}{1-2 \alpha_0}
	,
		\end{align}
having used the classical identity
$$
\int_0^1\frac{ds}{\sqrt{1-s^2}}=\frac\pi2.
$$	
	Finally, we can check from \eqref{est T}   that
	$$
	\lim_{\xi_0\mapsto0}\frac{T(\xi_0)}{\xi_0^2}=2\pi.
	$$
		{\bf{4)}} Integrating the second equation of \eqref{A} on one period we infer 
	$$         2\pi= 	\int_0^{T(\xi_0)}\dot\phi(t) dt=\int_0^{T(\xi_0)}\Big(\tfrac{1}{\mathtt{I}(t)}+\tfrac{\sin^2(\phi)}{y_0^2-\mathtt{I}\sin^2(\phi)}-\tfrac{\cos^2(\phi)}{y_0^2+\mathtt{I}\cos^2(\phi)}\Big)dt.
       $$
       Therefore,
       $$
       \omega_0(\xi_0)=\frac{1}{T(\xi_0)}\int_0^{T(\xi_0)}\Big(\tfrac{1}{\mathtt{I}(t)}+\tfrac{\sin^2(\phi)}{y_0^2-\mathtt{I}\sin^2(\phi)}-\tfrac{\cos^2(\phi)}{y_0^2+\mathtt{I}\cos^2(\phi)}\Big)dt
       $$
       and
       \begin{align}\label{diff-1}
      \dot\phi(t)-\omega_0=\tfrac{1}{\mathtt{I}}-\langle \tfrac{1}{\mathtt{I}}\rangle+\tfrac{\sin^2(\phi)}{y_0^2-\mathtt{I}\sin^2(\phi)}-\tfrac{\cos^2(\phi)}{y_0^2+\mathtt{I}\cos^2(\phi)}-\Big\langle \tfrac{\sin^2(\phi)}{y_0^2-\mathtt{I}\sin^2(\phi)}-\tfrac{\cos^2(\phi)}{y_0^2+\mathtt{I}\cos^2(\phi)}\Big\rangle,
     \end{align}
       where $\langle f\rangle $ denotes the average of $f$, that is,
       $$
       \langle f\rangle =\frac{1}{T(\xi_0)}\int_0^{T(\xi_0)}f(t)dt.
       $$
       Hence combining \eqref{diff-1} with \eqref{B-11} we deduce that $\dot\phi(t)-\omega_0$ is $T(\xi_0)-$periodic.
       Integrating the first equation of \eqref{A} yields
       \begin{align}\label{II-12}
          	\mathtt{I}(t)=\mathtt{I}_0e^{-\int_0^tg(s)\sin(2\phi(s))ds}, \quad g(s)\triangleq \Big[\tfrac{1}{y_0^2-\mathtt{I}(s)\sin^2(\phi(s))}+\tfrac{1}{y_0^2+\mathtt{I}(s)\cos^2(\phi(s))}\Big] .
	\end{align}	
	Therefore,
	$$
	\tfrac{1}{\mathtt{I}(t)}-\tfrac{1}{\mathtt{I}_0}=\tfrac{1}{\mathtt{I}_0}\Big(e^{\int_0^tg(s)\sin(2\phi(s))ds} -1\Big).
	$$
	From \eqref{bound xi 0t4}  one easily deduces,  by a symmetry argument, that 
	\begin{equation}\label{bound xi}
	\forall t\in\R,\quad | \xi(t)|\leqslant \xi_0
	\end{equation} 
	for all $\xi_0\in(0,\frac{y_0}{\sqrt{2}})$.
	It follows that
	\begin{equation}\label{bound yxi}
	\forall t\in\R,\quad \tfrac{y_0^2}{2}\leqslant y_0^2-\xi^2(t)=y_0^2-\mathtt{I}(t)\sin^2(\phi(s))\leqslant y_0^2.
	\end{equation}
Consequently,
$$
\forall t\in\R,\quad 0\leqslant g(s)\leqslant \tfrac{3}{y_0^2}.
$$
	Since $ \xi_0^2<\tfrac{y_0^2}{2}$, then using the estimate \eqref{est T}
      \begin{align}\label{Def-nn}
   \int_0^{T(\xi_0)}g(s)|\sin(2\phi(s))|ds&\leqslant \tfrac{3}{y_0^2} T(\xi_0)\leqslant  \tfrac{6\pi \xi_0^2 }{y_0^2(1-2\alpha_0)}.
      \end{align}
	It follows that
	  \begin{align}\label{est-radius}
	\sup_{t\in\R}\big|\tfrac{1}{\mathtt{I}(t)}-\tfrac{1}{\mathtt{I}_0}\big|
	&\leqslant   \tfrac{6 \pi}{y_0^2(1-2\alpha_0)}e^{ \tfrac{6\pi \alpha_0}{1-2\alpha_0}}.
	\end{align}
       Plugging this estimate into \eqref{diff-1} and using \eqref{bound yxi} yield
       \begin{align}\label{diff-2}
       \sup_{t\in\R}\big|\dot\phi(t)-\omega_0\big|\leqslant \tfrac{6 \pi}{y_0^2(1-2\alpha_0)}e^{ \tfrac{6\pi \alpha_0}{1-2\alpha_0}}.
     \end{align}
       {\bf{5)}} To  get the analyticity of the solutions to the ODE \eqref{A} it suffices to apply Cauchy-Kowaleski theorem on the analyticity of the solutions to the ODE, whose coefficients  are analytic.
	This achieves the proof of the lemma.
	\end{proof}
	\section{Contour dynamics equation and approximate solutions}\label{SEc-Lin-Ap}
We plan in this section to explore  some specific aspects concerning the dynamics of quartets of patches within the Euler equations \eqref{Euler-Eq}. Initially,  we will derive  the equations governing the leapfrogging with symmetric concentrated patches. Subsequently, we will examine some algebraic and analytical properties of the linearized operator associated with these equations. Finally, we will construct a good  approximate  solution to the leapfrogging, which  will serve as the  initial state for   Nash-Moser scheme.
 \subsection{Boundary equation}\label{Bound-motion}
 In this section, our aim is to formulate the contour dynamics equations that govern the motion of quartets of  concentrated vortex patches.  To achieve this, let's introduce $z_j(t)$, where $j=1,2$, as the core positions of two vortices, which satisfy the point vortex system given by equation \eqref{Vort-leap} and assume that $\overline{z_j(t)}$, $j=1,2$ are the core positions of the other two vortices. We will restrict the  discussion to the case where the orbits of the point vortices are time periodic as was previously analyzed  in Section \ref{sec:lep-vortex-pairs}.\\
For the sake of simplicity and without loss of generality, we can assume that the initial configuration is vertically aligned,  
as follows, 
$$
\textnormal{Re}\{z_1(0)\}=\textnormal{Re}\{z_2(0)\}=0.
$$
 According to Section \ref{sec:lep-vortex-pairs}, the dynamics of the point vortices is described through
\begin{equation}\label{pt vx zj}
z_k(t)=\tfrac{(-1)^{k+1}}{2}\big(\eta(t)+\ii \xi(t)\big)+\tfrac{1}{2}\big(x_0(t)+\ii y_0\big),\,\, k=1,2,
\end{equation}
where $y_0\triangleq\textnormal{Im}\{z_1(0)+z_2(0)\}$, $x_0(t)=\textnormal{Re}\{z_1(t)+z_2(t)\}$ and the coordinates $(\eta(t),\xi(t))$ solve the \mbox{system \eqref{A04pts0}.} 
Given $\varepsilon\in(0,1)$, we shall look for solutions in the form
\begin{equation}\label{omegateps}
\omega_{\varepsilon}(t)=\tfrac{1}{\varepsilon^2}{\bf{1}}_{{D}_{t,1}^\varepsilon}+\tfrac{1}{\varepsilon^2}{\bf{1}}_{{D}_{t,2}^{\varepsilon}}-\tfrac{1}{\varepsilon^2}{\bf{1}}_{{{D}_{t,3}^\varepsilon}}-\tfrac{1}{\varepsilon^2}{\bf{1}}_{{{D}_{t,4}^{\varepsilon}}}
\end{equation}
where $D_{t,1}^{\varepsilon}$, $D_{t,2}^{\varepsilon}$, $D_{t,3}^{\varepsilon}$ and $D_{t,4}^{\varepsilon}$ are given by  
\begin{equation}\label{domains0}
\begin{aligned}
{D}_{t,k}^\varepsilon &\triangleq\,   \varepsilon  {O}_{t,k}^\varepsilon+z_k(t), \\
{D}_{t,2+k}^\varepsilon &\triangleq\,   \varepsilon  {O}_{t,2+k}^\varepsilon+\overline{z_k(t)},
\end{aligned}
\qquad k=1,2,
\end{equation}
and  the domains  ${O}_{t,1}^\varepsilon$, ${O}_{t,2}^\varepsilon$, ${O}_{t,3}^\varepsilon$ and ${O}_{t,4}^\varepsilon$  represent  simply connected regions localized around the unit disc {\it with the same volume}. In this particular case the total  stream function takes the form
	\begin{equation}\label{def stream}
	\begin{split}
		\psi(t,z)&=\frac{1}{2\pi\varepsilon^2}\sum_{k=1}^2\bigg(\int_{D_{t,k}^{\varepsilon}}\log(|z-\zeta|)dA(\zeta)
		-\frac{1}{2\pi\varepsilon^2}\int_{{{D}_{t,2+k}^\varepsilon}}\log(|z-\zeta|)dA(\zeta)\bigg).
	\end{split}	
	\end{equation}
For each $k\in\{1,2,3,4 \}$, we consider a positively oriented parametrization $\gamma_k(t): \mathbb{T} \mapsto \partial {O}_{t,k}^\varepsilon$ of the boundary of the domain ${O}_{t,k}^\varepsilon$. We then define
\begin{equation}\label{param Dk}
w_k(t)\triangleq\varepsilon  \gamma_k(t)+z_k(t)
\end{equation}	
the corresponding  parametrization of the boundary $\partial {D}_{t,k}^\varepsilon$. It is  worth noting that in the contour dynamics reformulation, as discussed, for instance, in \cite[p. 174]{HMV}, the vortex patch equation can be expressed in the form,	
	\begin{equation}\label{vortex patches equation0}
		\forall k\in\{1,2\},\quad\partial_{t}w_k(t,\theta)\cdot \mathbf{n}_k(t,w_k(t,\theta))=v(t,w_k(t,\theta))\cdot \mathbf{n}_k(t,w_k(t,\theta)),
		\end{equation}	
		where  $\mathbf{n}_k(t,\cdot)$ refers to a normal vector to the boundary $\partial D_{t,k}$ of the patch. By identifying $\mathbb{R}^{2}$ to the  complex plane $\mathbb{C}$,  making the choice  $\mathbf{n}_k(t,w_k(t,\theta))=-\ii\partial_{\theta}w_k(t,\theta)$ and using the identity $v=2\ii \partial_{\overline{z}}\psi$, we get
		\begin{align*}
		 v(t,w_k(t,\theta))\cdot \mathbf{n}_k(t,w_k(t,\theta))&=\textnormal{Re}\big\{v(t,w_k(t,\theta))\ii\partial_{\theta}\overline{w_k(t,\theta)}\big\}\\&=-\textnormal{Re}\big\{2(\partial_{\overline{z}}\psi)(t,w_k(t,\theta))\partial_{\theta}\overline{w_k(t,\theta)}\big\}
		  \\&=-\partial_\theta\big[\psi(t,w_k(t,\theta))\big].
		\end{align*}
		Moreover, one has 
			\begin{align*}
		\partial_{t}w_k(t,\theta)\cdot \mathbf{n}_k(t,w_k(t,\theta))&=\textnormal{Im}\big\{\partial_{t}\overline{w_k(t,\theta)}\partial_{\theta}w_k(t,\theta)\big\}.
		\end{align*}
Plugging the last two identities into \eqref{vortex patches equation0} yield 
	\begin{equation}\label{vortex patches equation00}
\forall k\in\{1,2,3,4\},\quad G_k(t,\theta)\triangleq\textnormal{Im}\big\{ \partial_{t}\overline{w_k(t,\theta)}\partial_{\theta}  w_k(t,\theta)\big\}+\partial_{\theta}\big[\psi(t,w_k(t,\theta))\big]=0,
\end{equation}	
where, in view of \eqref{def stream},
		\begin{align*}
			\partial_{\theta}\big[\psi(t,w_k(t,\theta))\big]
			&=\frac{1}{2\pi \varepsilon^2}\partial_{\theta}\bigg[\int_{D_{t,1}^{\varepsilon}}\log(|w_k(t,\theta)-\zeta|)dA(\zeta)+\int_{D_{t,2}^{\varepsilon}}\log(|w_k(t,\theta)-\zeta|)dA(\zeta)\bigg]\\ &\quad -\frac{1}{2\pi \varepsilon^2}\partial_{\theta}\bigg[\int_{{D_{t,3}^{\varepsilon}}}\log(|w_k(t,\theta)-\zeta|)dA(\zeta)+\int_{{D_{t,4}^{\varepsilon}}}\log(|w_k(t,\theta)-\zeta|)dA(\zeta)\bigg].
							\end{align*}
Observe that by assuming that
$$
O_{t,2+k}^\varepsilon=\overline{O_{t,k}^\varepsilon},\quad  k=1,2, 
$$
that is $D_{t,2+k}^\varepsilon=\overline{D_{t,k}^\varepsilon}$,  and 
$$
w_{k+2}(t,\theta)=\overline{w_k(t,\theta)},\quad  k=1,2,
$$
we conclude, by  change of variables, that
		\begin{align*}
			\partial_{\theta}\big[\psi(t,w_{3}(t,\theta))\big]
			&=\frac{1}{2\pi \varepsilon^2}\partial_{\theta}\bigg[\int_{D_{t,1}^{\varepsilon}}\log(|\overline{w_1(t,\theta)}-\zeta|)dA(\zeta)+\int_{D_{t,2}^{\varepsilon}}\log(|\overline{w_1(t,\theta)}-\zeta|)dA(\zeta)\bigg]\\ &\quad -\frac{1}{2\pi \varepsilon^2}\partial_{\theta}\bigg[\int_{\overline{D_{t,1}^{\varepsilon}}}\log(|\overline{w_1(t,\theta)}-\zeta|)dA(\zeta)+\int_{\overline{D_{t,2}^{\varepsilon}}}\log(|\overline{w_1(t,\theta)}-\zeta|)dA(\zeta)\bigg]\\
			&=\frac{1}{2\pi \varepsilon^2}\partial_{\theta}\bigg[\int_{\overline{D_{t,1}^{\varepsilon}}}\log(|{w_1(t,\theta)}-\zeta|)dA(\zeta)+\int_{\overline{D_{t,2}^{\varepsilon}}}\log(|{w_1(t,\theta)}-\zeta|)dA(\zeta)\bigg]\\ &\quad -\frac{1}{2\pi \varepsilon^2}\partial_{\theta}\bigg[\int_{{D_{t,1}^{\varepsilon}}}\log(|{w_1(t,\theta)}-\zeta|)dA(\zeta)+\int_{{D_{t,2}^{\varepsilon}}}\log(|{w_1(t,\theta)}-\zeta|)dA(\zeta)\bigg]\\ &=-\partial_{\theta}\big[\psi(t,w_{1}(t,\theta))\big].
							\end{align*}
Similarly, we get 		
		\begin{align*}
			\partial_{\theta}\big[\psi(t,w_{4}(t,\theta))\big]=-\partial_{\theta}\big[\psi(t,w_{2}(t,\theta))\big].
							\end{align*}
Moreover,		one has
$$
\textnormal{Im}\Big\{ \partial_{t}\overline{w_{k+2}(t,\theta)}\partial_{\theta}  w_{k+2}(t,\theta)\Big\}=-\textnormal{Im}\Big\{ \partial_{t}\overline{w_k(t,\theta)}\partial_{\theta}  w_k(t,\theta)\Big\}, \quad k=1,2
$$					
Thus, by \eqref{vortex patches equation}, we find
$$
G_{k+2}(t,\theta)=-G_{k}(t,\theta), \quad k=1,2
$$	
and 	the problem might be reduced to the system of two scalar equations 
	\begin{equation}\label{vortex patches equation}
\forall k\in\{1,2\},\qquad G_k(t,\theta)=0.
\end{equation}						
By  \eqref{pt vx zj}, suitable change of variables, we infer 
		\begin{align*}
			\partial_{\theta}\big[\psi(t,w_k(t,\theta))\big]&=\frac{1}{2\pi }\partial_{\theta}\bigg[\int_{O_{t,k}^{\varepsilon}}\log(|  \gamma_k(t)-  \zeta|)dA(\zeta)\\ &\quad +\int_{O_{t,3-k}^{\varepsilon}}\log(|\varepsilon  \gamma_k(t)+z_k(t)-z_{3-k}(t)-\varepsilon \zeta|)dA(\zeta)
			\\ &\quad-\int_{\overline{{O}_{t,k}^\varepsilon}}\log(|\varepsilon  \gamma_k(t)+z_k(t)-\overline{z_k(t)}-\varepsilon  \zeta|)dA(\zeta)\\ &\quad-\int_{\overline{{O}_{t,3-k}^\varepsilon}}\log(|\varepsilon  \gamma_k(t)+z_k(t)-\overline{z_{3-k}(t)}-\varepsilon  \zeta|)dA(\zeta)\bigg].\nonumber
					\end{align*}
This can also be written as 
		\begin{align*}
			\partial_{\theta}\big[\psi(t,w_k(t,\theta))\big]&=\frac{1}{2\pi}\partial_{\theta}\bigg[\int_{O_{t,k}^{\varepsilon}}\log(|  \gamma_k(t)-  \zeta|)dA(\zeta)+\int_{O_{t,3-k}^{\varepsilon}}\log(|1+\varepsilon  \tfrac{\gamma_k(t)-\zeta}{z_k(t)-z_{3-k}(t)})dA(\zeta)
			\\ &\quad-\int_{\overline{{O}_{t,k}^\varepsilon}}\log(|1+\varepsilon  \tfrac{\gamma_k(t)-\zeta}{z_k(t)-\overline{z_k(t)}}|)dA(\zeta)-\int_{\overline{{O}_{t,3-k}^\varepsilon}}\log(|1+\varepsilon  \tfrac{\gamma_k(t)-\zeta}{z_k(t)-\overline{z_{3-k}(t)}}|)dA(\zeta)\bigg].\nonumber
					\end{align*}
Observe that
\begin{align*}
\partial_\theta\bigg[\frac{1}{2\pi} \int_{O_{t,3-k}^{\varepsilon}}\textnormal{Re}\Big\{\tfrac{\gamma_k(t,\theta)-\zeta}{z_k(t)-z_{3-k}(t)}\Big\}dA(\zeta)\bigg]&=\partial_\theta\bigg[ \textnormal{Re}\Big\{\tfrac{\gamma_k(t,\theta)}{z_k(t)-z_{3-k}(t)}\Big\}\frac{1}{2\pi}\int_{O_{t,3-k}^{\varepsilon}}dA(\zeta)\bigg]\\
 &= -{\tfrac12}\textnormal{Re}\Big\{\tfrac{{\partial_\theta}\gamma_k(t,\theta)}{z_k(t)-z_{3-k}(t)}\Big\}.
\end{align*}
where we have used the following identity, since the area of $O_{t,3-k}^{\varepsilon}$ is normalized  to $\pi$
\begin{align}\label{Volune-const}
\int_{O_{t,3-k}^{\varepsilon}}dA(\zeta)= | O_{t,3-k}^{\varepsilon}|=\pi. 
\end{align}
Similarly, we have
\begin{align*}
&\partial_\theta\bigg[\frac{1}{2\pi}  \int_{\overline{O_{t,k}^{\varepsilon}}}\textnormal{Re}\Big\{\tfrac{\gamma_k(t)-\zeta}{z_k(t)-\overline{z_{3-k}(t)}}\Big\}dA(\zeta)\bigg]= -{\tfrac12}\textnormal{Re}\Big\{\tfrac{{\partial_\theta}\gamma_k(t,\theta)}{z_k(t)-\overline{z_{k}(t)}}\Big\},\\
&\partial_\theta\bigg[\frac{1}{2\pi}  \int_{\overline{O_{t,3-k}^{\varepsilon}}}\textnormal{Re}\Big\{\tfrac{\gamma_k(t)-\zeta}{z_k(t)-\overline{z_{3-k}(t)}}\Big\}dA(\zeta)\bigg]= -{\tfrac12}\textnormal{Re}\Big\{\tfrac{{\partial_\theta}\gamma_k(t,\theta)}{z_k(t)-\overline{z_{3-k}(t)}}\Big\}.
\end{align*}
Therefore, from the previous computations, we conclude that
		\begin{align*}
		\partial_{\theta}\big[\psi(t,w_k(t,\theta))\big]
			&= \frac\varepsilon2 \textnormal{Re}\Big\{\Big[\tfrac{1}{{z_{k}(t)}-{z_{3-k}(t)}}-\tfrac{1}{{z_{k}(t)}-\overline{z_{3-k}(t)}}-\tfrac{1}{{z_{k}(t)}-\overline{z_{k}(t)}}\Big]\partial_{\theta}  \gamma_k(t,\theta)\Big\}
			\\ &\quad+\frac{1}{2\pi }\partial_{\theta}\bigg[\int_{O_{t,k}^{\varepsilon}}\log(|  \gamma_k(t)-  \zeta|)dA(\zeta)\\ &\quad +\int_{O_{t,3-k}^{\varepsilon}}\Big(\log\big(\big|1+\varepsilon \tfrac{  \gamma_k(t,\theta)-\zeta}{z_k(t)-z_{3-k}(t)}\big|\big)-\varepsilon \textnormal{Re}\big\{\tfrac{  \gamma_k(t,\theta)-\zeta}{z_k(t)-z_{3-k}(t)}\big\}\Big)dA(\zeta) \\ &\quad-\int_{\overline{O_{t,k}^{\varepsilon}}}\Big(\log\big(\big|1+\varepsilon \tfrac{  \gamma_k(t,\theta)-\zeta}{z_k(t)-\overline{z_k(t)}}\big|\big)-\varepsilon \textnormal{Re}\big\{\tfrac{  \gamma_k(t,\theta)-\zeta}{z_k(t)-\overline{z_k(t)}}\big\}\Big)dA(\zeta)
\\ &\quad -\int_{\overline{O_{t,3-k}^{\varepsilon}}}\Big(\log\big(\big|1+\varepsilon \tfrac{  \gamma_k(t,\theta)-\zeta}{z_k(t)-\overline{z_{3-k}(t)}}\big|\big)-\varepsilon \textnormal{Re}\big\{\tfrac{  \gamma_k(t,\theta)-\zeta}{z_k(t)-\overline{z_{3-k}(t)}}\big\}\Big)dA(\zeta)\bigg].					\end{align*}
On the other hand, from \eqref{param Dk} and \eqref{Vort-leap} one has 
$$
\begin{aligned}
\textnormal{Im}\Big\{ \partial_{t}\overline{w_k(t,\theta)}\partial_{\theta}  w_k(t,\theta)\Big\}
&=\varepsilon\textnormal{Im}\Big\{ \partial_{t}\overline{\dot{z}_k(t)}\partial_{\theta}  \gamma_k(t,\theta)\Big\}+\varepsilon^2\textnormal{Im}\Big\{ \partial_{t}\overline{\gamma_k(t,\theta)}\partial_{\theta}  \gamma_k(t,\theta)\Big\}\\
&=-\frac\varepsilon2 \textnormal{Re}\Big\{\Big[\tfrac{1}{{z_{k}(t)}-{z_{3-k}(t)}}-\tfrac{1}{{z_{k}(t)}-\overline{z_{3-k}(t)}}-\tfrac{1}{{z_{k}(t)}-\overline{z_{k}(t)}}\Big]\partial_{\theta}  \gamma_k(t,\theta)\Big\}\\ &\quad+\varepsilon^2\textnormal{Im}\Big\{ \partial_{t}\overline{\gamma_k(t,\theta)}\partial_{\theta}  \gamma_k(t,\theta)\Big\}.
\end{aligned}
$$
Inserting the two last identities into \eqref{vortex patches equation}  gives, for all $k\in\{1,2\}$, 
		\begin{align}\label{new vortex patches system}
&G_k(t,\theta)= \textnormal{Im}\Big\{ \varepsilon^2\partial_{t}\overline{\gamma_k(t,\theta)}\partial_{\theta}  \gamma_k(t,\theta)\Big\}+\frac{1}{2\pi }\partial_{\theta}\bigg[\int_{O_{t,k}^{\varepsilon}}\log(|  \gamma_k(t)-  \zeta|)dA(\zeta)
\\
 &+\int_{O_{t,3-k}^{\varepsilon}}\Big(\log\big(\big|1+(-1)^{k+1}\varepsilon \tfrac{  \gamma_k(t,\theta)-\zeta}{\eta(t)+\ii \xi(t)}\big|\big)-(-1)^{k+1}\varepsilon \textnormal{Re}\Big\{\tfrac{  \gamma_k(t,\theta)-\zeta}{\eta(t)+\ii \xi(t)}\Big\}\Big)dA(\zeta)\nonumber
\\ 
& -\int_{\overline{O_{t,k}^{\varepsilon}}}\Big(\log\big(\big|1+\varepsilon \tfrac{  \gamma_k(t,\theta)-\zeta}{(-1)^{k+1}\ii \xi(t)+\ii y_0}\big|\big)-\varepsilon \textnormal{Re}\Big\{\tfrac{  \gamma_k(t,\theta)-\zeta}{(-1)^{k+1}\ii \xi(t)+\ii y_0}\Big\}\Big)dA(\zeta)\nonumber
\\ & -\int_{\overline{O_{t,3-k}^{\varepsilon}}}\Big(\log\big(\big|1+\varepsilon \tfrac{  \gamma_k(t,\theta)-\zeta}{(-1)^{k+1}\eta(t)+\ii y_0 }\big|\big)-\varepsilon \textnormal{Re}\Big\{\tfrac{  \gamma_k(t,\theta)-\zeta}{(-1)^{k+1}\eta(t)+\ii y_0}\Big\}\Big)dA(\zeta)\bigg]=0,\nonumber
		\end{align}
where we have used the identities		
$$
\begin{aligned}
z_k(t)-z_{3-k}(t)&=(-1)^{k+1}\big(\eta(t)+\ii \xi(t)\big),\\ \quad z_k(t)-\overline{z_{k}(t)}&=(-1)^{k+1}\ii \xi(t)+\ii y_0,\\ \quad z_k(t)-\overline{z_{3-k}(t)}&=(-1)^{k+1}\eta(t)+\ii y_0, 
\end{aligned}
$$
that immediately follow from \eqref{pt vx zj}.

As a reminder from Section \ref{sec:lep-vortex-pairs},  it is  important to note that under the constraint $\frac{\xi_0^2}{y_0^2}<\frac{1}{2}$ the motion of the quartet of point vortices  in the translating frame is periodic with the period $T(\xi_0)$ given by \eqref{Period} and frequency $\omega_0=\frac{2\pi}{T(\xi_0)}$.		
From now on, we  fixe $y_0>0$, assume that $\xi_0\in(0,\frac{y_0}{\sqrt{2}})$ and  we shall focus on the periodic symmetric case where  $O_{t,2}^{\varepsilon}=O_{t+\frac{T(\xi_0)}{2},1}^{\varepsilon}$, $O_{t+T(\xi_0),1}^{\varepsilon}=O_{t,1}^{\varepsilon}$  and  
$$\gamma_2(t)=\gamma_1\big(t+\tfrac{T(\xi_0)}{2}\big)\quad\textnormal{and}\quad\gamma_1(t+T(\xi_0))=\gamma_1(t).$$
 Under this constraint, and using the fact that 
   $$\eta\big(t+\tfrac{T(\xi_0)}{2}\big)=-\eta(t), \quad  \xi\big(t+\tfrac{T(\xi_0)}{2}\big)=-\xi(t), \quad \eta\big(t+{T(\xi_0)}\big)=\eta(t), \quad  \xi\big(t+{T(\xi_0)}\big)=\xi(t),$$
which follows from the fact that the orbits in  \eqref{Alg-var} are two fold and the associated vector field is autonomous, one can easily check that 
 		\begin{align*}
&G_2\big(t+\tfrac{T(\xi_0)}{2},\theta\big)=\textnormal{Im}\Big\{ \varepsilon^2\partial_{t}\overline{\gamma_2\big(t+\tfrac{T(\xi_0)}{2},\theta\big)}\partial_{\theta}  \gamma_2\big(t+\tfrac{T(\xi_0)}{2},\theta\big)\Big\}\\ &+\frac{1}{2\pi }\partial_{\theta}\bigg[\int_{O_{t+\frac{T(\xi_0)}{2},2}^{\varepsilon}}\log(|  \gamma_2\big(t+\tfrac{T(\xi_0)}{2},\theta\big)-  \zeta|)dA(\zeta)
\\
 &+\int_{O_{t+\frac{T(\xi_0)}{2},1}^{\varepsilon}}\Big(\log\big(\big|1-\varepsilon \tfrac{  \gamma_2\big(t+\frac{T(\xi_0)}{2},\theta\big)-\zeta}{\eta\big(t+\frac{T(\xi_0)}{2}\big)+\ii \xi\big(t+\tfrac{T(\xi_0)}{2}\big)}\big|\big)+\varepsilon \textnormal{Re}\Big\{\tfrac{  \gamma_2\big(t+\frac{T(\xi_0)}{2},\theta\big)-\zeta}{\eta\big(t+\frac{T(\xi_0)}{2}\big)+\ii \xi\big(t+\frac{T(\xi_0)}{2}\big)}\Big\}\Big)dA(\zeta)\nonumber
\\ 
& -\int_{\overline{O_{t+\frac{T(\xi_0)}{2},2}^{\varepsilon}}}\Big(\log\big(\big|1+\varepsilon \tfrac{  \gamma_2\big(t+\frac{T(\xi_0)}{2},\theta\big)-\zeta}{-\ii \xi\big(t+\frac{T(\xi_0)}{2}\big)+\ii y_0}\big|\big)-\varepsilon \textnormal{Re}\Big\{\tfrac{  \gamma_2\big(t+\frac{T(\xi_0)}{2},\theta\big)-\zeta}{-\ii \xi\big(t+\frac{T(\xi_0)}{2}\big)+\ii y_0}\Big\}\Big)dA(\zeta)\nonumber
\\ & -\int_{\overline{O_{t+\frac{T(\xi_0)}{2},1}^{\varepsilon}}}\Big(\log\big(\big|1+\varepsilon \tfrac{  \gamma_2\big(t+\frac{T(\xi_0)}{2},\theta\big)-\zeta}{-\eta\big(t+\frac{T(\xi_0)}{2}\big)+\ii y_0 }\big|\big)-\varepsilon \textnormal{Re}\Big\{\tfrac{  \gamma_2\big(t+\frac{T(\xi_0)}{2},\theta\big)-\zeta}{-\eta\big(t+\frac{T(\xi_0)}{2}\big)+\ii y_0}\Big\}\Big)dA(\zeta)\bigg].\nonumber
		\end{align*}
Thus,
		\begin{align*}
G_2\big(t+\tfrac{T(\xi_0)}{2},\theta\big)&= \textnormal{Im}\Big\{ \varepsilon^2\partial_{t}\overline{\gamma_1(t,\theta)}\partial_{\theta}  \gamma_1(t,\theta)\Big\}+\frac{1}{2\pi }\partial_{\theta}\bigg[\int_{O_{t,1}^{\varepsilon}}\log(|  \gamma_1(t,\theta)-  \zeta|)dA(\zeta)
\\
 &\quad+\int_{O_{t,2}^{\varepsilon}}\Big(\log\big(\big|1+\varepsilon \tfrac{ \gamma_1(t,\theta)-\zeta}{\eta(t)+\ii \xi(t)}\big|\big)-\varepsilon \textnormal{Re}\Big\{\tfrac{ \gamma_1(t,\theta)-\zeta}{\eta(t)+\ii \xi(t)}\Big\}\Big)dA(\zeta)\nonumber
\\ 
& \quad -\int_{\overline{O_{t,1}^{\varepsilon}}}\Big(\log\big(\big|1+\varepsilon \tfrac{ \gamma_1(t,\theta)-\zeta}{\ii \xi(t)+\ii y_0}\big|\big)-\varepsilon \textnormal{Re}\Big\{\tfrac{ \gamma_1(t,\theta)-\zeta}{\ii \xi(t)+\ii y_0}\Big\}\Big)dA(\zeta)\nonumber
\\ & \quad -\int_{\overline{O_{t,2}^{\varepsilon}}}\Big(\log\big(\big|1+\varepsilon \tfrac{ \gamma_1(t,\theta)-\zeta}{\eta(t)+\ii y_0 }\big|\big)-\varepsilon \textnormal{Re}\Big\{\tfrac{ \gamma_1(t,\theta)-\zeta}{\eta(t)+\ii y_0}\Big\}\Big)dA(\zeta)\bigg]=G_1(t,\theta).\nonumber
		\end{align*}
Therefore the  system \eqref{new vortex patches system} might be reduced to only one scalar  equation
\begin{equation}\label{new vortex patches equation}
		\begin{aligned}
G_1(t,\theta)& =0.
 \end{aligned}
\end{equation}					
 We intend now to formulate the equations on the periodic setting and for this aim we set 
\begin{align}\label{lien-12}
\eta(t)+\ii \xi(t)=\sqrt{q(\omega_0 t)}e^{\ii\Theta(\omega_0 t)},
 \end{align}
where the pair  $\varphi\in\mathbb{T}\mapsto \big(q(\varphi),\Theta(\varphi)\big)$ satisfy, in view of \eqref{A},
\begin{equation}\label{asymp pt vortex}
\left\{ \begin{aligned}
&\omega_0\,{\dot{q}}+\Big[\frac{1}{y_0^2-\mathtt{I}\sin^2(\Theta)}+\frac{1}{y_0^2+q\cos^2(\Theta)}\Big]\sin(2\Theta)q=0, \\
&\omega_0\,\dot\Theta-\tfrac{ 1}{q}-\frac{\sin^2(\Theta)}{y_0^2-q\sin^2(\Theta)}+\frac{\cos^2(\Theta)}{y_0^2+q\cos^2(\Theta)}=0, \\
&\Theta(0)=\tfrac{\pi}{2}, \ \ q(0)=\xi_0^2.
\end{aligned}\right.
\end{equation}
Now, we restrict our search to the solutions  \eqref{new vortex patches equation} that share the  same period as the point vortex system, which is  $T(\xi_0)$. Therefore, the domain $O_{t,1}^\varepsilon$ can be parametrized as follows
\begin{align}\label{lien-123}
\gamma_1(t,\theta)=e^{\ii\Theta(\omega_0 t)} z(\omega_0 t,\theta)
\end{align}
with 
\begin{align}\label{def zk0}
		\nonumber	z (\varphi,\theta)&\triangleq\,R(\varphi,\theta)\,e^{\ii\theta}\\
			&\,=\sqrt{1+2\varepsilon r( \varphi,\theta)}\,e^{\ii\theta}.
		\end{align}
With this new variable, the domain  $O_{t,1}^\varepsilon$ will be simply denoted  $ e^{\ii\Theta(\omega_0 t)} O_{\varphi}^\varepsilon$. Thus,  by suitable change of variables in  \eqref{new vortex patches equation} and using the identity $\Theta(\varphi+\pi)=\Theta(\varphi)+\pi$,  we conclude that $z$ satisfies \begin{equation*}
		\begin{aligned}
&\varepsilon^2 \textnormal{Im}\Big\{ \omega_0\big[\partial_{\varphi}\overline{z(\varphi,\theta)}-\ii\, \dot{\Theta}(\varphi)\overline{z(\varphi,\theta)}\big]\partial_{\theta}  z(\varphi,\theta)\Big\}+\frac{1}{2\pi }\partial_{\theta}\bigg[\int_{O_{\varphi}^\varepsilon}\log(|  z(\varphi)-  \zeta|)dA(\zeta)
\\
 &+\int_{O_{\varphi+\pi}^\varepsilon}\Big(\log\big(\big|1+\varepsilon \tfrac{  z(\varphi,\theta)+\zeta}{\sqrt{q(\varphi)}}\big|\big)-\varepsilon \textnormal{Re}\Big\{\tfrac{  z(\varphi,\theta)+\zeta}{\sqrt{q(\varphi)}}\Big\}\Big)dA(\zeta)
\\ 
& -\int_{\overline{O_{\varphi}^\varepsilon}}\Big(\log\big(\big|1+\varepsilon \tfrac{ e^{\ii \Theta(\varphi)} z(\varphi,\theta)-e^{-\ii \Theta(\varphi)}\zeta}{\ii \sqrt{q(\varphi)}\sin( \Theta(\varphi))+\ii y_0}\big|\big)-\varepsilon \textnormal{Re}\Big\{\tfrac{ e^{\ii \Theta(\varphi)} z(\varphi,\theta)-e^{-\ii \Theta(\varphi)}\zeta}{\ii \sqrt{q(\varphi)}\sin( \Theta(\varphi))+\ii y_0}\Big\}\Big)dA(\zeta)
\\ & -\int_{\overline{O_{\varphi+\pi}^\varepsilon}}\Big(\log\big(\big|1+\varepsilon \tfrac{ e^{\ii \Theta(\varphi)} z(\varphi,\theta)+e^{-\ii \Theta(\varphi)}\zeta}{\sqrt{q(\varphi)}\cos( \Theta(\varphi))+\ii y_0 }\big|\big)-\varepsilon \textnormal{Re}\Big\{\tfrac{  e^{\ii \Theta(\varphi)} z(\varphi,\theta)+e^{-\ii \Theta(\varphi)}\zeta}{\sqrt{q(\varphi)}\cos( \Theta(\varphi))+\ii y_0 }\Big\}\Big)dA(\zeta)\bigg]=0. 
		\end{aligned}
\end{equation*}
Straightforward computations, using \eqref{def zk0}, lead to 
\begin{equation*}
\textnormal{Im}\Big\{ \omega_0\big[\partial_{\varphi}\overline{z(\varphi,\theta)}-\ii \dot{\Theta}(\varphi)\overline{z(\varphi,\theta)}\big]\partial_{\theta}  z(\varphi,\theta)\Big\} = \varepsilon\omega_0\partial_\varphi r(\varphi ,\theta)-\varepsilon\omega_0 \dot\Theta(\varphi ) \partial_\theta r(\varphi ,\theta)
\end{equation*}
	Thus,  $r$ defined through \eqref{def zk0} satisfies the following nonlinear equation: for all  $(\varphi,\theta)\in\mathbb{T}^2$, 
				\begin{align}
G(r) (\varphi ,\theta)  &\triangleq\varepsilon^3\omega_0\partial_\varphi r(\varphi ,\theta)-\varepsilon^3\omega_0 \dot\Theta(\varphi ) \partial_\theta r(\varphi ,\theta)+\sum_{n=1}^4\partial_{\theta} \Psi_n(\varepsilon,r)(\varphi ,\theta)=0,\label{Edc eq rkPsi}
		\end{align}
		where
		\begin{equation}\label{def psi1234}
	\begin{aligned}
		\Psi_1(\varepsilon, r)&\triangleq\int_{\mathbb{T}}\int_{0}^{ R(\varphi,\eta)}\log\big(\big| R(\varphi,\theta)e^{\ii\theta}- l e^{\ii\eta}\big|\big)l d l d\eta,\\
		\Psi_2(\varepsilon, r)&\triangleq\int_{\mathbb{T}}\int_{0}^{ R(\varphi+\pi,\eta)}\Big[\log\big(\big| 1+\varepsilon \tfrac{R(\varphi,\theta)e^{\ii\theta}+ l e^{\ii\eta}}{\sqrt{q(\varphi)}}\big|\big)-\varepsilon\textnormal{Re}\big\{\tfrac{R(\varphi,\theta)e^{\ii\theta}+ l e^{\ii\eta}}{\sqrt{q(\varphi)}}\big\}\Big]l d l d\eta,
		\\
				\Psi_{3+k}(\varepsilon, r)&\triangleq -\int_{\mathbb{T}}\int_{0}^{ R(\varphi+k\pi,\eta)}\Big[\log\Big(\Big| 1+ \varepsilon\tfrac{R(\varphi,\theta)e^{\ii(\theta+\Theta(\varphi))}+(-1)^{k+1}  l e^{-\ii(\eta+\Theta(\varphi))}}{\mathtt{w}_{3+k}(\varphi)}\Big|\Big)\\ &\qquad \qquad\qquad\qquad\qquad  -\varepsilon\textnormal{Re}\Big\{ \tfrac{R(\varphi,\theta)e^{\ii(\theta+ \Theta(\varphi))}+(-1)^{k+1} l e^{-\ii(\eta+\Theta(\varphi))}}{\mathtt{w}_{3+k}(\varphi)}\Big\}\Big]l d l d\eta		
		\end{aligned}
	\end{equation}	
with $k=0,1$ and
\begin{equation}\label{def ws}
\begin{aligned}
\mathtt{w}_3(\varphi)\triangleq \ii \sqrt{q(\varphi)}\sin( \Theta(\varphi))+\ii y_0,\qquad \mathtt{w}_4(\varphi)\triangleq \sqrt{q(\varphi)}\cos( \Theta(\varphi))+\ii y_0
\end{aligned}
\end{equation}		
		It's worthy to point out that the area  of the each single patch is conserved during the motion and satisfies 
		$$
		|O_\varphi|=\frac12\int_0^{2\pi}(1+2\varepsilon r(\varphi,\theta))d\theta=\pi
		$$
		where we impose the spatial  average to be zero, that is, $\int_{\T}r(\varphi,\theta)d\theta=0$. This latter property is preserved in time according to \eqref{Edc eq rkPsi}, which is compatible with the assumption \eqref{Volune-const}.
		
				 \subsection{Function spaces}\label{sec: func spa}
	
  We will introduce the complex Sobolev space within the periodic setting, denoted as $H^{s}(\mathbb{T}^{2},\mathbb{C})$, where the regularity index $s\in\R$. This space is defined as the collection of all the complex periodic functions $h:\mathbb{T}^{2}\to \mathbb{C}$ with the Fourier expansion
	$$
	{h=\sum_{(l,j)\in\mathbb{Z}^{2 }}h_{\ell,j}\,\mathbf{e}_{\ell,j}},\qquad \mathbf{e}_{\ell,j}(\varphi,\theta)\triangleq e^{\ii(l\cdot\varphi+j\theta)},\qquad h_{\ell,j}\triangleq\big\langle h,\mathbf{e}_{\ell,j}\big\rangle_{L^{2}(\mathbb{T}^{2})}
	$$
	equipped with the scalar product
	$$
	\big\langle h,\widetilde{h}\big\rangle_{H^{s}}\triangleq \sum_{(l,j)\in\mathbb{Z}^{2}}\langle \ell,j\rangle^{2s} h_{\ell,j}\overline{\widetilde h_{\ell,j}},\qquad\textnormal{with}\qquad \langle \ell,j\rangle^2\triangleq 1+|\ell|^2+|j|^2.
	$$
	 For $s=0$ this space coincides with the standard  $ L^{2}(\mathbb{T}^{2},\mathbb{C})$ space equipped with the scalar product
	$$
			\big\langle h,\widetilde{h}\big\rangle_{L^{2}(\mathbb{T}^{2})}\triangleq\bigintssss_{\mathbb{T}^{2}}h(\varphi,\theta)\overline{\widetilde{h}(\varphi,\theta)}d\varphi d\theta.
	$$
Another subspace that will be used frequently throughout this paper is $H^{s}_0(\mathbb{T}^{2},\mathbb{C})$	given by
$$
H^{s}_0(\mathbb{T}^{2},\mathbb{C})=\left\{ h\in H^{s}(\mathbb{T}^{2},\mathbb{C}), \hbox{s.t.}\quad\forall \varphi\in\T,\,  \int_{\T}h(\varphi,\theta) d\theta=0\right\}.
$$
	The anisotropic Sobolev space $H^{s_1,s_2}(\mathbb{T}^{2},\mathbb{C})$ is  the set of   functions $h:\mathbb{T}^{2}\to\mathbb{C}$ such that
$$
 h=\sum_{(\ell,j)\in \mathbb{Z}^{2}}h_{\ell,j}{\bf{e}}_{\ell,j}\quad\hbox{and}\quad \|h\|_{H^{s_1,s_2}}^2=\sum_{(\ell,j)\in \mathbb{Z}^{2}}\langle \ell,j\rangle^{2s_1}\langle j\rangle^{2s_2}|h_{\ell,j}|^2<\infty.
$$
		We shall also make use of the following mixed weighted Sobolev spaces with respect to a given  parameter  $\lambda\in(0,1)$. 
	{
	Let $\mathcal{O}$ be a nonempty subset of $\mathbb{R}$ and $\lambda\in(0,1]$, we  define the Banach spaces    }
	\begin{align*}
		\hbox{Lip}_\lambda(\mathcal{O},H^{s})&\triangleq\Big\lbrace h:\mathcal{O}\rightarrow H^{s}\quad\textnormal{s.t.}\quad\|h\|_{s}^{{\textnormal{Lip}(\lambda)}}<\infty\Big\rbrace,\\
		\hbox{Lip}_\lambda(\mathcal{O},\mathbb{C})&\triangleq\Big\lbrace h:\mathcal{O}\rightarrow\mathbb{C}\quad\textnormal{s.t.}\quad\|h\|^{{\textnormal{Lip}(\lambda)}}<\infty\Big\rbrace,
	\end{align*}
	with
	$$
	\|h\|_{s}^{\textnormal{Lip}(\lambda)}\triangleq\sup_{\xi_0\in{\mathcal{O}}}\|h(\xi_0,\cdot)\|_{H^{s}}+\lambda\sup_{\xi_{1}\neq\xi_{2}\in{\mathcal{O}}}\frac{\|h(\xi_1,\cdot)-h(\xi_2,\cdot)\|_{H^{s-1}}}{|\xi_1-\xi_2|},
	$$
	and
	$$
	\|h\|^{{\textnormal{Lip}(\lambda)}}\triangleq\sup_{\xi_0\in{\mathcal{O}}}|h(\xi_0)|+\lambda\sup_{\xi_1\neq\xi_2\in{\mathcal{O}}}\frac{|h(\xi_1)-h(\xi_2)|}{|\xi_1-\xi_2|}\cdot$$
	We emphasize that in Section \ref{N-M-S1} related to Nash-Moser scheme  we find it convenient to use the notation
	\begin{align}\label{Norm-not}
	\|h\|_{s,\mathcal{O}}^{\textnormal{Lip}(\lambda)}=\|h\|_{s}^{\textnormal{Lip}(\lambda)}.
	\end{align}
	We shall also need the anisotropic norm
	$$
	\|h\|_{s_1,s_2}^{\textnormal{Lip}(\lambda)}\triangleq\sup_{\xi_0\in{\mathcal{O}}}\|h(\xi_0,\cdot)\|_{H^{s_1,s_2}}+\lambda\sup_{\xi_1\neq\xi_2\in{\mathcal{O}}}\frac{\|h(\xi_1,\cdot)-h(\xi_2,\cdot)\|_{H^{s_1-1,s_2}}}{|\xi_1-\xi_2|}\cdot
	$$
To conclude this section, we will briefly recall the Hilbert transform on $\mathbb{T}$ . 
Let $h:\T\to\R$ be a continuous function with a zero average, we define its Hilbert transform by 
\begin{align}\label{H-Def}
\mathcal{H} h(\theta)=\int_{\T} h(\eta) \cot\left(\tfrac{\eta-\theta}{2}\right) d\eta.
\end{align}
It is well-known  that $\mathcal{H}$ is a Fourier multiplier with  
$$
\forall j\in\Z^*,\quad \mathcal{H} \mathbf{e}_j(\theta)=\ii \;\hbox{sign}(j)  \mathbf{e}_j(\theta).
$$

\subsection{Linearization}
Our next target is to perform a linearization of the functional ${G}(r)$ as defined in \eqref{Edc eq rkPsi}, for any small state $r$. Specifically, we aim to compute its Gateaux derivative, which can be demonstrated through standard arguments to be equivalent to the Fr\'echet derivative. More precisely, we will   prove the following result.	
	\begin{lemma}\label{prop:linearized} 
	There exists $\varepsilon_0\in(0,1)$ such that if $ r$ is smooth with zero average in space and  
		\begin{equation*}
		\varepsilon\leqslant\varepsilon_0\quad\textnormal{and}\quad\|r\|_{s_0+2}^{\textnormal{Lip}(\lambda)}\leqslant 1,
		\end{equation*}
		then the linearized operator of the map  ${G}(r)$, defined in  \eqref{Edc eq rkPsi},  at the  state $r$   in the direction  of a smooth function $h$ with zero average, $\displaystyle{\int_{\mathbb{T}}h(\varphi,\theta) d\theta=0}$,   is given by
\begin{align}
\partial_{r }{G}(r) [ h ] (\varphi,\theta)&=\varepsilon^3\omega_0\partial_\varphi  h (\varphi,\theta)+\varepsilon\partial_{\theta}\Big[\Big(\tfrac12- \tfrac{\varepsilon}{2} r(\varphi,\theta)-\varepsilon^2 \omega_0\dot\Theta (\varphi)-\tfrac{\varepsilon^2}{2}\mathtt{g}(\varphi,\theta)+\varepsilon^2{V}_1^\varepsilon(r)(\varphi,\theta)\nonumber\\ &\quad+\varepsilon^3{V}_2^\varepsilon(r)(\varphi,\theta)\Big) h (\varphi,\theta)\Big]-\tfrac{\varepsilon}{2}\mathcal{H}[h](\varphi,\theta)-\varepsilon^3  \partial_\theta\mathcal{Q}_0[h](\varphi,\theta)\nonumber\\ &\quad+\varepsilon^3 \partial_\theta\mathcal{R}_1^\varepsilon(r)[h](\varphi,\theta)+\varepsilon^4 \partial_\theta\mathcal{R}_2^\varepsilon(r)[h](\varphi,\theta),
\label{Edc eq rkPsi-dif}
		\end{align}		
where the real-valued function $\mathtt{g}$ takes the form 
\begin{align}
\mathtt{g}(\varphi,\theta)&\triangleq \textnormal{Re}\big\{ \mathtt{a}_2(\varphi)e^{\ii 2\theta} \big\}, \qquad \mathtt{a}_2(\varphi)\triangleq \tfrac{1}{q(\varphi)}-\tfrac{e^{\ii 2\Theta(\varphi)}}{\mathtt{w}_3^2(\varphi)}-\tfrac{e^{\ii 2\Theta(\varphi)}}{\mathtt{w}_4^2(\varphi)} \label{def g0}
\end{align}
the functions $\mathtt{w}_{n}(\varphi)$, $n\in\{3,4\}$, are given by \eqref{def ws} and the real functions ${V}_k^\varepsilon(r)$, $k\in\{1,2\}$ satisfy the estimates 
\begin{equation}\label{first est0}
\begin{aligned}
&\|{V}_1^\varepsilon(r)\|_{s}^{\textnormal{Lip}(\lambda)}\lesssim \|r\|_{s+1}^{\textnormal{Lip}(\lambda)}\|r\|_{s_0+1}^{\textnormal{Lip}(\lambda)};\quad 
\|{V}_2^\varepsilon(r)\|_{s}^{\textnormal{Lip}(\lambda)}\lesssim 1+\|r\|_{s}^{\textnormal{Lip}(\lambda)},
\\				
&\|\Delta_{12}{V}_1^\varepsilon(r)\|_{s}^{\textnormal{Lip}(\lambda)}\lesssim\|\Delta_{12}r\|_{s+1}^{\textnormal{Lip}(\lambda)}+\|\Delta_{12}r\|_{s_0+1}^{\textnormal{Lip}(\lambda)}\max_{\ell\in\{1,2\}}\|r_{\ell}\|_{s+1}^{\textnormal{Lip}(\lambda)},
\\				
&\|\Delta_{12}{V}_2^\varepsilon(r)\|_{s}^{\textnormal{Lip}(\lambda)}\lesssim\|\Delta_{12}r\|_{s}^{\textnormal{Lip}(\lambda)}+\|\Delta_{12}r\|_{s_0}^{\textnormal{Lip}(\lambda)}\max_{\ell\in\{1,2\}}\|r_{\ell}\|_{s}^{\textnormal{Lip}(\lambda)}.
\end{aligned}
\end{equation}
 The linear operators $\mathcal{Q}_0$ and $\mathcal{R}_k^\varepsilon(r)$ can be represented as follows
 \begin{equation}\label{def D0}
\begin{aligned}
\mathcal{Q}_0[h](\varphi,\theta)&\triangleq \int_{\mathbb{T}}  h (\varphi+\pi,\eta)\textnormal{Re}\Big\{\tfrac{e^{\ii(\theta+\eta)}}{q(\varphi)}-\tfrac{e^{\ii(\theta-\eta)}}{\mathtt{w}_{4}^2(\varphi)}\Big\} d\eta+\int_{\mathbb{T}}h(\varphi,\eta)\textnormal{Re}\Big\{ \tfrac{e^{\ii(\theta-\eta)}}{\mathtt{w}_{3}^2(\varphi)}\Big\} d\eta,\\
\mathcal{R}_1^\varepsilon(r)[h](\varphi,\theta)	&\triangleq \int_{\mathbb{T}} h (\varphi,\eta)K_1^\varepsilon(r)(\varphi,\theta,\eta)	d\eta,
\\
 \mathcal{R}_2^\varepsilon(r)[h](\varphi,\theta)	&\triangleq \sum_{k=0}^1\int_{\mathbb{T}} h (\varphi+k\pi,\eta)K_{2+k}^\varepsilon(r)(\varphi,\theta,\eta)	d\eta 
\end{aligned}
\end{equation}
with the estimates
\begin{equation}\label{first est}
\begin{aligned}
&\|K_1^\varepsilon(r)\|_{s}^{\textnormal{Lip}(\lambda)}\lesssim \|r\|_{s+1}^{\textnormal{Lip}(\lambda)}\|r\|_{s_0+1}^{\textnormal{Lip}(\lambda)},\\
&\|K_2^\varepsilon(r)\|_{s}^{\textnormal{Lip}(\lambda)}+\|K_3^\varepsilon(r)\|_{s}^{\textnormal{Lip}(\lambda)}\lesssim 1+\|r\|_{s}^{\textnormal{Lip}(\lambda)}.
\end{aligned}
\end{equation}
Moreover, one has 
\begin{equation}\label{hess G}
\begin{aligned} 
\partial_{r}^2{G}(r) [ h_1,h_2 ] (\varphi,\theta)&=- \tfrac{ \varepsilon^2}{2} \partial_{\theta}\Big[h_1(\varphi,\theta) h_2 (\varphi,\theta)\Big]+\varepsilon^3\mathcal{E}_1^\varepsilon(r)[h_1,h_2](\varphi,\theta),\\
\partial_{r}^3{G}(r) [ h_1,h_2,h_3 ] (\varphi,\theta)&=\varepsilon^3\mathcal{E}_2^\varepsilon(r)[h_1,h_2,h_3](\varphi,\theta),
		\end{aligned}
\end{equation}		
where
\begin{align*}
			\| \mathcal{E}_1^\varepsilon(r)[h_1,h_2]\|_{s}^{\textnormal{Lip}(\lambda)}&\lesssim\|h_{1}\|_{s_{0}+2}^{\textnormal{Lip}(\lambda)}\|h_{2}\|_{s+2}^{\textnormal{Lip}(\lambda)}+\Big(\|h_{1}\|_{s+2}^{\textnormal{Lip}(\lambda)}+\| r\|_{s+2}^{\textnormal{Lip}(\lambda)}\|h_{1}\|_{s_{0}+2}^{\textnormal{Lip}(\lambda)}\Big)\|h_{2}\|_{s_{0}+2}^{\textnormal{Lip}(\lambda)},\\
			\| \mathcal{E}_2^\varepsilon(r)[h,h,h]\|_{s}^{\textnormal{Lip}(\lambda)}&\lesssim \big(\|h\|_{s_{0}+2}^{\textnormal{Lip}(\lambda)}\big)^2\Big(\|h\|_{s+2}^{\textnormal{Lip}(\lambda)}+\|r\|_{s+2}^{\textnormal{Lip}(\lambda)}\|h\|_{s_0+2}^{\textnormal{Lip}(\lambda)}\Big).
			\end{align*}		
\end{lemma}
\begin{proof}
Differentiating \eqref{Edc eq rkPsi} with respect to $r$ in the direction $ h $ gives after straightforward computations 
	\begin{align}\label{linearized st0}
\partial_r {G}(r)[h](\varphi,\theta)&=\varepsilon^3\omega_0\partial_\varphi h(\varphi ,\theta)-\varepsilon^3\omega_0 \dot\Theta(\varphi ) \partial_\theta h(\varphi ,\theta)+\sum_{n=1}^4\partial_{\theta}\Big[\partial_r\Psi_n(\varepsilon,r)[h](\varphi ,\theta)\Big].
		\end{align}
		We write $\Psi_1$ in the form
\begin{equation*}
\Psi_1(\varepsilon,r)(\varphi,\theta)=\Psi\big(r, R(\varphi,\theta)e^{\ii\theta}\big)
\end{equation*}
with
	\begin{equation*}
	\begin{aligned}
		\Psi(r,z)(\varphi)&\triangleq\frac{1}{2\pi}\int_{0}^{2\pi}\int_{0}^{ R(\varphi,\eta)}\log\big(\big|z- l e^{\ii\eta}\big|\big)l dl d\eta,\\
		&= \frac{1}{2\pi}\int_{O_\varphi}\log\big(\big|z-\varepsilon \xi\big|\big)dA(\xi), \quad R(\varphi,\eta)=\left(1+2\varepsilon r( \varphi,\theta)\right)^{\frac{1}{2}}. 
	\end{aligned}
	\end{equation*}
Using the chain rule we get
 \begin{equation}\label{psi-psi1}
 \begin{aligned}
\partial_r\Psi_1(\varepsilon,r)(\varphi,\theta)&=(\partial_{r}\Psi)\big(r, R(\varphi,\theta)e^{\ii\theta}\big)[ h ](\varphi,\theta)
\\ &\quad +\textnormal{Re}\big\{2(\partial_{\overline{z}}\Psi)\big(r, R(\varphi,\theta)e^{\ii\theta}\big)\overline{\partial_{r}\big(R(\varphi,\theta)e^{\ii\theta}\big)[ h ]( \varphi,\theta) }\big\}.
\end{aligned}
\end{equation}	
Differentiating $\Psi(r,z)$ with respect to $r$ leads to
	\begin{align}\label{d z psi}
	(\partial_{r}\Psi)\big(r,z\big)[ h ](\varphi,\varphi)&=\varepsilon\int_{\mathbb{T}} h ( \varphi,\eta)\log\big(| z-\varepsilon R( \varphi,\eta)e^{\ii\eta}|\big)d\eta.
	\end{align}
	On the other hand, from the identity $\partial_{\overline{z}}\log(|z-\varepsilon\xi|)=-\tfrac1\varepsilon\partial_{\overline{\xi}}\log(|z-\varepsilon\xi|)$ one may write
		\begin{align*}
	2\partial_{\overline z} \Psi(r,z)(\varphi)&=-\frac{1}{\pi\varepsilon }\int_{O_\varphi}\partial_{\overline\xi} \log\big(\big|z-\varepsilon \xi\big|\big)  dA(\xi).
		\end{align*}
Applying Gauss-Green theorem yields
	\begin{align}\label{lkn part2}
		2\partial_{\overline z} \Psi(r,z)(\varphi)&= \frac{\ii}{2\pi \varepsilon} \int_{0}^{2\pi}\log\big(\big|z-\varepsilon R( \varphi,\eta) e^{\ii \eta}\big|\big)\partial_\eta\big(R( \varphi,\eta)e^{\ii \eta}\big)d\eta.
		\end{align}
	In view of \eqref{psi-psi1}, \eqref{d z psi} and \eqref{lkn part2}  we obtain
			\begin{align}\label{linearized psi st1}
& \partial_r\Psi_1(\varepsilon,r) [ h ]  (\varphi,\theta)=\varepsilon U_1^\varepsilon(r)(\varphi,\theta)h(\varphi,\theta)+\varepsilon I_1^\varepsilon(r)[h](\varphi,\theta)
		\end{align}
		with
	\begin{equation*}
	\begin{aligned}
I_1^\varepsilon(r)[h](\varphi,\theta)&= \int_{\mathbb{T}} h (\varphi,\eta)\log\big(\big| R(\varphi,\theta)e^{\ii\theta}- R(\varphi,\eta)e^{\ii\eta}\big|\big)d\eta,
	\\ 
	U_1^\varepsilon(r)(\varphi,\theta)&=-\frac{  1}{R(\varphi,\theta)}\int_{\mathbb{T}}\log\big(\big| R(\varphi,\theta)e^{\ii\theta}- R(\varphi,\eta)e^{\ii\eta}\big|\big)\partial_{\eta}\textnormal{Im}\big\{R(\varphi,\eta)e^{\ii(\eta-\theta)}\big\}d\eta.
\end{aligned}
\end{equation*}
For the sake of simplicity,   we shall omit the dependence of our quantities with respect to the variable $\varphi$. We write $I_1^\varepsilon(r)[ h ]$ as 
	\begin{equation*}
	\begin{aligned}
I_1^\varepsilon(r)[ h ](\theta)	 &=
	\int_{\mathbb{T}} h (\eta)\log\big(| e^{\ii\theta}- e^{\ii\eta}|\big)d\eta+\varepsilon\int_{\mathbb{T}}  h (\eta)\textnormal{Re}\big\{f_1(\theta,\eta)\big\}d\eta\\ &\quad + \int_{\mathbb{T}} h (\eta)\big[\log\big(\big| 1+\varepsilon f_1(\theta,\eta)\big|\big)-\varepsilon\textnormal{Re}\big\{f_1(\theta,\eta)\big\}\big]d\eta,
	\end{aligned}
	\end{equation*}
		where
	\begin{align}\label{def f1}
	f_1(\theta,\eta)&\triangleq \tfrac{(R(\theta)-1)e^{\ii\theta}- (R(\eta)-1)e^{\ii\eta}}{\varepsilon(e^{\ii\theta}- e^{\ii\eta})}.
	\end{align}
From the identity 
$
\textnormal{Re}\big\{\tfrac{ e^{\ii\eta}}{e^{\ii\theta}- e^{\ii\eta}}\big\}=-\tfrac12,
$
we conclude that
\begin{align}\label{re f1}
\textnormal{Re}\big\{f_1(\theta,\eta)\big\}&=\tfrac{1}{2\varepsilon}\big(R(\theta)-1\big)+\tfrac{1}{2\varepsilon}\big(R(\eta)-1\big).
\end{align}
Using the fact that $h$ has zero average, we obtain
	\begin{equation*}
	\begin{aligned}
I_1^\varepsilon(r)[ h ]	 &=
	\int_{\mathbb{T}} h (\eta)\log\big(| e^{\ii\theta}- e^{\ii\eta}|\big)d\eta+\tfrac12\int_{\mathbb{T}}  h (\eta)\big(R(\eta)-1\big)d\eta\\ &\quad + \int_{\mathbb{T}} h (\eta)\Big[\log\big(\big| 1+\varepsilon f_1(\theta,\eta)\big|\big)-\varepsilon\textnormal{Re}\big\{f_1(\theta,\eta)\big\}\Big]d\eta.
	\end{aligned}
	\end{equation*}  
Differentiating the last expression  with respect to  $\theta$ gives 
	\begin{equation}\label{new l1}
	\begin{aligned}
\partial_\theta\Big[I_1^\varepsilon(r)[ h ](\theta)\Big]	&=-\tfrac12
	\mathcal{H}[h](\theta)+\varepsilon^2\partial_\theta\mathcal{R}^\varepsilon_1(r)[h](\theta),
	\end{aligned}
	\end{equation}
where $\mathcal{H}$ is the Hilbert transform, defined by \eqref{H-Def}, and  $ \mathcal{R}^\varepsilon_1(r)[h]$ is given by \eqref{def D0} with
$$
K_1^\varepsilon(r)\triangleq \tfrac{1}{\varepsilon^2}\Big[\log\big(\big| \varepsilon f_1(\theta,\eta)+1\big|\big)-\varepsilon\textnormal{Re}\big\{f_1(\theta,\eta)\big\}\Big].
$$
In view of \eqref{def f1}, an application of \cite[Lem. 4.2]{HR22} gives 
\begin{equation}\label{estimate f1}
\|f_1\|_s^{\textnormal{Lip}(\lambda)}\lesssim \|r\|_{s+1}^{\textnormal{Lip}(\lambda)}.
\end{equation} 
By the composition laws in Lemma \ref{lem funct prop} we conclude  the estimate on $K_1^\varepsilon(r)$, stated in \eqref{first est}. 	
Next, in view of \eqref{def f1},   we write 
		\begin{align*}
U_1^\varepsilon(r)(\theta)&=-\tfrac{1}{R(\theta)}\int_{\mathbb{T}}\log\big(| e^{\ii\theta}-e^{\ii\eta}|\big)\cos(\eta-\theta)d\eta-\tfrac{1}{R(\theta)}\int_{\mathbb{T}}\log\big(\big| 1+\varepsilon f_1(\theta,\eta)\big|\big)\cos(\eta-\theta)d\eta\\ &\quad-\tfrac{1}{R(\theta)}\int_{\mathbb{T}}\log\big(| e^{\ii\theta}- e^{\ii\eta}|\big)\partial_\eta\big[\big(R(\eta)-1\big)\sin(\eta-\theta)\big]d\eta
\\ &\quad-\tfrac{1}{R(\theta)}\int_{\mathbb{T}}\log\big(\big| 1+\varepsilon f_1(\theta,\eta)\big|\big)\partial_\eta\big[\big(R(\eta)-1\big)\sin(\eta-\theta)\big]d\eta.
		\end{align*}
		Recall the following formula, which can be found in \cite[Lem. A.3]{CCG4},		
$$
\int_{\mathbb{T}}\log\big(| e^{\ii\theta}-e^{\ii\eta}|\big)\cos(\eta-\theta)d\eta=-\tfrac12.
$$
On the other hand, integrating by parts gives
	\begin{align*}
\int_{\mathbb{T}}\log\big(| e^{\ii\theta}- e^{\ii\eta}|\big)\partial_\eta\big[\big(R(\eta)-1\big)\sin(\eta-\theta)\big]d\eta
&=-\tfrac12\int_{\mathbb{T}}\big(1+\cos (\eta-\theta)\big)\big(R(\eta)-1\big)d\eta.
		\end{align*}
Moreover,  using the identity \eqref{re f1} we get
		$$
\int_{\mathbb{T}} \varepsilon\textnormal{Re}\big\{f_1(\theta,\eta)\big\}\cos(\eta-\theta)d\eta=\tfrac{1}{2}\int_{\mathbb{T}} \cos(\eta-\theta)\big(R(\eta)-1\big)d\eta.		
		$$		
Combining the four last identities, we get 
		\begin{align*}
U_1^\varepsilon(r)(\theta)&=\tfrac{1}{2R(\theta)}+\tfrac{1}{2R(\theta)}\int_{\mathbb{T}}\big(R(\eta)-1\big)d\eta\\ &\quad -\tfrac{1}{R(\theta)}\int_{\mathbb{T}}\Big[\log\big(\big| 1+\varepsilon f_1(\theta,\eta)\big|\big)-\varepsilon\textnormal{Re}\big\{f_1(\theta,\eta)\big\}\Big]\cos(\eta-\theta)d\eta
\\ &\quad-\tfrac{1}{R(\theta)}\int_{\mathbb{T}}\log\big(\big| 1+\varepsilon f_1(\theta,\eta)\big|\big)\partial_\eta\big[\big(R(\eta)-1\big)\sin(\eta-\theta)\big]d\eta.
		\end{align*}
Using the fact that $r$ has zero average in space   we conclude that
	\begin{align}\label{new V1}
U_1^\varepsilon(r)(\theta)&=\tfrac12-\tfrac{\varepsilon}{2} r(\theta)+\varepsilon^2 {V}_1^\varepsilon(r)(\theta),
		\end{align}
		where
		\begin{align}\label{new V1 tilde}
\nonumber {V}_1^\varepsilon(r)(\theta)&\triangleq\tfrac{1}{2\varepsilon^2}\Big[\tfrac{1}{\sqrt{1+2\varepsilon r(\theta)}}-1+{\varepsilon}r(\theta)\Big]+\tfrac{1}{2\varepsilon^2 R(\theta)}\int_{\mathbb{T}}\big(\sqrt{1+2\varepsilon r(\eta)}-1-\varepsilon r(\eta)\big)d\eta\\ \nonumber &\quad-\tfrac{1}{\varepsilon^2 R(\theta)}\int_{\mathbb{T}}\Big[\log\big(\big| 1+\varepsilon f_1(\theta,\eta)\big|\big)-\varepsilon\textnormal{Re}\big\{f_1(\theta,\eta)\big\}\Big]\cos(\eta-\theta)d\eta
\\ &\quad-\tfrac{1}{\varepsilon^2R(\theta)}\int_{\mathbb{T}}\log\big(\big| 1+\varepsilon f_1(\theta,\eta)\big|\big)\partial_\eta\big[\big(\sqrt{1+2\varepsilon r(\eta)}-1\big)\sin(\eta-\theta)\big]d\eta.
		\end{align}
As a consequence, composition laws in Lemma \ref{lem funct prop} together with \eqref{estimate f1} and  the smallness condition of $r$ imply the estimate on ${V}_1^\varepsilon(r)$, given by \eqref{first est0}.	Putting together \eqref{linearized psi st1}, \eqref{new l1}  and \eqref{new V1} we find 
			\begin{align}\label{linearized psi st12}
\nonumber\partial_\theta\Big[ \partial_r\Psi_1(\varepsilon,r) [ h ]  (\varphi,\theta)\Big]&=\varepsilon\partial_\theta\Big[  \Big(\tfrac12-\tfrac{\varepsilon}{2} r(\varphi,\theta)+\varepsilon^2 {V}_1^\varepsilon(r)(\varphi,\theta)\Big)h(\varphi,\theta)\Big]\\ &\quad-\tfrac{\varepsilon}{2} 
	\mathcal{H}[h](\varphi,\theta)+\varepsilon^3\partial_\theta\mathcal{R}^\varepsilon_1(r)[h](\varphi,\theta).
		\end{align}
Next, differentiating $\Psi_2(\varepsilon,r)$, defined trough \eqref{def psi1234}, with respect to $r$ gives 		
\begin{align}\label{linearized psi2 st1}
&\partial_r\Psi_2(\varepsilon,r) [ h ]  (\varphi,\theta)
 =\varepsilon^2 U_2^\varepsilon(r)(\varphi,\theta)h(\varphi,\theta)+\varepsilon I_2^\varepsilon(r)[h](\varphi,\theta),
		\end{align}		
where
\begin{equation*}
	\begin{aligned}
	U_2^\varepsilon(r)(\varphi,\theta)&\triangleq \tfrac{1}{R(\varphi,\theta)}\int_{\mathbb{T}}\int_{0}^{ R(\varphi+\pi,\eta)}\bigg[\textnormal{Re}\Big\{\tfrac{e^{\ii\theta}}{\sqrt{q(\varphi)}+\varepsilon (R(\varphi,\theta)e^{\ii\theta}+  l e^{\ii\eta})}\Big\}-\textnormal{Re}\Big\{\tfrac{e^{\ii\theta}}{\sqrt{q(\varphi)}}\Big\}\bigg]l d l d\eta,\\
	I_2^\varepsilon(r)[h](\varphi,\theta)&\triangleq\int_{\mathbb{T}} h(\varphi+\pi,\eta)\Big[\log\big(\big| 1+\varepsilon \tfrac{R(\varphi,\theta)e^{\ii\theta}+   R(\varphi+\pi,\eta) e^{\ii\eta}}{\sqrt{q(\varphi)}}\big|\big)-\varepsilon\textnormal{Re}\big\{\tfrac{R(\varphi,\theta)e^{\ii\theta}+ R(\varphi+\pi,\eta) e^{\ii\eta}}{\sqrt{q(\varphi)}}\big\}\Big] d\eta.
	\end{aligned}
	\end{equation*}
Notice that ${I}_{2}^\varepsilon(r)[ h ],$ can be expressed as 
	\begin{align*}
&I_2^\varepsilon(r)[ h ](\varphi,\theta)
 =-\frac{\varepsilon^{2}}{2 }\int_{\mathbb{T}}  h (\varphi+\pi,\eta)\textnormal{Re}\Big\{\tfrac{f_2(\varphi,\theta,\eta)^2}{q(\varphi)}\Big\}d\eta\\ &\quad+\int_{\mathbb{T}} h (\varphi+\pi,\eta)\partial_\theta\Big[\log\big(\big|\tfrac{\varepsilon}{\sqrt{q(\varphi)}} f_2(\varphi,\theta,\eta)+1\big|\big)-\varepsilon\textnormal{Re}\Big\{\tfrac{f_2(\varphi,\theta,\eta)}{\sqrt{q(\varphi)}}\Big\}+\tfrac{\varepsilon^2}{2}\textnormal{Re}\Big\{\tfrac{f_2(\varphi,\theta,\eta)^2}{q(\varphi)}\Big\}\Big]d\eta,
 \end{align*}
 	where
\begin{align}\label{def f2}
	f_2(\varphi,\theta,\eta)&\triangleq R(\varphi,\theta)e^{\ii\theta}+R(\varphi,\eta) e^{\ii\eta}.
	\end{align}
Since $h$ has zero average, in the space variable, we get
 \begin{align*}
 \int_{\mathbb{T}}  h (\varphi+\pi,\eta)\textnormal{Re}\Big\{\tfrac{f_2(\varphi,\theta,\eta)^2}{q(\varphi)}\Big\}d\eta&=2\int_{\mathbb{T}}  h (\varphi+\pi,\eta)R(\varphi,\theta)R(\varphi,\eta)\textnormal{Re}\Big\{\tfrac{e^{\ii(\theta+\eta)}}{q(\varphi)}\Big\}d\eta\\ &\quad+\tfrac{1}{q(\varphi)}\int_{\mathbb{T}}  h (\varphi+\pi,\eta)R(\varphi,\eta)^2\cos(2\eta)d\eta.
 \end{align*}
This implies that
\begin{align}\label{new l2}
\nonumber \partial_\theta\big[I_2^\varepsilon(r)[ h ](\varphi,\theta)\big]&=\varepsilon^2\partial_\theta\bigg[-\int_{\mathbb{T}}  h (\varphi+\pi,\eta)\textnormal{Re}\Big\{\tfrac{e^{\ii(\theta+\eta)}}{q(\varphi)}\Big\}d\eta\\ &\qquad\qquad +\varepsilon   \int_{\mathbb{T}} h (\varphi+\pi,\eta)K_{2,2}^\varepsilon(r)(\varphi,\theta,\eta)	d\eta\bigg],
 \end{align}
 where 
   \begin{align*}
\nonumber K_{2,2}^\varepsilon(r)(\varphi,\theta,\eta)	&\triangleq\tfrac{1}{\varepsilon^3}\Big[\log\big(\big|\tfrac{\varepsilon}{\sqrt{q(\varphi)}} f_2(\varphi,\theta,\eta)+1\big|\big)-\varepsilon\textnormal{Re}\Big\{\tfrac{f_2(\varphi,\theta,\eta)}{\sqrt{q(\varphi)}}\Big\}+\tfrac{\varepsilon^2}{2}\textnormal{Re}\Big\{\tfrac{f_2(\varphi,\theta,\eta)^2}{q(\varphi)}\Big\}\Big]\\  &\qquad-\tfrac{1}{\varepsilon }\Big[\big(R(\varphi,\theta)R(\varphi,\eta)-1\big)\textnormal{Re}\Big\{\tfrac{e^{\ii(\theta+\eta)}}{q(\varphi)}\Big\}\Big].
	\end{align*}
Let us now move to the function $U_2^\varepsilon(r)$  which writes, in view of \eqref{def f2},
\begin{align}\label{exp v2 new}
U_2^\varepsilon(r)(\theta)&= -\frac{\varepsilon}{R(\varphi,\theta)}\int_{\mathbb{T}}\int_{0}^{ R(\varphi+\pi,\eta)}\textnormal{Re}\Big\{\tfrac{(R(\varphi,\theta)e^{\ii\theta}+ l e^{\ii\eta})e^{\ii\theta}}{{q(\varphi)}+\varepsilon \sqrt{q(\varphi)} (R(\varphi,\theta)e^{\ii\theta}+  l e^{\ii\eta})}\Big\}l d l d\eta\nonumber\\
 &= -\varepsilon\int_{\mathbb{T}}\int_{0}^{1}\textnormal{Re}\Big\{\tfrac{(e^{\ii\theta}+  l e^{\ii\eta})e^{\ii\theta}}{{q(\varphi)}}\Big\}l d l d\eta +\varepsilon^2 V_{2,2}^\varepsilon(r)(\varphi,\theta),
 \end{align}
 where
  \begin{align*}
V_{2,2}^\varepsilon(r)(\varphi,\theta)
& \triangleq \frac1\varepsilon\bigg[\int_{\mathbb{T}}\int_{0}^{1}\textnormal{Re}\Big\{\tfrac{(e^{\ii\theta}+  l e^{\ii\eta})e^{\ii\theta}}{{q(\varphi)}}\Big\}l d l d\eta \nonumber \\ &\qquad- \frac{1}{R(\varphi,\theta)}\int_{\mathbb{T}}\int_{0}^{ R(\varphi+\pi,\eta)}\textnormal{Re}\Big\{\tfrac{(R(\varphi,\theta)e^{\ii\theta}+  l e^{\ii\eta})e^{\ii\theta}}{{q(\varphi)}+\varepsilon \sqrt{q(\varphi)}(R(\varphi,\theta)e^{\ii\theta}+  l e^{\ii\eta})}\Big\}l d l d\eta\bigg].
 \end{align*}
Thus, straightforward computations lead to
\begin{align}\label{new V2}
&U_2^\varepsilon(r)(\varphi,\theta) =-\tfrac\varepsilon2\textnormal{Re}\Big\{\tfrac{e^{\ii 2\theta}}{q(\varphi)}\Big\}+\varepsilon^2 {V}_{2,2}^\varepsilon(r)(\varphi,\theta), 
 \end{align}
 Gathering \eqref{linearized psi2 st1}, \eqref{new l2} and  \eqref{new V2} gives
 			\begin{align}\label{linearized psi 2 st12}
\nonumber &\partial_\theta\Big[ \partial_r\Psi_2(\varepsilon,r) [ h ]  (\varphi,\theta)\Big]=\varepsilon^3\partial_\theta\Big[ - \textnormal{Re}\Big\{\tfrac{e^{\ii 2\theta}}{q(\varphi)}\Big\}+\varepsilon {V}_{2,2}^\varepsilon(r)(\varphi,\theta)\Big)h(\theta)\Big]\\ &\quad+\varepsilon^3\partial_\theta\bigg[-\int_{\mathbb{T}}  h (\varphi+\pi,\eta)\textnormal{Re}\Big\{\tfrac{e^{\ii(\theta+\eta)}}{q(\varphi)}\Big\}d\eta +\varepsilon\int_{\mathbb{T}} h (\varphi+\pi,\eta)K_{2,2}^\varepsilon(r)(\varphi,\theta,\eta)	d\eta\bigg]
		\end{align}
As for the terms $\Psi_{3+k}(\varepsilon,r)$, $k=0,1$, given by \eqref{def psi1234}, we have
$$
\partial_r\Psi_{3+k}(\varepsilon,r)[h](\varphi ,\theta)=\varepsilon^2 U_{3+k}^\varepsilon(r)(\varphi,\theta)h(\varphi,\theta)+\varepsilon I_{3+k}^\varepsilon(r)[h](\varphi,\theta),
$$
where
		\begin{equation*}
	\begin{aligned}
	U_{3+k}^\varepsilon(r)(\varphi,\theta)&\triangleq \tfrac{-1}{R(\varphi,\theta)}\int_{\mathbb{T}}\int_{0}^{ R(\varphi+k\pi,\eta)} \bigg[\textnormal{Re}\Big\{ \tfrac{e^{\ii(\theta+\Theta(\varphi))}}{\mathtt{w}_{3+k}(\varphi)+\varepsilon \big(R(\varphi,\theta)e^{\ii(\theta+\Theta(\varphi))}+(-1)^{k+1}  l e^{-\ii(\eta+\Theta(\varphi))}\big)}\Big\}\\ &\qquad \qquad \qquad \qquad \qquad\qquad \qquad \qquad \qquad  -\textnormal{Re}\Big\{ \tfrac{e^{\ii(\theta+\Theta(\varphi))}}{\mathtt{w}_{3+k}(\varphi)}\Big\}\bigg]l d l d\eta,	
\\	 
I_{3+k}^\varepsilon(r)[h](\varphi,\theta)&\triangleq-\int_{\mathbb{T}}h(\varphi+k\pi,\eta)\Big[\log\Big(\Big| 1+\varepsilon \tfrac{R(\varphi,\theta)e^{\ii(\theta+\Theta(\varphi))}+(-1)^{k+1}  R(\varphi+k\pi,\eta) e^{-\ii(\eta+\Theta(\varphi))}}{\mathtt{w}_{3+k}(\varphi)}\Big|\Big)\\ &\qquad \qquad \qquad \qquad \qquad -\varepsilon \textnormal{Re}\Big\{ \tfrac{R(\varphi,\theta)e^{\ii(\theta+\Theta(\varphi))}+(-1)^{k+1}  R(\varphi+k\pi,\eta) e^{-\ii(\eta+\Theta(\varphi))}}{\mathtt{w}_{3+k}(\varphi)}\Big\}\Big] d\eta.		\end{aligned}
	\end{equation*}
	The function $U_{3+k}^\varepsilon(r)(\varphi,\theta)$ writes
		\begin{equation*}
	\begin{aligned}
	U_{3+k}^\varepsilon(r)(\varphi,\theta)&=\tfrac{\varepsilon}{R(\varphi,\theta)}\int_{\mathbb{T}}\int_{0}^{ R(\varphi+k\pi,\eta)}\textnormal{Re}\Big\{ \tfrac{e^{\ii(\theta+\Theta(\varphi))}\big(R(\varphi,\theta)e^{\ii(\theta+\Theta(\varphi))}+(-1)^{k+1}  l e^{-\ii(\eta+\Theta(\varphi))}\big)}{\mathtt{w}_{3+k}^2(\varphi)+\varepsilon \mathtt{w}_{3+k}(\varphi)\big(R(\varphi,\theta)e^{\ii(\theta+\Theta(\varphi))}+(-1)^{k+1}  l e^{-\ii(\eta+\Theta(\varphi))}\big)}\Big\}l d l d\eta\\ &=\varepsilon \int_{\mathbb{T}}\int_{0}^{1}\textnormal{Re}\Big\{ \tfrac{e^{\ii 2(\theta+\Theta(\varphi))} +(-1)^{k+1} l e^{\ii(\theta-\eta)}}{\mathtt{w}_{3+k}^2(\varphi)}\Big\}l d l d\eta + \varepsilon^2 {V}_{2,k}^\varepsilon(r)(\varphi,\theta),
		\end{aligned}
	\end{equation*}	
		where
		\begin{equation}\label{def v2k}
	\begin{aligned}	
				& {V}_{2,k}^\varepsilon(r)(\varphi,\theta)\triangleq \frac1\varepsilon\bigg[-\int_{\mathbb{T}}\int_{0}^{1}\textnormal{Re}\Big\{ \tfrac{e^{\ii 2(\theta+\Theta(\varphi))}+(-1)^{k+1}  l e^{\ii(\theta-\eta)}}{\mathtt{w}_{3+k}^2(\varphi)}\Big\}l d l d\eta\\ &\quad +\tfrac{1}{ R(\varphi,\theta)}\int_{\mathbb{T}}\int_{0}^{ R(\varphi+\pi,\eta)}\textnormal{Re}\Big\{ \tfrac{e^{\ii(\theta+\Theta(\varphi))}\big(R(\varphi,\theta)e^{\ii(\theta+\Theta(\varphi))}+(-1)^{k+1}  l e^{-\ii(\eta+\Theta(\varphi))}\big)}{\mathtt{w}_{3+k}^2(\varphi)+\varepsilon \mathtt{w}_{3+k}(\varphi)\big(R(\varphi,\theta)e^{\ii(\theta+\Theta(\varphi))}+(-1)^{k+1}  l e^{-\ii(\eta+\Theta(\varphi))}\big)}\Big\}l d l d\eta\bigg].
\end{aligned}
	\end{equation}
	Hence, from the identity 
	$$
	\int_{\mathbb{T}}\int_{0}^{1}\textnormal{Re}\Big\{ \tfrac{e^{\ii 2(\theta+\Theta(\varphi))} +(-1)^{k+1} l e^{\ii(\theta-\eta)}}{\mathtt{w}_{3+k}^2(\varphi)}\Big\}l d l d\eta=\tfrac{1}{2}\textnormal{Re}\Big\{ \tfrac{e^{\ii 2(\theta+\Theta(\varphi))}}{\mathtt{w}_{3+k}^2(\varphi)}\Big\}
	$$
	we conclude that
		\begin{equation*}
	\begin{aligned}
	U_{3+k}^\varepsilon(r)(\varphi,\theta)&= \tfrac{\varepsilon}{2}\textnormal{Re}\Big\{ \tfrac{e^{\ii 2(\theta+\Theta(\varphi))}}{\mathtt{w}_{3+k}^2(\varphi)}\Big\} + \varepsilon^2 {V}_{2,k}^\varepsilon(r)(\varphi,\theta).
		\end{aligned}
	\end{equation*}	
On the other hand, one has
		\begin{equation*}
	\begin{aligned}
	I_{3+k}^\varepsilon(r)[h](\varphi,\theta)&=
\frac{\varepsilon^2}{2}\int_{\mathbb{T}}h(\varphi+k\pi,\eta)\textnormal{Re}\Big\{ \Big(\tfrac{e^{\ii(\theta+\Theta(\varphi))} +(-1)^{k+1}   e^{-\ii(\eta+\Theta(\varphi))}}{\mathtt{w}_{3+k}(\varphi)}\Big)^2\Big\} d\eta\\ &\quad +\varepsilon^3\int_{\mathbb{T}} h (\varphi+k\pi,\eta)K_{2,k}^\varepsilon(r)(\varphi,\theta,\eta)	d\eta,
		\end{aligned}
	\end{equation*}
	where
		\begin{equation}\label{def f2k}
	\begin{aligned}	
	K_{2,k}^\varepsilon(r)(\varphi,\theta,\eta) &\triangleq \frac{1}{2\varepsilon}\textnormal{Re}\Big\{ \Big(\tfrac{f_{3+k}(\varphi,\theta,\eta)}{\mathtt{w}_{3+k}(\varphi)}\Big)^2-\Big(\tfrac{e^{\ii(\theta+\Theta(\varphi))}+(-1)^{k+1}  e^{-\ii(\eta+\Theta(\varphi))}}{\mathtt{w}_{3+k}(\varphi)}\Big)^2\Big\}, \\
	&-\frac{1}{\varepsilon^3}\Big[\log\Big(\Big| 1+ \tfrac{\varepsilon {f}_{3+k}(\varphi,\theta,\eta)}{\mathtt{w}_3(\varphi)}\Big|\Big) -\textnormal{Re}\Big\{ \tfrac{\varepsilon {f}_{3+k}(\varphi,\theta,\eta)}{\mathtt{w}_3(\varphi)}\Big\}+\tfrac{\varepsilon^2}{2}\textnormal{Re}\Big\{ \Big(\tfrac{ {f}_3(\varphi,\theta,\eta)}{\mathtt{w}_{3+k}(\varphi)}\Big)^2\Big\}\Big],\\
	f_{3+k}(\varphi,\theta,\eta)&\triangleq R(\varphi,\theta)e^{\ii(\theta+\Theta(\varphi))}+(-1)^{k+1}  R(\varphi+k\pi ,\eta) e^{-\ii(\eta+\Theta(\varphi))}.
		\end{aligned}
	\end{equation}
Since $h$ has zero average in the space variable, we get
		\begin{equation*}
	\begin{aligned}
	\partial_{\theta}\big[I_{3+k}^\varepsilon(r)[h](\varphi,\theta)\big]&=\partial_\theta\bigg[(-1)^{k+1}\varepsilon^2\int_{\mathbb{T}}h(\varphi+k\pi,\eta)\textnormal{Re}\Big\{ \Big(\tfrac{e^{\ii(\theta-\eta)}}{\mathtt{w}_{3+k}^2(\varphi)}\Big)\Big\} d\eta\\ &\qquad +\varepsilon^3\int_{\mathbb{T}} h (\varphi+k\pi,\eta)K_{2,k}^\varepsilon(r)(\varphi,\theta,\eta)	d\eta\bigg].
		\end{aligned}
	\end{equation*} 
Consequently,	
\begin{align}\label{dif-psi-3}
	&\partial_{\theta}\big[\partial_r\Psi_{3+k}(\varepsilon,r)[h](\varphi ,\theta)\big]=\varepsilon^3\partial_\theta\Big[\Big(\tfrac{1}{2}\textnormal{Re}\Big\{ \tfrac{e^{\ii 2(\theta+\Theta(\varphi))}}{\mathtt{w}_{3+k}^2(\varphi)}\Big\}+\varepsilon {V}_{2,k}^\varepsilon(r)(\varphi,\theta)\Big)h(\varphi,\theta)\Big]\\
	&+
	\varepsilon^3\partial_\theta\bigg[(-1)^{k+1}\int_{\mathbb{T}}h(\varphi+k\pi,\eta)\textnormal{Re}\Big\{ \Big(\tfrac{e^{\ii(\theta-\eta)}}{\mathtt{w}_{3+k}^2(\varphi)}\Big)\Big\} d\eta+\varepsilon\int_{\mathbb{T}} h (\varphi+k\pi,\eta)K_{2,k}^\varepsilon(r)(\varphi,\theta,\eta)	d\eta\bigg].\nonumber
		\end{align}
By inserting  \eqref{linearized psi  st12}, \eqref{linearized psi 2 st12} and   \eqref{dif-psi-3} into  \eqref{linearized st0}  we obtain \eqref{Edc eq rkPsi-dif} with
		\begin{equation*}
	\begin{aligned}
 {V}_2^\varepsilon(r)(\varphi,\theta)&\triangleq {V}_{2,0}^\varepsilon(r)(\varphi,\theta)+ {V}_{2,1}^\varepsilon(r)(\varphi,\theta)+{V}_{2,2}^\varepsilon(r)(\varphi,\theta),\\
K_{2}^\varepsilon(r)(\varphi,\theta,\eta)&\triangleq  K_{2,0}^\varepsilon(r)(\varphi,\theta,\eta),\\
K_{3}^\varepsilon(r)(\varphi,\theta,\eta) &\triangleq  K_{2,1}^\varepsilon(r)(\varphi,\theta,\eta)+K_{2,2}^\varepsilon(r)(\varphi,\theta,\eta).
		\end{aligned}
	\end{equation*}
From \eqref{def f2} and \eqref{def f2k} we immediately get
\begin{equation}\label{estimate f2}
\|f_2\|_s^{\textnormal{Lip}(\lambda)}+\|f_3\|_s^{\textnormal{Lip}(\lambda)}+\|f_4\|_s^{\textnormal{Lip}(\lambda)}\lesssim 1+\|r\|_{s}^{\textnormal{Lip}(\lambda)}.
\end{equation}	
Hence, the estimate of $K_2^\varepsilon(r)$ and $K_3^\varepsilon(r)$ claimed in \eqref{first est} follow  by the composition laws in Lemma \ref{lem funct prop}.	
Similarly, from \eqref{estimate f2} and the composition laws in Lemma \ref{lem funct prop} we get the estimates on ${V}_2^\varepsilon(r)$ in \eqref{first est0}. 
Finally,  \eqref{hess G}  follows by differentiating \eqref{Edc eq rkPsi-dif}. The corresponding estimates  follow in a straightforward manner. 
This ends the proof of Lemma~\ref{prop:linearized}.
\end{proof}

\subsection{Construction of an approximate solution}\label{Approxim-sol1}

In this section, we shall construct an approximate  solution to the nonlinear equation  $G(r)=0$, given by \eqref{Edc eq rkPsi}, up to an error  of size $O(\varepsilon^5)$. This will be  needed later during the initialization step along  Nash-Moser scheme in order to deal with the growth in $\varepsilon$ of the approximate right inverse of the linearized operator. As we shall see here, the construction is basically implemented in two steps by inverting  the main part of the linearized operator that involves only the spatial variable. However this part is degenerating at the mode one and some  specific cancellation structures are discovered  allowing  to handle this issue. We emphasize that by proceeding in that way one can delay the use   of Cantor sets, which will be used in the next step during the reduction to a Fourier multiplier the full transport operator. There, we shall require that the non constant coefficients parts enjoy suitable smallness in $\varepsilon.$\\
 Before stating our main result on the approximate solutions, we shall discuss an intermediate useful one about  the natural approximation by the radial case. This approximation, on its own, does not yield a sufficiently small error in terms of $\varepsilon.$ To achieve the right approximation, we must introduce appropriate perturbations around this radial state.
	\begin{lemma}\label{lem eq EDC r}
	Let $G$ as in \eqref{Edc eq rkPsi}. Then the following assertion holds true.
\begin{align*}
\forall \, (\varphi,\theta)\in\mathbb{T}^2,\quad {G}(0)(\varphi,\theta)&= \sum_{k\geqslant 2}\tfrac{(-\varepsilon)^{k}}{2}\textnormal{Im}\big\{\mathtt{a}_k(\varphi)e^{\ii k\theta} \big\}, \quad \mathtt{a}_k(\varphi)\triangleq \tfrac{1}{\sqrt{q(\varphi)}^k}-\tfrac{e^{\ii k\Theta(\varphi)}}{\mathtt{w}_3^k(\varphi)}-\tfrac{e^{\ii k\Theta(\varphi)}}{\mathtt{w}_4^k(\varphi)}\cdot
\end{align*}

	\end{lemma}

	\begin{proof}
 Substituting $r=0$ into  \eqref{Edc eq rkPsi} gives
				\begin{align*}
G(0) (\varphi ,\theta)  &=\sum_{n=1}^4\partial_{\theta}\big[ \Psi_n(\varepsilon,0)(\varphi ,\theta)\big].
		\end{align*}
In view of \eqref{def psi1234} we have
\begin{equation*}
		\partial_\theta\big[\Psi_1(\varepsilon,  0)(\varphi ,\theta)\big]=\frac{1}{2\pi }\partial_\theta\bigg[\int_{0}^{2\pi}\int_{0}^{1}\log\big(\big| \varepsilon -\varepsilon l  e^{\ii(\eta-\theta)}\big|\big)l\,  dl  d\eta\bigg]=0
	\end{equation*}
and	
		\begin{equation*}
	\begin{aligned}
		\Psi_2(\varepsilon, 0)(\varphi ,\theta)&=\int_{\mathbb{T}}\int_{0}^{1}\Big[\log\big(\big| 1+\varepsilon \tfrac{e^{\ii\theta}+  l e^{\ii\eta}}{\sqrt{q(\varphi)}}\big|\big)-\varepsilon\textnormal{Re}\big\{\tfrac{e^{\ii\theta}+  l e^{\ii\eta}}{\sqrt{q(\varphi)}}\big\}\Big]l d l d\eta,
		\\
		\Psi_{3+j}(\varepsilon, 0)(\varphi ,\theta)&=-\int_{\mathbb{T}}\int_{0}^{ 1}\Big[\log\Big(\Big| 1+ \tfrac{\varepsilon\big(e^{\ii(\theta+\Theta(\varphi))}+(-1)^{j+1}  l e^{-\ii(\eta+\Theta(\varphi))}\big)}{\mathtt{w}_{3+j}(\varphi)}\Big|\Big)\\ &\qquad \qquad\qquad\qquad -\textnormal{Re}\Big\{ \tfrac{\varepsilon\big(e^{\ii(\theta+ \Theta(\varphi))}+(-1)^{j+1} l e^{-\ii(\eta+\Theta(\varphi))}\big)}{\mathtt{w}_{3+j}(\varphi)}\Big\}\Big]l d l d\eta, \quad k=0,1		
		\end{aligned}
	\end{equation*}	
Using the expansion 
$$
\forall |z|<1,\quad \log|z+1|=\sum_{k=1}^{\infty}\tfrac{(-1)^{k+1}}{k}\textnormal{Re}\big\{z^k\big\}
$$
we get by direct computations,
	\begin{equation*}
	\begin{aligned}
		\partial_\theta\big[\Psi_2(\varepsilon, 0)(\varphi ,\theta)\big]&=\partial_\theta\bigg[\sum_{k=2}^{\infty}\tfrac{(-1)^{k+1}}{  k}\varepsilon^k\int_{\mathbb{T}}\int_{0}^{1}\textnormal{Re}\Big\{\big(\tfrac{e^{\ii\theta}+  l e^{\ii\eta}}{\mathtt{w}_2(\varphi)}\big)^k\Big\}l dl  d\eta\bigg]
		\\ &=\partial_\theta\bigg[\sum_{k=2}^{\infty}\tfrac{(-1)^{k+1}}{2k}\varepsilon^k \textnormal{Re}\Big\{\tfrac{e^{\ii k\theta}}{\sqrt{q(\varphi)}^k}\Big\}\bigg].
	\end{aligned}
	\end{equation*}
	Similarly, we get
				\begin{equation*}
	\begin{aligned}
		\partial_\theta\big[\Psi_{3+j}(\varepsilon,0)(\varphi,\theta)\big]&=\sum_{k=2}^{\infty}\tfrac{(-1)^{k}}{k}\varepsilon^k\partial_\theta\int_{\mathbb{T}}\int_{0}^{ 1}\textnormal{Re}\Big\{ \big(\tfrac{e^{\ii(\theta+\Theta(\varphi))}+(-1)^{j+1}  l e^{-\ii(\eta+\Theta(\varphi))}}{\mathtt{w}_{3+j}(\varphi)}\big)^k\Big\}l d l d\eta\\
		&=\sum_{k=2}^{\infty}\tfrac{(-1)^{k}}{2k}\varepsilon^k\partial_\theta\textnormal{Re}\Big\{ \tfrac{e^{\ii k(\theta+\Theta(\varphi))}}{\mathtt{w}_{3+j}^k(\varphi)}\Big\}.
		\end{aligned}
	\end{equation*}
This completes the proof of Lemma \ref{lem eq EDC r}.		
	\end{proof}
Now, we shall discuss the construction of a good approximation to the solution of \eqref{Edc eq rkPsi}.
\begin{lemma}\label{lem: construction of appx sol}
 Given the conditions \eqref{cond1}-\eqref{cond-interval}. There exists $\varepsilon_0>0$ small enough such that the following occurs. For any  $\varepsilon \in(0,\varepsilon_0)$, there exists   $r_\varepsilon\in \textnormal{Lip}_\lambda\big(\mathcal{O}, C^\infty(\T^2)\big)$ such that, for any  $s>0$
 $$
 \|r_\varepsilon- \mathtt{g}\|_{s}^{\textnormal{Lip}(\lambda)}\lesssim \varepsilon\quad\hbox{and}\quad  \|G( \varepsilon r_\varepsilon)\|_{s}^{\textnormal{Lip}(\lambda)}\lesssim \varepsilon^5,
 $$
 where the function $\mathtt{g}$ is given by \eqref{def g0}.
\end{lemma}
 \begin{proof}
 According to Taylor formula, one gets
 \begin{align*}
G(r)&= G(0)+\partial_rG(0)[r]+\tfrac{1}{2} \partial_r^2G(0)[r,r]+\tfrac{1}{2}\int_0^1(1-\tau)^2\partial_r^3G( \tau r)[r,r,r]d\tau.
 \end{align*}
Then, decomposing $r=\varepsilon r_0+\varepsilon^2 r_1$ leads to
\begin{equation}\label{taylor g}
 \begin{aligned}
&G(\varepsilon r_0+\varepsilon^2 r_1)= G(0)+\varepsilon \partial_rG(0)[r_0]+\tfrac{\varepsilon^2}{2} \partial_r^2G(0)[r_0,r_0]+\varepsilon^2\partial_rG(0)[r_1]\\
 &\quad+\varepsilon^3 \partial_r^2G(0)[r_0,r_1]+\tfrac{\varepsilon^4}{2} \partial_r^2G(0)[r_1,r_1]+\tfrac{1}{2}\int_0^1(1-\tau)^2\partial_r^3G( \tau r)[r,r,r]d\tau.
 \end{aligned}
 \end{equation}
 Applying  Lemma  \ref{prop:linearized} we find
 \begin{align*}
\nonumber & G(0)+\varepsilon\partial_rG(0)[r_0]+\tfrac{\varepsilon^2}{2} \partial_r^2G(0)[r_0,r_0]=G(0)+\tfrac{\varepsilon^2}{2}\big[\partial_\theta -\mathcal{H}\big] r_0+\varepsilon^4\omega_0\partial_\varphi  r_0\\ & -\varepsilon^4\partial_{\theta}\big[\big(\omega_0\Theta+\tfrac12\mathtt{g}-\varepsilon{V}_2^\varepsilon(0)\big) r_0 \big]-\tfrac{\varepsilon^4}{q(\varphi)} \mathcal{Q}_0[r_0]-\tfrac14\varepsilon^4\partial_{\theta}\big(r_0^2\big)+\varepsilon^5 \mathcal{R}_2^\varepsilon(0)[r_0]+\tfrac12\varepsilon^5 \mathcal{E}^\varepsilon_1(0)[r_0,r_0].
 \end{align*}		
At this level, we shall  impose the equation 
 \begin{equation*}
		\tfrac{\varepsilon^2}{2}\big[\partial_\theta -\mathcal{H}\big] r_0+ G(0)=0. 		
\end{equation*}	
where $G(0)$ is described by Lemma \ref{lem eq EDC r}. Notice that this equations is uniquely solvable since $G(0)$ does not contain the mode $1$. In addition, a direct computation gives
	\begin{equation}\label{exp r0}
r_0=\mathtt{g}+\varepsilon B_0,\qquad\textnormal{where}  \qquad 
			B_0(\varphi,\theta)\triangleq \sum_{k=3}^\infty\tfrac{(-1)^{k}}{k-1}\varepsilon^{k-3} \textnormal{Re}\big\{\mathtt{a}_k(\varphi)e^{\ii k\theta} \big\},
\end{equation}	
where $\mathtt{g}$ is defined by \eqref{def g0}.
According to \eqref{def D0}, the operator $\mathcal{Q}_0$ excites only the modes $\pm1$. These modes are entirely absent in the  expression of $r_0$, thus yielding the following result,
$$
\mathcal{Q}_0[r_0]=0.
$$
Combining the last four identities  we obtain
\begin{align}\label{part r0}
&G(0)+\varepsilon\partial_{r }{G}(0) [ r_0 ]+\tfrac{\varepsilon^2}{2}\partial_r^2G(0)[r_0,r_0] =\varepsilon^4 A_1+\varepsilon^5 B_1.
\end{align}
where
\begin{align*}
&  A_1\triangleq \omega_0\partial_\varphi  \mathtt{g}-\tfrac34\partial_{\theta}\big(\mathtt{g}^2\big),\\
&
 B_1\triangleq \omega_0\partial_\varphi B_0-\partial_{\theta}\Big[{\tfrac{1}{4}}(r_0+4\omega_0\dot\Theta+3\mathtt{g})B_0-{V}_2^\varepsilon(0) r_0\Big]+ \mathcal{R}_2^\varepsilon(0)[r_0]+ {\tfrac12}\mathcal{E}^\varepsilon_1(0)[r_0,r_0].
		\end{align*}
Notice that in view of \eqref{exp r0}, the function $ A_1$ has only even modes in space. Thus, one has
$$
\int_0^{2\pi}  A_1(\varphi,\theta)\cos(\theta)d\theta=0\qquad\textnormal{and}\qquad
\int_0^{2\pi}  A_1(\varphi,\theta)\sin(\theta)d\theta=0.
$$				
This latter fact allows to show that the equation 
 		\begin{equation*}
		\tfrac{1}{2}\big[\partial_\theta-\mathcal{H}\big] r_1+\varepsilon A_1=0		
		\end{equation*}
		admits a solution, enjoying the estimate
		\begin{align}\label{r1-estimate}
		\forall s\in\R,\quad \|r_1\|_{s}^{\textnormal{Lip}(\lambda)}\lesssim \varepsilon,
		\end{align}
		and satisfies 
		$$
\mathcal{Q}_0[r_1]=0.
$$
Thus, direct computations, using \eqref{Edc eq rkPsi-dif}, give
 \begin{align}\label{part r1}
\varepsilon^4  A_1+\varepsilon^2 \partial_rG(0)[r_1]&=\varepsilon^5\omega_0\partial_\varphi  r_1 -\varepsilon^5\partial_{\theta}\big[\big(\omega_0\dot{\Theta}+\tfrac12\mathtt{g}-\varepsilon{V}_2^\varepsilon(0)\big) r_1 \big]+\varepsilon^6 \mathcal{R}_2^\varepsilon(0)[r_1].
 \end{align}
Define,
\begin{equation*}
 r_\varepsilon \triangleq r_0+\varepsilon r_1,
\end{equation*}				
then, inserting \eqref{part r0} and \eqref{part r1} into \eqref{taylor g} we conclude that		
		 \begin{align*}
G(\varepsilon r_\varepsilon)&= \varepsilon^5 B_1+\varepsilon^5\omega_0\partial_\varphi  r_1 -\varepsilon^5\partial_{\theta}\big[\big(\omega_0\dot{\Theta}+\tfrac12\mathtt{g}-\varepsilon{V}_2^\varepsilon(0)\big) r_1 \big]+\varepsilon^6 \mathcal{R}_2^\varepsilon(0)[r_1]+\varepsilon^3 \partial_r^2G(0)[r_0,r_1]\\
 &+\tfrac{\varepsilon^4}{2} \partial_r^2G(0)[r_1,r_1]+\tfrac{\varepsilon^3}{2}\int_0^1(1-\tau)^2\partial_r^3G( \tau r_\varepsilon)[r_\varepsilon,r_\varepsilon,r_\varepsilon]d\tau.
 \end{align*}
From \eqref{exp r0} and \eqref{r1-estimate} we immediately get
\begin{equation*}
\|r_0-\mathtt{g}\|_{s}^{\textnormal{Lip}(\lambda)}\lesssim \varepsilon,\quad \|r_1\|_{s}^{\textnormal{Lip}(\lambda)}\lesssim \varepsilon\quad\textnormal{and}\quad \|r_\varepsilon-\mathtt{g}\|_{s}^{\textnormal{Lip}(\lambda)}\lesssim \varepsilon.
\end{equation*}
 The estimate on $G(\varepsilon r_\varepsilon)$ follows  from the preceding estimates together with  \eqref{hess G}. This completes the proof of Lemma \ref{lem: construction of appx sol}.
\end{proof}
\subsection{Rescaled functional and linearization}\label{rescaled func}
In this section we shall work with a new rescaled functional. Then, we will  proceed with its linearization and perform suitable tame estimates. Let $\mu\in(0,1)$ an arbitrary parmeter, then the equation \eqref{Edc eq rkPsi} is equivalent  to
\begin{equation}\label{def nonlinear functional fN}
\mathcal{F}(\rho)\triangleq\tfrac{1}{\varepsilon^{2+\mu}} G( \varepsilon r_\varepsilon+\varepsilon^{1+\mu}\rho)=0.
\end{equation}
The first result deals with  the linearized operator of $\mathcal{F}$ around a small state. 
\begin{proposition}\label{prop-size}
Given the conditions \eqref{cond1}-\eqref{cond-interval}. There exists $\varepsilon_0\in(0,1)$ such that if 
		\begin{equation*}
		\varepsilon\leqslant\varepsilon_0\quad\textnormal{and}\quad\|\rho\|_{s_0+2}^{\textnormal{Lip}(\lambda)}\leqslant 1,
		\end{equation*}
		then
the linearized operator of the map  $\mathcal{F}$, defined by \eqref{def nonlinear functional fN},  at the state $\rho$   in the direction $h$ $($with zero space average$)$ takes the form
\begin{align}
\nonumber\mathcal{L}_0 h(\varphi,\theta) \triangleq\partial_\rho \mathcal{F}(\rho)[h](\varphi,\theta) &=\varepsilon^2\omega_0\partial_\varphi  h (\varphi,\theta)+\partial_{\theta}\big[\mathcal{V}^\varepsilon(\rho)(\varphi,\theta) h (\varphi,\theta)\big]-\tfrac{1}{2}\mathcal{H}[h](\varphi,\theta)\nonumber\\ &\quad-{\varepsilon^2}\partial_\theta\mathcal{Q}_0[h](\varphi,\theta)+\varepsilon^3 \partial_\theta \mathcal{R}^\varepsilon_0(\rho)[h](\varphi,\theta),\label{linearized f}
\end{align}
where the function $\mathcal{V}^\varepsilon(\rho)$ is given by
\begin{equation}\label{def calV}
\begin{aligned}
\mathcal{V}^\varepsilon(\rho)(\varphi,\theta)&\triangleq \tfrac12-{\varepsilon^2}\big(\omega_0 \dot{\Theta} (\varphi)+\mathtt{g}(\varphi,\theta)\big)- \tfrac{\varepsilon^{2+\mu}}{2} \rho(\varphi,\theta)+\varepsilon^{3}{V}^\varepsilon(\rho)(\varphi,\theta),
\end{aligned}
\end{equation}
the function $\mathtt{g}$ is given by \eqref{def g0} and ${V}^\varepsilon(\rho)$ satisfies the estimates: for any $s\geqslant s_0$
\begin{equation}\label{est Veps}
\begin{aligned}
\|{V}^\varepsilon(\rho)\|_{s}^{\textnormal{Lip}(\lambda)}&\lesssim 1+ \varepsilon^\mu\|\rho\|_{s+1}^{\textnormal{Lip}(\lambda)},
\\				
\|\Delta_{12}{V}^\varepsilon(\rho)\|_{s}^{\textnormal{Lip}(\lambda)}&\lesssim \varepsilon^\mu \|\Delta_{12}\rho\|_{s+1}^{\textnormal{Lip}(\lambda)}+ \varepsilon^\mu\|\Delta_{12}\rho\|_{s_0+1}^{\textnormal{Lip}(\lambda)}\max_{\ell\in\{1,2\}}\|\rho_{\ell}\|_{s+1}^{\textnormal{Lip}(\lambda)}.
\end{aligned}
\end{equation}
The operator $\mathcal{Q}_0$ is given by \eqref{def D0} and the operator $\mathcal{R}_0^\varepsilon(\rho)$ is expressed as 
\begin{align*}
\mathcal{R}_0^\varepsilon(\rho)[h](\varphi,\theta)	&\triangleq \sum_{k=0}^1\int_{\mathbb{T}} h (\varphi+k\pi,\eta)K_{0,k}^\varepsilon(\rho)(\varphi,\theta,\eta)	d\eta,
\end{align*}
where the kernel $K_0^\varepsilon(\rho)$ satisfies: for any $s\geqslant s_0$
\begin{align}\label{est K0 R0}
\|K_{0,0}^\varepsilon(\rho)\|_{s}^{\textnormal{Lip}(\lambda)}+\|K_{0,1}^\varepsilon(\rho)\|_{s}^{\textnormal{Lip}(\lambda)}&\lesssim 1+  \varepsilon^\mu\|\rho\|_{s+1}^{\textnormal{Lip}(\lambda)}.
\end{align}

 \end{proposition}		

\begin{proof}
Linearizing the equation  \eqref{def nonlinear functional fN} together with  Lemma \ref{prop:linearized} we infer
\begin{align*}
\partial_\rho \mathcal{F}(\rho)[h](\varphi,\theta)&=\tfrac{1}{\varepsilon} (\partial_rG)(\varepsilon r_\varepsilon+\varepsilon^{1+\mu}\rho)[h](\varphi,\theta)\\ \nonumber&=\varepsilon^2\omega_0\partial_\varphi  h (\varphi,\theta)+\partial_{\theta}\Big[\Big(\tfrac12- \tfrac{\varepsilon^2}{2} r_\varepsilon(\varphi,\theta)- \tfrac{\varepsilon^{2+\mu}}{2} \rho(\varphi,\theta)-\varepsilon^2\omega_0\dot{\Theta}-\tfrac{\varepsilon^2}{2}\mathtt{g}(\varphi,\theta)\\ 
\nonumber&+\varepsilon^2{V}_1^\varepsilon(\varepsilon r_\varepsilon+\varepsilon^{1+\mu}\rho)(\varphi,\theta)+\varepsilon^3{V}_2^\varepsilon(\varepsilon r_\varepsilon+\varepsilon^{1+\mu}\rho)(\varphi,\theta)\Big) h (\varphi,\theta)\Big]-\tfrac{1}{2}\mathcal{H}[h](\varphi,\theta)\nonumber\\ 
\nonumber&-\varepsilon^2\mathcal{Q}_0[h](\varphi,\theta)+\varepsilon^2 \partial_\theta\mathcal{R}_1^\varepsilon(\varepsilon r_\varepsilon+\varepsilon^{1+\mu}\rho)[h](\varphi,\theta)+\varepsilon^3 \partial_\theta\mathcal{R}_2^\varepsilon(\varepsilon r_\varepsilon+\varepsilon^{1+\mu}\rho)[h](\varphi,\theta).
\end{align*}
In view of Lemma \ref{lem: construction of appx sol} and 
from the expression of  $\mathtt{g}(\varphi,\theta)$ described by \eqref{def g0}, we find
\begin{equation}\label{est r2}
 r_\varepsilon(\varphi,\theta)+\mathtt{g}(\varphi,\theta)=2\mathtt{g}(\varphi,\theta)+2\varepsilon r_2(\varphi,\theta) \end{equation}
with 
$$
\|r_2\|_{s}^{\textnormal{Lip}(\lambda)}\lesssim 1.
$$	
Therefore, by setting 
\begin{align*}
{V}^\varepsilon(\rho)&\triangleq {-} r_2+\tfrac1\varepsilon {V}_1^\varepsilon(\varepsilon r_\varepsilon+\varepsilon^{1+\mu}\rho)+{V}_2^\varepsilon(\varepsilon r_\varepsilon+\varepsilon^{1+\mu}\rho),\\
\mathcal{R}_0^\varepsilon(\rho)[h]&\triangleq \tfrac1\varepsilon\mathcal{R}_1^\varepsilon(\varepsilon r_\varepsilon+\varepsilon^{1+\mu}\rho)[h]+\mathcal{R}_2^\varepsilon(\varepsilon r_\varepsilon+\varepsilon^{1+\mu}\rho)[h],
\end{align*}
we deduce \eqref{linearized f} and \eqref{def calV}. The estimates \eqref{est Veps} follow from \eqref{first est0} and \eqref{est r2}. Concerning the opereator $\mathcal{R}_0^\varepsilon(\rho)$, we write
\begin{align*}
\mathcal{R}_0^\varepsilon(\rho)[h](\varphi,\theta)	&\triangleq \sum_{k=0}^1\int_{\mathbb{T}} h (\varphi+k\pi,\eta)K_{0,k}^\varepsilon(\rho)(\varphi,\theta,\eta)	d\eta,
\end{align*}
with
$$
\begin{aligned}
K_{0,0}^\varepsilon(\rho)\triangleq\tfrac1\varepsilon K_1^\varepsilon(\varepsilon r_\varepsilon+\varepsilon^{1+\mu}\rho)+K_2^\varepsilon(\varepsilon r_\varepsilon+\varepsilon^{1+\mu}\rho),
\\
K_{0,1}^\varepsilon(\rho)\triangleq\tfrac1\varepsilon K_3^\varepsilon(\varepsilon r_\varepsilon+\varepsilon^{1+\mu}\rho).
\end{aligned}
$$
Thus, the estimates of $K_{0,0}^\varepsilon(\rho)$ and $K_{0,1}^\varepsilon(\rho)$ can be obtained through the composition laws from \eqref{first est}. This ends the proof of Proposition~\ref{prop-size}.
\end{proof}
The estimates below on the linearized operator follow easily  from Proposition \ref{prop-size} and \eqref{hess G}.
	\begin{lemma}\label{Tame-estimates-F}
		Given the conditions \eqref{cond1}-\eqref{cond-interval}. There exists $\varepsilon_0\in(0,1)$ such that if 
		$$\varepsilon\leqslant\varepsilon_0\quad\textnormal{and}\quad\|\rho\|_{s_0+2}^{\textnormal{Lip}(\lambda)}\leqslant 1,$$ 
		then   the following  estimates hold true: for any $s\geqslant s_0$,
		\begin{enumerate}
			\item $\big\| \partial_{\rho}\mathcal{F}(\rho)[h]\big\|_{s}^{\textnormal{Lip}(\lambda)}\lesssim \|\,h\,\|_{s+2}^{\textnormal{Lip}(\lambda)}+\|\rho\|_{s+2}^{\textnormal{Lip}(\lambda)}\| h\,\|_{s_0+2}^{\textnormal{Lip}(\lambda)}.$
			\item $\big\| \partial_{\rho}^{2}\mathcal{F}(\rho)[\,h, h\,]\big\|_{s}^{\textnormal{Lip}(\lambda)}\lesssim \varepsilon^2 \|\,h\,\|_{s+2}^{\textnormal{Lip}(\lambda)}\|\,h\,\|_{s_0+2}^{\textnormal{Lip}(\lambda)}+\varepsilon^2\|\,\rho\,\|_{s+2}^{\textnormal{Lip}(\lambda)}\big(\|\,h\,\|_{s_0+2}^{\textnormal{Lip}(\lambda)}\big)^2.$
	\end{enumerate}
\end{lemma}

 \section{Reduction of the transport part}\label{Sec-Red-Tran}
 This section is devoted to the straightening of the transport part  of the operator $\mathcal{L}_0$ defined in   \eqref{linearized f}. This will be performed  through some steps using various transformations.   
 \subsection{Change of coordinates of the first kind}\label{Sec-Changeof-coordinates1}
We plan to reduce to  a Fourier multiplier  the transport part of the operator $\mathcal{L}_0$  introduced in Proposition \ref{prop-size}. Currently, implementing the KAM scheme, as done in \cite{FGMP19,HHM21,HR21}, is not feasible due to the presence of multiple degeneracies in both the operator and the Cantor sets that will be utilized. One of the crucial factor  is that some coefficients in $\mathcal{L}_0$   scales exactly at the rate $\varepsilon^2$ which is incompatible with the requirements of   KAM scheme. The strategy involves first   intermediate partial change of coordinates to avoid spatial time resonance and  prepare the new operator for subsequent application of the KAM scheme. This will be the goal of this  section and our first main result reads as follows.
 \begin{proposition}\label{QP-change0}
 Let $\mathcal{V}^\varepsilon(\rho) $ be as in \eqref{def calV} and assume that the conditions \eqref{cond1}, \eqref{cond-interval} hold.  There exists ${\varepsilon}_0>0$ 
such that if 
		$$\varepsilon\leqslant\varepsilon_0\quad\textnormal{and}\quad\|\rho\|_{s_0+2}^{\textnormal{Lip}(\lambda)}\leqslant 1,$$ 
then  there exist  $\mathfrak{b}$ in the form
$$
\mathfrak{b}(\varphi,\theta)={\Theta}(\varphi)-\varphi+\varepsilon\widehat{\Theta}_1(\varphi)+\varepsilon^2\textnormal{Im}\big\{\mathtt{a}_2(\varphi)e^{\ii 2\theta} \big\}+{\varepsilon^{2+\mu}}\mathfrak{r}(\varphi,\theta),$$
where $\mathtt{a}_2(\varphi)$ is defined by \eqref{def g0},
 and 
\begin{equation}\label{change of variables0}
\begin{aligned}
\mathfrak{B} h(\varphi,\theta) & \triangleq \big(1+\partial_{\theta}\mathfrak{b}(\varphi,\theta)\big) h\big(\varphi,\theta+\mathfrak{b}(\varphi,\theta)\big) 
\end{aligned}
\end{equation}
such that 
\begin{align}\label{transformation KAM step transport2}
\mathfrak{B}^{-1}\Big(\varepsilon^2\omega_0\partial_\varphi&+\partial_\theta\big[\mathcal{V}^\varepsilon(\rho) \cdot \big]\Big)\mathfrak{B}=\varepsilon^2\omega_0\partial_\varphi +\partial_{\theta}\big[\big(\mathtt{c}_0+{\varepsilon^{6} }{\mathcal{V}}_0^\varepsilon({\rho}) \big)\cdot\big].
\end{align}
Furthermore, we have
\begin{enumerate}
\item The constant $\mathtt{c}_0$ is given by
\begin{equation}\label{def c0}
\mathtt{c}_0\triangleq \tfrac12-\varepsilon^2\omega_0+\varepsilon^{3}\mathtt{c}_1, \quad \| \mathtt{c}_1 \|^{\textnormal{Lip}(\lambda)}\lesssim 1.\end{equation}
\item  The functions $\mathfrak{b}$, $\mathfrak{r}$, $\widehat{\Theta}_1$ and  ${\mathcal{V}}_0^\varepsilon({\rho})$  satisfy the following estimates for all $s\in[s_{0},S],$ 
\begin{equation}\label{est b1 b2 b b-theta}
\begin{aligned}
\|\mathfrak{b}\|_{s}^{\textnormal{Lip}(\lambda)} +\|\widehat{\Theta}_1\|_{s}^{\textnormal{Lip}(\lambda)}+\|\mathfrak{r}\|_{s}^{\textnormal{Lip}(\lambda)} \lesssim 1+\|\rho\|_{s+{2}}^{\textnormal{Lip}(\lambda)}\\
\|{\mathcal{V}}_0^\varepsilon({\rho})\|_{s}^{\textnormal{Lip}(\lambda)} \lesssim 1+\|\rho\|_{s+{3}}^{\textnormal{Lip}(\lambda)}
\end{aligned}
\end{equation}
\item The transformations $\mathfrak{B}^{\pm 1}$  satisfy the following estimates for all $s\in[s_{0},S],$
\begin{equation}\label{est B1B2B}
\begin{aligned}
\|\mathfrak{B}^{\pm 1} h\|_{s}^{\textnormal{Lip}(\lambda)}
 \lesssim \|h\|_{s}^{\textnormal{Lip}(\lambda)}+\varepsilon\| \rho\|_{s+{3}}^{\textnormal{Lip}(\lambda)}\|h\|_{s_{0}}^{\textnormal{Lip}(\lambda)}.
\end{aligned}
\end{equation} 
	\item Given two small states $\rho_{1}$ and $\rho_{2}$, then 
				\begin{align}\label{diff Vpm 00}
				\|\Delta_{12}\mathtt{c}_0\|^{\textnormal{Lip}(\lambda)}&\lesssim\varepsilon^{3}\| \Delta_{12}\rho\|_{s_0+2}^{\textnormal{Lip}(\lambda)},
				\\ 
				\label{ain-z10}
				\|\Delta_{12}\mathcal{V}_0^{\varepsilon }\|_{s}^{\textnormal{Lip}(\lambda)}
&\lesssim\Big(\|\Delta_{12} \rho\|_{s+3}^{\textnormal{Lip}(\lambda)}+{ \|\Delta_{12} \rho\|_{s_0+3}^{\textnormal{Lip}(\lambda)}
\max_{\ell=1,2} \|\rho_\ell\|_{s+3}^{\textnormal{Lip}(\lambda)}}
\Big).
			\end{align}
\end{enumerate}

\end{proposition}

\begin{proof} 
Consider the symplectic periodic change of coordinates close to the identity in the form
\begin{equation}\label{change of variables}
\begin{array}{rcl}
\mathfrak{B} h(\varphi,\theta) &=& \big(1+\partial_{\theta}\mathfrak{b}(\varphi,\theta)\big){\rm {B}} h(\varphi,\theta)\\
& \triangleq& \big(1+\partial_{\theta}\mathfrak{b}(\varphi,\theta)\big) h\big(\varphi,\theta+\mathfrak{b}(\varphi,\theta)\big).
\end{array}
\end{equation}
The goal is to construct this suitable change of coordinates to reduce the size of the coefficients. It will be  a smooth diffeomorphism of the space $X^s$ introduced in \eqref{Xs-space} and  its symplectic structure is  useful to  ensure the persistence  of the spatial zero average. Actually, one has
$$
\int_{\mathbb{T}}h(\varphi,\theta) d\theta=0\Longrightarrow \int_{\mathbb{T}}\mathfrak{B} h(\varphi,\theta) d\theta=0.
$$
According to Lemma \ref{algeb1}, the conjugation of the transport operator by $\mathfrak{B}$  writes
$$
\mathfrak{B}^{-1}\Big(\varepsilon^2\omega_0\partial_\varphi+\partial_\theta\big[\mathcal{V}^\varepsilon(\rho) \cdot \big]\Big)\mathfrak{B}=\varepsilon^2\omega_0\partial_\varphi+\partial_y \Big[{\rm B}^{-1}\big(\mathcal{V}_1^\varepsilon(\rho)\big)\cdot\Big],
$$
where, by  \eqref{def calV}, 
\begin{align*}
\mathcal{V}_1^\varepsilon(\rho)&\triangleq \varepsilon^2\omega_0\partial_{\varphi} \mathfrak{b}+\mathcal{V}^\varepsilon(\rho)\big(1+\partial_\theta \mathfrak{b}\big)
\\
&=\varepsilon^2\omega_0\partial_\varphi \mathfrak{b}+\Big(\tfrac12-\varepsilon^2\omega_0 \dot{\Theta} (\varphi)-{\varepsilon^2}\mathtt{g}(\varphi,\theta)-\tfrac{\varepsilon^{2+\mu}}{2}\rho+\varepsilon^{3}{V}^\varepsilon(\rho)\Big)\big(1+\partial_\theta \mathfrak{b} \big).
\end{align*}
Let us decompose $\mathfrak{b}$ in the form
\begin{equation}\label{def bfrak}
 \mathfrak{b}(\varphi,\theta)=\varepsilon^{2} \mathfrak{b}_{1,1}(\varphi,\theta)+\varepsilon^{4} \mathfrak{b}_{1,2}(\varphi,\theta)+\mathfrak{b}_{2,1}(\varphi)+\varepsilon^{2}\mathfrak{b}_{2,2}(\varphi)
\end{equation}
and we shall   impose   the following constraints 
\begin{eqnarray}  \label{equation satisfied by g-21}
       \left\{\begin{array}{ll}
          	&\tfrac12\,\partial_{\theta}\mathfrak{b}_{1,1}=\underbrace{\omega_0 \dot{\Theta}+\mathtt{g}+ \tfrac{1}{2}\varepsilon^\mu \rho-\varepsilon {V}^\varepsilon(\rho)}_{\triangleq F_1}-\langle F_1\rangle_{\theta}, \\
&\omega_0\partial_\varphi \mathfrak{b}_{2,1}=\langle F_1\rangle_{\theta}-\langle F_1\rangle_{\theta,\varphi}
       \end{array}\right.
\end{eqnarray}	
and
\begin{eqnarray}  \label{equation satisfied by g-22}
       \left\{\begin{array}{ll}
          	&\tfrac12\,\partial_{\theta}\mathfrak{b}_{1,2}=\underbrace{{-}\omega_0\partial_\varphi\mathfrak{b}_{1,1}+ (\omega_0 \dot{\Theta}+\mathtt{g})\partial_{\theta}\mathfrak{b}_{1,1}+ \tfrac{1}{2}\varepsilon^\mu \rho\partial_{\theta}\mathfrak{b}_{1,1}-\varepsilon {V}^\varepsilon(\rho)\partial_{\theta}\mathfrak{b}_{1,1}}_{\triangleq F_2}-\langle  F_2\rangle_{\theta}, \\
&\omega_0\partial_\varphi \mathfrak{b}_{2,2}=\langle  F_2\rangle_{\theta}-\langle  F_2\rangle_{\theta,\varphi}.
       \end{array}\right.
\end{eqnarray}	
Thus, straightforward computations yield
$$
\mathcal{V}_1^\varepsilon(\rho)=\tfrac12-\varepsilon^2\langle F_1\rangle_{\theta,\varphi}-\varepsilon^4\langle F_2\rangle_{\theta,\varphi} +\varepsilon^6\mathcal{V}_2^\varepsilon(\rho)
$$ 
with
\begin{align}\label{dim-11}
\mathcal{V}_2^\varepsilon(\rho)=\omega_0\partial_\varphi  \mathfrak{b}_{1,2} -\big(\omega_0 \dot{\Theta}+\mathtt{g}+\tfrac12\varepsilon^\mu \rho +\varepsilon V^\varepsilon(\rho)\big)\partial_\theta \mathfrak{b}_{1,2}.
\end{align}
Using  the facts $\displaystyle{\langle\dot\Theta\rangle_{\varphi} =1, \langle \rho\rangle_{\theta}=0}$ and \eqref{def calV}, we infer
$$
\mathcal{V}_1^\varepsilon(\rho)=\tfrac12-\varepsilon^2\omega_0+\varepsilon^{3} \langle {V}^\varepsilon(\rho)\rangle_{\theta,\varphi}-\varepsilon^4\langle F_2\rangle_{\theta,\varphi} +\varepsilon^6 \mathcal{V}_2^\varepsilon(\rho).
$$
 By defining
\begin{equation}\label{def f0}
\mathcal{V}_0^\varepsilon({\rho})\triangleq {\rm B}^{-1}\mathcal{V}_2^\varepsilon(\rho)\quad\textnormal{and}\quad \mathtt{c}_0\triangleq \tfrac12-\varepsilon^2\omega_0+\varepsilon^{3} \underbrace{\big(\langle {V}^\varepsilon(\rho)\rangle_{\theta,\varphi}-\varepsilon\langle F_2\rangle_{\theta,\varphi} \big)}_{\triangleq \mathtt{c}_1}\end{equation}
 we get  the identity \eqref{transformation KAM step transport2}.
To solve  the first equation in \eqref{equation satisfied by g-21}, we write in view of \eqref{def calV}
$$
\tfrac12\,\partial_{\theta}\mathfrak{b}_{1,1}(\varphi,\theta)= \mathtt{g}+ \tfrac{1}{2} {\varepsilon^\mu}\rho-\varepsilon {V}^\varepsilon(\rho)+\varepsilon \langle  {V}^\varepsilon(\rho)\rangle_{\theta},
$$
where $\mathtt{g}$ is given by \eqref{def g0}.
Notice that $\rho\in X^s$ is  with  zero spatial average. Thus  there is only one solution to this equation with zero  spatial average, that can be for instance  solved using Fourier expansion,  and taking the form
\begin{align*}
\mathfrak{b}_{1,1}(\varphi,\theta)=\textnormal{Im}\big\{\mathtt{a}_2(\varphi)e^{\ii 2\theta} \big\}+{\varepsilon^\mu}\mathfrak{r}_{1}(\varphi,\theta),
\end{align*}
where the last term can be estimated as follows according to \eqref{est Veps},
\begin{align*}
\|\mathfrak{r}_{1}\|_{s}^{\textnormal{Lip}(\lambda)}\lesssim  1+ \varepsilon^\mu\|\rho\|_{s+1}^{\textnormal{Lip}(\lambda)}.
\end{align*}
As to the second equation in \eqref{equation satisfied by g-21}, we first write  from \eqref{def calV},
$$
\langle \mathtt{g}\rangle_{\theta}=0, \quad\hbox{and}\quad \langle \Theta\rangle_{\theta,\varphi}=\omega_0.
$$
Thus we deduce
$$
\omega_0\partial_\varphi \mathfrak{b}_{2,1}=\omega_0\big(\dot\Theta(\varphi)-1\big) +\varepsilon \big[\langle  {V}^\varepsilon(\rho)\rangle_{\theta,\varphi}-\langle  {V}^\varepsilon(\rho)\rangle_{\theta}\big].
$$
It follows that
\begin{equation*}
\mathfrak{b}_{2,1}(\varphi)={\Theta}(\varphi)-\varphi+\varepsilon\,\mathfrak{r}_{2}(\varphi)
\end{equation*}
with the following estimate that can derived from \eqref{est Veps},
$$
\|\mathfrak{r}_{2}\|_{s}^{\textnormal{Lip}(\lambda)}\lesssim  1+ \varepsilon^\mu\|\rho\|_{s}^{\textnormal{Lip}(\lambda)}.
$$
As a byproduct of the preceding discussion, we infer
\begin{align}\label{est b1 b2}
\|\mathfrak{b}_{1,1}\|_{s}^{\textnormal{Lip}(\lambda)}\lesssim 1+\|\rho\|_{{s+1}}^{\textnormal{Lip}(\lambda)}, \quad \|\mathfrak{b}_{2,1}\|_{s}^{\textnormal{Lip}(\lambda)}\lesssim \varepsilon\big(1+\|\rho\|_{{s}}^{\textnormal{Lip}(\lambda)}\big).
\end{align}
Similar arguments can be applied to the system \eqref{equation satisfied by g-22} and one has
\begin{align*}
\|\mathfrak{b}_{1,2}\|_{s}^{\textnormal{Lip}(\lambda)}\lesssim 1+\|\rho\|_{{s+2}}^{\textnormal{Lip}(\lambda)},\quad \|\mathfrak{b}_{2,2}\|_{s}^{\textnormal{Lip}(\lambda)}\lesssim 1+\|\rho\|_{{s+1}}^{\textnormal{Lip}(\lambda)}.
\end{align*}
Combining the preceding estimates and \eqref{def bfrak} we find
$$
\|\mathfrak{b}\|_{s}^{\textnormal{Lip}(\lambda)}\lesssim 1+\|\rho\|_{{s+2}}^{\textnormal{Lip}(\lambda)}.
$$
Thus, applying  Lemma  \ref{Compos1-lemm}-$1$ leads to \eqref{est B1B2B}. 
Now, from \eqref{equation satisfied by g-22}, \eqref{def f0}, \eqref{est b1 b2} and \eqref{est Veps} we find
$$
\| \mathtt{c}_1 \|^{\textnormal{Lip}(\lambda)}\lesssim 1.
$$
Moreover, by \eqref{dim-11}, \eqref{est Veps}, \eqref{est b1 b2}, and the product laws in Lemma \ref{lem funct prop}  we get
\begin{align*}
  \|\mathcal{V}_2^\varepsilon(\rho) \|_{s}^{\textnormal{Lip}(\lambda)}&\lesssim 
 1+\|{\rho} \|_{s+{3}}^{\textnormal{Lip}(\lambda)}.
\end{align*}
Applying  once again Lemma  \ref{Compos1-lemm}-$1$ and using the last estimate and the smallness condition \mbox{on $\rho$,} we obtain  
\begin{align*}
\|\mathcal{V}_0^\varepsilon({\rho})\|_{s}^{\textnormal{Lip}(\lambda)}&=\|{\rm B}^{-1}\mathcal{V}_2^\varepsilon(\rho) \|_{s}^{\textnormal{Lip}(\lambda)}\\
& \leqslant  \|\mathcal{V}_2^\varepsilon(\rho) \|_{s}^{\textnormal{Lip}(\lambda)}\big(1+C\|\mathfrak{b}\|_{s_0}^{\textnormal{Lip}(\lambda)}\big)+C\|\mathfrak{b}\|_{s}^{\textnormal{Lip}(\lambda)}\|\mathcal{V}_2^\varepsilon(\rho) \|_{s_0}^{\textnormal{Lip}(\lambda)}\Big)\nonumber\\
&\lesssim 1+\|\rho\|_{s+{3}}^{\textnormal{Lip}(\lambda)}.
\end{align*}
Similarly,  from Lemma \ref{Compos1-lemm}-$3$, the second estimate in \eqref{est Veps}, the smallness condition on $\rho$ we conclude  \eqref{diff Vpm 00}. Finally, by \eqref{def c0} and  \eqref{est Veps} we find  \eqref{ain-z10}.
One also gets
\begin{equation}\label{b1-expand}
\begin{aligned}
\mathfrak{b}(\varphi,\theta)&=\underbrace{{\Theta}(\varphi)-\varphi}_{\triangleq \widehat{\Theta}(\varphi)}+ \varepsilon\,\underbrace{\big(\mathfrak{r}_{2}(\varphi)+\varepsilon\mathfrak{b}_{2,2}(\varphi)\big)}_{\triangleq \widehat{\Theta}_1(\varphi)}+\varepsilon^2\underbrace{\textnormal{Im}\big\{\mathtt{a}_2(\varphi)e^{\ii 2\theta} \big\}}_{\triangleq \mathfrak{b}_{0}(\varphi,\theta)}\\ &\quad +{\varepsilon^{2+\mu}}\underbrace{\big(\mathfrak{r}_{1}(\varphi,\theta)+\varepsilon^{2-\mu}\mathfrak{b}_{1,2}(\varphi,\theta)\big)}_{\triangleq \mathfrak{r}(\varphi,\theta)}.
\end{aligned}
\end{equation}
Notice that
\begin{equation}\label{mathscrB1}
 \mathfrak{B}^{-1}h(\varphi,y)=\Big(1+\varepsilon^2\partial_y\widehat{\mathfrak{b}}(\varphi,y)\Big)h\big(\varphi,y+\widehat{\mathfrak{b}}(\varphi,y)\big)
\end{equation}
with 
\begin{align}\label{IInv-77}
 y=\theta+\mathfrak{b}(\varphi,\theta)\Longleftrightarrow \theta=y+\widehat{\mathfrak{b}}(\varphi,y).
\end{align} 
Now, we intend  to describe the asymptotics of $\widehat{\mathfrak{b}}$ that will be used later in Proposition \ref{lemma-beta1}. 
Using \mbox{Lemma \ref{beta-inv-asym}} allows to get
\begin{align}\label{inverse-thetaL}
\widehat{\mathfrak{b}}(\varphi,\theta)=-\underbrace{\big(\widehat{\Theta}(\varphi)+ \varepsilon\,\widehat{\Theta}_1(\varphi)\big)}_{\triangleq \Theta_1(\varphi)}-\varepsilon^2\mathfrak{b}_{0}\big(\varphi,\theta- \widehat{\Theta}(\varphi)\big)+\varepsilon^{2+\mu}\mathfrak{b}_{1}\big(\varphi,\theta\big)
\end{align}
with the estimate
\begin{align}\label{inverse-thetaML}
\|\mathfrak{b}_1\|_{s}^{\textnormal{Lip}(\lambda)}\lesssim 1+\|\rho\|_{{s+3}}^{\textnormal{Lip}(\lambda)}.
\end{align}
This achieves the proof of Proposition~\ref{QP-change0}.

\end{proof}
\subsection{Change of coordinates of the second kind}\label{Section-Change2}
In section \ref{Sec-Changeof-coordinates1}, we successfully  implemented a first change of variables  in order to reduce the size of the variable coefficients of the transport part in \eqref{transformation KAM step transport2} to the order $\varepsilon^6$. Our current objective is to achieve a complete reduction of this new operator  by employing a KAM scheme, following an approach  akin to the one detailed in \cite{FGMP19,HHM21,HR21}, with some adaptations due to the degeneracy in the time direction. Before stating our main result, let us first introduce the following sequence needed later in the cut-off projectors,
\begin{equation}\label{definition of Nm}
N_{-1}\triangleq 1,\quad \forall n\in\mathbb{N},\quad N_{n}\triangleq N_{0}^{\left(\frac{3}{2}\right)^{n}},\quad N_0\geqslant2.
 \end{equation}
 The  following result follows immediately  from the general statement of \mbox{Proposition \ref{Thm transport}.}
\begin{proposition}\label{QP-change}
Given the conditions \eqref{cond1}-\eqref{cond-interval} and let $\mu_2$ a real number  satisfying 
\begin{equation}\label{Conv-Trans0}
\mu_{2}\geqslant 4\tau+3. 
\end{equation}
There exists ${\varepsilon}_0>0$ such that if    
\begin{align}
\label{small-C2-0}\|\rho\|_{\frac{3}{2}\mu_{2}+2s_{0}+2\tau+{4}}^{\textnormal{Lip}(\lambda)}\leqslant1, \quad \lambda \varepsilon^{{-}2}\leqslant 1\quad\textnormal{and}\quad N_{0}^{\mu_{2}}\varepsilon^{{6}} {\lambda^{-1}}\leqslant{\varepsilon}_0,
\end{align}
then, we can construct $\mathtt{c}\triangleq \mathtt{c}(\xi_0,\rho)\in \textnormal{Lip}_\lambda(\mathcal{O},\mathbb{R})$ and $\beta\in \bigcap_{{s\in[s_{0},S]}} {\textnormal{Lip}_\lambda(\mathcal{O},H^{s})}$
such that with $\mathscr{B}$ as in \eqref{definition symplectic change of variables} one gets the following results.
\begin{enumerate}
\item  The function $\xi_0\mapsto\mathtt{c}(\xi_0,\rho)$ $($which is constant with respect to the time-space variables$)$  satisfies the following estimate, 
\begin{equation*}
\| \mathtt{c}-\mathtt{c}_0 \|^{\textnormal{Lip}(\lambda)}\lesssim {\varepsilon}^{{6}},
\end{equation*}
where $\mathtt{c}_0$ is defined in \eqref{def c0}.
\item The transformations $\mathscr{B}^{\pm 1}, 
 {\beta}$ and the functions  $\widehat{\beta}$ satisfy, for all $s\in[s_{0},S],$ 
\begin{equation*}
\|\mathscr{B}^{\pm 1}h\|_{s}^{\textnormal{Lip}(\lambda)}
 \lesssim\|h\|_{s}^{\textnormal{Lip}(\lambda)}+{\varepsilon^{{6}}\lambda ^{-1}}\| \rho\|_{s+2\tau+{5}}^{\textnormal{Lip}(\lambda)}\|h\|_{s_{0}}^{\textnormal{Lip}(\lambda)}
\end{equation*}
and 
\begin{equation}\label{est-beta-r0}
\|\widehat{\beta}\|_{s}^{\textnormal{Lip}(\lambda)}\lesssim\|\beta\|_{s}^{\textnormal{Lip}(\lambda)}\lesssim \varepsilon^{{6}}\lambda^{-1}\left(1+\| \rho\|_{s+2\tau+{4}}^{\textnormal{Lip}(\lambda)}\right).
\end{equation}
\item Let $n\in\N$, then on  the Cantor set
\begin{equation}\label{Cantor set0}
{\mathcal{O}_{n}^{1}(\rho)}=\bigcap_{(\ell,j)\in\mathbb{Z}^{2}\atop 1\leqslant |j|\leqslant N_{n}}\Big\lbrace \xi_0\in \mathcal{O};\;\, \big|\varepsilon^2\omega(\xi_0)  \ell+j \mathtt{c}(\xi_0,\rho)\big|\geqslant{\lambda}{| j|^{-\tau}}\Big\rbrace
\end{equation}
we have 
\begin{align*}
\mathscr{B}^{-1}\Big(&\varepsilon^2\omega_0\partial_\varphi +\partial_{\theta}\big[\big(\mathtt{c}_0+\varepsilon^{6} {\mathcal{V}}^\varepsilon({\rho}) \big)\cdot\big]\Big)\mathscr{B}=\varepsilon^2\omega_0\partial_\varphi+ \mathtt{c}(\xi_0,\rho)\partial_{\theta}+\mathtt{E}_{n}
\end{align*}
with $\mathtt{E}_{n}$ a linear operator satisfying
\begin{equation}\label{estimate En0}
\|\mathtt{E}_{n}h\|_{s_0}^{\textnormal{Lip}(\lambda)}\lesssim \varepsilon^{{6}} N_{0}^{\mu_{2}}N_{n+1}^{-\mu_{2}}\|h\|_{s_{0}+2}^{\textnormal{Lip}(\lambda)}.
\end{equation}
	\item Given two functions $\rho_{1}$ and $\rho_{2}$ both satisfying \eqref{small-C2-0}, we have 
			\begin{align}\label{diff Vpm}
				\|\Delta_{12}\mathtt{c}\|^{\textnormal{Lip}(\lambda)}&\lesssim\varepsilon^{3}\| \Delta_{12}\rho\|_{2{s}_{0}+2\tau+3}^{\textnormal{Lip}(\lambda)}.
			\end{align}
\end{enumerate}
\end{proposition}
The next goal is to describe the asymptotic structure of the operator $\mathcal{L}_0$  in \eqref{linearized f} after successive conjugation by the transformations of first and second kind.
\begin{proposition}\label{lemma-beta1} 
Given the conditions \eqref{cond1}, \eqref{cond-interval}, \eqref{Conv-Trans0} and \eqref{small-C2-0}.  Let $\mathfrak{B}$ as in \eqref{change of variables0} and $\mathscr{B}$ as in Proposition \ref{QP-change}. Then, the following claims  hold true.
\begin{enumerate}
\item The operator 
\begin{equation*}
\mathscr{A}\triangleq \mathfrak{B}\mathscr{B}
\end{equation*}
writes 
\begin{align*}
\mathscr{A} h(\varphi,\theta)&= \big(1+\partial_{\theta}\mathfrak{a}(\varphi,\theta)\big)  h\big(\varphi,\theta+\mathfrak{a}(\varphi,\theta)\big),
\end{align*}
with the estimate
\begin{equation*}
\| \mathfrak{a}\|_{s}^{\textnormal{Lip}(\lambda)}\lesssim  
1+\| \rho\|_{s+2\tau+{4}}^{\textnormal{Lip}(\lambda)}.
\end{equation*}
Moreover, $\mathscr{A}$ is invertible and satisfies the estimate
\begin{align*}
\|\mathscr{A}^{\pm 1} h\|_{s}^{\textnormal{Lip}(\lambda)}& \lesssim   \|h\|_{s}^{\textnormal{Lip}(\lambda)}+ \| \rho\|_{s+2\tau+{5}}^{\textnormal{Lip}(\lambda)}\|h\|_{s_0}^{\textnormal{Lip}(\lambda)}.
\end{align*}
\item For any $\xi_0$ in the Cantor set ${\mathcal{O}_{n}^{1}(\rho)}$, defined in \eqref{Cantor set0},
one has 
\begin{align*}
\mathcal{L}_1\triangleq\mathscr{A}^{-1}\mathcal{L}_0\mathscr{A}&= \varepsilon^2\omega_0\partial_\varphi +\mathtt{c}(\xi_0,\rho)\partial_{\theta}-\tfrac{1}{2}\mathcal{H}-{\varepsilon^2 } \partial_\theta\mathcal{Q}_1 +{\partial_\theta}\mathcal{R}_{1}+\mathtt{E}_{n}
\end{align*}
where the linear operator $\mathtt{E}_{n}$ is the same as in Proposition \ref{QP-change} and 
\begin{align*}
\mathcal{Q}_1[h](\varphi,\theta)&\triangleq  \int_{\mathbb{T}} h( \varphi, \eta) \textnormal{Re}\Big\{\tfrac{e^{\ii(\theta-\eta)}}{\mathtt{w}_{3}^2(\varphi)}-\tfrac12\Big(\tfrac{e^{-\ii 2 \widehat{\Theta}(\varphi)}}{q(\varphi)}-\tfrac{e^{\ii 2\varphi}}{\mathtt{w}_3^2(\varphi)}-\tfrac{e^{\ii 2\varphi}}{\mathtt{w}_4^2(\varphi)}\Big)e^{\ii (\theta+\eta)} \Big\}  \,d\eta\nonumber
\\
&\quad+\int_{\mathbb{T}}h(\varphi+\pi,\eta)\textnormal{Re}\Big\{ \tfrac{e^{\ii(\theta+\eta-2\widehat\Theta(\varphi))}}{q(\varphi)}-\tfrac{e^{\ii(\theta-\eta)}}{\mathtt{w}_{4}^2(\varphi)}\Big\} d\eta.
\end{align*}
In addition, the operator $\partial_\theta \mathcal{R}_{1}$ satisfies  the estimates: for any $s\in[s_0,S],\,N\geqslant0$, 
\begin{equation*}
\begin{aligned}
\| \partial_\theta \mathcal{R}_{1} h\|_{s,N}^{\textnormal{Lip}(\lambda)}&\lesssim \big(\varepsilon^{2+{\mu}} +\varepsilon^{6} \lambda^{-1}\big) \Big(\|h\|_{{s}}^{\textnormal{Lip}(\lambda)}\big(1+ \|\rho\|_{s_0+2\tau+{5}+N}^{\textnormal{Lip}(\lambda)}\big) + \|h\|_{{s_0}}^{\textnormal{Lip}(\lambda)}\|\rho\|_{s+2\tau+{5}+N}^{\textnormal{Lip}(\lambda)} \Big).
\end{aligned}
\end{equation*}

\end{enumerate}

\end{proposition}

\begin{proof}

${\bf 1}.$ Applying Lemma \ref{algeb1} leads to \eqref{change of variables0} with
$$
{\mathfrak{a}(\varphi,\theta)\triangleq \mathfrak{b}(\varphi,\theta)+\beta\big(\varphi,\theta+\mathfrak{b}(\varphi,\theta)\big).}
$$
Hence, using  Lemma \ref{Compos1-lemm}-$1$,   \eqref{est b1 b2 b b-theta} and  \eqref{est-beta-r0},  we get the desired estimates.\\
${\bf 2}.$ In view of Propositions \ref{prop-size}, \ref{QP-change0}, \ref{QP-change}, we deduce that for all $\xi_0\in \mathcal{O}_{n}^{1}(\rho)$ 
\begin{align}\label{conjug A}
\mathscr{A}^{-1}\mathcal{L}_0\mathscr{A}&=\varepsilon^2\omega_0\partial_{\varphi}+\mathtt{c}(\xi_0,\rho)\partial_{\theta}+\mathtt{E}_{n}-\tfrac12\mathscr{A}^{-1}\mathcal{H} \mathscr{A}-\varepsilon^2 \mathscr{A}^{-1}\partial_\theta\mathcal{Q}_0  \mathscr{A}\\ &\quad+
 \varepsilon^3 \mathscr{A}^{-1} \partial_\theta\mathcal{R}^\varepsilon_0(\rho) \mathscr{A}.\nonumber
\end{align}
Now, we shall use the following decomposition
{\begin{align}\label{comm H}
\nonumber\mathscr{A}^{-1}\mathcal{H} \mathscr{A}-\mathcal{H}&=\mathscr{B}^{-1}\mathfrak{B}^{-1}\mathcal{H} \mathfrak{B}\mathscr{B}-\mathcal{H}
\\
\nonumber&=\mathscr{B}^{-1}\big[\mathfrak{B}^{-1}\mathcal{H} \mathfrak{B}-\mathcal{H}\big]\mathscr{B}+\mathscr{B}^{-1}\mathcal{H}\mathscr{B}-\mathcal{H}
\\
&=\underbrace{\mathfrak{B}^{-1}\mathcal{H} \mathfrak{B}-\mathcal{H}}_{\triangleq\mathcal{H}_0}+(\mathscr{B}^{-1}\mathcal{H}_0\mathscr{B}-\mathcal{H}_0)+\mathscr{B}^{-1}\mathcal{H}\mathscr{B}-\mathcal{H}.
\end{align}
}
From direct computations, using  \eqref{mathscrB1} and \eqref{inverse-thetaL}, we find
\begin{align*}
(\mathfrak{B}^{-1}\mathcal{H}\mathfrak{B}-\mathcal{H}) h ( \varphi, \theta) &= \partial_\theta\int_{\mathbb{T}}  {\mathfrak K}_0( \varphi, \theta,\eta) h( \varphi, \eta)\,d\eta
\end{align*}
with
\begin{align*}
{\mathfrak K}_0( \varphi, \theta,\eta)\triangleq \ln\left| \frac{e^{\ii\,(\theta+\widehat{\mathfrak{b}}(\varphi,\theta))}- e^{\ii \,(\eta+\widehat{\mathfrak{b}}(\varphi,\eta))}}{e^{\ii\, \theta}-e^{\ii\, \eta}} \right|.
\end{align*}
Using \eqref{inverse-thetaL} and employing Taylor expansion yield
\begin{align*}
e^{\ii \Theta_1(\varphi)} \frac{e^{\ii\,(\theta+\widehat{\mathfrak{b}}(\varphi,\theta))}- e^{\ii \,(\eta+\widehat{\mathfrak{b}}(\varphi,\eta))}}{e^{\ii\, \theta}-e^{\ii\, \eta}} &=1-\ii \varepsilon^2\frac{e^{\ii\,\theta}{\mathfrak{b}}_{0}\big(\varphi,\theta- \widehat{\Theta}(\varphi)\big)-e^{\ii\,\eta}{\mathfrak{b}}_{0}\big(\varphi,\eta- \widehat{\Theta}(\varphi)\big)}{e^{\ii\,\theta}-e^{\ii\,\eta}}\\ &\quad  +\varepsilon^{2+\mu} \widehat{\mathfrak{K}}_{0}( \varphi, \theta,\eta)
\end{align*}
with
\begin{align}\label{grenada3}
\|\widehat{\mathfrak{K}}_{0}\|_{s}^{\textnormal{Lip}(\lambda)}\lesssim  1+\|\rho\|_{s+{3}}^{\textnormal{Lip}(\lambda)}.
\end{align}
Hence, from straightforward computations, we obtain
\begin{align*}
{\mathfrak K}_0( \varphi, \theta,\eta)&=\varepsilon^2\hbox{Im}\bigg\{\frac{e^{\ii\,\theta}{\mathfrak{b}}_{0}\big(\varphi,\theta- \widehat{\Theta}(\varphi)\big)-e^{\ii\,\eta}{\mathfrak{b}}_{0}\big(\varphi,\eta- \widehat{\Theta}(\varphi)\big)}{e^{\ii\,\theta}-e^{\ii\,\eta}}\bigg\}+\varepsilon^{2+\mu} {\mathfrak K}_1( \varphi, \theta,\eta)
\\
&=\varepsilon^2\Big({\mathfrak{b}}_{0}(\varphi,\theta- \widehat{\Theta}(\varphi))-{\mathfrak{b}}_{0}(\varphi,\eta- \widehat{\Theta}(\varphi))\Big)\hbox{Im}\Big\{\tfrac{e^{\ii\eta}}{e^{\ii\theta}-e^{\ii\eta}}\Big\}+\varepsilon^{2+\mu} {\mathfrak K}_1( \varphi, \theta,\eta)
\end{align*}
with the estimate
\begin{align}\label{grenada3}
\|{\mathfrak K}_1\|_{s}^{\textnormal{Lip}(\lambda)}\lesssim  1+\|\rho\|_{s+{3}}^{\textnormal{Lip}(\lambda)}.
\end{align}
Then, using \eqref{b1-expand} gives
\begin{align*}
{\mathfrak K}_0( \varphi, \theta,\eta)&=\varepsilon^2\textnormal{Im}\Big\{\mathtt{a}_2(\varphi)e^{-\ii 2 \widehat{\Theta}(\varphi)}\big(e^{\ii 2\theta} -e^{\ii 2} \big)\Big\}\hbox{Im}\Big\{\tfrac{e^{\ii\eta}}{e^{\ii\theta}-e^{\ii\eta}}\Big\}+\varepsilon^{2+\mu} {\mathfrak K}_1( \varphi, \theta,\eta).
\end{align*}
Moreover, one may easily check that
\begin{align*}
&\textnormal{Im}\Big\{\mathtt{a}_2(\varphi)e^{-\ii 2 \widehat{\Theta}(\varphi)}\big(e^{\ii 2\theta} -e^{\ii 2\eta} \big)\Big\}\hbox{Im}\Big\{\tfrac{e^{\ii\eta}}{e^{\ii\theta}-e^{\ii\eta}}\Big\}\\
&\quad =\textnormal{Im}\Big\{\tfrac{1}{2\ii}\mathtt{a}_2(\varphi)e^{-\ii 2 \widehat{\Theta}(\varphi)}\big(e^{\ii \theta} -e^{\ii \eta} \big)\big(e^{\ii \theta} +e^{\ii \eta} \big)\Big(\tfrac{e^{\ii\eta}}{e^{\ii\theta}-e^{\ii\eta}}-\tfrac{e^{-\ii\,\eta}}{e^{-\ii\theta}-e^{-\ii\eta}}\Big)\Big\}\\
&\quad =-\tfrac12\textnormal{Re}\Big\{\mathtt{a}_2(\varphi)e^{-\ii 2 \widehat{\Theta}(\varphi)}\big(e^{\ii \theta} +e^{\ii \eta} \big)^2\Big\}.
\end{align*}
It follows, from using \eqref{def g0}  and the identity ${\Theta}(\varphi)-\widehat{\Theta}(\varphi)=\varphi$,  that
\begin{align*}
{\mathfrak K}_0( \varphi, \theta,\eta)&=-\tfrac{\varepsilon^2}{2}\textnormal{Re}\big\{\mathtt{a}_2(\varphi)e^{-\ii 2 \widehat{\Theta}(\varphi)}\big(e^{\ii \theta} +e^{\ii \eta} \big)^2\big\}+\varepsilon^{2+\mu} {\mathfrak K}_1( \varphi, \theta,\eta)\\
&=-\tfrac{\varepsilon^2}{2}\textnormal{Re}\Big\{\Big(\tfrac{e^{-\ii 2 \widehat{\Theta}(\varphi)}}{q(\varphi)}-\tfrac{e^{\ii 2\varphi}}{\mathtt{w}_3^2(\varphi)}-\tfrac{e^{\ii 2\varphi}}{\mathtt{w}_4^2(\varphi)}\Big)\big(e^{\ii \theta} +e^{\ii \eta} \big)^2\Big\}+\varepsilon^{2+\mu} {\mathfrak K}_1( \varphi, \theta,\eta).
\end{align*}
Therefore, since $\int_{\mathbb{T}}h(\varphi,\eta)d\eta=0$ then
\begin{equation}\label{commutator with H}
\begin{aligned}
(\mathfrak{B}^{-1}\mathcal{H}\mathfrak{B}-\mathcal{H}) [h] ( \varphi, \theta) &=- {\varepsilon^2} \, \partial_\theta\underbrace{\int_{\mathbb{T}}  \textnormal{Re}\Big\{\Big(\tfrac{e^{-\ii 2 \widehat{\Theta}(\varphi)}}{q(\varphi)}-\tfrac{e^{\ii 2\varphi}}{\mathtt{w}_3^2(\varphi)}-\tfrac{e^{\ii 2\varphi}}{\mathtt{w}_4^2(\varphi)}\Big)e^{\ii (\theta+\eta)} \Big\} h( \varphi, \eta)\,d\eta}_{\triangleq \widehat{\mathcal{H}}_0[h]( \varphi, \theta) }.\\
&\quad+\varepsilon^{2+\mu} \underbrace{\partial_\theta\int_{\mathbb{T}}{\mathfrak K}_1( \varphi, \theta,\eta) h( \varphi, \eta)\,d\eta}_{\triangleq \widehat{\mathcal{H}}_1[h]( \varphi, \theta) }.
\end{aligned}
\end{equation}
On the other hand, since $\mathscr{A}= \mathfrak{B}\mathscr{B}$ then one has
{
\begin{align}\label{comm Q0}
\nonumber\mathscr{A}^{-1}\partial_\theta\mathcal{Q}_{0}  \mathscr{A}&=\mathscr{B}^{-1}\mathfrak{B}^{-1}\partial_\theta\mathcal{Q}_{0}  \mathfrak{B} \mathscr{B}
\\ &=\underbrace{\mathfrak{B}^{-1}\partial_\theta\mathcal{Q}_{0}  \mathfrak{B}}_{\triangleq \widehat{\mathcal{Q}}_0}+\big[\mathscr{B}^{-1}\widehat{\mathcal{Q}}_0 \mathscr{B}-\widehat{\mathcal{Q}}_0\big].
\end{align}
}
From  \eqref{def D0} and \eqref{IInv-77}-\eqref{inverse-thetaL} we can easily get 
\begin{align}\label{cumm Q}
\nonumber\widehat{\mathcal{Q}}_0[h](\varphi,\theta)&=\partial_\theta\bigg[\int_{\mathbb{T}}h(\varphi+\pi,\eta)\textnormal{Re}\Big\{ \tfrac{e^{\ii(\theta+\eta+\widehat{\mathfrak{b}}(\varphi,\theta)+\widehat{\mathfrak{b}}(\varphi,\eta))}}{q(\varphi)}- \tfrac{e^{\ii(\theta-\eta+\widehat{\mathfrak{b}}(\varphi,\theta)-\widehat{\mathfrak{b}}(\varphi,\eta))}}{\mathtt{w}_{4}^2(\varphi)}\Big\} d\eta \nonumber\\ &\qquad+\int_{\mathbb{T}}h(\varphi,\eta)\textnormal{Re}\Big\{ \tfrac{e^{\ii(\theta-\eta+\widehat{\mathfrak{b}}(\varphi,\theta)-\widehat{\mathfrak{b}}(\varphi,\eta))}}{\mathtt{w}_{3}^2(\varphi)}\Big\} d\eta\bigg]\nonumber\\
&=\partial_\theta \underbrace{\int_{\mathbb{T}}h(\varphi+\pi,\eta)\textnormal{Re}\Big\{ \tfrac{e^{\ii(\theta+\eta-2\widehat\Theta(\varphi))}}{q(\varphi)}- \tfrac{e^{\ii(\theta-\eta)}}{\mathtt{w}_{4}^2(\varphi)}\Big\} d\eta+\int_{\mathbb{T}}h(\varphi,\eta)\textnormal{Re}\Big\{ \tfrac{e^{\ii(\theta-\eta)}}{\mathtt{w}_{3}^2(\varphi)}\Big\} d\eta}_{\triangleq  \widehat{\mathcal{Q}}_1[h](\varphi,\theta) }.\nonumber\\
&\quad+\varepsilon\underbrace{{\partial_\theta}\bigg[\sum_{k=0}^1\int_{\mathbb{T}}  h (\varphi+k\pi,\eta){\mathfrak K}_{2+k}(\varphi,\theta,\eta)d\eta\bigg]}_{\triangleq \widehat{\mathcal{Q}}_2[h](\varphi,\theta)},
\end{align}
where 
 $$
 \begin{aligned}
 {\mathfrak K}_2(\varphi,\theta,\eta)&\triangleq \tfrac{1}{\varepsilon}\Big[\textnormal{Re}\Big\{ \tfrac{e^{\ii(\theta-\eta+\widehat{\mathfrak{b}}(\varphi,\theta)-\widehat{\mathfrak{b}}(\varphi,\eta))}}{\mathtt{w}_{3}^2(\varphi)}\Big\}-\textnormal{Re}\Big\{ \tfrac{e^{\ii(\theta-\eta)}}{\mathtt{w}_{3}^2(\varphi)}\Big\}\Big],\\
 {\mathfrak K}_3(\varphi,\theta,\eta)&\triangleq \tfrac{1}{\varepsilon}\Big[\textnormal{Re}\Big\{ \tfrac{e^{\ii(\theta+\eta+\widehat{\mathfrak{b}}(\varphi,\theta)+\widehat{\mathfrak{b}}(\varphi,\eta))}}{q(\varphi)}\Big\}-\textnormal{Re}\Big\{ \tfrac{e^{\ii(\theta+\eta-2\widehat\Theta(\varphi))}}{q(\varphi)}\Big\}\Big]\\ &\quad -\tfrac1\varepsilon\Big[\textnormal{Re}\Big\{ \tfrac{e^{\ii(\theta-\eta+\widehat{\mathfrak{b}}(\varphi,\theta)-\widehat{\mathfrak{b}}(\varphi,\eta))}}{\mathtt{w}_{4}^2(\varphi)}\Big\}-\textnormal{Re}\Big\{ \tfrac{e^{\ii(\theta-\eta)}}{\mathtt{w}_{4}^2(\varphi)}\Big\}\Big].
 \end{aligned}
 $$
 Applying Taylor expansion together with \eqref{inverse-thetaL}, \eqref{inverse-thetaML} and  \eqref{est b1 b2 b b-theta}, we find  
$$
\begin{aligned}
\|{\mathfrak K}_2\|_{s}^{\textnormal{Lip}(\lambda)}+\|{\mathfrak K}_3\|_{s}^{\textnormal{Lip}(\lambda)}\lesssim 1+\|\rho\|_{s+{3}}^{\textnormal{Lip}(\lambda)}.
\end{aligned}
$$
Inserting \eqref{comm H}, \eqref{comm Q0}, \eqref{commutator with H} and  \eqref{cumm Q}  into \eqref{conjug A} yields
\begin{align*}
\mathscr{A}^{-1}\mathcal{L}_0 \mathscr{A}&=\varepsilon^2\omega_0\partial_{\varphi}+\mathtt{c}\partial_{\theta}+\mathtt{E}_{n}^{0}-\tfrac12\mathcal{H} -\varepsilon^2 \partial_\theta\underbrace{(\widehat{\mathcal{Q}}_1-\tfrac12\widehat{\mathcal{H}}_0)}_{\triangleq\mathcal{Q}_1}+\partial_\theta\mathcal{R}^\varepsilon_1(\rho),
\end{align*}
where 
\begin{equation}\label{def R1eps}
\begin{aligned}
\partial_\theta\mathcal{R}^\varepsilon_1(\rho)&\triangleq  \varepsilon^3 \mathscr{A}^{-1}\partial_\theta\mathcal{R}^\varepsilon_0(\rho) \mathscr{A}  -\tfrac12[\mathscr{B}^{-1}\mathcal{H}_0\mathscr{B}-\mathcal{H}_0]-\tfrac12\big[\mathscr{B}^{-1}\mathcal{H}\mathscr{B}-\mathcal{H}\big]
\\ &\quad
 -\varepsilon^2\big[\mathscr{B}^{-1}\widehat{\mathcal{Q}}_0 \mathscr{B}-\widehat{\mathcal{Q}}_0\big]-\varepsilon^3 \widehat{\mathcal{Q}}_2-{\tfrac12}\varepsilon^{2+\mu} \widehat{\mathcal{H}}_1.
\end{aligned}
\end{equation}
In view of  Lemma \cite[Lemma 2.36]{BertiMontalto},  \eqref{est-beta-r0} and \eqref{small-C2-0} the operator
\begin{align*}
		( {\mathscr B}^{-1} {\cal H} {\mathscr B} - {\cal H}) h(\varphi, \theta) 
		&= \int_{\mathbb{T}}  {\mathscr  K}_1( \varphi, \theta,\eta) h( \varphi, \eta)\,d\eta,
\end{align*}
		defines  an integral operator whose kernel satisfies the   estimates : for all $s\in[ s_0,S],$  
\begin{align}\label{e-Box H}
 \| {\mathscr  K}_1\|_{s}^{\textnormal{Lip}(\lambda)}\lesssim_s \varepsilon^{{6}}\lambda^{-1}\left(1+\| \rho\|_{s+2\tau+{5}}^{\textnormal{Lip}(\lambda)}\right).
\end{align}
Moreover, according to Lemma \ref{lem CVAR kernel}, \eqref{est b1 b2 b b-theta}, \eqref{est-beta-r0} and \eqref{small-C2-0}, \eqref{est K0 R0}, we have
\begin{align*}
\nonumber				\big(\mathscr{B}^{-1}\widehat{\mathcal{Q}}_0  \mathscr{B}-\widehat{\mathcal{Q}}_0\big)h(\varphi,\theta)&=\int_{\mathbb{T}}h(\varphi,{\eta}){\mathscr  K}_2(\varphi,\theta,{\eta})d{\eta},\\
		\big(\mathscr{B}^{-1}\mathcal{H}_0\mathscr{B}-\mathcal{H}_0\big)h(\varphi,\theta)&=\int_{\mathbb{T}}h(\varphi,{\eta}){\mathscr  K}_3(\varphi,\theta,{\eta})d{\eta},
		\\
\nonumber \mathscr{A}^{-1}\partial_\theta\mathcal{R}^\varepsilon_0(\rho) \mathscr{A}h ( \varphi, \theta)&=\int_{\mathbb{T}}h(\varphi,{\eta}){\mathscr  K}_4(\varphi,\theta,{\eta})d{\eta},
		\end{align*}
		with
\begin{align*}
\|{\mathscr  K}_2\|_{s}^{\textnormal{Lip}(\lambda)}+\|{\mathscr  K}_3\|_{s}^{\textnormal{Lip}(\lambda)}&\lesssim \varepsilon^{6}\lambda^{-1}\left(1+\| \rho\|_{s+2\tau+{5}}^{\textnormal{Lip}(\lambda)}\right),
\\ 
\|{\mathscr  K}_4\|_{s}^{\textnormal{Lip}(\lambda)}&\lesssim  1+\| \rho\|_{s+2\tau+{5}}^{\textnormal{Lip}(\lambda)}.
\end{align*}
Consequently, we deduce that the operator  $\partial_\theta\mathcal{R}^\varepsilon_1(\rho)$, defined in  \eqref{def R1eps}, is an integral operator  with the kernel 
$$
\mathcal{K}_1(\rho)\triangleq\varepsilon^3 {\mathscr  K}_4-\tfrac12{\mathscr K}_3-\tfrac12{\mathscr  K}_1-\varepsilon^2{\mathscr  K}_2-\varepsilon^3{\mathfrak  K}_2-\tfrac12\varepsilon^{2+\mu} \mathfrak{K}_1,
$$
which satisfies the estimate 
$$
\|\mathcal{K}_1(\rho)\|_{s}^{\textnormal{Lip}(\lambda)}\lesssim  \big(\varepsilon^{2+{\mu}} +\varepsilon^{6} \lambda^{-1}\big)\left(1+\| \rho\|_{s+2\tau+{5}}^{\textnormal{Lip}(\lambda)}\right)
$$ 
Finally, the  estimate of $\partial_\theta \mathcal{R}_1$ can be derived from Lemma \ref{kernel-est}, which achieves  the proof of Proposition \ref{lemma-beta1}.
\end{proof}

\section{Invertibility of the linearized operator}

The main goal this section is  to conjugate   the linearized operator of the functional $\mathcal{F}(\rho)$, as defined in  \eqref{Edc eq rkPsi}, into a Fourier multiplier up to a small error. This step is crucial in order to find  an approximation of the  right inverse, which is required along Nash-Moser scheme. The reduction process will unfold in several sequential steps. Initially, we will elucidate the asymptotic structure of the linearized operator around  a small state. Subsequently, we will transform it into an   operator with  constant coefficients up to quasi-linear terms  with small size in $\varepsilon.$ Finally, we will perform KAM techniques to reduce the positive  order of the linear part  into a  Fourier multiplier.  
 As we shall see,  by these techniques we may invert the linearized operator out from the spatial modes $\pm1$, which  are degenerating with small $\varepsilon$. To investigate  the invertibility of the restricted operator on    these modes, a careful study on the linearized operator is required, using  especially its structure  at the order $\varepsilon^2$. We will tackle this issue  through  a suitable description of the monodromy matrix.\\
We need to split the phase space as follows,
\begin{align}\label{Xs-space}
X^s=\hbox{Lip}_\lambda(\mathcal{O},H_0^{s})=X^{s}_{\circ}\oplus X^s_{\perp}
\end{align}
with
$$
X^{s}_{\circ}=\big\{ h\in \textnormal{Lip}_\lambda(\mathcal{O},H^{s}_0)\quad\hbox{s.t.}\quad  \Pi h=h\big\}
$$
and
$$
X^{s}_{\perp}=\big\{ h\in \textnormal{Lip}_\lambda(\mathcal{O},H^{s}_0)\quad\hbox{s.t.}\quad  \Pi^\perp h=h\big\},
$$
where $\Pi$ is  the projector defined by
\begin{align}\label{Orto-1}
h=\sum_{n\in\ZZ^\star}h_n(\theta)\Longrightarrow \Pi h=\sum_{n=\pm1} h_n(\varphi) e_n(\theta), \quad \Pi^\perp=\hbox{Id}-\Pi.
\end{align}
We shall also need the following spaces: $\displaystyle{X^{\infty}_{\circ}=\cap_{s\geq0}X^{s}_{\circ}}$ and $X^{\infty}_{\perp}=\cap_{s\geq0}X^{s}_{\perp}.$
To start, let us remind the structure of the operator $\mathcal{L}_1$, as detailed in Proposition \ref{lemma-beta1},
\begin{align}\label{L1-deff}
\mathcal{L}_1&= \underbrace{\varepsilon^2\omega(\xi_0)\partial_\varphi +\mathtt{c}(\xi_0,\rho) \partial_{\theta}-\tfrac{1}{2}\mathcal{H}-{\varepsilon^2 }\partial_\theta\mathcal{Q}_1}_{\triangleq\,\mathbb{L}_1}+\partial_\theta\mathcal{R}_{1}+\mathtt{E}_{n}
\end{align}
with
\begin{align}\label{def Q10}
\mathcal{Q}_1[h](\varphi,\theta)&=\textnormal{Re}\Big\{\Big(\tfrac{1}{\mathtt{w}_3^2(\varphi)}h_1(\varphi) -\tfrac{1}{2}\Big(\tfrac{e^{-\ii 2 \widehat{\Theta}(\varphi)}}{q(\varphi)}-\tfrac{e^{\ii 2\varphi}}{\mathtt{w}_3^2(\varphi)}-\tfrac{e^{\ii 2\varphi}}{\mathtt{w}_4^2(\varphi)}\Big)h_{-1}(\varphi)\nonumber
\\ &\quad\qquad 
+\tfrac{e^{\ii 2 \widehat{\Theta}(\varphi)}}{q(\varphi)}h_{-1}(\varphi+\pi) -\tfrac{1}{\mathtt{w}_4^2(\varphi)}h_1(\varphi+\pi)\Big)e^{\ii \theta}\Big\}.
\end{align}
The operator $\partial_\theta \mathcal{R}_1(\rho)$ is defined through Proposition \ref{lemma-beta1} 
and the constant $\mathtt{c}(\xi_0,\rho)$ takes the form, by virtue of Proposition \ref{QP-change}-1 and \eqref{def c0},
\begin{align}\label{c-2-def}
\nonumber \mathtt{c}(\xi_0,\rho)&=
 \tfrac12-\varepsilon^2\omega_0 +\varepsilon^{3}\mathtt{c}_1+(\mathtt{c}(\xi_0,\rho)-\mathtt{c}_0),\\
&\triangleq  \tfrac12-\varepsilon^2\big(\omega_0-\varepsilon \mathtt{c}_2\big),
\end{align}
with
\begin{align}\label{c2-estimate}
\|\mathtt{c}_2\|^{\textnormal{Lip}(\lambda)}\leqslant C.
\end{align}
From straightforward analysis, one may observe that  the linear operator 
$$\mathbb{L}_1+\partial_\theta \mathcal{R}_1:\hbox{Lip}_\lambda(\mathcal{O},X^{s})\to \textnormal{Lip}_\lambda(\mathcal{O},X^{s-1})$$ is well-defined and its action is  equivalent to the matrix operator $\displaystyle{\mathbb{M}: X^{s}_{\circ}\times X^{s}_{\perp}\to X^{s-1}_{\circ}\times X^{s-1}_{\perp}}$ with
\begin{align}\label{M-def}
\mathbb{M}&\triangleq\begin{pmatrix}
\Pi\mathbb{L}_1\Pi&0 \\
0 & \Pi^\perp\mathbb{L}_1\Pi^\perp
\end{pmatrix}+\begin{pmatrix}
\Pi\partial_\theta\mathcal{R}_1\Pi&\Pi\partial_\theta\mathcal{R}_1\Pi^\perp \\
\Pi^\perp\partial_\theta\mathcal{R}_1\Pi &\Pi^\perp\partial_\theta\mathcal{R}_1\Pi^\perp
\end{pmatrix}\\
\nonumber &\triangleq\mathbb{M}_1+\partial_\theta\mathbb{R}_1.
\end{align}
Therefore, the invertibility of the scalar operator $\mathbb{L}_1+\partial_\theta\mathcal{R}_1$ will be analyzed through the invertibility of the  matrix operator $\mathbb{M}$. This will be performed by inverting its main part $\mathbb{M}_1$ combined with  perturbative arguments. For that purpose, we shall proceed in several  steps.

\subsection{Invertibility of $\Pi\mathbb{L}_1\Pi$ and monodromy matrix}\label{sec-monodromy}
The principal task is to show the invertibility of the operator $\Pi\mathbb{L}_1\Pi$ introduced in \eqref{M-def} and \eqref{L1-deff}. We point out  that this operator localizes in the Fourier side at the spatial  modes $\pm1$  and exhibits a degeneracy in $\varepsilon.$ To invert it in the periodic setting we are led to solve  a system of first order differential equations with periodic coefficients. We will accomplish this by leveraging some specific structures  of the monodromy matrix. Before stating the main result let us introduce, for given $0<\sigma<1$, the set   
\begin{equation}\label{cond-interval2}
	\mathcal{O}_\sigma=\Big\{\xi_0\in [\xi_*,\xi^*]; \quad d(x,\mathbb{S})\geqslant\sigma\Big\}
\end{equation}
with  $\mathbb{S}$ a discrete set defined in Lemma \ref{Reso-R0}.
The main result reads as follows.
\begin{proposition}\label{propo-monod}
 Given the conditions \eqref{cond1}-\eqref{cond-interval}. There exist $\mathtt{n}\in\mathbb{N}$ and  $\varepsilon_1>0$ such that under the smallness condition
 \begin{equation}\label{small sigma}
 	\varepsilon\sigma^{-\mathtt{n}}\leqslant \varepsilon_1
 \end{equation} 
  there exists a linear operator   operator $\mathtt{S}:X^{s-1}_{\circ}\to X^{s}_{\circ}$, for all $s>1$, satisfying
  \begin{align*}
 \|\mathtt{S} [h]\|_{s}^{\textnormal{Lip}(\lambda)}\leqslant C(s)\varepsilon^{-2}\sigma^{-\mathtt{n}} \|h\|_{{s-1}}^{\textnormal{Lip}(\lambda)}
\end{align*}
and  by restricting the parameter $\xi_0$ on  $\mathcal{O}_\sigma$ one has the identity
$$
(\Pi\mathbb{L}_1\Pi)\mathtt{S}=\textnormal{Id}.
$$

\end{proposition}
\begin{proof}
First, we write by virtue  of  \eqref{L1-deff}, \eqref{c-2-def} and the identity
$$
(\partial_\theta-\mathcal{H})\Pi=0
$$
that
\begin{align*}
\nonumber \Pi\mathbb{L}_1\Pi&=\varepsilon^2\Big(\omega_0\partial_\varphi -\big(\omega_0-\varepsilon \mathtt{c}_2\big)\partial_{\theta}-\Pi\partial_\theta\mathcal{Q}_1\Pi\Big)\\
&\triangleq \mathcal{L}_{1,\circ}.
\end{align*}
Consider an arbitrary  real-valued function 
$$(\varphi,\theta)\mapsto g(\varphi,\theta)=\displaystyle{\sum_{n=\pm1}g_n(\varphi)e^{\ii n\theta}\in X^{s-1}_{\circ}}
$$ 
and we want  to solve in the real space  $ X^{s}_{\circ}$ the equation
$$
\mathcal{L}_{1,\circ}h= g.
$$
Since $g$ is real then  $g_{-n}=\overline{g_n}$ and similarly with $h$ we get $h_{-n}=\overline{h_n}$. Thus, 
using Fourier expansion together with \eqref{def Q10},  this equation is equivalent to solve
\begin{equation}\label{system h1h-1}
 \begin{aligned}
\omega_0\partial_\varphi h_1(\varphi)=\ii &\Big(\omega_0+\tfrac{1}{2\mathtt{w}_3^2(\varphi)}-\varepsilon \mathtt{c}_2\Big)h_1(\varphi) 
-\tfrac{\ii }{4}\Big(\tfrac{e^{-\ii 2 \widehat{\Theta}(\varphi)}}{q(\varphi)}-\tfrac{e^{\ii 2\varphi}}{\mathtt{w}_3^2(\varphi)}-\tfrac{e^{\ii 2\varphi}}{\mathtt{w}_4^2(\varphi)}\Big)h_{-1}(\varphi)
\\
&-\tfrac{\ii }{2{\mathtt{w}_4^2(\varphi)}} h_1(\varphi+\pi)+\tfrac{ \ii e^{\ii 2 \widehat{\Theta}(\varphi)}}{2q(\varphi)}h_{-1}(\varphi+\pi)+\varepsilon^{-2} g_1(\varphi).
\end{aligned}
\end{equation}
Denote by
 \begin{equation}\label{rhon-1}
 \begin{aligned}
\rho_1(\varepsilon, \varphi)&\triangleq \tfrac{\ii}{\omega_0}\Big(\omega_0+\tfrac{1}{2\mathtt{w}_3^2(\varphi)}-\varepsilon \mathtt{c}_2\Big),\qquad 
\rho_2(\varphi)\triangleq -\tfrac{\ii}{4\omega_0}\Big(\tfrac{e^{-\ii 2 \widehat{\Theta}(\varphi)}}{q(\varphi)}-\tfrac{e^{\ii 2\varphi}}{\mathtt{w}_3^2(\varphi)}-\tfrac{e^{\ii 2\varphi}}{\mathtt{w}_4^2(\varphi)}\Big),\\
\rho_3(\varphi)&\triangleq -\tfrac{\ii}{2\omega_0 \mathtt{w}_4^2(\varphi)} , \qquad \qquad \qquad \quad\,
\rho_4(\varphi)\triangleq \tfrac{\ii e^{\ii 2 \widehat{\Theta}(\varphi)}}{2\omega_0 q(\varphi)},
 \end{aligned}
 \end{equation}

\begin{align}\label{Matrix-dec0}
A(\varepsilon,\xi_0,\varphi)\triangleq  {\small \begin{pmatrix}
\rho_1(\varepsilon,\varphi)&\rho_2(\varphi) &\rho_3(\varphi) & \rho_4(\varphi)
\\
\overline{\rho_2(\varphi)} &\overline{\rho_1(\varepsilon,\varphi)} & \overline{\rho_4(\varphi)} &\overline{\rho_3(\varphi)}
\\
\rho_3(\varphi+\pi)& \rho_4(\varphi+\pi) &\rho_1(\varepsilon,\varphi+\pi) & \rho_2(\varphi+\pi)
\\
\overline{\rho_4(\varphi+\pi)}  & \overline{\rho_3(\varphi+\pi)} & \overline{\rho_2(\varphi+\pi)} & \overline{\rho_1(\varepsilon,\varphi+\pi)}
\end{pmatrix}},
\end{align}
and $$
G(\varphi)\triangleq \tfrac{1}{\omega_0}{\small \begin{pmatrix}
g_1(\varphi) \\
\overline{g_1(\varphi)}\\
g_1(\varphi+\pi) \\
\overline{g_1(\varphi+\pi)}
\end{pmatrix}}\in\mathbb{C}^4 ,\qquad H(\varphi)\triangleq {\small \begin{pmatrix}
h_1(\varphi) \\
\overline{h_1(\varphi)}\\
h_1(\varphi+\pi) \\
\overline{h_1(\varphi+\pi)}
\end{pmatrix}}\in\mathbb{C}^4.
$$
Then the  system \eqref{system h1h-1} is equivalent to  
\begin{equation}\label{sol-1}
\partial_\varphi H-A(\varepsilon,\xi_0,\varphi) H=\varepsilon^{-2} G.
\end{equation}
Let $(\varphi,\phi)\in\R^2 \mapsto \mathcal{M}(\varphi,\phi)$ be  the fundamental matrix defined through the $2\times2$ matrix ODE
\begin{equation*}
\left\lbrace\begin{array}{ll}
		\partial_\varphi \mathcal{M}(\varphi,\phi)- A(\varepsilon,\xi_0,\varphi) \mathcal{M}(\varphi,\phi)=0 
		\\
		\mathcal{M}(\phi,\phi)=\textnormal{Id}.
	\end{array}\right.
	\end{equation*}
Then the solution $H$ can be expressed in the form
\begin{align}\label{Def-M1}
H(\varphi)=\mathcal{M}(\varphi,0) H(0)+\varepsilon^{-2}\int_0^\varphi \mathcal{M}(\varphi,\phi)G(\phi) d\phi.
\end{align}
From \eqref{Matrix-dec0} and \eqref{rhon-1} we easily see that the matrix $A$ is $2\pi-$periodic, that is,   
$$\forall \varphi\in\R,\quad A(\varepsilon,\xi_0,\varphi+2\pi)=A(\varepsilon,\xi_0,\varphi).
$$ Then $H$ is $2\pi-$periodic if and only if
$$
H(2\pi)=H(0),
$$
which  is equivalent in view of \eqref{Def-M1} to 
\begin{align}\label{Iden-P1}
\big(\hbox{Id}-\mathcal{M}(2\pi,0)\big)H(0)= \varepsilon^{-2}\int_0^{2\pi} \mathcal{M}(2\pi,\phi)G(\phi) d\phi.
\end{align}
In order to find a unique solution to this equation it is enough to show  that the matrix $\hbox{Id}-\mathcal{M}(2\pi,0)$ is invertible. To this aim, we first use the decomposition
$$
A(\varepsilon,\xi_0,\varphi)=A(0,\xi_0,\varphi)+ A_1(\varepsilon,\xi_0,\varphi)$$
where $A_1$ can be estimated according to \eqref{c-2-def} and \eqref{c2-estimate} as
\begin{align}\label{A1-est-1}
\sup_{\varphi\in\R}\|A_1(\varepsilon,\xi_0,\varphi)\|\leqslant C\varepsilon.
\end{align}
 Next, let us consider the fundamental solution of the unperturbed problem
 \begin{equation}\label{Matrix-Fund1}
\left\lbrace\begin{array}{ll}
		\partial_\varphi \mathcal{M}_0(\varphi)- A(0,\xi_,\varphi) \mathcal{M}_0(\varphi)=0 
		\\
		\mathcal{M}_0(0)=\textnormal{Id}.
	\end{array}\right.
	\end{equation}
Then one has the decomposition
$$\mathcal{M}(\varphi,0)=\mathcal{M}_0(\varphi)+ \mathcal{M}_1(\varepsilon,\varphi),
$$
with 
$$
\partial_\varphi \mathcal{M}_1(\varphi)- A(0,\xi_0,\varphi) \mathcal{M}_1(\varphi)=-A_1(\varepsilon,\xi_0,\varphi)\mathcal{M},\;\, \mathcal{M}_1(0)=0.
$$
Therefore we get by virtue of  \eqref{A1-est-1}
\begin{align*}
\sup_{\varphi\in[0,2\pi]}\|\mathcal{M}_1(\varphi)\|\leqslant C\varepsilon.
\end{align*}
Using Lemma \ref{Reso-R0}, we get that the matrix $\mathcal{M}_0(2\pi)-\textnormal{Id}$ is invertible on $\mathcal{O}_\sigma$ and 
there exists $\mathtt{n}\in\mathbb{N}$ such that the inverse satisfies
$$
\|\big(\mathcal{M}_0(2\pi)-\textnormal{Id}\big)^{-1}\|^{\textnormal{Lip}(\lambda)}_{\mathcal{O}_\sigma}\leqslant C\sigma^{-\mathtt{n}}.
$$
 Consequently, by a perturbation argument  the matrix $\mathcal{M}(2\pi)-\textnormal{Id}$ is invertible on $\mathcal{O}_\sigma$ and under the smallness condition \eqref{small sigma} 
 we have
$$
\|\big(\mathcal{M}(2\pi)-\textnormal{Id}\big)^{-1}\|^{\textnormal{Lip}(\lambda)}_{\mathcal{O}_\sigma}\leqslant C\sigma^{-\mathtt{n}}.
$$This implies that \eqref{Iden-P1} admits only one solution
and 
\begin{align}\label{Iden-P2}
\|H(0)\|^{\textnormal{Lip}(\lambda)}_{\mathcal{O}_\sigma}\leqslant C\varepsilon^{-2}\sigma^{-\mathtt{n}}\|G\|_{L^2(\mathbb{T})},
\end{align}
where we have used for $\varepsilon$ small enough  the estimate
$$
\sup_{\varphi,\phi\in[0,2\pi]}\| \mathcal{M}(\varphi,\phi)\|\leqslant C.
$$
Coming back to \eqref{sol-1} we get that this equation admits only one solution. Moreover, 
inserting \eqref{Iden-P2} into  \eqref{Def-M1} and using straightforward estimates based on law products we get for $s>1$
\begin{align*}
\|H\|_{s,{\mathcal{O}_\sigma}}^{\textnormal{Lip}(\lambda)}\leqslant C\varepsilon^{-2}\sigma^{-\mathtt{n}}\|G\|_{s-1}^{\textnormal{Lip}(\lambda)}.
\end{align*}
Note that in obtaining this estimate,  we have specifically used  that  the functions $ q, \Theta$ are smooth according to Lemma \ref{lem-period}, implying that   for any $s\in\R$
$$
\| \mathcal{M}\|_{s}^{\textnormal{Lip}(\lambda)}\leqslant C(s).
$$
It follows that  the linear operator $\mathcal{L}_{1,\circ}:X^{s}_{\circ}\to X^{s-1}_{\circ}$ is invertible on ${\mathcal{O}_\sigma}$  and 
\begin{align*}
\|\mathcal{L}_{1,\circ}^{-1}g\|_{s,{\mathcal{O}_\sigma}}^{\textnormal{Lip}(\lambda)}\leqslant C\varepsilon^{-2}\sigma^{-\mathtt{n}}\|g\|_{s-1}^{\textnormal{Lip}(\lambda)}.
\end{align*}
{In view of Lemma \ref{thm-extend}, the operator $\mathcal{L}_{1,\circ}^{-1} $ admits an extension to the set $\mathcal{O}$, that we  denote by $\mathtt{S}$, which satisfies the same Lipschitz  estimate.}
This achieves the proof of the desired result.
\end{proof}
The next goal is to explore suitable properties, which are essential for proving Proposition \ref{propo-monod},  of  the monodromy matrix associated with the fundamental matrix defined in \eqref{Matrix-Fund1}. Recall that the matrix map $\varphi\in\R\mapsto A(0,\xi_0,\varphi)$  is $2\pi-$periodic and defined through the expression \eqref{Matrix-dec0}. 
Applying Lemma \ref{lem-period} 
 together with  \eqref{rhon-1} and \eqref{def ws} we deduce  that for all $\delta_0>$ small enough,
\begin{equation}\label{rhon-3}
\tfrac{\xi_0^2}{y_0^2}<\tfrac{1-\delta_0}{2}\Longrightarrow \sup_{\varphi\in\R}\big|\rho_1(0, \varphi)-\ii\big|+\sup_{\varphi\in\R}\big|\rho_2( \varphi)+\tfrac{\ii}{4}\big|+\sup_{\varphi\in\R}\big|\rho_3( \varphi)\big|+\sup_{\varphi\in\R}\big|\rho_4( \varphi)-\tfrac{\ii}{2}\big|\leqslant C\tfrac{\xi_0^2}{y_0^2},
\end{equation}
Putting together  \eqref{Matrix-dec0}, \eqref{rhon-1} and \eqref{rhon-3} allows to get
\begin{align}\label{matrix-decomp}
A(0,\xi_0,\varphi)= \underbrace{\lim_{\xi_0\to 0}A(0,\xi_0,\varphi)}_{\triangleq A_0}+A_2(\xi_0,\varphi),\quad A_0= \begin{pmatrix}
\ii& -\tfrac{\ii}{4} & 0& \tfrac{\ii}{2} \\
\tfrac{\ii}{4}&-\ii&-\tfrac{\ii}{2} & 0  \\
0 &\tfrac{\ii}{2} &\ii&  -\tfrac{\ii}{4} \\
 -\tfrac{\ii}{2}&0  &\tfrac{\ii}{4}  &-\ii 
\end{pmatrix}
\end{align}
with the estimate
\begin{align}\label{ESt-resolv-11}
\tfrac{\xi_0^2}{y_0^2}<\tfrac{1-\delta_0}{2}\Longrightarrow \sup_{\varphi\in\R}|A_2(\xi_0,\varphi)|\leqslant C\tfrac{\xi_0^2}{y_0^2}.
\end{align}
The monodromy matrix is defined by $\mathcal{M}_0(2\pi)$, and we want to establish appropriate  conditions under which the matrix $\mathcal{M}_0(2\pi)-\textnormal{Id}$ is invertible. Our result reads as follows.
\begin{lemma}\label{Reso-R0} 
Let $y_0>0$. Then, there exists a countable set $\mathbb{S}\subset (0,\frac{y_0}{\sqrt{2}})$ depending only on $y_0$ and   with at most one accumulation point at $\frac{y_0}{\sqrt{2}}$ such that
$$
\forall \xi_0\in \big(0,\tfrac{y_0}{\sqrt{2}}\big)\backslash\mathbb{S}, \quad \mathcal{M}_0(2\pi)-\textnormal{Id}\quad \textnormal {is invertible}.
$$ 
Moreover, there exist 
 $\mathtt{n}\in\mathbb{N}$ and  $C>0$, independent of $\sigma$,  such that 
$$
\|\big(\mathcal{M}_0(2\pi)-\textnormal{Id}\big)^{-1}\|^{\textnormal{Lip}(\lambda)}_{\mathcal{O}_\sigma}\leqslant C\sigma^{-\mathtt{n}}.
$$
\end{lemma}
\begin{proof} 
We shall first consider the generalized  resolvant equation 
\begin{align}\label{Cauch-res}
\partial_\varphi \mathcal{M}_0(\varphi,\phi)- A(0,\xi_0,\varphi) \mathcal{M}_0(\varphi,\phi)=0,\;\, \mathcal{M}_0(\phi,\phi)=\textnormal{Id}.
\end{align}
Then the link to the solution of the equation  \eqref{Matrix-Fund1} is given by the relation  $\mathcal{M}_0(\varphi,0)=\mathcal{M}_0(\varphi).$
From the classical theory, one has
\begin{align}\label{res-hm-11}
\sup_{\varphi,\phi\in[0,2\pi]} |\mathcal{M}_0(\varphi,\phi)|\leqslant C.
 \end{align}
By using  the  decomposition \eqref{matrix-decomp} we find
$$
\partial_\varphi \mathcal{M}_0(\varphi,\phi)- A_0  \mathcal{M}_0(\varphi,\phi)=A_2(\xi_0,\varphi)\mathcal{M}_0(\varphi,\phi), \;\,\mathcal{M}_0(\phi,\phi)=\textnormal{Id}.
$$
Since the matrix $A_0$ is independent of $\varphi$ then by Duhamel formulae we infer
$$
\mathcal{M}_0(\varphi)=e^{\varphi A_0}+\int_0^\varphi e^{(\varphi-\phi)A_0}A_2(\xi_0,\phi)\mathcal{M}_0(\phi) d\phi.
$$
Thus, we find in view of \eqref{ESt-resolv-11} and \eqref{res-hm-11}
\begin{align}\label{Hm-d-1-1}
\tfrac{\xi_0^2}{y_0^2} \leqslant \tfrac{1-\delta_0}{2}\Longrightarrow |\mathcal{M}_0(2\pi)-e^{2\pi A_0}|\leqslant {C}\tfrac{\xi_0^2}{y_0^2}.
\end{align}
From standard computations we find 
$$
e^{\varphi A_0}=\begin{pmatrix}
a_1(\varphi)& a_2(\varphi) & a_3(\varphi) & a_4(\varphi) \\
\overline{a_2(\varphi)} &\overline{a_1(\varphi)} &\overline{a_4(\varphi)} & \overline{a_3(\varphi)} \\
a_3(\varphi)& a_4(\varphi) &a_1(\varphi) & a_2(\varphi) \\
\overline{a_4(\varphi)}& \overline{a_3(\varphi)} &\overline{a_2(\varphi)} & \overline{a_1(\varphi)}
\end{pmatrix}
$$
with
$$
 \begin{aligned} 
a_1(\varphi)&=\frac12\cos\big(\tfrac{\sqrt{15}}{4}\varphi\big)+\frac12\cos\big(\tfrac{\sqrt{7}}{4}\varphi\big)+\ii\Big[\tfrac{2}{\sqrt{15}}\sin\big(\tfrac{\sqrt{15}}{4}\varphi\big)+\tfrac{2}{\sqrt{7}}\sin\big(\tfrac{\sqrt{7}}{4}\varphi\big)\Big],\\
a_2(\varphi)&=\tfrac\ii2\Big[\tfrac{1}{\sqrt{15}}\sin\big(\tfrac{\sqrt{15}}{4}\varphi\big)-\tfrac{3}{\sqrt{7}}\sin\big(\tfrac{\sqrt{7}}{4}\varphi\big)\Big],\\
a_3(\varphi)&=\frac12\cos\big(\tfrac{\sqrt{15}}{4}\varphi\big)-\frac12\cos\big(\tfrac{\sqrt{7}}{4}\varphi\big)+\ii\Big[\tfrac{2}{\sqrt{15}}\sin\big(\tfrac{\sqrt{15}}{4}\varphi\big)-\tfrac{2}{\sqrt{7}}\sin\big(\tfrac{\sqrt{7}}{4}\varphi\big)\Big],\\
a_4(\varphi)&=\tfrac\ii2\Big[\tfrac{1}{\sqrt{15}}\sin\big(\tfrac{\sqrt{15}}{4}\varphi\big)+\tfrac{3}{\sqrt{7}}\sin\big(\tfrac{\sqrt{7}}{4}\varphi\big)\Big].
\end{aligned}
$$
Then one can easily check that
$$
\hbox{det}(e^{\varphi A_0}-\hbox{Id})=16\sin^2\big(\tfrac{\sqrt{15}}{8}\varphi\big)\sin^2\big(\tfrac{\sqrt{7}}{8}\varphi\big).
$$
In particular, one has in view of \eqref{Hm-d-1-1}
\begin{equation}\label{det}
\hbox{det}(e^{2\pi A_0}-\hbox{Id})\approx 0.121262 >0,
\end{equation}
implying that the matrix $e^{2\pi A_0}-\hbox{Id}$ is invertible. Now, we consider the function 
$$\xi_0\in\big(0,\tfrac{y_0}{\sqrt{2}}\big)\mapsto g(\xi_0)\triangleq \hbox{det}(\mathcal{M}_0(2\pi)-\hbox{Id}).$$
This function is real analytic  according to classical theorems on the analyticity of solutions to ODE together with \eqref{Cauch-res}, \eqref{Matrix-dec0} and  Lemma \ref{lem-period}. Moreover, from \eqref{det}, the analytic function $\xi_0\in\big(0,\tfrac{y_0}{\sqrt{2}}\big)\mapsto g(\xi_0)$ is not vanishing close to zero and therefore it  admits at most a countable  set of zeros in this interval. Thus, we conclude that the set
$$
\mathbb{S}\triangleq \Big\{\xi_0\in \big(0,\tfrac{y_0}{\sqrt{2}}\big); \; g(\xi_0)=0\Big\}
$$ 
is a countable set with at most one accumulation point at $\tfrac{y_0}{\sqrt{2}}$.
Next, given $\xi_0\in \mathbb{S}$, we denote by  $n_{\xi_0}\in \mathbb{N}$ the order of multiplicity of $\xi_0$, that is,
$$
\forall n< n_{\xi_0}, \quad g^{(n)}(\xi_0)=0 \quad\textnormal{and}\quad g^{(n_{\xi_0})}(\xi_0)\neq0 .
$$
Notice that by analyticity of $g$  the integer $n_{\xi_0}$ is well defined. We set 
$$
\mathtt{n}_0\triangleq\max\big\{n_{\xi_0};\; \xi_0\in\mathbb{S}\cap[\xi_*,\xi^*]\big\}+1.
$$
This is well defined since $ \mathbb{S}\cap[\xi_*,\xi^*]$ is a finite set.
Thus, by Taylor expansion we  get a constant $C>0$, independent of $\sigma$, such that
$$
\inf_{\xi_0\in \mathcal{O}_\sigma}|g(\xi_0)|\geqslant C\sigma^{\mathtt{n}}.
$$
Therefore,
$$
\|(\mathcal{M}_0(2\pi)-\textnormal{Id})^{-1}\|^{\textnormal{Lip}(\lambda)}_{\mathcal{O}_\sigma}\leqslant C\sigma^{-\mathtt{n}}.
$$
This ends the proof of the Lemma \ref{Reso-R0}.
\end{proof}

\subsection{Invertibility of $\Pi^\perp\mathbb{L}_1\Pi^\perp$}\label{sec-Norm76}
The aim of this section is to investigate  the right invertibility of the operator $\Pi^\perp\mathbb{L}_1\Pi^\perp$, which is  defined through  \eqref{L1-deff} and \eqref{Orto-1}.
One may easily check that  the operator 
$$\mathbb{L}_{1,\perp}\triangleq \Pi^\perp\mathbb{L}_1\Pi^\perp:X^s_{\perp}\to X^{s-1}_{\perp}
$$ is well-defined and assumes the structure 
\begin{align}\label{perp-ope}
\mathbb{L}_{1,\perp}
&=
 \varepsilon^2\omega_0\partial_\varphi +\underbrace{\mathtt{c}(\xi_0,\rho)\partial_{\theta}-\tfrac{1}{2}\mathcal{H}}_{\triangleq\,\mathcal{D}_1},
\end{align}
where we have used the identity
$$
\Pi^\perp\mathcal{Q}_1\Pi^\perp=0
$$
which follows from the fact that the operator $\mathcal{Q}_1$ localizes on the spatial  modes $n=\pm1,$ according to \eqref{def Q10}.
We shall start with exploring  the asymptotic structure of the operator $\mathbb{L}_{1,\perp}.$
The following property follows from \eqref{c-2-def},
\begin{align}\label{D-1-LL}
					\forall \ell\in \mathbb{Z}, |j|\geqslant 2,\quad  \mathcal{D}_{1}\mathbf{e}_{\ell,j}=\ii\mu_{j,2}\,\mathbf{e}_{\ell,j},
					\end{align}
					with
					\begin{align}\label{mu-j1}
					\mu_{j,2}(\xi_0,\rho)&=j\left( \tfrac12-\varepsilon^2\big(\omega_0-\varepsilon \mathtt{c}_2\big)\right)-\tfrac12\tfrac{j}{|j|}	.
								\end{align}
Next, we intend  to  discuss the existence of an approximate right  inverse for  $\mathbb{L}_{1,\perp}$.
\begin{proposition}\label{prop-perp}
Let $ \tau,\lambda>0$ as in \eqref{cond1}. There exists $\varepsilon_0>0$ small enough  such that for any  $\varepsilon\in(0,\varepsilon_0),$
  there exists a family of  linear operators $\big(\mathtt{T}_n\big)_{n\in\mathbb{N}}$  with  the estimate
$$
\forall \,N\geqslant 0,  s\in[s_0,S],\quad \sup_{n\in\mathbb{N}}\|\mathtt{T}_nh\|_{s,N}^{\textnormal{Lip}(\lambda)}\leqslant C\lambda^{-1}\|h\|_{s,2\tau+N}^{\textnormal{Lip}(\lambda)}
$$
and  by restricting the parameter $\xi_0$ on  the Cantor set
\begin{equation}\label{Cantor second}
	\mathcal{O}_{n}^2(\rho)=\Big\{\xi_0\in \mathcal{O},\, \forall \ell\in\mathbb{Z}, 2\leqslant |j|\leqslant N_n, \,\,|\varepsilon^2 \omega(\xi_0)\ell+\mu_{j,2}(\xi_0,\rho)|\geqslant  \tfrac{\lambda}{| j|^\tau}\Big\}	
\end{equation}
we get
$$
\mathbb{L}_{1,\perp}\mathtt{T}_n=\textnormal{Id}+{\mathtt{E}_{n}^{2}}
$$
with
\begin{align*}
\|\mathtt{E}_{n}^{2}h\|_{s_0}^{\textnormal{Lip}(\lambda)}
 &\leqslant  C\lambda^{-1}N_n^{s_0-s} \|h\|_{{s,2\tau+1}}^{\textnormal{Lip}(\lambda)}.\end{align*}

\end{proposition}

\begin{proof}
From \eqref{perp-ope}, we  may use the splitting
\begin{align}\label{dekomp1}
\mathbb{L}_{1,\perp}&=\underbrace{\varepsilon^2\omega_0\partial_\varphi + \Pi_{N_n}\mathcal{D}_1}_{\triangleq\,\mathtt{L}_n}-\Pi_{N_n}^\perp\mathcal{D}_1,
\end{align}
where the projector $\Pi_{N}$ is defined by
$$
h=\sum_{\ell\in\Z\atop 2\leqslant |j|}h_{\ell,j}{\bf e}_{\ell,j},\quad \Pi_{N}h =\sum_{\ell\in\Z \atop 2\leqslant |j|\leqslant N }h_{\ell,j}{\bf e}_{\ell,j}.
$$
By definition
$$\mathtt{L}_n{\bf e}_{\ell,j}={\bf e}_{\ell,j}\left\lbrace\begin{array}{rcl}
 \ii\big(\varepsilon^2\omega(\xi_0)\,\ell+\mu_{j,2}\big);& \hbox{if}& \quad  2\leqslant |j|\leqslant N_n,\quad \ell\in\ZZ,\\
  \ii\,\varepsilon^2\omega(\xi_0)\,\ell;& \hbox{if} & \quad |j|> N_n,\quad \ell\in\ZZ.
\end{array}\right.$$
Define the diagonal  operator  $\mathtt{T}_n$ by 
\begin{eqnarray*}
\mathtt{T}_{n}h(\varphi,\theta)&\triangleq&
-\ii\sum_{2\leqslant |j|\leqslant N_n \atop 
  \ell\in\ZZ}\tfrac{\chi\left((\varepsilon^2\omega_0\,\ell+\mu_{j,2})\lambda^{-1}| j|^{\tau}\right)}{\varepsilon^2\omega_0\, \ell+\mu_{j,2}}h_{\ell,j}\,{\bf e}_{\ell,j}(\varphi,\theta),\\
\end{eqnarray*}
where $\chi\in\mathscr{C}^\infty(\mathbb{R},[0,1])$ is an even positive cut-off function  such that 
\begin{equation}\label{properties cut-off function first reduction} 
	\chi(\xi)=\left\{ \begin{array}{ll}
		0\quad \hbox{if}\quad |\xi|\leqslant\frac13,&\\
		1\quad \hbox{if}\quad |\xi|\geqslant\frac12.
	\end{array}\right.
\end{equation}
Thus, in  the Cantor set $\mathcal{O}_{n,2}$ one has
\begin{align}\label{LT11}
\mathtt{L}_n\mathtt{T}_{n}=\hbox{Id}-\Pi_{N_n}^\perp.
\end{align}
One can easily  check  from Fourier side  that for any $N\geqslant 0$
$$
\sup_{n\in\mathbb{N}}\|\mathtt{T}_nh\|_{s,N}\leqslant C\lambda^{-1}\|h\|_{s,\tau+N}.
$$
On the other hand, denote by
\begin{align*}
g_{\ell,j}(\xi_0)&=\tfrac{\chi\left((\varepsilon^2\omega_0\cdot \ell+\mu_{j,2})\lambda^{-1}| j|^{\tau}\right)}{\varepsilon^2\omega_0\cdot \ell+\mu_{j,2}}\\
&=a_{j}\,\widehat{\chi}\big(a_{j}A_{\ell,j}(\xi_0)\big),
\end{align*}
with
\begin{align*}\widehat{\chi}(x)\triangleq\,\tfrac{\chi(x)}{x},\quad A_{\ell,j}(\xi_0)&\triangleq\,\varepsilon^2\omega_0\,\ell+\mu_{j,2},\quad a_{j}\triangleq\,\lambda^{-1}| j|^{\tau}.\nonumber
				\end{align*}
				Notice that $\widehat{\chi}$ is $C^{\infty}$ with bounded derivatives and $\widehat{\chi}(0)=0.$  Then
				$$
				\|g_{\ell,j}\|_{L^\infty}\lesssim  \lambda^{-1}| j|^{\tau}.
				$$
				Taking the Lipschitz norm in $\xi_0$,  we get
				$$
				\|A_{\ell,j}\|_{\textnormal{Lip}}\lesssim \langle \ell,j\rangle. 
				$$
				Therefore
				$$
				\|g_{\ell,j}\|_{\textnormal{Lip}}\lesssim a_{j}^2\|A_{\ell,j}\|_{\textnormal{Lip}}\lesssim \lambda^{-2}| j|^{2\tau}  \langle \ell,j\rangle. 
				$$
				It follows that
				\begin{align*}
				\sup_{\xi_1\neq\xi_2\in(a,b)}\frac{\|(\mathtt{T}_nh)(\xi_1)-(\mathtt{T}_nh)(\xi_2)\|_{H^{s-1}}}{|\xi_1-\xi_2|}&\lesssim \lambda^{-1} \sup_{\xi_1\neq\xi_2\in (a,b)}\frac{\|h(\xi_1)-h(\xi_2)\|_{H^{s-1,\tau}}}{|\xi_1-\xi_2|}\\
				&\quad+\lambda^{-2} \sup_{\xi\in (a,b)}\|h(\xi)\|_{H^{s,2\tau}}.
				\end{align*}
				Consequently
				\begin{align*}
				\sup_{n\in\mathbb{N}}\|\mathtt{T}_nh\|_{s}^{\textnormal{Lip}(\lambda)}&\leqslant C\lambda^{-1}\|h\|_{s,2\tau}^{\textnormal{Lip}(\lambda)}.
				\end{align*}
				Similarly, we find for any $N\geqslant 0$
				\begin{align}\label{Tn-01}
				\sup_{n\in\mathbb{N}}\|\mathtt{T}_nh\|_{s,N}^{\textnormal{Lip}(\lambda)}&\leqslant C\lambda^{-1}\|h\|_{s,2\tau+N}^{\textnormal{Lip}(\lambda)}.
				\end{align}
				Putting together \eqref{dekomp1} with \eqref{LT11} yields on   the Cantor set $\mathcal{O}_{n,2}$ to the identity
\begin{align}\label{LT011}
\mathbb{L}_{1,\perp}\mathtt{T}_{n}=\hbox{Id}\underbrace{-\Pi_{N_n}^\perp-\Pi_{N_n}^\perp\mathcal{D}_1\mathtt{T}_{n}}_{\triangleq \mathtt{E}_{n}^{2}}.
\end{align}
From straightforward estimates, using in particular \eqref{Tn-01} and \eqref{D-1-LL}, we find
\begin{align*}
\|\Pi_{N_n}^\perp\mathcal{D}_1\mathtt{T}_nh\|_{s_0}^{\textnormal{Lip}(\lambda)}&\leqslant CN_n^{s_0-s}\|\mathtt{T}_nh\|_{s,1}^{\textnormal{Lip}(\lambda)}\\
&\leqslant CN_n^{s_0-s}\lambda^{-1}\|h\|_{s,2\tau+1}^{\textnormal{Lip}(\lambda)}.
\end{align*}
This ends the proof of the desired result.
\end{proof}
\subsection{Invertibility of $\mathbb{M}$}
The next goal is to revisit   the matrix operator introduced in \eqref{M-def} and explore its invertibility. For this aim we fix $\sigma$, defined throughs \eqref{cond-interval2} as
\begin{equation}\label{def sigma}
	\sigma^\mathtt{n}=\varepsilon^{-2}\lambda,
\end{equation}
where $\mathtt{n}$ is defined in Proposition \ref{propo-monod}.  Here is our key result.
\begin{proposition}\label{propo-Hm-1}
Under the assumptions of Proposition $\ref{prop-perp}$ and the smallness condition
\begin{equation}\label{hao11}
\varepsilon^{-1}\lambda +{\varepsilon^{2+{\mu}} \lambda^{-1}\leqslant\varepsilon_0},\quad  \|\rho\|_{s_0+4\tau+5}^{\textnormal{Lip}(\lambda)} \leqslant 1,
\end{equation}
we get the following results. 
There exists a linear map $\mathbb{P}_n$ satisfying for any $s\in[s_0,S],$
$$
\|\mathbb{P}_nH\|_{s}^{\textnormal{Lip}(\lambda)}\leqslant C\lambda^{-1}\Big(\|H\|_{s+2\tau}^{\textnormal{Lip}(\lambda)}+\|\rho\|_{s+2\tau}^{\textnormal{Lip}(\lambda)}\|H\|_{s_0+2\tau}^{\textnormal{Lip}(\lambda)}\Big)
,
$$
such that in the Cantor set $\mathcal{O}_{n}^2\cap\mathcal{O}_{\sigma}$, where $\mathcal{O}_{n}^2$ is defined in Proposition $\ref{Reso-R0}$  and $\mathcal{O}_{\sigma}$ is introduced in  $\eqref{cond-interval2}$,  we have
$$
\mathbb{M}\mathbb{P}_n=\textnormal{Id}_{X_{\circ}^s\times X_{\perp}^s}+\mathbb{E}_{n}^{2},
$$
with
\begin{align*}
 \|\mathbb{E}_{n}^{2}H\|_{s_0}^{\textnormal{Lip}(\lambda)}
 &\leqslant C\lambda^{-1}N_n^{s_0-s}\Big( \|H\|_{{s+2\tau+1}}^{\textnormal{Lip}(\lambda)}+\|\rho\|_{s+4\tau+5}^{\textnormal{Lip}(\lambda)}  \|H\|_{{s_0+2\tau}}^{\textnormal{Lip}(\lambda)}\Big).
 \end{align*}

\end{proposition}
\begin{proof}
Recall from  \eqref{M-def} the following structure
\begin{align*}
\mathbb{M}&=\begin{pmatrix}
\Pi\mathbb{L}_1\Pi&0 \\
0 & \Pi^\perp\mathbb{L}_1\Pi^\perp
\end{pmatrix}+\begin{pmatrix}
\Pi\partial_\theta\mathcal{R}_1\Pi&\Pi\mathcal{R}_1\Pi^\perp \\
\Pi^\perp\partial_\theta\mathcal{R}_1\Pi &\Pi^\perp\partial_\theta\mathcal{R}_1\Pi^\perp
\end{pmatrix}\\
\nonumber &=\mathbb{M}_1+\partial_\theta\mathbb{R}_1.
\end{align*}
Set
\begin{align*}
\mathbb{K}_n&=\begin{pmatrix}
\mathtt{S}&0 \\
0 & \mathtt{T}_n
\end{pmatrix}
\end{align*}
where the invertibility of the operator $\Pi\mathbb{L}_1\Pi$  was previously discussed  in Proposition \ref{propo-monod} and the operator $\mathtt{T}_n$ was defined in Proposition \ref{prop-perp}.
Then, for all $\xi_0\in \mathcal{O}_{n}^2\cap\mathcal{O}_{\sigma}$ one has the identity
 \begin{align}\label{jeud-10}
\mathbb{M}_1\mathbb{K}_n&=\hbox{Id}_{X_\circ^s\times X_{\perp}^s}+\begin{pmatrix}
0&0 \\
0 & \mathtt{E}_{n}^{2}
\end{pmatrix}.
\end{align}
Consider the operator 
\begin{align}\label{Pn-def}
\mathbb{P}_n=\big(\hbox{Id}_{X_\circ^s\times X_{\perp}^s}+\mathbb{K}_n\partial_\theta\mathbb{R}_1\big)^{-1}\mathbb{K}_n.
\end{align}
Then for $H=(h_1,h_2)\in X_\circ^s\times X_{\perp}^s$ we get according to Proposition \ref{propo-monod}, Proposition \ref{prop-perp}, \eqref{def sigma} and \eqref{hao11} we get
\begin{align}\label{jeud-4}
\nonumber \|\mathbb{K}_nH\|_{s}^{\textnormal{Lip}(\lambda)} &\leqslant C \lambda^{-1}\|h_1\|_s^{\textnormal{Lip}(\lambda)}+C\lambda^{-1}\|h_2\|_{s,2\tau}^{\textnormal{Lip}(\lambda)}\\
&\leqslant C\lambda^{-1}\|H\|_{s,2\tau}^\lambda.
\end{align}
On the other hand, according to  Proposition \ref{lemma-beta1}, 
one gets that  for any $s\in[s_0,S]$ and $N\geqslant0$, 
$$
\| \partial_\theta\mathcal{R}_{1} h\|_{s,N}^{\textnormal{Lip}(\lambda)}\lesssim \big(\varepsilon^{2+{\mu}} +\varepsilon^{6} \lambda^{-1}\big) \Big(  \|h\|_{{s}}^{\textnormal{Lip}(\lambda)}\big(1+\|\rho\|_{s_0+2\tau+5+N}^{\textnormal{Lip}(\lambda)}\big)+\|\rho\|_{s+2\tau+5+N}^{\textnormal{Lip}(\lambda)}  \|h\|_{{s_0}}^{\textnormal{Lip}(\lambda)}\Big).
$$
 Hence, we obtain in view of \eqref{hao11} and $\mu\in(0,1)$
\begin{equation}\label{jeud-12}
\|\partial_\theta\mathbb{R}_1H\|_{s,N}^{\textnormal{Lip}(\lambda)}
\leqslant C\varepsilon^{2+{\mu}} \Big( \|H\|_{{s}}^{\textnormal{Lip}(\lambda)}\big(1+\|\rho\|_{s_0+N+2\tau+5}^{\textnormal{Lip}(\lambda)}\big)+\|\rho\|_{s+N+2\tau+5}^{\textnormal{Lip}(\lambda)}  \|H\|_{{s_0}}^{\textnormal{Lip}(\lambda)}\Big).
\end{equation}
Plugging \eqref{jeud-12} into \eqref{jeud-4} yields according to  \eqref{hao11} and  $\mu\in(0,1)$,
\begin{align}\label{Tn-induc}
\nonumber\|\mathbb{K}_n\partial_\theta\mathbb{R}_1H\|_{s}^{\textnormal{Lip}(\lambda)} &\leqslant  C\lambda^{-1}\|\partial_\theta\mathbb{R}_1H\|_{s,2\tau}^{\textnormal{Lip}(\lambda)}
\\ 
&\leqslant C\varepsilon^{2+\mu}\lambda^{-1}\Big( \|H\|_{{s}}^{\textnormal{Lip}(\lambda)}+\|\rho\|_{s+4\tau+5}^{\textnormal{Lip}(\lambda)}  \|H\|_{{s_0}}^{\textnormal{Lip}(\lambda)}\Big).
\end{align}
 It follows  under  \eqref{hao11} that  the operator 
$$\hbox{Id}_{X_\circ^{s_0}\times X_{\perp}^{s_0}}+\mathbb{K}_n\partial_\theta\mathbb{R}_1: {X_\circ^{s_0}\times X_{\perp}^{s_0}}\to {X_\circ^{s_0}\times X_{\perp}^{s_0}}
$$ is invertible  with
$$
\|(\hbox{Id}_{X_\circ^s\times X_{\perp}^s}+\mathbb{K}_n\partial_\theta\mathbb{R}_1)^{-1}H\|_{s_0}^{\textnormal{Lip}(\lambda)}\leqslant 2 \|H\|_{s_0}^{\textnormal{Lip}(\lambda)}.
$$
As to the invertibility for $s\in[s_0,S]$, one can check by induction from \eqref{Tn-induc}, $\forall m\geqslant 1,$
\begin{align*}
 &\|(\mathbb{K}_n\partial_\theta\mathbb{R}_1)^mH\|_{s}^{\textnormal{Lip}(\lambda)} \leqslant \Big(C\varepsilon^{2+\mu}\lambda^{-1}\Big)^m \|H\|_{{s}}^{\textnormal{Lip}(\lambda)}\\
&+C\varepsilon^{2+\mu}\lambda^{-1}(m+1) 2^{m-2}\Big(C\varepsilon^{2+\mu}\lambda^{-1}\Big)^{m-1}\|\rho\|_{s+4\tau+5}^{\textnormal{Lip}(\lambda)}  \|H\|_{{s_0}}^{\textnormal{Lip}(\lambda)}.
\end{align*}
Consequently, we get under the smallness condition \eqref{hao11} 
\begin{align*}
\sum_{m\geqslant 0}\|(\mathbb{K}_n\partial_\theta\mathbb{R}_1)^mH\|_{s}^{\textnormal{Lip}(\lambda)} &\leqslant C \|H\|_{{s}}^{\textnormal{Lip}(\lambda)}+C\|\rho\|_{s+4\tau+5}^{\textnormal{Lip}(\lambda)}  \|H\|_{{s_0}}^{\textnormal{Lip}(\lambda)}.
\end{align*}
It follows that
\begin{align}\label{jeud-2}
\|(\hbox{Id}_{X_\circ^s\times X_{\perp}^s}+\mathbb{K}_n\partial_\theta\mathbb{R}_1)^{-1}H\|_{s}^{\textnormal{Lip}(\lambda)} &\leqslant C \|H\|_{{s}}^{\textnormal{Lip}(\lambda)}+C\|\rho\|_{s+4\tau+5}^{\textnormal{Lip}(\lambda)}  \|H\|_{{s_0}}^{\textnormal{Lip}(\lambda)}.
\end{align}
Therefore, we deduce from \eqref{Pn-def}, \eqref{jeud-2}, \eqref{jeud-4} and \eqref{hao11}
\begin{align}\label{jeud-3}
\nonumber \|\mathbb{P}_nH\|_{s}^{\textnormal{Lip}(\lambda)} &\leqslant C \|\mathbb{K}_nH\|_{{s}}^{\textnormal{Lip}(\lambda)}+C\|\rho\|_{s+4\tau+5}^{\textnormal{Lip}(\lambda)}  \|\mathbb{K}_nH\|_{{s_0}}^{\textnormal{Lip}(\lambda)}\\
&\leqslant C\lambda^{-1}\Big(\|H\|_{s,2\tau}^\lambda+\|\rho\|_{s+4\tau+5}^{\textnormal{Lip}(\lambda)}\|H\|_{s_0,2\tau}^{\textnormal{Lip}(\lambda)}\Big).
\end{align}
By virtue of \eqref{jeud-10} and \eqref{Pn-def} we deduce that on the Cantor set $\mathcal{O}_{\sigma}\cap \mathcal{O}_n^2$ one has
\begin{align}\label{Inv-PP1}
\mathbb{M}\mathbb{P}_n&=\left(\mathbb{M}_1+\left(\mathbb{M}_1\mathbb{K}_n-\begin{pmatrix}
0&0 \\
0 & \mathtt{E}_{n}^{2}
\end{pmatrix}\right)\partial_\theta\mathbb{R}_1\right)\left(\hbox{Id}_{X_\circ^s\times X_{\perp}^s}+\mathbb{K}_n\partial_\theta\mathbb{R}_1\right)^{-1}\mathbb{K}_n\\
\nonumber&=\mathbb{M}_1\mathbb{K}_n-\begin{pmatrix}
0&0 \\
0 & \mathtt{E}_{n}^{2}
\end{pmatrix}\partial_\theta\mathbb{R}_1\mathbb{P}_n\\
\nonumber&=\hbox{Id}_{X_\circ^s\times X_{\perp}^s}+\underbrace{\begin{pmatrix}
0&0 \\
0 & \mathtt{E}_{n}^{2}
\end{pmatrix}-\begin{pmatrix}
0&0 \\
0 & \mathtt{E}_{n}^{2}
\end{pmatrix}\partial_\theta\mathbb{R}_1\mathbb{P}_n}_{\triangleq\,\mathbb{E}_{n}^{2}}.
\end{align}
Applying Proposition \ref{prop-perp} 
\begin{align}\label{Eq-1L}
\|\mathbb{E}_{n}^{2}H\|_{s_0}^{\textnormal{Lip}(\lambda)}
 &\leqslant C\lambda^{-1}N_n^{s_0-s}\Big( \|H\|_{{s,2\tau+1}}^{\textnormal{Lip}(\lambda)}+ \|\partial_\theta\mathbb{R}_1\mathbb{P}_nH\|_{{s,2\tau+1}}^{\textnormal{Lip}(\lambda)}\Big).
\end{align}
Putting together \eqref{jeud-12} and \eqref{jeud-3} with \eqref{hao11} yields 
\begin{align}\label{jeud-13}
\nonumber \|\partial_\theta\mathbb{R}_1\mathbb{P}_n H\|_{s,2\tau+1}^\lambda 
&\leqslant C\varepsilon^{2+{\mu}}\Big( \|\mathbb{P}_n H\|_{{s}}^{\textnormal{Lip}(\lambda)}+\|\rho\|_{s+4\tau+5}^{\textnormal{Lip}(\lambda)}  \|\mathbb{P}_n H\|_{{s_0}}^{\textnormal{Lip}(\lambda)}\Big)\\
&\leqslant C\varepsilon^{2+{\mu}} \lambda^{-1}\Big(\|H\|_{s+2\tau}^{\textnormal{Lip}(\lambda)}+\|\rho\|_{s+4\tau+5}^{\textnormal{Lip}(\lambda)}\|H\|_{s_0+2\tau}^{\textnormal{Lip}(\lambda)}\Big).
\end{align}
Inserting \eqref{jeud-13} into \eqref{Eq-1L} leads to 
\begin{align}\label{Eq-1L0}
 \|\mathbb{E}_{n}^{2}H\|_{s}^{\textnormal{Lip}(\lambda)}
 &\leqslant C\lambda^{-1}N_n^{s_0-s} \|H\|_{{s+2\tau+1}}^{\textnormal{Lip}(\lambda)}\\
 \nonumber&+C\varepsilon^{2+{\mu}}\lambda^{-2}N_n^{s_0-s}\Big(\|H\|_{s+2\tau}^{\textnormal{Lip}(\lambda)}+\|\rho\|_{s+4\tau+5}^{\textnormal{Lip}(\lambda)}\|H\|_{s_0+2\tau}^{\textnormal{Lip}(\lambda)}\Big).
\end{align}
Thus, we infer from \eqref{hao11} and \eqref{Eq-1L0}
\begin{align*}
 \|\mathbb{E}_{n}^{2}H\|_{s_0}^{\textnormal{Lip}(\lambda)}
 &\leqslant C\lambda^{-1}N_n^{s_0-s}\Big( \|H\|_{{s+2\tau+1}}^{\textnormal{Lip}(\lambda)}+\|\rho\|_{s+4\tau+5}^{\textnormal{Lip}(\lambda)}  \|H\|_{{s_0+2\tau}}^{\textnormal{Lip}(\lambda)}\Big).\end{align*}
 Remark that, one gets by construction the algebraic properties
 \begin{align}\label{Alg-Iden-Jul}
\mathbb{M}=\begin{pmatrix}
\Pi&0 \\
0 & \Pi^\perp
\end{pmatrix} \mathbb{M}\begin{pmatrix}
\Pi&0 \\
0 & \Pi^\perp
\end{pmatrix}\quad\hbox{and}\quad \mathbb{P}_n=\begin{pmatrix}
\Pi&0 \\
0 & \Pi^\perp
\end{pmatrix} \mathbb{P}_n\begin{pmatrix}
\Pi&0 \\
0 & \Pi^\perp
\end{pmatrix}.
 \end{align}
 This completes the proof of the desired result.
\end{proof}
\subsection{Invertibility of ${\mathcal{L}_1}$ and ${\mathcal{L}_0}$ .}
In this section we intend to explore the   existence of an approximate right inverse to   the operators ${\mathcal{L}_1}$ and ${\mathcal{L}_0}$ defined in \eqref{L1-deff} and \eqref{linearized f}, respectively.
\begin{proposition}\label{prop-inverse}
Given the conditions \eqref{cond1}, \eqref{Conv-Trans0} and  \eqref{cond-interval2}. 
There exists $\epsilon_0>0$ such that under the assumptions
\begin{equation}\label{hao11MM}
\varepsilon^{-1}\lambda ++\varepsilon^{4+\mu}\lambda^{-1}N_0^{\mu_2}+{\varepsilon^{2+{\mu}} \lambda^{-1}\leqslant\epsilon_0},\quad  \|\rho\|_{2s_0+2\tau+4+\frac32\mu_2}^{\textnormal{Lip}(\lambda)} \leqslant 1,
\end{equation}
the following assertions  holds true. 

\begin{enumerate}
\item There exists a family of linear  operators $\big({\mathbf{T}}_{n}\big)_{n\in\mathbb{N}}$ satisfying
\begin{equation*}
				\forall \, s\in\,[ s_0, S],\quad\sup_{n\in\mathbb{N}}\|{\mathbf{T}}_{n}h\|_{s}^{\textnormal{Lip}(\lambda)}\leqslant  C\lambda^{-1}\Big(\|h\|_{s+2\tau}^{\textnormal{Lip}(\lambda)}+\|\rho\|_{s+2\tau}^{\textnormal{Lip}(\lambda)}\|h\|_{s_0+2\tau}^{\textnormal{Lip}(\lambda)}\Big)
	\end{equation*}
			and such that in the Cantor set  $\mathcal{O}_{n}^2\cap\mathcal{O}_{\sigma}$, where $\mathcal{O}_{n}^2$ is introduced in Proposition $\ref{prop-perp}$ and $\mathcal{O}_{\sigma}$ is given by \eqref{cond-interval2},
			we have
			$$
			{\mathcal{L}_1}\,{\mathbf{T}}_{n}=\textnormal{Id}+\mathbf{E}_n,
			$$
			where $\mathbf{E}_n$ satisfies the following estimate
			\begin{align*}
				\forall\, s\in [s_0,S],\quad  \|\mathbf{E}_n h\|_{s_0}^{\textnormal{Lip}(\lambda)}
				\nonumber&\leqslant C \lambda^{-1}N_n^{s_0-s}\Big( \|h\|_{s+2\tau+1}^{\textnormal{Lip}(\lambda)}+\| \rho\|_{s+4\tau+5}^{\textnormal{Lip}(\lambda)}\|h\|_{s_{0}+2\tau}^{\textnormal{Lip}(\lambda)} \Big)\\
				&\quad +\varepsilon^6\lambda^{-1}N_{0}^{\mu_{2}}N_{n+1}^{-\mu_{2}}\|h\|_{s_{0}+2\tau+2}^{\textnormal{Lip}(\lambda)}.
			\end{align*}

				\item 
There exists a family of linear  operators $\big({\mathcal{T}}_{n}\big)_{n\in\mathbb{N}}$ satisfying
\begin{equation*}
				\forall \, s\in\,[ s_0, S],\quad\sup_{n\in\mathbb{N}}\|{\mathcal{T}}_{n}h\|_{s}^{\textnormal{Lip}(\lambda)}\leqslant C\lambda^{-1}\Big(\|h\|_{s+2\tau}^{\textnormal{Lip}(\lambda)}+\|\rho\|_{s+{4\tau+5}}^{\textnormal{Lip}(\lambda)}\|h\|_{s_0+2\tau}^{\textnormal{Lip}(\lambda)}\Big)
			\end{equation*}
			and such that in the Cantor set $\mathcal{O}_n^{1}(\rho)\cap \mathcal{O}_n^{2}(\rho)\cap\mathcal{O}_{\sigma}$ $($See \eqref{Cantor set0} for $\mathcal{O}_n^{1}(\rho))$
			we have
			$$
			{\mathcal{L}_0}\,{\cal{T}}_{n}=\textnormal{Id}+{\cal{E}}_{n},
			$$
			with  the following estimate
			\begin{align*}
				 \|\mathcal{E}_{n}h\|_{s_0}^{\textnormal{Lip}(\lambda)}
				\nonumber&\leqslant C\lambda^{-1}N_n^{s_0-s}\big( \|h\|_{s+2\tau+1}^{\textnormal{Lip}(\lambda)}+\| \rho\|_{s+4\tau{+6}}^{\textnormal{Lip}(\lambda)}\|h\|_{s_{0}+2\tau}^{\textnormal{Lip}(\lambda)}\big)\\ &+C\varepsilon^6\lambda^{-1}N_{0}^{\mu_{2}}N_{n+1}^{-\mu_{2}}\|h\|_{s_{0}+2\tau+2}^{\textnormal{Lip}(\lambda)}.
			\end{align*}
		\end{enumerate}
\end{proposition}
\begin{proof}
{\bf{1.}} 
From \eqref{M-def}, \eqref{Pn-def} and \eqref{Inv-PP1} we may write
\begin{align*}
\mathbb{M}&=\begin{pmatrix}
M_1&M_2 \\
M_3 & M_4
\end{pmatrix},\quad\mathbb{P}_n=\begin{pmatrix}
P_1&P_2 \\
P_3&P_4
\end{pmatrix}\quad\hbox{and}\quad \mathbb{E}_{n}^{2}=\begin{pmatrix}
E_{n,1}^2&E_{n,2}^2 \\
E_{n,3}^2&E_{n,4}^2
\end{pmatrix}.
\end{align*}
Then one can check from \eqref{L1-deff} that
$$
\mathcal{L}_1=\sum_{j=1}^4M_j+\mathtt{E}_{n}.
$$
Denote 
$$
{\mathbf{T}_n}= \sum_{i=1}^4P_i \quad\hbox{and}\quad {\mathbf{E}_n}=\sum_{i=1}^4 E_{n,i}^2+\mathtt{E}_{n}{\mathbf{T}_n}
$$
Then from the algebraic structure \eqref{Alg-Iden-Jul} and the identity \eqref{Inv-PP1} we infer
$$
\mathcal{L}_1{\mathbf{T}_n}=\hbox{Id}+ {\mathbf{E}_n}.
$$
The desired estimates on $\mathbf{T}_n$ and $\mathbf{E}_n$ follow from Proposition \ref{propo-Hm-1} and  Proposition  \ref{QP-change}.\\
{\bf{2.}} Recall that $\mathcal{L}_1=\mathscr{A}^{-1}\mathcal{L}_0\mathscr{A},$ then
$$
\mathscr{A}^{-1}\mathcal{L}_0\mathscr{A}{\mathbf{T}_n}=\hbox{Id}+ {\mathbf{E}_n}.
$$
Thus by setting
$$
\mathcal{T}_n:=\mathscr{A}{\mathbf{T}_n}\mathscr{A}^{-1}
$$
we find
$$
\mathcal{L}_0{\mathcal{T}_n}=\hbox{Id}+ \mathscr{A}{\mathbf{E}_n}\mathscr{A}^{-1}:=\hbox{Id}+{\mathcal{E}_n}.
$$
The estimate of ${\mathcal{E}_n}$ follows from  Proposition \ref{prop-inverse}-1 and Proposition \ref{lemma-beta1} -1 together with \eqref{hao11MM}.
This achieves the proof of Proposition \ref{prop-inverse}.\end{proof}

\section{Construction of the solutions}\label{N-M-S1}
In this section, our primary objective is to establish the main result discussed in \mbox{Theorem \ref{TH-Main1}.} We will accomplish this through a series of steps. First, we will construct approximate solutions using a modified Nash-Moser scheme, aligning with the approach developed in the papers \cite{Baldi-berti,BB13}. During this phase, Proposition \ref{prop-inverse} will play a pivotal role in the induction step, as it provides the existence of an approximate right inverse with tame  estimates.
Subsequently, we will analyze the convergence of the scheme and demonstrate the existence of solutions under the condition that the parameter $\xi_0$ belongs to a Cantor-like set. Finally,  we will address the estimation of the Cantor set and prove that its Lebesgue measure is nearly full. 
\subsection{Nash-Moser scheme}
The main aim of this section is to construct solutions to the nonlinear equation 
\begin{equation}\label{main-eq1}
\mathcal{F}(\rho)\triangleq\tfrac{1}{\varepsilon^{2+\mu}} G( \varepsilon r_\varepsilon+\varepsilon^{\mu+1}\rho)=0
\end{equation}
seen before in \eqref{def nonlinear functional fN}. We will achieve this goal  by employing a modified  Nash-Moser scheme in the spirit of  \cite{Baldi-berti,BB13,BB10,BertiMontalto}. As we shall see,  we will  devise a recursive explicit scheme where at each step we generate an approximate solution lying in  the  finite-dimensional space  
$$
E_{n}\triangleq\Big\{h: \mathcal{O}\times\T^2\to\R;\, \Pi_nh=h\Big\}$$
where $\mathcal{O}=[\xi_*,\xi^*]$  and  $\Pi_{n}$ is the projector defined   
$$
h(\varphi,\theta)=\sum_{\ell\in\Z\atop j\in\Z^\star}h_{\ell,j}e^{\ii(\ell \varphi+j\theta)},\quad\Pi_{n}h(\varphi,\theta)=\sum_{|\ell|+|j|\leqslant N_{n}}h_{\ell,j}e^{\ii(\ell \varphi+j\theta)},
$$
{and  the sequence of numbers  $(N_n)_{n}$ was defined  in \eqref{definition of Nm}.  Here we will make use of the parameters introduced in \eqref{cond1} supplemented  with the following   number
\begin{align}\label{choice-f1}
 \quad \mathtt{b}_0=3-\mu.
\end{align}
In addition, we shall fix the number  $N_0$ used  \eqref{definition of Nm} and $\lambda$ (the parameter of the Cantor sets $\mathcal{O}_{n}^1(\rho)$ and $\mathcal{O}_{n}^2(\rho)$) with the respect to $\varepsilon$ as below
				\begin{equation}\label{lambda-choice}
N_{0}\triangleq\varepsilon^{-{\delta}}, \quad \lambda= \varepsilon^{{\delta}}{\varepsilon^2}.\end{equation}
 Moreover, we shall  impose the following constraints required along the  Nash-Moser scheme,
\begin{equation}\label{Assump-DRP1}\left\lbrace\begin{array}{rcl}
						1+\tau&<&a_2\\3s_0+12\tau+15+\tfrac32 a_2&<& a_1\\
												\frac{2}{3}a_{1}&< & \mu_{2}\\
	0<\delta&<& \min\big(\mu,\tfrac{1-\mu}{a_1+2},{\tfrac{2+\mu}{1+\mu_2}\big)}\\
						12\tau+3+\tfrac{6}{\delta}&<&\mu_1\\
						\max\left(s_{0}+4\tau+3+\frac{2}{3}\mu_{1}+a_{1} +{\frac4\delta},3s_0+4\tau+6+3\mu_2\right)&< & b_{1}.
					\end{array}\right.
				\end{equation}	
				Note that we can select these parameters as follows: Given the values $s_0,\tau, \mu$ as specified  in \eqref{cond1}, then we generate  the parameters successively  in the following order $a_2,a_1,\mu_2,\delta,\mu_1$ and $b_1$ satisfying the  conditions outlined in   \eqref{Assump-DRP1} in the prescribed order. This allows for a multitude of valid parameter choices, and any admissible selection will suffice to get the results discussed in  this section. Now, let's proceed to our central result, which revolves around the implementation of a Nash-Moser scheme to generate approximate solutions for \eqref{main-eq1}.
\begin{proposition}[Nash-Moser scheme]\label{Nash-Moser}
Assume the conditions \eqref{choice-f1},\eqref{lambda-choice} and   \eqref{Assump-DRP1}. There exist $C_{\ast}>0$ and ${\varepsilon}_0>0$ such that for any $\varepsilon\in[0,\varepsilon_0]$   we get  for all $n\in\mathbb{N}$ the following properties,
\begin{itemize}
\item  $(\mathcal{P}1)_{n}$ There exists a Lipschitz function 
$$\rho_{n}:\begin{array}[t]{rcl}
\mathcal{O} & \rightarrow &  E_{n-1}\\
\xi_0 & \mapsto & \rho_n
\end{array}$$
satisfying 
$$
\rho_{0}=0\quad\mbox{ and }\quad\,\| \rho_{n}\|_{{2s_0+2\tau+3}}^{\textnormal{Lip}(\lambda)}\leqslant C_{\ast}\varepsilon^{\mathtt{b}_0}\lambda^{-1}\quad \hbox{for}\quad n\geqslant1.
$$
By setting 
$$
\quad {u}_{n} =\rho_{n}-\rho_{n-1} \quad \hbox{for}\quad n\geqslant1,
$$
 we have 
$$
\forall s\in[s_0,S], \,\| {u}_{1}\|_{s}^{{\textnormal{Lip}(\lambda)}}\leqslant\tfrac12 C_{\ast}\varepsilon^{\mathtt{b}_0}\lambda^{-1}\quad\hbox{and}\quad \| {u}_{k}\|_{{{2s_0+2\tau+3}}}^{{\textnormal{Lip}(\lambda)}}\leqslant C_{\ast}\varepsilon^{\mathtt{b}_0}\lambda^{-1}N_{k-1}^{-a_{2}}\quad  \forall\,\, 2\leqslant k\leqslant n.
$$
\item $(\mathcal{P}2)_{n}$ Set 
$$
\mathcal{A}_{0}=\mathcal{O}\quad \mbox{ and }\quad \mathcal{A}_{n+1}=\mathcal{A}_{n}\cap\mathcal{O}_{n}^1(\rho_n)\cap\mathcal{O}_{n}^2(\rho_n)\cap\mathcal{O}_{\sigma} \quad\forall n\in\mathbb{N}.
$$
Then we have the following estimate 
$$\|\mathcal{F}(\rho_{n})\|_{s_{0},\mathcal{A}_{n}}^{{\textnormal{Lip}(\lambda)}}\leqslant C_{\ast}\varepsilon^{\mathtt{b}_0} N_{n-1}^{-a_{1}}.
$$
\item $(\mathcal{P}3)_{n}$ High regularity estimate: $\| \rho_{n}\|_{b_1}^{{\textnormal{Lip}(\lambda)}}\leqslant C_{\ast}\varepsilon^{\mathtt{b}_0}\lambda^{-1} N_{n-1}^{\mu_1}.$
\end{itemize}
Here, we have used  the notation \eqref{Norm-not}.
\end{proposition}

\begin{proof}
The proof will be implemented  using an induction principle.\\
 $\bullet$ \textit{ Initialization.}  According to Lemma \ref{lem: construction of appx sol}, \eqref{main-eq1} and \eqref{choice-f1} one has 
 \begin{equation}\label{F-zero}
\|\mathcal{F}(0)\|_{s,\mathcal{O}}^{\textnormal{Lip}(\lambda)}=O(\varepsilon^{3-\mu})=O(\varepsilon^{\mathtt{b}_0}).
 \end{equation}
The properties $(\mathcal{P}1)_{0},$ $(\mathcal{P}2)_{0}$ and $(\mathcal{P}3)_{0}$ then follow immediately.\\
$\bullet$ {\it{Induction step:}} Given $n\in\mathbb{N}$ and  assume that  we have constructed $\rho_n$ satisfying the assumptions 
$(\mathcal{P}1)_{k},$ $(\mathcal{P}2)_{k}$ and $(\mathcal{P}3)_{k}$   for all $k\in\llbracket 0,n\rrbracket$ and let us check them at the next order $n+1$. As we shall explain now,    the next  approximation $\rho_{n+1}$ will be performed through  a modified Nash-Moser scheme.  First, let us verify   one by one the assumptions \eqref{hao11MM} needed for the invertibility of the linearized operator.  For the first one, it can be obtained from \eqref{lambda-choice} and \eqref{Assump-DRP1} together with   \eqref{lambda-choice} \mbox{and \eqref{Assump-DRP1}}, which yield,
\begin{align*}
\varepsilon^{4+\mu}\lambda^{-1}N_0^{\mu_2}&=\varepsilon^{2+\mu-\delta(1+\mu_2)}\\&\leqslant \epsilon_0. 
\end{align*}
As to the second one in  \eqref{hao11MM}, we may use Sobolev embeddings and the interpolation inequality from Lemma \ref{lem funct prop}
\begin{align*}
\|u_k\|_{2s_{0}+2\tau+4+\frac32\mu_2}^{\textnormal{Lip}(\lambda)}&\leqslant \|u_k\|_{\frac{s_{0}}{2}+1+\frac{b_1}{2}}^{\textnormal{Lip}(\lambda)}\\
&\leqslant \left( \|u_k\|_{s_{0}+2}^{\textnormal{Lip}(\lambda)}\right)^{\frac12}\left( \|u_k\|_{b_1}^{\textnormal{Lip}(\lambda)}\right)^{\frac12}
\end{align*}
provided that
\begin{align*}
3s_0+4\tau+6+3\mu_2\leqslant b_1,
\end{align*}
which follows from \eqref{Assump-DRP1}.
Applying  $(\mathcal{P}1)_{n},$ $(\mathcal{P}1)_{k},$ with $1\leqslant k\leqslant n$,   
\begin{align*}
2\leqslant k\leqslant n,\quad \|u_k\|_{2s_{0}+2\tau+4+\frac32\mu_2}^{\textnormal{Lip}(\lambda)}&\leqslant C_{\ast}\varepsilon^{\mathtt{b}_0}\lambda^{-1}N_{k-1}^{\frac{\mu_1-a_2}{2}}
\end{align*}
and
\begin{align*}
\|u_1\|_{2s_{0}+2\tau+4+\frac32\mu_2}^{\textnormal{Lip}(\lambda)}&\leqslant \tfrac12 C_{\ast}\varepsilon^{\mathtt{b}_0}\lambda^{-1}.
\end{align*}
Consequently, we infer by the triangle inequality, \eqref{lambda-choice}, \eqref{choice-f1} 
\begin{align*}
 \|\rho_n\|_{2s_{0}+2\tau+4+\frac32\mu_2}^{\textnormal{Lip}(\lambda)}&\leqslant \sum_{k=1}^n\|u_k\|_{2s_{0}+2\tau+4+\frac32\mu_2}^{\textnormal{Lip}(\lambda)}\\
 &\lesssim C_{\ast}\varepsilon^{\mathtt{b}_0}\lambda^{-1}\\
 &\lesssim C_{\ast}\varepsilon^{1-\mu-\delta}.
\end{align*}
From \eqref{Assump-DRP1} we deduce   for small $\varepsilon$ 
$$
 \|\rho_n\|_{2s_{0}+2\tau+4+\frac32\mu_2}^{\textnormal{Lip}(\lambda)}\leqslant 1.
$$
Hence, we can apply Proposition \ref{prop-inverse}-2. with the operator  
$$\mathcal{L}_{n}\triangleq\,\partial_\rho\mathcal{F}(\rho_n)$$
implying  the existence of an operator $\mathcal{T}_n$ well-defined in the whole set of parameters $\mathcal{O}$ and such that \begin{equation}\label{estimate Tm}	
				\forall \, s\in\,[ s_0, S],\quad\|{\mathcal{T}}_{n}h\|_{s}^{\textnormal{Lip}(\lambda)}\leqslant C\lambda^{-1}\Big(\|h\|_{s+2\tau}^{\textnormal{Lip}(\lambda)}+\|\rho_n\|_{s+4\tau+5}^{\textnormal{Lip}(\lambda)}\|h\|_{s_0+2\tau}^{\textnormal{Lip}(\lambda)}\Big)
			\end{equation}
			such that on  the Cantor set $\mathcal{A}_{n+1}$ we have
			$$
				{\mathcal{L}_n}\,{\cal{T}}_{n}=\textnormal{Id}+{\cal{E}}_{n}
			$$
			with suitable  estimates on the remainder  ${\cal{E}}_{n}$.
				Now, according  to  $(\mathcal{P}1)_{n}$ and \eqref{estimate Tm} we infer 
				\begin{align}\label{estimate Tm in norm s0}
					\nonumber \|{\mathcal{T}}_{n}h\|_{s_0}^{\textnormal{Lip}(\lambda)}&\leqslant   C\lambda^{-1}\|h\|_{s_0+2\tau}^{\textnormal{Lip}(\lambda)}\Big(1+C_{\ast}\varepsilon^{\mathtt{b}_0}\lambda^{-1}{}\Big)\\
			&\leqslant   C\lambda^{-1}\|h\|_{s_0+2\tau}^{\textnormal{Lip}(\lambda)},	
	\end{align}
	due to the estimate below which follows  from \eqref{choice-f1}, \eqref{lambda-choice} and \eqref{Assump-DRP1}
	\begin{align}\label{small-PP1}
	\nonumber \varepsilon^{\mathtt{b}_0}\lambda^{-1}&=\varepsilon^{1-\mu-\delta}\\
				&\leqslant \epsilon_0.
\end{align}
Now,  we define
				\begin{align}\label{def-un}
				{\rho}_{n+1}\triangleq \rho_{n}+{u}_{n+1}\quad\mbox{ with }\quad {u}_{n+1}\triangleq -{\Pi}_{n}\hbox{Ext}\mathcal{T}_{n}\Pi_{n}\mathcal{F}(\rho_{n})\in E_{n},
				\end{align}
				where   $\textnormal{Ext} f$ denotes a Lipschitz extension of $f$ from the set $\mathcal{A}_n$ to the full interval $[a,b]=\mathcal{O}$ as stated  in Lemma \ref{thm-extend}. In particular, we get by 
				\begin{align}\label{ext-Hm1}
				\forall \xi_0\in \mathcal{A}_n,\quad {u}_{n+1}=-{\Pi}_{n}\mathcal{T}_{n}\Pi_{n}\mathcal{F}(\rho_{n})\quad\hbox{and}\quad 
				\|u_{n+1}\|_{s,\mathcal{O}}^{{\textnormal{Lip}(\lambda)}} \lesssim \|{u}_{n+1}\|_{s,\mathcal{A}_n}^{{\textnormal{Lip}(\lambda)}}.
				\end{align}
				It is worthy to point out that the following  estimate from $(\mathcal{P}1)_{n}$ 
				$$
				\forall s\in[s_0,S], \,\| {u}_{1}\|_{s}^{{\textnormal{Lip}(\lambda)}}\leqslant\tfrac12 C_{\ast}\varepsilon^{\mathtt{b}_0}\lambda^{-1}
				$$
				can be derived  from Proposition \ref{prop-inverse}-$(2)$ and \eqref{F-zero}  as follows
				\begin{align*}
				\forall s\in[s_0,S],\quad \| {u}_{1}\|_{s}^{{\textnormal{Lip}(\lambda)}}&\leqslant C\| \mathcal{T}_{0}\Pi_{0}\mathcal{F}(0)\|_{s}^{{\textnormal{Lip}(\lambda)}}\\
				&\leqslant C\lambda^{-1}\| \mathcal{F}(0)\|_{s+2\tau}^{{\textnormal{Lip}(\lambda)}}\\
				&\leqslant C\lambda^{-1}\varepsilon^{\mathtt{b}_0}.
				\end{align*}
				Then to match with the desired estimate,  it is enough to  take $C_*$ large enough.
			 Next, we introduce  the quadratic function
				\begin{align}\label{Def-Qm}
					Q_{n}\triangleq \mathcal{F}(\rho_{n}+{u}_{n+1})-\mathcal{F}(\rho_{n})-\mathcal{L}_{n}{u}_{n+1}.
				\end{align}
Then, in the set $\mathcal{A}_n$ and using simple transformations we get
				\begin{align}\label{Decom-RTT1}
					\nonumber \mathcal{F}({\rho}_{n+1})& =  \mathcal{F}(\rho_{n})-\mathcal{L}_{n}{\Pi}_{n}\mathcal{T}_{n}\Pi_{n}\mathcal{F}(\rho_{n})+Q_{n}\\
					\nonumber& =  \mathcal{F}(\rho_{n})-\mathcal{L}_{n}\mathcal{T}_{n}\Pi_{n}\mathcal{F}(\rho_{n})+\mathcal{L}_{n}{\Pi}_{n}^{\perp}\mathcal{T}_{n}\Pi_{n}\mathcal{F}(\rho_{n})+Q_{n}\\
					\nonumber& =  \mathcal{F}(\rho_{n})-\Pi_{n}\mathcal{L}_{n}\mathcal{T}_{n}\Pi_{n}\mathcal{F}(\rho_{n})+(\mathcal{L}_{n}{\Pi}_{n}^{\perp}-\Pi_{n}^{\perp}\mathcal{L}_{n})\mathcal{T}_{n}\Pi_{n}\mathcal{F}(\rho_{n})+Q_{n}\\
					& =  \Pi_{n}^{\perp}\mathcal{F}(\rho_{n})-\Pi_{n}(\mathcal{L}_{n}\mathcal{T}_{n}-\textnormal{Id})\Pi_{n}\mathcal{F}(\rho_{n})+(L_{n}{\Pi}_{n}^{\perp}-\Pi_{n}^{\perp}\mathcal{L}_{n})\mathcal{T}_{n}\Pi_{n}\mathcal{F}(U_{n})+Q_{n}.
				\end{align}	
				\noindent $\blacktriangleright$ \textbf{Estimates of $\mathcal{F}({\rho}_{n+1})$.}
				 In the sequel, we intend to  prove  that
				$$\|\mathcal{F}({\rho}_{n+1})\|_{s_{0},\mathcal{A}_{n+1}}^{\textnormal{Lip}(\lambda)}\leqslant C_{\ast}\varepsilon^{\mathtt{b}_0} N_{n}^{-a_{1}}.$$
			To derive  this estimate it is enough to check it  with all the right hand side terms in \eqref{Decom-RTT1}.\\				
				 $\bullet$ \textit{Estimate of $\Pi_{n}^{\perp}\mathcal{F}(\rho_{n}).$} 
				We apply Taylor formula combined with \eqref{F-zero}, Lemma \ref{Tame-estimates-F} and $(\mathcal{P}1)_{n},$ leading to  
				\begin{align}\label{link mathcalF(Um) and Wm}
					\nonumber \forall s\geqslant s_{0},\quad\|\mathcal{F}(\rho_{n})\|_{s}^{\textnormal{Lip}(\lambda)}&\leqslant\|\mathcal{F}(0)\|_{s}^{\textnormal{Lip}(\lambda)}+\|\mathcal{F}(\rho_{n})-\mathcal{F}(0)\|_{s}^{\textnormal{Lip}(\lambda)}\\
					&\lesssim\varepsilon^{\mathtt{b}_0}+\| \rho_{n}\|_{s+2}^{\textnormal{Lip}(\lambda)},
				\end{align}
				where  we have used the following estimate, which follows from $(\mathcal{P}1)_{n}$ and \eqref{small-PP1},
				\begin{align*}
				\nonumber\| \rho_{n}\|_{s_{0}+2}^{\textnormal{Lip}(\lambda)}&\leqslant C_{\ast}\varepsilon^{\mathtt{b}_0}\lambda^{-1}\\			
				&\leqslant \varepsilon_0.
				\end{align*}
				From the standard decay properties of the projectors $\Pi_n$ and \eqref{link mathcalF(Um) and Wm}, we deduce that
				\begin{align}\label{pipFuns0}
					\|\Pi_{n}^{\perp}\mathcal{F}(\rho_{n})\|_{s_0}^{\textnormal{Lip}(\lambda)}&\leqslant N_{n}^{s_{0}-b_{1}}\|\mathcal{F}(\rho_{n})\|_{b_1}^{\textnormal{Lip}(\lambda)}\nonumber\\
					&\lesssim  N_{n}^{s_0-b_{1}}\left(\varepsilon^{\mathtt{b}_0}+\| \rho_{n}\|_{b_1+2}^{\textnormal{Lip}(\lambda)}\right).
				\end{align}
				Now, $(\mathcal{P}3)_{n}$ together with  \eqref{definition of Nm}  yield 
				\begin{align}\label{Wm in high norm}
					\nonumber\varepsilon^{\mathtt{b}_0}+\| \rho_{n}\|_{b_1+2}^{\textnormal{Lip}(\lambda)}&\leqslant\varepsilon^{\mathtt{b}_0}+C_{\ast}\varepsilon^{\mathtt{b}_0} \lambda^{-1}N_{n-1}^{\mu_{1}+2}\\
						\nonumber&\leqslant\varepsilon^{\mathtt{b}_0}+C_{\ast}\varepsilon^{\mathtt{b}_0} N_0^{1+\frac2\delta}N_{n-1}^{\mu_{1}+2}\\
					&\leqslant 2C_{\ast}\varepsilon^{\mathtt{b}_0} N_{n}^{\frac{2}{3}\mu_{1}+3+{\frac2\delta}},
				\end{align}
				where we have used from \eqref{lambda-choice}.
				\begin{equation}\label{fra-1}
				\lambda^{-1}=N_0^{1+\frac{2}{\delta}}.
				\end{equation}
				By putting together \eqref{Wm in high norm} and  \eqref{pipFuns0}, we find that
				\begin{align}\label{final estimate PiperpF(Um)}
					\|\Pi_{n}^{\perp}\mathcal{F}(\rho_{n})\|_{s_0}^{\textnormal{Lip}(\lambda)}&\lesssim  C_{\ast}\varepsilon^{\mathtt{b}_0} N_{n}^{s_{0}+\frac{2}{3}\mu_{1}+3+{\frac2\delta}-b_{1}}.
				\end{align}
				Remark that one also obtains, combining \eqref{link mathcalF(Um) and Wm}, \eqref{fra-1}, \eqref{Wm in high norm} and the fact that $\rho_n\in E_{n-1}$,
				\begin{align}\label{HDP10}
				\nonumber	\|\mathcal{F}(\rho_{n})\|_{b_{1}}^{\textnormal{Lip}(\lambda)} & \lesssim\varepsilon^{\mathtt{b}_0}+\| \rho_{n}\|_{b_{1}+2}^{\textnormal{Lip}(\lambda)}\\
				\nonumber	& \lesssim\varepsilon^{\mathtt{b}_0}+N_{n-1}^2\| \rho_{n}\|_{b_{1}}^{\textnormal{Lip}(\lambda)}\\
				&	\lesssim  		C_{\ast}\varepsilon^{\mathtt{b}_0} N_{n}^{\frac{2}{3}\mu_{1}+{3+\frac2\delta}}.
				\end{align}
				\\
				$\bullet$ \textit{Estimate of $\Pi_{n}(\mathcal{L}_{n}\mathcal{T}_{n}-\textnormal{Id})\Pi_{n}\mathcal{F}(\rho_{n})$.} 	By applying Proposition \ref{prop-inverse}-2 we can write in  $\mathcal{A}_{n+1}$
				$$\Pi_{n}(\mathcal{L}_{n}\mathcal{T}_{n}-\textnormal{Id})\Pi_{n}\mathcal{F}(\rho_{n})=\Pi_{n}\mathcal{E}_{n,1}\Pi_{n}\mathcal{F}(\rho_{n})+\Pi_{n}\mathcal{E}_{n,2}\Pi_{n}\mathcal{F}(\rho_{n})\triangleq \mathscr{E}_{n,1}+\mathscr{E}_{n,2}
				$$				with 
				
				\begin{align}\label{dimanche1}
				\nonumber \forall\, s\in [s_0,S],\quad  \|\mathscr{E}_{n,1}\|_{s_0,\mathcal{A}_{n}}^{\textnormal{Lip}(\lambda)}
			&\leqslant C N_n^{s_0-s}\lambda^{-1}\Big( \|\Pi_{n}\mathcal{F}(\rho_{n})\|_{s+2\tau+1}^{\textnormal{Lip}(\lambda)}+\| \rho_n\|_{s+4\tau+6}^{\textnormal{Lip}(\lambda)}\|\Pi_{n}\mathcal{F}(\rho_{n})\|_{s_{0}+2\tau}^{\textnormal{Lip}(\lambda)} \Big)\\
			 \|\mathscr{E}_{n,2}\|_{s_0,\mathcal{A}_{n}}^{\textnormal{Lip}(\lambda)} & \leqslant C\varepsilon^6\lambda^{-1} N_{0}^{\mu_{2}}N_{n+1}^{-\mu_{2}}\|\Pi_{n}\mathcal{F}(\rho_{n})\|_{s_{0}+2\tau+2}^{\textnormal{Lip}(\lambda)}.
			\end{align}
							 It follows that
				\begin{equation}\label{e-ai-NM}
					\|\Pi_{n}(\mathcal{L}_{n}\mathcal{T}_{n}-\textnormal{Id})\Pi_{n}\mathcal{F}(\rho_{n})\|_{s_0,\mathcal{A}_{n+1}}^{\textnormal{Lip}(\lambda)}\leqslant
					\|\mathscr{E}_{n,1}\|_{s_0,\mathcal{A}_{n}}^{\textnormal{Lip}(\lambda)}+\|\mathscr{E}_{n,2}\|_{s_0,\mathcal{A}_{n}}^{\textnormal{Lip}(\lambda)}.
				\end{equation}
		Applying \eqref{dimanche1} with $s=b_{1}$ and using \eqref{HDP10}, $(\mathcal{P}_2)_n$, $(\mathcal{P}_3)_n$, \eqref{Assump-DRP1} and \eqref{fra-1}, we get for $n\geqslant 1,$
				\begin{align}\label{Esc2n}
					\|\mathscr{E}_{n,1}\|_{s_0,\mathcal{A}_{n}}^{\textnormal{Lip}(\lambda)}&\lesssim\lambda^{-1}N_n^{s_0-b_1}\left(\|\Pi_{n}\mathcal{F}(\rho_n)\|_{b_1+2\tau+1,\mathcal{A}_{n}}^{\textnormal{Lip}(\lambda)}+\|\rho_n\|_{b_1+4\tau+6}^{\textnormal{Lip}(\lambda)}\|\Pi_{n}\mathcal{F}(\rho_n)\|_{s_0+2\tau,\mathcal{A}_{n}}^{\textnormal{Lip}(\lambda)}\right)\nonumber\\
					&\lesssim \lambda^{-1}N_n^{s_0-b_1}\left(N_n^{2\tau+1}\|\mathcal{F}(\rho_n)\|_{b_1,\mathcal{A}_{n}}^{\textnormal{Lip}(\lambda)}+N_n^{4\tau+6}\|\rho_n\|_{b_1}^{\textnormal{Lip}(\lambda)}\|\mathcal{F}(\rho_n)\|_{s_0,\mathcal{A}_{n}}^{\textnormal{Lip}(\lambda)}\right)\nonumber\\
					&\lesssim C_{\ast}\varepsilon^{\mathtt{b}_0} N_n^{s_0+\frac{2}{3}\mu_{1}+2\tau+5-b_1+{\frac{4}{\delta}}}+C_{\ast}\varepsilon^{\mathtt{b}_0} \varepsilon^{\mathtt{b}_0}N_n^{s_0+4\tau+\frac{2}{3}\mu_{1}+8-\frac{2}{3}a_{1}-b_1+{\frac{4}{\delta}}}\nonumber\\
					&\lesssim C_{\ast}\varepsilon^{\mathtt{b}_0} N_n^{s_0+4\tau+\frac{2}{3}\mu_{1}+3-b_1+{\frac{4}{\delta}}}.
				\end{align}
	On the other hand, we find from $(\mathcal{P}_2)_n$, \eqref{dimanche1}, \eqref{Assump-DRP1} and \eqref{lambda-choice}
				\begin{align}\label{Esc02n}
			\nonumber \|\mathscr{E}_{n,2}\|_{s_0,\mathcal{A}_{n}}^{\textnormal{Lip}(\lambda)}&\leqslant C\varepsilon^6\lambda^{-1} N_{0}^{\mu_{2}}N_{n+1}^{-\mu_{2}}\|\Pi_{n}\mathcal{F}(\rho_{n})\|_{s_{0}+2,\mathcal{A}_{n}}^{\textnormal{Lip}(\lambda)}\\
		\nonumber &\leqslant C C_\ast\underbrace{\varepsilon^6\lambda^{-1} N_{0}^{\mu_{2}}}_{\leqslant 1}\,\varepsilon^{\mathtt{b}_0}N_{n}^{-\mu_{2}-\frac23 a_1+2}\\
		&\leqslant C C_\ast \varepsilon^{\mathtt{b}_0} N_{n}^{-\mu_{2}-\frac23 a_1+2}.
			\end{align}
				Putting together \eqref{e-ai-NM},  \eqref{Esc2n} and \eqref{Esc02n}, we obtain for $n\geqslant 1$
				\begin{equation}\label{est-aait}
					\|\Pi_{n}(\mathcal{L}_{n}\mathcal{T}_{n}-\textnormal{Id})\Pi_{n}\mathcal{F}(\rho_{n})\|_{s_0,\mathcal{A}_{n+1}}^{\textnormal{Lip}(\lambda)}\leqslant CC_{\ast}\varepsilon^{\mathtt{b}_0}\left(N_n^{s_0+4\tau+\frac{2}{3}\mu_{1}+3-b_1+{\frac{4}{\delta}}}+N_{n}^{-\mu_{2}-\frac23 a_1+2}\right).
				\end{equation}
				For $n=0$, we deduce from the first line of \eqref{Esc2n} and \eqref{F-zero}
				\begin{align}\label{Esc2nL}
					\nonumber \|\mathscr{E}_{0,1}\|_{s_0}^{\textnormal{Lip}(\lambda)}&\lesssim\lambda^{-1}N_0^{s_0-b_1}\|\Pi_{0}\mathcal{F}(0)\|_{b_1+2\tau+1}^{\textnormal{Lip}(\lambda)}\\
	&\lesssim\lambda^{-1}N_0^{s_0-b_1}\varepsilon^{\mathtt{b}_0}.									
				\end{align}
				As to the inequality \eqref{Esc02n} it becomes in view of \eqref{lambda-choice} and \eqref{F-zero}
				\begin{align}\label{Esc02N}
			\nonumber \|\mathscr{E}_{0,2}\|_{s_0}^{\textnormal{Lip}(\lambda)}&\leqslant C\varepsilon^6\lambda^{-1} N_{0}^{\mu_{2}}N_{1}^{-\mu_{2}}\|\Pi_{0}\mathcal{F}(0)\|_{s_{0}+2}^{\textnormal{Lip}(\lambda)}\\
		\nonumber &\leqslant C C_\ast\underbrace{\varepsilon^6\lambda^{-1} N_{0}^{\mu_{2}}}_{\leqslant 1}N_{1}^{-\mu_{2}}\varepsilon^{\mathtt{b}_0}\\
		&\leqslant C C_\ast \varepsilon^{\mathtt{b}_0} N_{0}^{-\frac32\mu_{2}}.
			\end{align}
		Thus, we deduce from \eqref{Esc2nL}, \eqref{Esc02N}	and \eqref{fra-1}
							\begin{align}\label{e-ai-0}
					\|\Pi_{0}(\mathcal{L}_{0}\mathcal{T}_{0}-\textnormal{Id})\Pi_{0}\mathcal{F}(0)\|_{s_0}^{\textnormal{Lip}(\lambda)}&\leqslant\|\mathscr{E}_{0,1}\|_{s_0}^{\textnormal{Lip}(\lambda)}+\|\mathscr{E}_{0,2}\|_{s_0}^{\textnormal{Lip}(\lambda)}\nonumber\\
					& \lesssim C_\ast\varepsilon^{\mathtt{b}_0} N_{0}^{s_{0}+1+\frac2\delta-b_{1}}+C_\ast \varepsilon^{\mathtt{b}_0} N_{0}^{-\frac32\mu_{2}}\cdot
				\end{align}
				$\bullet$  \textit{Estimate of $\big(\mathcal{L}_{n}{\Pi}_{n}^{\perp}-\Pi_{n}^{\perp}\mathcal{L}_{n}\big)\mathcal{T}_{n}\Pi_{n}\mathcal{F}(\rho_{n}).$} 		
				From the structure of the operator $\mathcal{L}_{n}$ described by  \eqref{linearized f} one gets in view  of straightforward computations that 		
				$$
				\|(\mathcal{L}_{n}{\Pi}_{n}^{\perp}-\Pi_{n}^{\perp}\mathcal{L}_{n})h\|_{s_0}^{\textnormal{Lip}(\lambda)}\lesssim\varepsilon^2 N_{n}^{s_{0}-b_{1}}\left(\|h\|_{b_1+1}^{\textnormal{Lip}(\lambda)}+\|\rho_{n}\|_{b_1+1}^{\textnormal{Lip}(\lambda)}\|h\|_{s_0+1}^{\textnormal{Lip}(\lambda)}\right).
				$$
				Consequently,
				\begin{align}\label{commu-t1}
			\nonumber	\|(\mathcal{L}_{n}{\Pi}_{n}^{\perp}-\Pi_{n}^{\perp}\mathcal{L}_{n})\mathcal{T}_{n}\Pi_{n}\mathcal{F}(\rho_{n})\|_{s_0,\mathcal{A}_{n+1}}^{\textnormal{Lip}(\lambda)}&\lesssim\varepsilon^2 N_{n}^{s_{0}-b_{1}}\|\mathcal{T}_{n}\Pi_{n}\mathcal{F}(\rho_{n})\|_{b_1+1,\mathcal{A}_{n+1}}^{\textnormal{Lip}(\lambda)}\\
				&+\varepsilon^2 N_{n}^{s_{0}-b_{1}}\|\rho_{n}\|_{b_1+1}^{\textnormal{Lip}(\lambda)}\|\mathcal{T}_{n}\Pi_{n}\mathcal{F}(\rho_{n})\|_{s_0+1,\mathcal{A}_{n+1}}^{\textnormal{Lip}(\lambda)}.
				\end{align}
				Applying Proposition \ref{prop-inverse} allows to get			
				\begin{equation*}
				\forall \, s\in\,[ s_0, S],\quad\|\mathcal{T}_{n}\Pi_{n}\mathcal{F}(\rho_{n})\|_{s,\mathcal{A}_{n}}^{\textnormal{Lip}(\lambda)}\leqslant C\lambda^{-1}\Big(\|\Pi_{n}\mathcal{F}(\rho_{n})\|_{s+2\tau,\mathcal{A}_{n}}^{\textnormal{Lip}(\lambda)}+\|\rho_n\|_{s+4\tau+5}^{\textnormal{Lip}(\lambda)}\|\Pi_{n}\mathcal{F}(\rho_{n})\|_{s_0+2\tau,\mathcal{A}_{n}}^{\textnormal{Lip}(\lambda)}\Big)
			\end{equation*}
			Thus, it gives for $s=s_0+1$
			\begin{align*}
				\|\mathcal{T}_{n}\Pi_{n}\mathcal{F}(\rho_{n})\|_{s_0+1,\mathcal{A}_{n}}^{\textnormal{Lip}(\lambda)}&\leqslant C\lambda^{-1}\Big(N_n^{2\tau+1}\|\mathcal{F}(\rho_{n})\|_{s_0,\mathcal{A}_{n}}^{\textnormal{Lip}(\lambda)}+N_n^{2\tau}\|\rho_n\|_{s_0+4\tau+6}^{\textnormal{Lip}(\lambda)}\|\mathcal{F}(\rho_{n})\|_{s_0,\mathcal{A}_{n}}^{\textnormal{Lip}(\lambda)}\Big)\\
	&\leqslant CC_*\lambda^{-1}\varepsilon^{\mathtt{b}_0}\Big(N_n^{2\tau+1-\frac23 a_1} +N_n^{2\tau-\frac23 a_1}\underbrace{C_* \varepsilon^{\mathtt{b}_0}\lambda^{-1}{}}_{\leqslant   1}\Big)\\
	&\leqslant CC_*\lambda^{-1}\varepsilon^{\mathtt{b}_0}N_n^{2\tau+1-\frac23 a_1}.			
			\end{align*}
			For $s=b_1+1$ we find by virtue of \eqref{HDP10}, \eqref{Assump-DRP1}, $(\mathcal{P}2)_{n}$ and  $(\mathcal{P}3)_{n}$
			\begin{align*}
				\|\mathcal{T}_{n}\Pi_{n}\mathcal{F}(\rho_{n})\|_{b_1+1,\mathcal{A}_{n}}^{\textnormal{Lip}(\lambda)}&\leqslant C\lambda^{-1}\Big(\|\Pi_{n}\mathcal{F}(\rho_{n})\|_{b_1+2\tau+1,\mathcal{A}_{n}}^{\textnormal{Lip}(\lambda)}+\|\rho_n\|_{b_1+4\tau+6}^{\textnormal{Lip}(\lambda)}\|\Pi_{n}\mathcal{F}(\rho_{n})\|_{s_0+2\tau,\mathcal{A}_{n}}^{\textnormal{Lip}(\lambda)}\Big)\\
			&\leqslant CC_*\lambda^{-1}\varepsilon^{\mathtt{b}_0}\Big(N_{n}^{\frac{2}{3}\mu_{1}+2\tau+4+\tfrac2\delta}+\underbrace{C_*\lambda^{-1}\varepsilon^{\mathtt{b}_0}}_{\leqslant 1} N_{n-1}^{4\tau+6+\mu_1-a_1}N_{n}^{2} \Big)\\
				&\leqslant CC_*\lambda^{-1}\varepsilon^{\mathtt{b}_0}N_{n}^{2\tau+4+\frac{2}{3}\mu_{1}+\frac2\delta}.
			\end{align*}
			Plugging  the preceding estimates  into \eqref{commu-t1} and using \eqref{lambda-choice} yields
			\begin{align}\label{final estimate commutator}
			\nonumber 	\|(\mathcal{L}_{n}{\Pi}_{n}^{\perp}-\Pi_{n}^{\perp}\mathcal{L}_{n})\mathcal{T}_{n}\Pi_{n}\mathcal{F}(\rho_{n})\|_{s_0,\mathcal{A}_{n+1}}^{\textnormal{Lip}(\lambda)}&\lesssim \lambda^{-1}\varepsilon^{2} N_{n}^{s_{0}-b_{1}}\left(\varepsilon^{\mathtt{b}_0}N_{n}^{2\tau+4+\frac{2}{3}\mu_{1}+\frac2\delta}+\varepsilon^{2\mathtt{b}_0}\lambda^{-1} N_n^{2+2\tau+\frac23\mu_1-\frac23 a_1}\right)\\
				\nonumber&\lesssim \lambda^{-1}\varepsilon^{2+\mathtt{b}_0} (1+\varepsilon^{\mathtt{b}_0} \lambda^{-1})N_{n}^{s_{0}-b_{1}+2\tau+4+\frac{2}{3}\mu_{1}+\frac{2}{\delta}}\\
				&\lesssim \varepsilon^{\mathtt{b}_0} N_{n}^{s_{0}-b_{1}+2\tau+5+\frac{2}{3}\mu_{1}+\frac{2}{\delta}}.
				\end{align}
				For $n=0$, we get in view of \eqref{commu-t1} and Proposition \ref{prop-inverse}-2 and \eqref{F-zero} with \eqref{lambda-choice}
								\begin{align}\label{Thm-1}
	\nonumber			\|(\mathcal{L}_{0}{\Pi}_{0}^{\perp}-\Pi_{0}^{\perp}\mathcal{L}_{0})\mathcal{T}_{0}\Pi_{0}\mathcal{F}(\rho_{0})\|_{s_0}^{\textnormal{Lip}(\lambda)}&\lesssim\varepsilon^2 N_{0}^{s_{0}-b_{1}}\|\mathcal{T}_{0}\Pi_{0}\mathcal{F}(0)\|_{b_1+1}^{\textnormal{Lip}(\lambda)}\\
\nonumber&\lesssim\varepsilon^2\lambda^{-1} N_{0}^{s_{0}-b_{1}}	\|\mathcal{F}(0)\|_{b_1+2\tau+1}^{\textnormal{Lip}(\lambda)}\\
&\lesssim C_*\varepsilon^{\mathtt{b}_0}N_{0}^{s_{0}+1-b_{1}}			.
				\end{align}
				$\bullet$  \textit{Estimate of $Q_{n}$.} Applying Taylor formula together with \eqref{Def-Qm} lead  to
				$$Q_{n}=\int_{0}^{1}(1-t)d_{\rho}^{2}\mathcal{F}(\rho_{n}+t{u}_{n+1})[{u}_{n+1},{u}_{n+1}]dt.$$
				Thus,   we deduce from  Lemma \ref{Tame-estimates-F}-{(2)} 
				\begin{align}\label{mahma-YDa1}
					\| Q_{n}\|_{s_0,\mathcal{A}_{n}}^{\textnormal{Lip}(\lambda)}\lesssim\varepsilon^2\left(1+\|\rho_{n}\|_{s_0+2}^{\textnormal{Lip}(\lambda)}+\| {u}_{n+1}\|_{s_0+2,\mathcal{A}_{n}}^{\textnormal{Lip}(\lambda)}\right)\left(\| {u}_{n+1}\|_{s_0+2,\mathcal{A}_{n}}^{\textnormal{Lip}(\lambda)}\right)^{2}.
				\end{align}
				Combining \eqref{def-un}, \eqref{estimate Tm}, \eqref{link mathcalF(Um) and Wm} and $(\mathcal{P}2)_{n}$ , we find for all $s\in[s_{0},S]$ 
				\begin{align}\label{Tu-L1}
					\| {u}_{n+1}\|_{s,\mathcal{A}_{n}}^{\textnormal{Lip}(\lambda)} & =  \|{\Pi}_{n}\mathcal{T}_{n}\Pi_{n}\mathcal{F}(\rho_{n})\|_{s,\mathcal{A}_{n}}^{\textnormal{Lip}(\lambda)}\nonumber\\
					& \lesssim  \lambda^{-1}\left(\|\Pi_{n}\mathcal{F}(\rho_{n})\|_{s+2\tau,\mathcal{A}_{n}}^{\textnormal{Lip}(\lambda)}+\|\rho_{n}\|_{s+2\tau}^{\textnormal{Lip}(\lambda)}\|\Pi_{n}\mathcal{F}(\rho_{n})\|_{s_0+2\tau,\mathcal{A}_{n}}^{\textnormal{Lip}(\lambda)}\right)\nonumber\\
				\nonumber 	& \lesssim  \lambda^{-1}\left(N_n^{2\tau}\|\Pi_{n}\mathcal{F}(\rho_{n})\|_{s,\mathcal{A}_{n}}^{\textnormal{Lip}(\lambda)}+N_n^{2\tau}\|\rho_{n}\|_{s+2\tau}^{\textnormal{Lip}(\lambda)}\|\Pi_{n}\mathcal{F}(\rho_{n})\|_{s_0,\mathcal{A}_{n}}^{\textnormal{Lip}(\lambda)}\right)\nonumber\\
					& \lesssim  \lambda^{-1}N_{n}^{2\tau}\left(\varepsilon^{\mathtt{b}_0}+\| \rho_{n}\|_{s+2\tau}^{\textnormal{Lip}(\lambda)}\right).
				\end{align}
				For $s={2s_0+2\tau+3}$, we get  according to the third line of  \eqref{Tu-L1}, \eqref{small-PP1}, $(\mathcal{P}1)_{n}$ and $(\mathcal{P}2)_{n}$, 
				\begin{align}\label{prat0}
					\nonumber \| {u}_{n+1}\|_{{2s_0+2\tau+3},\mathcal{A}_{n}}^{\textnormal{Lip}(\lambda)}&\lesssim\lambda^{-1}N_{n}^{{2s_0+4\tau+3}}\|\mathcal{F}(\rho_{n})\|_{s_{0},\mathcal{A}_{n}}^{\textnormal{Lip}(\lambda)}\\
					\nonumber&\lesssim C_{\ast}\varepsilon^{\mathtt{b}_0}\lambda^{-1} N_{n}^{{2s_0+4\tau+3}-\frac23 a_1}\\
					&\lesssim C_{\ast} N_{n}^{{2s_0+4\tau+3}-\frac23 a_1}.
				\end{align}
				Choosing $\varepsilon$ small enough and using $(\mathcal{P}1)_{n}$ and \eqref{small-PP1}, we find
								\begin{align*}
					\|\rho_{n}\|_{2s_0+2\tau+3}^{\textnormal{Lip}(\lambda)}+\| {u}_{n+1}\|_{2s_0+2\tau+3,\mathcal{A}_{n}}^{\textnormal{Lip}(\lambda)} & \leqslant C_{\ast}\varepsilon^{\mathtt{b}_0}\lambda^{-1}+C_{\ast}\varepsilon^{\mathtt{b}_0}\lambda^{-1} N_{n}^{2s_0+4\tau+3-\frac23 a_1}\\
					& \leqslant 2,
				\end{align*}
				where we have used that
				$$
				{a_1\geqslant 3s_0+6\tau+5}
				$$
				which follows from \eqref{Assump-DRP1}. In a similar way to \eqref{prat0}, we get
				\begin{align*}
					\| {u}_{n+1}\|_{s_0+2,\mathcal{A}_{n}}^{\textnormal{Lip}(\lambda)} &\leqslant  C_{\ast}\varepsilon^{\mathtt{b}_0}\lambda^{-1} N_{n}^{s_0+2\tau+2-\frac23 a_1}\\
					& \leqslant 2.
				\end{align*}
				Hence,  plugging this estimate  into \eqref{mahma-YDa1}  and using the second line of \eqref{prat0} together with \eqref{lambda-choice},\eqref{small-PP1} and   \eqref{fra-1}, we find   $n\geqslant 1$
				\begin{align}\label{final estimate for Qm}
					\nonumber\| Q_{n}\|_{s_{0},\mathcal{A}_{n}}^{\textnormal{Lip}(\lambda)} & \lesssim \varepsilon^2\left(\| {u}_{n+1}\|_{s_{0}+2,\mathcal{A}_{n}}^{\textnormal{Lip}(\lambda)}\right)^{2}\\
								\nonumber 	& \leqslant C C_{\ast}^2\varepsilon^{2+2\mathtt{b}_0}\lambda^{-2} N_{n}^{4\tau+4-\frac43a_1}\\
									& \leqslant C C_{\ast}\varepsilon^{\mathtt{b}_0}N_{n}^{4\tau+5-\frac43a_1}.
				\end{align}
				For $n=0$, we come back to the second line of \eqref{Tu-L1} and use \eqref{F-zero} to obtain for all $s\in[s_{0},S]$
				\begin{align}\label{H1}
					\|{u}_{1}\|_{s}^{\textnormal{Lip}(\lambda)}&\lesssim\lambda^{-1}\|\Pi_{0}\mathcal{F}(0)\|_{s+2\tau}^{\textnormal{Lip}(\lambda)}\nonumber\\
					&\leqslant C_{\ast}\lambda^{-1}\varepsilon^{\mathtt{b}_0}.
				\end{align}
				Finally, the inequality \eqref{final estimate for Qm} becomes for $n=0$, in view of \eqref{H1}
								\begin{align}\label{e-Q0}
				 \|Q_{0}\|_{s_{0}}^{\textnormal{Lip}(\lambda)}&\lesssim C_{\ast}\varepsilon^{2}(\lambda^{-1}\varepsilon^{\mathtt{b}_0})^2.
				\end{align}
				 \textit{Conclusion.}
				Inserting  \eqref{final estimate PiperpF(Um)}, \eqref{est-aait}, \eqref{final estimate commutator} and \eqref{final estimate for Qm}, into \eqref{Decom-RTT1} implies for $n\in\mathbb{N}^{*}$,
				\begin{align*}
					\|\mathcal{F}({\rho}_{n+1})\|_{s_{0},\mathcal{A}_{n+1}}^{\textnormal{Lip}(\lambda)}&\leqslant CC_{\ast}\varepsilon^{\mathtt{b}_0}\left(N_{n}^{s_{0}+\frac{2}{3}\mu_{1}+3+{\frac2\delta}-b_{1}}+N_n^{s_0+4\tau+\frac{2}{3}\mu_{1}+3+{\frac4\delta}-b_1}+N_{n}^{-\mu_{2}-\frac23 a_1+2}\right)\\
					&\qquad +C C_\ast\varepsilon^{\mathtt{b}_0} N_{n}^{s_{0}-b_{1}+2\tau+\frac{2}{3}\mu_{1}+5+{\frac2\delta}}+C C_{\ast}\varepsilon^{\mathtt{b}_0} N_{n}^{4\tau+5-\frac43a_1}.
				\end{align*} 
				The parameters conditions stated in \eqref{Assump-DRP1} give
				\begin{equation}\label{Assump-DR1}\left\lbrace\begin{array}{rcl}
						s_{0}+4\tau+3+\frac{2}{3}\mu_{1}+{\frac4\delta}+a_1&< & b_{1}\\
						\frac{1}{3}a_{1}+2 &< & \mu_{2}\\
						
						4\tau+5&< & \frac{1}{3}a_{1}.
					\end{array}\right.
				\end{equation}
				Thus, by taking $N_{0}$ large enough, that is $\varepsilon$ small enough, we obtain for  $n\in\mathbb{N},$
				\begin{align}\label{Dabdoub1}
					\|\mathcal{F}({\rho}_{n+1})\|_{s_{0},\mathcal{A}_{n+1}}^{\textnormal{Lip}(\lambda)}\leqslant C_{\ast}\varepsilon^{\mathtt{b}_0} N_{n}^{-a_{1}}.
				\end{align}
				However, for  $n=0$, we plug \eqref{final estimate PiperpF(Um)}, \eqref{e-ai-0},  \eqref{Thm-1} and \eqref{e-Q0} into \eqref{Decom-RTT1}, using \eqref{lambda-choice}  in order to get
				\begin{align*}\|\mathcal{F}({\rho}_{1})\|_{s_{0}}^{\textnormal{Lip}(\lambda)}&\leqslant CC_{\ast}\varepsilon^{\mathtt{b}_0}\left(N_{0}^{s_{0}+\frac{2}{3}\mu_{1}+3+\frac2\delta-b_{1}}+N_{0}^{s_{0}+1+\frac2\delta-b_{1}}+  N_{0}^{-\frac32\mu_{2}}+ N_{0}^{s_{0}+1-b_{1}}+\varepsilon^{2}\lambda^{-2}\varepsilon^{\mathtt{b}_0}\right)\\
				&\leqslant CC_{\ast}\varepsilon^{\mathtt{b}_0}\left(N_{0}^{s_{0}+\frac{2}{3}\mu_{1}+3+\frac2\delta-b_{1}}+  N_{0}^{-\frac32\mu_{2}}+N_0^{-\frac{{\mathtt{b}_0}-2-2\delta}{\delta}}\right).
				\end{align*}
At this level, we impose 
\begin{equation}\label{Assump-DRL1}\left\lbrace\begin{array}{rcl}
						s_{0}+\frac23\mu_1+3+\frac2\delta+a_1&< & b_{1}\\
						a_1&<&\frac32\mu_2\\
						a_1&< & \frac{{\mathtt{b}_0}-2-2\delta}{\delta}
					\end{array}\right.
				\end{equation}
				in order to get for $\varepsilon$ small enough
				\begin{align*}\|\mathcal{F}({\rho}_{1})\|_{s_{0}}^{\textnormal{Lip}(\lambda)}&\leqslant C_{\ast}\varepsilon^{\mathtt{b}_0} N_0^{-a_1}.
				\end{align*}			
				We point out that the first assumption in \eqref{Assump-DRL1} follows from  the first one in \eqref{Assump-DR1} and the third one is equivalent to  
				\begin{align}\label{dran-1}0<\delta< \tfrac{\mathtt{b}_0-2}{a_1+2}=\tfrac{1-\mu}{a_1+2}\cdot
				\end{align}
				As to the second one, it follows immediately from the second one in \eqref{Assump-DRP1}.
				Hence
				$$\|\mathcal{F}({\rho}_{1})\|_{s_{0}}^{\textnormal{Lip}(\lambda)}\leqslant C_{\ast}\varepsilon^{\mathtt{b}_0} N_{0}^{-a_{1}}.$$ 
				This completes the proof of the estimates in $(\mathcal{P}2)_{n+1}.$ \\
				
				\noindent $\blacktriangleright$ \textbf{Verification of $(\mathcal{P}1)_{n+1}-(\mathcal{P}3)_{n+1}.$} 
Using \eqref{ext-Hm1}, \eqref{Tu-L1}, $(\mathcal{P}3)_{n}$ and \eqref{fra-1} we deduce that
\begin{align}\label{Tu-LT1}
					\nonumber\| {u}_{n+1}\|_{b_1}^{\textnormal{Lip}(\lambda)} 
					& \leqslant C  \lambda^{-1}N_{n}^{2\tau}\left(\varepsilon^{\mathtt{b}_0}+\| \rho_{n}\|_{b_1+2\tau}^{\textnormal{Lip}(\lambda)}\right)\\
					\nonumber&\leqslant C C_*\varepsilon^{\mathtt{b}_0} \lambda^{-2}N_{n}^{4\tau+\frac23\mu_1}\\
					&\leqslant C C_*\varepsilon^{\mathtt{b}_0}\lambda^{-1}N_{n}^{4\tau+\frac23\mu_1+1+\frac2\delta}.
				\end{align}
Now gathering \eqref{Tu-LT1} and $(\mathcal{P}3)_{n}$ allows to write
				\begin{align*}
					\|\rho_{n+1}\|_{b_1}^{\textnormal{Lip}(\lambda)}  & \leqslant  \| \rho_{n}\|_{b_1}^{\textnormal{Lip}(\lambda)}+ \|u_{n+1}\|_{b_1}^{\textnormal{Lip}(\lambda)} \\
					& \leqslant  C_{\ast} \varepsilon^{\mathtt{b}_0}\lambda^{-1}N_{n}^{\frac23\mu_{1}}+CC_\ast\varepsilon^{\mathtt{b}_0}\lambda^{-1}N_{n}^{4\tau+\frac23\mu_1+1+\frac2\delta}\\
					& \leqslant  C_{\ast} \varepsilon^{\mathtt{b}_0}\lambda^{-1}N_{n}^{\mu_{1}}
				\end{align*}
				provided that
\begin{align}\label{bling3}
4\tau+1+\tfrac{2}{\delta}<\tfrac13\mu_1,
\end{align}
which follows from \eqref{Assump-DRP1}.		
This achieves $(\mathcal{P}3)_{n+1}$. \\
Using the Frequency localization of $u_{n+1}$, that is ${u}_{n+1}\in E_n$, together with  \eqref{prat0} yield
\begin{align*}
					\nonumber \| {u}_{n+1}\|_{2s_{0}+2\tau+3}^{\textnormal{Lip}(\lambda)}	 
					\nonumber &\leqslant C C_{\ast}\varepsilon^{\mathtt{b}_0}\lambda^{-1} N_{n}^{2s_0+4\tau+3-\frac23 a_1}\\
					&\leqslant C C_{\ast}\varepsilon^{\mathtt{b}_0}\lambda^{-1} N_{n}^{-a_2}
				\end{align*}
				provided that
				\begin{align*}
				{2s_0+4\tau+3+a_2}<\tfrac23 a_1,
				\end{align*}
				which is a consequence of  \eqref{Assump-DRP1}. 
								By $(\mathcal{P}1)_{n}$, we infer
								\begin{align*}
					\|\rho_{n+1}\|_{2s_{0}+2\tau+3}^{\textnormal{Lip}(\lambda)}&\leqslant\|u_{1}\|_{2s_{0}+2\tau+3}^{\textnormal{Lip}(\lambda)}+\sum_{k=2}^{n+1}\|u_{k}\|_{2s_{0}+2\tau+3}^{\textnormal{Lip}(\lambda)}\\
					&\leqslant \tfrac{1}{2}C_{\ast}\varepsilon^{\mathtt{b}_0}\lambda^{-1}+C_{\ast}\varepsilon^{\mathtt{b}_0}\lambda^{-1}\sum_{k=0}^{\infty}N_{k}^{-a_2}\\
					&\leqslant \tfrac{1}{2}C_{\ast}\varepsilon^{\mathtt{b}_0}\lambda^{-1}+CN_{0}^{-a_2}C_{\ast}\varepsilon^{\mathtt{b}_0}\lambda^{-1}\\
					&\leqslant C_{\ast}\varepsilon^{\mathtt{b}_0}\lambda^{-1}.
				\end{align*}
				This completes the proof of $(\mathcal{P}1)_{n+1}.$  The proof of Proposition \ref{Nash-Moser} is now complete.
\end{proof}
The next target is to study the convergence of Nash-Moser  scheme stated in Proposition \ref{Nash-Moser} and show that the limit is a solution to the problem \eqref{main-eq1}. For this aim, we need to introduce the final Cantor set,
\begin{equation}\label{Cantor-set01}\mathtt{C}_{\infty}\triangleq\bigcap_{m\in\mathbb{N}}\mathcal{A}_{m}.
\end{equation}
\begin{coro}\label{prop-construction}
There exists $\xi_0\in\mathcal{O}\mapsto \rho_\infty$ satisfying
$$\|\rho_\infty\|_{2s_{0}+2\tau+3}^{\textnormal{Lip}(\lambda)}\leqslant C_{\ast}\varepsilon^{\mathtt{b}_0}\lambda^{-1}\quad \hbox{and}\quad \|\rho_\infty-\rho_m\|_{2s_{0}+2\tau+3}^{\textnormal{Lip}(\lambda)}\leqslant  C_{\ast}\varepsilon^{\mathtt{b}_0}\lambda^{-1}N_{m}^{-a_{2}}
$$
such that  
$$\forall \xi_0\in \mathtt{C}_{\infty}, \quad \mathcal{F}\big(\rho_{\infty}(\xi_0)\big)=0.
$$
\end{coro}
\begin{proof}
According to Proposition \ref{Nash-Moser} one may write for each $m\geqslant1,$
$$
\rho_m=\sum_{n=1}^m u_n.
$$
Define the formal infinite sum
$$
\rho_\infty\triangleq\,\sum_{n=1}^\infty u_n.
$$
By using the estimates of $(\mathcal{P}1)_{m}$, \eqref{definition of Nm} and \eqref{lambda-choice} we get
\begin{align*}
\|\rho_\infty\|_{2s_{0}+2\tau+3}^{\textnormal{Lip}(\lambda)}&\leqslant\sum_{n=1}^\infty \|u_n\|_{2s_0+2\tau+3}^{\textnormal{Lip}(\lambda)}\\
&\leqslant \tfrac12 C_{\ast}\varepsilon^{\mathtt{b}_0}\lambda^{-1}+C_{\ast}\varepsilon^{\mathtt{b}_0}\lambda^{-1}\sum_{n=1}^\infty N_{n}^{-a_{2}}\\
&\leqslant \tfrac12 C_{\ast}\varepsilon^{\mathtt{b}_0}\lambda^{-1}+C_{\ast}\varepsilon^{\mathtt{b}_0}\lambda^{-1}N_{0}^{-a_{2}}.
\end{align*}
Thus for $ \varepsilon$ small enough,  we deduce that
\begin{align*}
\|\rho_\infty\|_{2s_{0}+2\tau+3}^{\textnormal{Lip}(\lambda)}
&\leqslant  C_{\ast}\varepsilon^{\mathtt{b}_0}\lambda^{-1}.
\end{align*}
On the other hand, we get in a similar way
\begin{align*}
\|\rho_\infty-\rho_m\|_{2s_{0}+2\tau+3}^{\textnormal{Lip}(\lambda)}&\leqslant\sum_{n=m+1}^\infty \|u_n\|_{2s_{0}+2\tau+3}^{\textnormal{Lip}(\lambda)}\\
&\leqslant C_{\ast}\varepsilon^{\mathtt{b}_0}\lambda^{-1}\sum_{n=m}^\infty N_{n}^{-a_{2}}\\
&\leqslant C_{\ast}\varepsilon^{\mathtt{b}_0}\lambda^{-1}N_{m}^{-a_{2}}.
\end{align*}
It follows that the sequence $(\rho_m)_{m\geqslant1}$ converges pointwisely to $\rho_\infty$. Applying $(\mathcal{P}2)_{m}$ yields 
$$\|\mathcal{F}(\rho_{m})\|_{s_{0},\mathtt{C}_\infty}^{{\textnormal{Lip}(\lambda)}}\leqslant C_{\ast}\varepsilon^{\mathtt{b}_0} N_{m-1}^{-a_{1}}.
$$
Hence passing to the limit leads to 
$$\forall \xi_0\in \mathtt{C}_{\infty}, \quad \mathcal{F}\big(\rho_{\infty}(\xi_0)\big)=0.
$$
This ends  the proof of the desired result.
\end{proof}
\subsection{Cantor set measure}
This section is dedicated  to examining  the measure of  the final  Cantor set $\mathtt{C}_{\infty}$ defined by  \eqref{Cantor-set01}.  It's worth noting that this set is related  to the spectrum of the linearized operators associated with the approximate  sequence $(\rho_m)_{m\in\N}$,  and for its measure we find it convenient  to  extract  a subset related only  to the solution $\rho_\infty$  constructed in Corollary \ref{prop-construction}. The measure of this latter set will be performed   straightforwardly.\\
To streamline our notation, we will unify the spectrum within the Cantor sets $\mathcal{O}_{n}^1(\rho)$ and $\mathcal{O}_{n}^2(\rho)$, which are defined in \eqref{Cantor set0} and \eqref{D-1-LL}, as follows
\begin{align}\label{mu-j1}
					\nonumber \mu_{j,k}(\xi_0,\rho)&= j\mathtt{c}(\xi_0,\rho)-\tfrac{k-1}{2}\tfrac{j}{|j|}\\
					&=j\left( \tfrac12-\varepsilon^2(\omega_0-\varepsilon \mathtt{c}_2)\right)-\tfrac{k-1}{2}\tfrac{j}{|j|}	\cdot
								\end{align}
Now, consider the following  the sets,
$$
k=1,2,\quad\mathcal{C}_{\infty}^{k}=\Big\{\xi_0\in\mathcal{O}, \; \forall \ell\in\mathbb{Z},\;  \forall    |j|\geqslant k, \; \, |\varepsilon^2 \omega(\xi_0)\ell+\mu_{j,k}(\xi_0,\rho_\infty)|\geqslant  \tfrac{2\lambda}{| j |^\tau}\Big\}.
$$
We have the following result
\begin{lemma}\label{Lem-Cantor-measu}
Under \eqref{Assump-DRP1},  we have the inclusion
$$
\mathcal{C}_{\infty}^{1}\cap \mathcal{C}_{\infty}^{2}\cap \mathcal{O}_\sigma\subset \mathtt{C}_{\infty}.
$$
In addition, there exists $\varepsilon_0>0$ and $C>0$ such that for any $\varepsilon\in(0,\varepsilon_0)$ 
$$|\mathcal{O}\backslash  \mathtt{C}_{\infty}|\leqslant C\varepsilon^{\frac{\delta}{\mathtt{n}}},
$$
where $\delta$ is defined through \eqref{lambda-choice}-\eqref{Assump-DRP1} and $\mathtt{n}$ is defined in Proposition \ref{propo-monod}.
\end{lemma}
\begin{proof}
One can check from the definitions of $\mathcal{O}_{n}^1(\rho)$ and $\mathcal{O}_{n}^2(\rho)$ seen in \eqref{Cantor set0} and \eqref{Cantor second} that
\begin{align*}\mathtt{C}_{\infty}\triangleq\bigcap_{m\in\mathbb{N}}\mathcal{A}_{m}=\mathcal{O}_{\infty,1}\cap\mathcal{O}_{\infty,2}\cap \mathcal{O}_\sigma.
\end{align*}
with 
$$
\mathcal{O}_{\infty,k}=\Big\{\xi_0\in \mathcal{O},\; \forall \ell\in\mathbb{Z},\; n\in\N,\;  k\leqslant |j|\leqslant N_n, \;\,|\varepsilon^2 \omega(\xi_0)\ell+\mu_{j,k}(\xi_0,\rho_n)|\geqslant  \tfrac{\lambda}{| j |^\tau}\Big\}.
$$
We plan to  check that for $k=1,2$
\begin{equation}\label{Mahma-l}
\mathcal{C}_{\infty}^{k}\subset \mathcal{O}_{\infty,k},
\end{equation}
which will ensure the inclusion $
\mathcal{C}_{\infty}^{1}\cap \mathcal{C}_{\infty}^{2}\subset \mathtt{C}_{\infty}
$.
For this aim, we pick $\xi_0\in \mathcal{C}_{\infty}^{k}$ and $n\in\N$, then for $k\leqslant |j|\leqslant N_n$ we get from the triangle inequality
\begin{align*}
|\varepsilon^2 \omega(\xi_0)\ell+\mu_{j,k}(\xi_0,\rho_n)|&\geqslant |\varepsilon^2 \omega(\xi_0)\ell+\mu_{j,k}(\xi_0,\rho_\infty)|-|\mu_{j,k}(\xi_0,\rho_\infty)-\mu_{j,k}(\xi_0,\rho_n)|\\
&\geqslant \tfrac{2\lambda}{| j |^\tau}-|\mu_{j,k}(\xi_0,\rho_\infty)-\mu_{j,k}(\xi_0,\rho_n)|.
\end{align*}
Applying \eqref{diff Vpm}, \eqref{lambda-choice}  and Corollary \ref{prop-construction} we get for $\varepsilon$ small enough 
\begin{align*}
|\mu_{j,k}(\xi_0,\rho_\infty)-\mu_{j,k}(\xi_0,\rho_n)|&
\leqslant C\varepsilon^{3}|j|\|\rho_{\infty}-\rho_n\|_{2s_{0}+2\tau+3}^{\textnormal{Lip}(\lambda)}\\
&  \leqslant C_{\ast} \varepsilon^{3+\mathtt{b}_0}\lambda^{-1}|j|N_n^{-a_2}\\
&\leqslant  C_{\ast}\lambda|j|^{-\tau}\varepsilon^{\mathtt{b}_0-1-2\delta}N_n^{1+\tau-a_2}.
\end{align*}
From \eqref{choice-f1} and the assumption $1+\tau<a_2$ and $\delta<\frac{2-\mu}{2}$ (see \eqref{Assump-DRP1}), we deduce  for small $\varepsilon$
\begin{align*}
|\mu_{j,k}(\xi_0,\rho_\infty)-\mu_{j,k}(\xi_0,\rho_n)|&
\leqslant  C_{\ast}\lambda|j|^{-\tau}\varepsilon^{2-\mu-2\delta}N_n^{1+\tau-a_2}\\
&\leqslant  \lambda|j|^{-\tau}.
\end{align*} 
Hence, for any $n\in\N$, for any $\ell\in\Z$ and $k\leqslant |j|\leqslant N_n$ we get
\begin{align*}
|\varepsilon^2 \omega(\xi_0)\ell+\mu_{j,k}(\xi_0,\rho_n)|&\geqslant \tfrac{\lambda}{| j|^\tau}\cdot
\end{align*}
This implies that $\xi_0\in \mathcal{O}_{\infty,k}$, which concludes the proof of \eqref{Mahma-l}.\\ Let us now move to the measure  estimate. First, we write
$$
\big|\mathcal{O}\backslash  \mathtt{C}_{\infty}\big|\leqslant \big|\mathcal{O}\backslash  \mathcal{C}_{\infty}^{1}\big|+\big|\mathcal{O}\backslash \mathcal{C}_{\infty}^{2}\big|+\big|\mathcal{O}\backslash  \mathcal{O}_{\sigma}\big|.
$$ 
Now, we  write,  by \eqref{cond-interval2} \eqref{def sigma} and \eqref{lambda-choice},   
\begin{equation}\label{cond-int3}
\big|\mathcal{O}\setminus \mathcal{O}_\sigma \big|\leqslant C\varepsilon^{\frac{\delta}{\mathtt{n}}}.
\end{equation}
Moreover, from \eqref{lambda-choice}, one has
$$\mathcal{C}_{\infty}^{k}=\bigcap_{n\in\N}\bigcap_{(\ell,j)\in\mathbb{Z}^{2}\atop k\leqslant |j|\leqslant N_{n}}A_{\ell,j}^k.
$$
with
$$A_{\ell,j}^k=\left\lbrace \xi_0\in \mathcal{O}=[\xi_*,\xi^*];\;\, \big|\varepsilon^2\omega(\xi_0)  \ell+\mu_{j,k}(\xi_0,\rho_\infty)\big|\geqslant 2\tfrac{\varepsilon^{2+\delta}}{| j|^{\tau}}\right\rbrace.
$$
Notice that
$$
\bigcap_{(\ell,j)\in\mathbb{Z}^{2}\atop k\leqslant |j|}A_{\ell,j}^k\subset \mathcal{C}_{\infty}^{k}.
$$
From Lemma \ref{lem-period}, we know that $\xi_0\mapsto \omega_0(\xi_0)$ is positive and strictly decreasing. Thus
$$
\forall \xi_0\in[\xi_*,\xi^*],\quad 0<\overline{b}=\omega_0(\xi^*)\leqslant \omega(\xi_0)\leqslant \omega_0(\xi_*)=\overline{a}.
$$
Moreover, from \eqref{mu-j1}, we obtain for $0<\varepsilon\leqslant \varepsilon_0$ 
\begin{align}\label{est-c_0}
\tfrac13\leqslant \mathtt{c}(\xi_0,\rho_\infty)\leqslant \tfrac23.
\end{align}
Hereafter, we will distinguish different cases.\\
$\bullet$ Case $\varepsilon^2 \overline{b}|\ell|\geqslant |j|\geqslant 1, k=1.$ One has by the triangle inequality 
\begin{align*}
 \big|\varepsilon^2\omega(\xi_0)  \ell+j \mathtt{c}(\xi_0,\rho_\infty)\big|&\geqslant \varepsilon^2\overline{b} |\ell|-\tfrac23|j|\\
 &\geqslant\tfrac13|j|\\
 &\geqslant \tfrac{\varepsilon^{2+\delta}}{| j|^{\tau}}\cdot
 \end{align*}
Hence
$$
 A_{\ell,j}^1=[\xi_*,\xi^*].
 $$
$\bullet$ Case $\varepsilon^2 \overline{b}|\ell|\geqslant |j|\geqslant 2, k=2.$ One has by the triangle inequality 
\begin{align*}
 \big|\varepsilon^2\omega(\xi_0)  \ell+j \mathtt{c}(\xi_0,\rho_\infty)-\tfrac{j}{2|j|}\big|&\geqslant \varepsilon^2\overline{b} |\ell|-\tfrac23|j|-\tfrac12\\
 &\geqslant\tfrac13|j|-\tfrac12\geqslant \tfrac16\\
 &\geqslant \tfrac{\varepsilon^{2+\delta}}{|j|^{\tau}}\cdot
 \end{align*}
 Therefore
 $$
 A_{\ell,j}^2=[\xi_*,\xi^*].
 $$
 $\bullet$ Case $|j|\geqslant 24\varepsilon^2 \overline{a}|\ell|, k=1.$ One gets by the triangle inequality and \eqref{est-c_0}, together with $|j|\geqslant 1$
 \begin{align*}
 \big|\varepsilon^2\omega(\xi_0)  \ell+j \mathtt{c}(\xi_0,\rho_\infty)\big|&\geqslant \tfrac13|j|-\varepsilon^2\overline{a}|\ell|\\
 &\geqslant \big(\tfrac{1}{24}|j|-\varepsilon^2\overline{a}|\ell|\big) +\tfrac{7}{24}|j|\\
 &\geqslant\tfrac{7}{24}\cdot
 \end{align*}
  Thus, for small $\varepsilon$ we infer
  \begin{align*}
 \big|\varepsilon^2\omega(\xi_0)  \ell+j \mathtt{c}(\xi_0,\rho_\infty)\big|
 &\geqslant  \tfrac{\varepsilon^{2+\delta}}{| j|^{\tau}}\cdot
 \end{align*}
 Hence
 $$
  A_{\ell,j}^1=[\xi_*,\xi^*].
  $$
  $\bullet$ Case $|j|\geqslant  24\varepsilon^2 \overline{a}|\ell|, k=2.$ One gets by the triangle inequality and \eqref{est-c_0}, together with $|j|\geqslant 2$
 \begin{align*}
 \big|\varepsilon^2\omega(\xi_0)  \ell+j \mathtt{c}(\xi_0,\rho_\infty)-\tfrac{j}{2|j|}\big|&\geqslant \tfrac13|j|-\varepsilon^2\overline{a}|\ell|-\tfrac12\\
 &\geqslant \big(\tfrac{1}{24}|j|-\varepsilon^2\overline{a}|\ell|\big) +\tfrac{7}{24}|j|-\tfrac12\\
 &\geqslant\tfrac{1}{12}\cdot
 \end{align*}
 Thus, for small $\varepsilon$ we infer
 \begin{align*}
 \big|\varepsilon^2\omega(\xi_0)  \ell+j  \mathtt{c}(\xi_0,\rho_\infty)-\tfrac{j}{2|j|}\big|&\geqslant  \tfrac{\varepsilon^{2+\delta}}{| j|^{\tau}}
 \end{align*}
 implying that
 $$
 A_{\ell,j}^2=[\xi_*,\xi^*].
 $$
 It follows that
 \begin{align}\label{Cinfty}\bigcap_{|j|\leqslant 24\varepsilon^2 \overline{a}|\ell|\atop \varepsilon^2 \overline{b}|\ell|\leqslant |j|}\ A_{\ell,j}^k\subset \mathcal{C}_{\infty}^{k}.
 \end{align}
 Define
 $$
 f_{\ell,j}(\xi_0)=\varepsilon^2\omega(\xi_0)  \ell+j  \mathtt{c}(\xi_0,\rho_\infty)-(k-1)\tfrac{j}{2|j|}\cdot
 $$
 Differentiating in $\xi_0$ and using \eqref{mu-j1} yields
 $$
  f_{\ell,j}^\prime(\xi_0)=\varepsilon^2\omega^\prime(\xi_0)  \ell-j\varepsilon^2\omega^\prime(\xi_0) +\varepsilon^{3} j   \mathtt{c}_2^\prime(\xi_0).
   $$
 Applying Lemma \ref{lem-period}, \eqref{c-2-def} and  \eqref{c2-estimate} we infer that 
 $$\forall \xi_0\in[\xi_*,\xi^*],\quad 0<\underline{c}\leqslant |\omega_0^\prime(\xi_0)|\leqslant \overline{c}\quad\hbox{and}\quad |\mathtt{c}_2^\prime(\xi_0)|\leqslant c_1
 $$ and thus
 \begin{align*}
 \forall \xi_0\in[a,b],\quad  |f_{j,\ell}^\prime(\xi_0)|\geqslant \underline{c}\varepsilon^2  |\ell|-\overline{c}|j|\varepsilon^2-\varepsilon^{3} |j| c_1.
  \end{align*}
  Hence, as $|j|\leqslant 24\varepsilon^2 \overline{a}|\ell|,$ we get
  \begin{align*}
 | f_{j,\ell}^\prime(\xi_0)|\geqslant |j|\big(\tfrac{\underline{c}}{24\overline{a}}-\overline{c}\varepsilon^2-c_1\varepsilon^{3} \big),
  \end{align*}
  which implies for  $\varepsilon$ small enough 
   \begin{align*}
  |f_{j,\ell}^\prime(\xi_0)|\geqslant \tfrac{\overline{c}}{48\overline{a}} |j|.
   \end{align*}
  Applying Lemma \ref{Piralt}, we deduce that
  \begin{align*}
  \left|\left\lbrace \xi_0\in [\xi_*,\xi^*];\;\, \big|\varepsilon^2\omega(\xi_0)  \ell+j \mathtt{c}(\xi_0,\rho_\infty)-(k-1)\tfrac{j}{2|j|}\big|<2 \tfrac{\varepsilon^{2+\delta}}{| j|^{\tau}}\right\rbrace\right|\leqslant C\tfrac{\varepsilon^{2+\delta}}{| j|^{1+\tau}}\cdot
  \end{align*}
  Consequently we get from \eqref{Cinfty}
  \begin{align*}
  \big|[\xi_*,\xi^*]\backslash \mathcal{C}_{\infty}^{k}\big|&\leqslant C \sum_{ \varepsilon^2 \overline{b}|\ell|\leqslant |j|}\tfrac{\varepsilon^{2+\delta}}{| j|^{1+\tau}} \\
  &\leqslant C \varepsilon^{\delta} \sum_{   |j|\geqslant 1}\tfrac{1}{| j|^{\tau}} \\
  &\leqslant C \varepsilon^{\delta}.
  \end{align*}
  Finally, we get
  \begin{align*}
 \big|[\xi_*,\xi^*]\backslash \mathtt{C}_{\infty}\big|\leqslant  \big|[\xi_*,\xi^*]\backslash \mathcal{C}_{\infty}^{1}\big|+\big|[\xi_*,\xi^*]\backslash \mathcal{C}_{\infty}^{2}\big|+\big|[\xi_*,\xi^*]\setminus\mathcal{O}_\sigma\big|
  &\leqslant C \varepsilon^{\frac{\delta}{\mathtt{n}}}.
  \end{align*}
  It is  important to note that we are making a slight abuse of the argument here, as the function $f_{j,\ell}$ is Lipschitz and not $C^1$. However, we can rigorously justify the previous argument by manipulating the Lipschitz norm.
  This achieves the proof of the desired result.
\end{proof}
\appendix
\section{Appendix: Toolbox}
This section is devoted to an introduction to various technical lemmas, a collection of useful results on periodic change of coordinates together with some estimates on integral operators and the transport reduction using KAM techniques.

\subsection{Technical lemmas}

		In the  next lemma we shall collect   some useful classical results related to various actions over  weighted Sobolev spaces. The proofs are standard and can be found for instance in \cite{BFM21,BFM,BertiMontalto}.
	\begin{lemma}\label{lem funct prop}
		Let  $(\lambda,s_{0},s)$ satisfy \eqref{cond1}, then the following assertions hold true.
			\begin{enumerate}
				\item Let $\rho$ be a smooth function,  then for all $N\in\mathbb{N}^{*}$ and $t>0$,
				$$\|\Pi_{N}\rho\|_{s+t}^{\textnormal{Lip}(\lambda)}\leqslant N^{t}\|\rho\|_{s}^{\textnormal{Lip}(\lambda)}\qquad\textnormal{and}\qquad\|\Pi_{N}^{\perp}\rho\|_{s}^{\textnormal{Lip}(\lambda)}\leqslant N^{-t}\|\rho\|_{s+t}^{\textnormal{Lip}(\lambda)},
				$$
				where the cut-off projectors are defined by
	\begin{equation}\label{def projectors PiN}
		\Pi_{N} h\triangleq\sum_{\underset{\langle \ell,j\rangle\leqslant N}{(\ell,j)\in\Z^{2}}}h_{\ell,j}\mathbf{e}_{\ell,j}\qquad\textnormal{and}\qquad \Pi^{\perp}_{N}\triangleq\textnormal{Id}-\Pi_{N}.
	\end{equation}
				\item Product law : 
				Let $\rho_{1},\rho_{2}\in \textnormal{Lip}_\lambda(\mathcal{O},H^{s}).$ Then $\rho_{1}\rho_{2}\in \textnormal{Lip}_\lambda(\mathcal{O},H^{s})$ and 
				$$\| \rho_{1}\rho_{2}\|_{s}^{\textnormal{Lip}(\lambda)}\lesssim\| \rho_{1}\|_{s_{0}}^{\textnormal{Lip}(\lambda)}\| \rho_{2}\|_{s}^{\textnormal{Lip}(\lambda)}+\| \rho_{1}\|_{s}^{\textnormal{Lip}(\lambda)}\| \rho_{2}\|_{s_{0}}^{\textnormal{Lip}(\lambda)}.$$
				\item Composition law: Let $f\in C^{\infty}(\mathcal{O}\times\mathbb{R},\mathbb{R})$ and  $\rho_{1},\rho_{2}\in \textnormal{Lip}_\lambda(\mathcal{O},H^{s})$  such that $$\| \rho_{1}\|_{s}^{\textnormal{Lip}(\lambda)},\|\rho_{2}\|_{s}^{\textnormal{Lip}(\lambda)}\leqslant C_{0}$$ for an  arbitrary  constant  $C_{0}>0$ and define the pointwise composition $$\forall (\xi,\varphi,\theta)\in \mathcal{O}\times\mathbb{T}^{2},\quad f(\rho)(\xi,\varphi,\theta)\triangleq  f\big(\xi,\rho(\xi,\varphi,\theta)\big).$$
				Then 
				$$\| f(\rho_{1})-f(\rho_{2})\|_{s}^{\textnormal{Lip}(\lambda)}\leqslant C(s,d,q,f,C_{0})\| \rho_{1}-\rho_{2}\|_{s}^{\textnormal{Lip}(\lambda)}.$$
				\item Interpolation inequality : Let $1<s_{1}\leqslant s_{3}\leqslant s_{2}$ and $\overline{\theta}\in[0,1],$ with  $s_{3}=\overline{\theta} s_{1}+(1-\overline{\theta})s_{2}.$\\
				If $\rho\in \textnormal{Lip}_\lambda(\mathcal{O},H^{s_2})$, then  $\rho\in \textnormal{Lip}_\lambda(\mathcal{O},H^{s_3})$ and
				$$\|\rho\|_{s_{3}}^{\textnormal{Lip}(\lambda)}\lesssim\left(\|\rho\|_{s_{1}}^{\textnormal{Lip}(\lambda)}\right)^{\overline{\theta}}\left(\|\rho\|_{s_{2}}^{\textnormal{Lip}(\lambda)}\right)^{1-\overline{\theta}}.$$
		\end{enumerate}
	\end{lemma}

	We shall state a particular statement of Kirszbraun Theorem \cite{Kirsz}.
		\begin{lemma}[Kirszbraun Theorem]\label{thm-extend}
	Given $a<b,$ $ U$ a subset of $[a,b]$ and $H$ a Hilbert space. Let $ f:U\to H$ be  a Lipschitz function, then $f$ admits a Lipschitz  extension $F=\textnormal{Ext}f: [a,b]\to H$ with the same Lipschitz constant.
	\end{lemma}
		We recall the following classical result on bi-Lipschitz functions and measure theory.
	\begin{lemma}\label{Piralt}
Let  $(\alpha,\beta)\in(\mathbb{R}_{+}^{*})^{2}$ and $f:[a,b]\to \RR$ be  a bi-Lipschitz  function   such that 
$$\forall x,y \in[a,b],\quad |f(x)-f(y)|\geqslant\beta |x-y|.$$
Then  there exists $C>0$ independent of $ \|f\|_{\textnormal{Lip}}$ such that 
$$\Big|\big\lbrace x\in [a,b];\;\, |f(x)|\leqslant\alpha\big\rbrace\Big|\leqslant C\tfrac{\alpha}{\beta}\cdot $$
\end{lemma}

\subsection{Change of coordinates system}
The main goal of this section is to discuss useful  results related to some change of coordinates system. For the proofs we refer the reader to the papers \cite{Baldi-Montalto21,BFM,FGMP19}.
Let $\beta: \mathcal{O}\times \T^{2}\to \R$ be a smooth function such that $\displaystyle\sup_{\xi_0\in \mathcal{O}}\|\beta(\xi_0,\cdot,\centerdot)\|_{\textnormal{Lip}}<1$ 
then there exists $\widehat\beta: \mathcal{O}\times \T^{2}\to \R$ smooth 
such that
\begin{equation}\label{def betahat}
	y=\theta+\beta(\xi_0,\varphi,\theta)\Longleftrightarrow \theta=y+\widehat\beta(\xi_0,\varphi,y).
\end{equation}
Define the operators
\begin{equation}\label{definition symplectic change of variables}
	\mathscr{B}=(1+\partial_{\theta}\beta)\mathcal{B}, \qquad \mathcal{B}h(\xi_0,\varphi,\theta)=h\big(\xi_0,\varphi,\theta+\beta(\xi_0,\varphi,\theta)\big).
\end{equation}
By straightforward  computations we obtain, see for instance \cite{HR21},
\begin{equation}\label{mathscrB1}
	\mathscr{B}^{-1}=(1+\partial_{\theta}\widehat\beta)\mathcal{B}^{-1} ,\qquad \mathcal{B}^{-1} h(\xi_0,\varphi,y)=h\big(\xi_0,\varphi,y+\widehat{\beta}(\xi_0,\varphi,y)\big).
\end{equation}
	We shall now give some elementary algebraic properties for $\mathcal{B}^{\pm 1}$ and $\mathscr{B}^{\pm 1}$ which can be checked by straightforward computations, for more details we refer to \cite{Baldi-Montalto21,BFM,FGMP19}.
			\begin{lemma}\label{algeb1}
				The following assertions hold true.
				\begin{enumerate}
					\item Let $\mathscr{B}_1,\mathscr{B}_2$  be two periodic  change of variables as  in \eqref{definition symplectic change of variables}, then
					$$
					\mathscr{B}_{{1}}\mathscr{B}_2=(1+\partial_{\theta}\beta)\mathcal{B}
					$$
					with
					$$
\beta(\varphi,\theta)\triangleq \beta_1(\varphi,\theta)+\beta_2\big(\varphi,\theta+\beta_1(\varphi,\theta)\big).
$$					\item
					 The conjugation of the transport operator by $\mathscr{B}$  keeps the same structure
					$$
					\mathscr{B}^{-1}\Big(\omega\cdot\partial_\varphi+\partial_\theta\big(V(\varphi,\theta)\cdot\big)\Big)\mathscr{B}=\omega\cdot\partial_\varphi+\partial_y\big(\mathscr{V}(\varphi,y)\cdot\big)
					$$
					with
					$$
					\mathscr{V}(\varphi,y)\triangleq\,\mathcal{B}^{-1}\Big(\omega\cdot\partial_{\varphi} \beta(\varphi,\theta)+V(\varphi,\theta)\big(1+\partial_\theta \beta(\varphi,\theta)\big)\Big).
					$$
				\end{enumerate}
			\end{lemma}
			In what follows, and in the rest of this appendix, we assume that $(\lambda,s,s_0)$ satisfy \eqref{cond1} and we consider  $\beta\in \textnormal{Lip}_\lambda(\mathcal{O},H^{s}(\T^{2})) $ satisfying the smallness condition 
				\begin{equation}\label{small beta lem}
					\|\beta \|_{2s_0}^{\textnormal{Lip}(\lambda)}\leqslant \varepsilon_0,
				\end{equation}
			with $\varepsilon_{0}$ small enough.\\
			The following result is proved in \cite{FGMP19}. We also refer to \cite[(A.2)]{BFM}.
			\begin{lemma}\label{Compos1-lemm}
The following assertions hold true.
				\begin{enumerate}
					\item The linear operators $\mathcal{B},\mathscr{B}:\textnormal{Lip}_\lambda(\mathcal{O},H^{s}(\T^{2}))\to \textnormal{Lip}_\lambda(\mathcal{O},H^{s}(\T^{2}))$ are continuous and invertible, with 
					\begin{align}
						 \|\mathcal{B}^{\pm1}h\|_{s}^{\textnormal{Lip}(\lambda)}\leqslant \|h\|_{s}^{\textnormal{Lip}(\lambda)}\left(1+C\|\beta\|_{s_{0}}^{\textnormal{Lip}(\lambda)}\right)+C\|\beta\|_{s}^{\textnormal{Lip}(\lambda)}\|h\|_{s_{0}}^{\textnormal{Lip}(\lambda)}, \label{tame comp}
\\
						\|\mathscr{B}^{\pm1}h\|_{s}^{\textnormal{Lip}(\lambda)}\leqslant \|h\|_{s}^{\textnormal{Lip}(\lambda)}\left(1+C\|\beta\|_{s_{0}}^{\textnormal{Lip}(\lambda)}\right)+C\|\beta\|_{s+1}^{\textnormal{Lip}(\lambda)}\|h\|_{s_{0}}^{\textnormal{Lip}(\lambda)}. \label{tame comp symp}
					\end{align}
					\item The functions $\beta$ and $\widehat{\beta}$ defined through \eqref{def betahat} satisfy the estimates
					\begin{equation*}
					\|\widehat{\beta}\|_{s}^{\textnormal{Lip}(\lambda)}\leqslant C\|\beta\|_{s}^{\textnormal{Lip}(\lambda)}.
					\end{equation*}
				\end{enumerate}
			\end{lemma}
			The next result was useful. 
			\begin{lemma} \label{beta-inv-asym}
			Let $\varepsilon\in[0,1),$ $f:\mathcal{O}\times\R\to\R$ and $g:\mathcal{O}\times \T^2\to\R$ be two smooth  functions such that
			
			$$\|g\|_{2s_0}^{\textnormal{Lip}(\lambda)}\leqslant 1
			$$
			  and $f(\xi_0,\cdot)$ or $f(\xi_0,\cdot)-\textnormal{Id}$ is $2\pi-$periodic.
			Then the transformation $\Phi(\xi_0):\R^2\to\R^2$ defined by
			$$
			\Phi(\xi_0,\varphi,\theta)=\big(\varphi,\theta+f(\xi_0,\varphi)+ \varepsilon g(\xi_0,\varphi,\theta)\big)
			$$
			is a diffeomorphism, with
			$$
			\Phi^{-1}(\xi_0,\varphi,\theta)=\big(\varphi,\theta-f(\xi_0,\varphi)- \varepsilon g\big(\xi_0,\varphi,\theta-f(\xi_0,\varphi)\big)+\varepsilon^2 \mathtt{r}(\xi_0,\varphi,\theta)\big)
			$$
			and $ \mathtt{r}$ satisfies  the estimate 
			$$
			\forall s\geqslant s_0,\quad \|\mathtt{r}\|_{s}^{\textnormal{Lip}(\lambda)}\leqslant C(s,\|f\|_{s}^{\textnormal{Lip}(\lambda)}) \|g\|_{s+1}^{\textnormal{Lip}(\lambda)}.
			$$
			\end{lemma}
			\begin{proof}
			For any $\xi_0\in\mathcal{O},$ the map $\Phi(\xi_0)$ is of class  $C^1$ and its Jacobian is not vanishing for any  $\varepsilon\in[0,1)$. Thus by the inversion theorem, 
			it is a local diffeormorphism. Using classical arguments we show that it is a global diffeomorphism. To compute the inverse it suffices to invert the  scalar equation in $\theta$, taking $\varphi$ as a parameter,
			$$
			\theta+f(\xi_0,\varphi)+ \varepsilon g(\xi_0,\varphi,\theta)=x$$
			which is equivalent to the fixed point problem
			\begin{align}\label{Fix-1}
			\theta=x-f(\xi_0,\varphi)- \varepsilon g(\xi_0,\varphi,\theta).
			\end{align}
			As we can check, this equation admits a unique solution. To find its asymptotic we write the ansatz
			$$
			\theta=x-f(\xi_0,\varphi)-\varepsilon g\big(\xi_0,\varphi, x-f(\varphi)\big)+\varepsilon^2\mathtt{r}(\xi_0,\varphi,x).
			$$
			Inserting this form  into \eqref{Fix-1} and using Taylor formula  yield ( to simplify  we remove $\xi_0$ )
			\begin{align*}
			\mathtt{r}(\varphi,x)&=\frac1\varepsilon\Big[g\big(\varphi, x-f(\xi_0,\varphi)\big)-g\Big(\varphi, x-f(\varphi)-\varepsilon g\big(\varphi, x-f(\varphi)\big)+\varepsilon^2\mathtt{r}(\varphi,x)\Big)\Big]\\
			\nonumber&=-\Big[g\big(\varphi, x-f(\varphi)\big)-\varepsilon\mathtt{r}(\varphi,x)\Big]\int_0^1\partial_2g\big[\varphi, x-f(\varphi)-\varepsilon \tau g(\varphi, x-f(\varphi))+\varepsilon^2\tau\mathtt{r}(\varphi,x)\big] d\tau.
			\end{align*}
			Using \eqref{tame comp} together with by the smallness condition $\|g\|_{s_0+1}^{\textnormal{Lip}(\lambda)}\leqslant 1$ and the law products one can show that
			\begin{align*}
			\|r\|_{s}^{\textnormal{Lip}(\lambda)}&\le C\big(s,\|f\|_{s}^{\textnormal{Lip}(\lambda)}\big)\Big[ \Big(\|g\|_{s}^{\textnormal{Lip}(\lambda)}+\varepsilon \|r\|_s^{\textnormal{Lip}(\lambda)}\Big)\|g\|_{s_0+1}^{\textnormal{Lip}(\lambda)}\big(1+\varepsilon^2\|\mathtt{r}\|_{s_0}^{\textnormal{Lip}(\lambda)}\big)\\
			&+ \big(\|g\|_{s_0}^{\textnormal{Lip}(\lambda)}+\varepsilon \|r\|_{s_0}^{\textnormal{Lip}(\lambda)}\big)\|g\|_{s+1}^{\textnormal{Lip}(\lambda)}\big(1+\varepsilon^2\|\mathtt{r}\|_{s_0}^{\textnormal{Lip}(\lambda)}\big)\\
			&+  \big(\|g\|_{s_0}^{\textnormal{Lip}(\lambda)}+\varepsilon \|r\|_{s_0}^{\textnormal{Lip}(\lambda)}\big)\|g\|_{s_0+1}^{\textnormal{Lip}(\lambda)}\big(1+\varepsilon\|g\|_s^{\textnormal{Lip}(\lambda)}+ \varepsilon^2\|\mathtt{r}\|_{s}^{\textnormal{Lip}(\lambda)}\big)\Big]
			\end{align*}
			By taking $s=s_0$ we deduce by the smallness condition in $g$  that
			$$
			\|r\|_{s_0}^{\textnormal{Lip}(\lambda)}\le C(f) \|g\|_{s_0+1}^{\textnormal{Lip}(\lambda)}
			$$
			 which implies in turn that
			 $$
			\|r\|_{s}^{\textnormal{Lip}(\lambda)}\le C(f) \|g\|_{s+1}^{\textnormal{Lip}(\lambda)}.
			$$
			This ends the proof of the desired result.
			\end{proof}			
			\subsection{Integral operators}
	We shall consider an integral operator taking the form 
	\begin{equation}\label{Top-op1}(\mathcal{T}_Kh)(\xi_0,\varphi,\theta)\triangleq\int_{\mathbb{T}}K(\xi_0,\varphi,\theta,\eta)h(\xi_0,\varphi,\eta)d\eta,
	\end{equation}
	where the kernel function $K$ may be smooth or singular at the  diagonal set $\{\theta=\eta\}$. The next result deals with some  operator estimates whose proof is a consequence of  \cite[Lem. 4.4]{HR21}.
	\begin{lemma}\label{kernel-est}
		Let   
		$\mathcal{T}_K$ be an integral operator with a real-valued kernel $K$. Then, 
 for any $s_1,s_2\geqslant0,$
		\begin{align*}
			\| \mathcal{T}_Kh\|_{s_1,s_2}^{\textnormal{Lip}(\lambda)}&\lesssim \|h\|_{s_0}^{\textnormal{Lip}(\lambda)}\|K\|_{s_1+s_2}^{\textnormal{Lip}(\lambda)} +\|h\|_{s_1}^{{\textnormal{Lip}(\lambda)}} \|K\|_{s_0+s_2}^{\textnormal{Lip}(\lambda)}.
		\end{align*}

	\end{lemma}		
The following Lemma can be found in \cite[Lem. 2.3]{BertiMontalto}.			
					\begin{lemma}\label{lem CVAR kernel}
			Consider a smooth real-valued kernel
		$$K:(\xi_0,\varphi,\theta,\eta)\mapsto K(\xi_0,\varphi,\theta,\eta)$$
		and the periodic change of variables $\mathscr{B}$, $\mathcal{B}$  defined by \eqref{definition symplectic change of variables}.
		Then the operators $\mathscr{B}^{-1}\mathcal{T}_K\mathscr{B}$ and $\mathscr{B}^{-1}\mathcal{T}_{K}\mathscr{B}-\mathcal{T}_{K}$ are  integral operators, 		
		\begin{align*}\big(\mathscr{B}^{-1}\mathcal{T}_{K}\mathscr{B}\big)h(\xi_0,\varphi,\theta)&=\int_{\mathbb{T}}h(\xi_0,\varphi,{\eta})\widehat{K}(\xi_0,\varphi,\theta,{\eta})d{\eta}\\
		\big(\mathscr{B}^{-1}\mathcal{T}_{K}\mathscr{B}-\mathcal{T}_{K}\big)h(\xi_0,\varphi,\theta)&=\int_{\mathbb{T}}h(\xi_0,\varphi,{\eta})\widetilde{K}(\xi_0,\varphi,\theta,{\eta})d{\eta}
		\end{align*}
		with
\begin{align*}
\|\widehat{K}\|_{s}^{\textnormal{Lip}(\lambda)}&\lesssim \|K\|_{s}^{\textnormal{Lip}(\lambda)}+\|K\|_{s_0}^{\textnormal{Lip}(\lambda)}\|\beta\|_{s+1}^{\textnormal{Lip}(\lambda)},\\
\|\widetilde{K}\|_{s}^{\textnormal{Lip}(\lambda)}&\lesssim\|K\|_{s+1}^{\textnormal{Lip}(\lambda)}\|\beta\|_{s_0}^{\textnormal{Lip}(\lambda)}+\|K\|_{s_0}^{\textnormal{Lip}(\lambda)}\|\beta\|_{s+1}^{\textnormal{Lip}(\lambda)}.
\end{align*}
	\end{lemma}

\subsection{KAM reduction of   transport operators}
Let  $\rho\in \bigcap_{{s\in[s_{0},S]}} \textnormal{Lip}_\lambda(\mathcal{O},H^{s})$ and  $\xi_0\in\mathcal{O}.$ Let  $\epsilon_1, \epsilon_2\in(0,\overline\epsilon]$ be  two parameters with $\overline{\epsilon}>0$ small enough. Consider the transport operator 
$$
\mathscr{T}_0(\xi_0)\triangleq{\epsilon_1}\Omega(\xi_0)\partial_\varphi+\partial_{\theta}\big[\big(\mathtt{c}_{-1}(\xi_0,\rho)+ {\epsilon_2}f_{0}(\xi_0,\rho) \big)\cdot\big]
$$
and assume that there exists a constant $C_0>0$ independent of $\xi_0,\epsilon_1,\epsilon_2$  such that 
\begin{itemize}
\item The functions  $\xi_0\mapsto\Omega(\xi_0), \xi_0\mapsto\mathtt{c}_{-1}(\xi_0,\rho)\in \textnormal{Lip}_\lambda(\mathcal{O},\mathbb{R})$ (they do not depend on the space and time variables) satisfy
\begin{equation}\label{est Omg}
\|\mathtt{c}_{-1}\|^{\textnormal{Lip}(\lambda)}+\|\Omega \|^{\textnormal{Lip}(\lambda)}\leqslant C_0, \qquad \forall \xi_0\in \mathcal{O},\quad  |\Omega(\xi_0)|\geqslant \tfrac{1}{C_0}.
\end{equation} 
\item  There exists a universal number  $\sigma_0\geqslant 0$ such that  
 \begin{equation}\label{boundedness of V0}
\|\rho\|_{s_0+2}^{\textnormal{Lip}(\lambda)}\lesssim 1\; \Longrightarrow\; \forall s\in[s_{0},S],\;  \|f_{0}\|_{s}^{\textnormal{Lip}(\lambda)}\leqslant C_0\big(1+\|\rho\|_{s+\sigma_0}^{\textnormal{Lip}(\lambda)}\big).
 \end{equation}

\end{itemize}
The main goal of the following result  is  to conjugate  the linear periodic transport operator $\mathscr{T}_0$ into an operator with constant coefficient up to a small remainder. This result generalizes previous results obtained in \cite{Baldi-berti,Baldi-Montalto21,BFM21,FGMP19}.
\begin{proposition}\label{Thm transport}
Given the conditions \eqref{cond1}, \eqref{est Omg} and \eqref{boundedness of V0}. Let $\mu_2$ and $s_h$ satisfy  
\begin{equation}\label{Conv-Trans2}
\mu_{2}\geqslant 4\tau+3, \quad s_h=  \tfrac{3}{2}\mu_{2}+2s_{0}+2\tau+1.
\end{equation}
There exists $N_0\geqslant2$ large enough and ${\epsilon}_0>0$ small enough such that if  
\begin{align}\label{small-C2}
 \epsilon_1^{-1}\lambda+ N_{0}^{\mu_{2}}\epsilon_2\lambda^{-1}\leqslant{\epsilon}_0, \qquad \|\rho\|_{s_{h}+\sigma_0}^{\textnormal{Lip}(\lambda)}\leqslant1,
\end{align}
then, we can construct $\mathtt{c}\triangleq \mathtt{c}(\xi_0,\rho)\in \textnormal{Lip}_\lambda(\mathcal{O},\mathbb{R})$ and $\beta\in \bigcap_{{s\in[s_{0},S]}} \textnormal{Lip}_\lambda(\mathcal{O},H^{s}) $
such that with $\mathscr{B}$ as in \eqref{definition symplectic change of variables} one gets the following results.
\begin{enumerate}
\item The function $\xi_0\mapsto\mathtt{c}(\xi_0,\rho)$ (which is constant with respect to the time-space variables)  satisfies the following estimate,
\begin{equation*}
\| \mathtt{c}-\langle f_0\rangle_{\varphi,\theta} \|^{\textnormal{Lip}(\lambda)}\lesssim N_0^{-\frac{\mu_2}{2}}.
\end{equation*}
\item The transformations $\mathscr{B}^{\pm 1},{\mathcal{B}}^{\pm 1}, {\beta}$ and $\widehat{\beta}$ satisfy the following estimates for all $s\in[s_{0},S],$ 
\begin{equation*}
\|\mathscr{B}^{\pm 1}h\|_{s}^{\textnormal{Lip}(\lambda)}
 \lesssim\|h\|_{s}^{\textnormal{Lip}(\lambda)}+\epsilon_2\lambda ^{-1}\| \rho\|_{s+2\tau+2+\sigma_0}^{\textnormal{Lip}(\lambda)}\|h\|_{s_{0}}^{\textnormal{Lip}(\lambda)},
\end{equation*}
and 
\begin{equation*}
\|\widehat{\beta}\|_{s}^{\textnormal{Lip}(\lambda)}\lesssim\|\beta\|_{s}^{\textnormal{Lip}(\lambda)}\lesssim \epsilon_2\lambda^{-1}\left(1+\| \rho\|_{s+2\tau+1+\sigma_0}^{\textnormal{Lip}(\lambda)}\right).
\end{equation*}
\item For any $\xi_0$ in the Cantor set
\begin{equation*}
\mathcal{O}_{\infty,n}^{\tau}(\rho)=\bigcap_{(\ell,j)\in\mathbb{Z}^{2}\atop1\leqslant |j|\leqslant N_{n}}\left\lbrace \xi_0\in \mathcal{O};\;\, \big|\epsilon_1\Omega(\xi_0)  \ell+j (\mathtt{c}_{-1}+ {\epsilon_2}\mathtt{c}(\xi_0,\rho))\big|\geqslant\tfrac{2\lambda}{| j |^{\tau}}\right\rbrace
\end{equation*}
we have 
\begin{align*}
\mathscr{B}^{-1}\Big(& \epsilon_1\Omega(\xi_0)\partial_\varphi +\partial_{\theta}\big[\big(\mathtt{c}_{-1}+ {\epsilon_2}f_0 \big)\cdot\big]\Big)\mathscr{B}=\epsilon_1\Omega(\xi_0)\partial_\varphi+ (\mathtt{c}_{-1}+ {\epsilon_2}\mathtt{c}(\xi_0,\rho))\partial_{\theta}+\mathtt{E}_{n}^{0}
\end{align*}
with $\mathtt{E}_{n}^{0}$ a linear operator satisfying
$$\|\mathtt{E}_{n}^{0}h\|_{s_0}^{\textnormal{Lip}(\lambda)}\lesssim \epsilon_2 N_{0}^{\mu_{2}}N_{n+1}^{-\mu_{2}}\|h\|_{s_{0}+2}^{\textnormal{Lip}(\lambda)}.$$
	\item Given two functions $\rho_{1}$ and $\rho_{2}$ both satisfying \eqref{small-C2}, we have 
			\begin{align*}
				\|\Delta_{12}\mathtt{c}\|^{\textnormal{Lip}(\lambda)}&\lesssim \| \Delta_{12}\mathtt{c}_0\|^{\textnormal{Lip}(\lambda)}+\epsilon_2\|\Delta_{12}\rho\|_{2s_{0}+2\tau+3}^{\textnormal{Lip}(\lambda)}.
			\end{align*}
\end{enumerate}
The sequence $(N_m)_m$ was  defined in \eqref{definition of Nm}.
\end{proposition}
\begin{proof}
 The proof can be performed following the same lines of \cite[Prop. 6.2]{HR21} and \cite{Baldi-berti}. It revolves  around KAM scheme where at each step we   conjugate the linear operator into a new one with constant coefficients up to smaller perturbations. In our presentation, we will offer a succinct overview of the core concepts, establish certain critical estimates, and defer the remaining aspects, which can be derived similarly to the existing proof of \cite[Prop. 6.2]{HR21}.  \\
$\blacktriangleright$ \textbf{KAM scheme}.  
\smallskip
\noindent {\textit{Induction.}}  Let us start initially with the operator
$$
\mathscr{T}_0(\xi_0)\triangleq\epsilon_1\Omega(\xi_0)\partial_\varphi +\partial_{\theta}\big[\big(\mathtt{c}_{-1}+{\epsilon_2}\mathtt{c}_0+{\epsilon_2} f_0\big)\cdot\big],\quad\hbox{with}\quad \mathtt{c}_0(\xi_0)=0
$$
then in view of \eqref{boundedness of V0} one has 
 \begin{align}\label{estimate delta0 and I02}
\forall s\in[s_0,S],\quad \delta_{0}(s)&\triangleq\lambda^{-1}{\epsilon_2} \|f_{0}\|_{s}^{\textnormal{Lip}(\lambda)}\nonumber\\
&\leqslant C\lambda^{-1}\epsilon_2\big(1+\|\rho\|_{s+\sigma_0}^{\textnormal{Lip}(\lambda)}\big).
\end{align}
By the smallness condition \eqref{small-C2} we obtain
\begin{equation}\label{initial smallness condition in sh norm}
N_{0}^{\mu_{2}}\delta_{0}(s_{h})\leqslant {2} C{\epsilon}_0. 
\end{equation}
From \eqref{initial smallness condition in sh norm}  and using the assumption on $\mu_2$ we infer
\begin{align*}
N_{0}^{2\tau+1}\delta_{0}(s_{0})
&\leqslant N_{0}^{2\tau+1-\mu_{2}}N_{0}^{\mu_{2}}\delta_{0}({s_{h}})\\
&\leqslant {2} C{\epsilon}_0N_{0}^{-1}.
\end{align*}
Taking  $N_{0}$ large enough such that 
\begin{equation}\label{Est-N0}
{2}CN_{0}^{-1}\leqslant 1
\end{equation}
we find
$$
N_{0}^{2\tau+1}\delta_{0}(s_{0})\leqslant{\epsilon}_0.
$$
Now, assume  that  for some $m\in\N$ we have a transport operator the form, 
$$
\mathscr{T}_m(\xi_0)\triangleq\epsilon_1\Omega(\xi_0)\partial_\varphi +\partial_{\theta}\big[\big(\mathtt{c}_{-1}+{\epsilon_2}\mathtt{c}_m(\xi_0)+{\epsilon_2} f_m(\xi_0)\big)\cdot\big],
$$
where $\mathtt{c}_{m}=\mathtt{c}_{m}(\xi_0)$, $f_{m}=f_{m}(\xi_0,\varphi,\theta)$ satisfy  the  properties
\begin{equation}\label{hypothesis of induction deltam}
\delta_{m}(s_l)\leqslant\delta_{0}(s_{h})N_{0}^{\mu_{2}}N_{m}^{-\mu_{2}},\quad\quad\delta_{m}(s_{h})\leqslant \big(2-\tfrac{1}{m+1}\big)\delta_{0}(s_{h})
\end{equation}
and
\begin{equation}\label{assumptions KAM iterations}
\|\mathtt{c}_{m}\|^{\textnormal{Lip}(\lambda)}\leqslant C,\qquad\quad N_{m}^{2\tau +1}\delta_{m}(s_{0})\leqslant{\epsilon}_0,
\end{equation}
with
\begin{equation}\label{def sl}
\delta_{m}(s)\triangleq \lambda^{-1}{\epsilon_2}\| f_{m}\|_{s}^{\textnormal{Lip}(\lambda)},
 \quad s_l\triangleq 2s_{0}+2\tau+1,
\end{equation}
the parameter $s_h$ is defined in \eqref{Conv-Trans2}.
Consider a symplectic periodic change of variables $\mathscr{B}_{m}$ in the form 
\begin{align*}
\mathscr{B}_{m} h(\xi_0,\varphi,\theta)&\triangleq\big(1+{\epsilon_2}\partial_{\theta}g_{m}(\xi_0,\varphi,\theta)\big){\mathcal B}_{m} h(\xi_0,\varphi,\theta)\\
&=\big(1+{\epsilon_2}\partial_{\theta}g_{m}(\xi_0,\varphi,\theta)\big) h\big(\xi_0,\varphi,\theta+{\epsilon_2}g_{m}(\xi_0,\varphi,\theta)\big)
\end{align*}
with $g_m:\mathcal{O}\times\mathbb{T}^{2}\rightarrow\mathbb{R}$ being a    function to be adjusted later  with respect to  $f_m.$ 
The main goal  is to obtain after this transformation a new transport operator in the form
\begin{align}\label{Conjug-ope-tr}
\nonumber\mathscr{B}_{m}^{-1}\Big(\epsilon_1\Omega(\xi_0)\partial_\varphi&+\partial_{\theta}\big[\big(\underbrace{\mathtt{c}_{-1}+{\epsilon_2}\mathtt{c}_{m}}_{\triangleq \widehat{\mathtt{c}}_{m}}+{\epsilon_2}\ f_{m}(\xi_0)\big)\cdot\big]\Big)\mathscr{B}_{m}=\\ &\quad\epsilon_1\Omega(\xi_0)\partial_{\varphi}+\partial_{\theta}\big[\big(\widehat{\mathtt{c}}_{m+1}+ {\epsilon_2}\ f_{m+1}(\xi_0)\big)\cdot\big],
\end{align}
where
$\widehat{\mathtt{c}}_{m+1}(\xi_0)\in\mathbb{R}$, $f_{m+1}(\xi_0)=f_{m+1}(\xi_0,\varphi,\theta)$ quadratically smaller than $f_m$ 
 and satisfying  the constraints \eqref{hypothesis of induction deltam} and \eqref{assumptions KAM iterations} at the order $m+1$. According to Lemma \ref{algeb1},  we may write 
\begin{align}\label{transformation KAM step transport}
\nonumber&\mathscr{B}_m^{-1}\Big(\epsilon_1\Omega(\xi_0)\partial_\varphi +\partial_{\theta}\big[\big(\widehat{\mathtt{c}}_{m}+ {\epsilon_2}f_m(\xi_0)\big)\cdot\big]\Big)\mathscr{B}_m
 =\epsilon_1 \Omega(\xi_0)\partial_\varphi\\
\nonumber&+\partial_{\theta}\Big[{\mathcal B}_m^{-1}\Big(\widehat{\mathtt{c}}_{m}
+\epsilon_1{\epsilon_2}\Omega(\xi_0)\partial_\varphi g_m(\xi_0)+{\epsilon_2}\widehat{\mathtt{c}}_m\partial_{\theta}g_m(\xi_0)+ {\epsilon_2}P_{N_m} f_m(\xi_0)\\
&+ {\epsilon_2}P_{N_m}^{\perp} f_m(\xi_0)+ {\epsilon_2^2} f_m(\xi_0)\partial_{\theta}g_m(\xi_0)\Big)\cdot\Big],
\end{align}
where, for all $N\geqslant 1$ we define the  projector $P_{N}$ and $P_{N}^\perp$ as follows
	\begin{equation*}
		P_{N} \Big(\sum_{(\ell,j)\in \mathbb{Z}^{2}}h_{\ell,j}{\bf{e}}_{\ell,j}\Big)\triangleq\sum_{\underset{|j|\leqslant N}{(\ell,j)\in\Z^{2}}}h_{\ell,j}\mathbf{e}_{\ell,j}\qquad\textnormal{and}\qquad P^{\perp}_{N}\triangleq\textnormal{Id}-P_{N}. 
	\end{equation*}
We shall   impose  to the perturbation $g_m$ the following homological equation 
\begin{eqnarray}  \label{equation satisfied by g-212}
          	&\epsilon_1 \Omega_0(\xi_0)\partial_\varphi g_{m}+\widehat{\mathtt{c}}_{m}\partial_{\theta}g_{m}+P_{N_m} f_m= \langle f_m\rangle_{\varphi,\theta}, 
\end{eqnarray}	
with  
$$ \langle f_m\rangle_{\varphi,\theta}=\frac{1}{(2\pi)^{2}}\int_{\mathbb{T}^{2}}f_m(\xi_0,\varphi,\theta)d\varphi d\theta.$$
To solve  the equation  \eqref{equation satisfied by g-212}, we use Fourier series in order to recover the solution in the form
\begin{align}
\nonumber g_{m}(\xi_0,\varphi,\theta)&=\sum_{\ell\in\mathbb{Z}\backslash\{0\} }\frac{\ii\,f_m)_{\ell,0}(\xi_0)}{\epsilon_1\Omega(\xi_0)  \ell}e^{\ii \ell \varphi}+\sum_{\ell\in\mathbb{Z} \atop 1\leqslant |j|\leqslant N_m}\frac{\ii\,(f_m)_{\ell,j}(\xi_0)}{\epsilon_1\Omega(\xi_0)  \ell+j\mathtt{c}_m(\xi_0)}e^{\ii (\ell \varphi+j\theta)},\\
&\triangleq g_{m,1}(\xi_0,\varphi)+g_{m,2}(\xi_0,\varphi,\theta)\label{definition of g1}.
\end{align}
The estimate of $ g_{m,1}$ follows easily from direct calculations and \eqref{est Omg}, 	
 	\begin{equation*}
\|g_{m,1}\|_{s+1}^{\textnormal{Lip}(\lambda)}\leqslant {C}{\epsilon_1}^{-1}\| f_m\|_{s}^{\textnormal{Lip}(\lambda)}.
				\end{equation*}	
As to $ g_{m,2}$, we need to localize  the parameters around the following Cantor sets
\begin{equation*}
\mathcal{O}_{m+1}^{\lambda}\triangleq\bigcap_{\ell\in\mathbb{Z} \atop 1\leqslant |j|\leqslant N_m}\left\lbrace\xi_0\in \mathcal{O}_{m}^{\lambda };\;\, \big|\epsilon_1\Omega(\xi_0)\ell+j\,\widehat{\mathtt{c}}_{m}(\xi_0)\big|>\lambda{| j|^{-\tau}}\right\rbrace,\quad\textnormal{where}\quad \mathcal{O}_{0}^{\lambda}\triangleq \mathcal{O}. 
\end{equation*}
  We extend the Fourier coefficients of $g_{m,2}$ in \eqref{definition of g1}  in the following way 
\begin{align*}
g_{m,2}(\xi_0,\varphi,\theta)&= \epsilon_2\sum_{\ell\in\mathbb{Z} \atop 1\leqslant |j|\leqslant N_m}\ii\, \lambda^{-1}|j|^{\tau}\widehat{\chi}\Big(\lambda^{-1}|j|^{\tau} \big(\epsilon_1\Omega(\xi_0)  \ell+j\,\widehat{\mathtt{c}}_{m}(\xi_0)\big)\Big) (f_{m})_{\ell,j}(\xi_0)\,e^{\ii (\ell\varphi+j\theta)},  
\end{align*}
where $\widehat{\chi}(x)\triangleq\tfrac{\chi(x)}{x}$ and  $\chi$ is the cut-off function defined in \eqref{properties cut-off function first reduction}. 
Note that $g_{m,2}$ is well-defined on the whole set of parameters $\mathcal{O}$ and when it is  restricted to the  Cantor set  $\mathcal{O}_{m+1}^{\lambda}$
it solves  the {homological equation}  \eqref{equation satisfied by g-212}. 
				Since $\widehat{\chi}$ is $C^{\infty}$ with bounded derivatives and $\widehat{\chi}(0)=0,$ then applying Lemma \ref{lem funct prop} and using \eqref{est Omg} and the first assumption in \eqref{assumptions KAM iterations}, we obtain
\begin{equation*}
\forall k\in\{0,1\},\quad  \|\partial_\theta^{(k)} g_{m,1}\|_{s}^{\textnormal{Lip}(\lambda)}\leqslant C \lambda^{-1}\min\left(\|P_{N_m}f_m\|_{s+2\tau+k}^{\textnormal{Lip}(\lambda)}, N_m^{2\tau+k}\|P_{N_m}f_m\|_{s}^{\textnormal{Lip}(\lambda)} \right).				\end{equation*}
Hence, combining this latter estimate with  \eqref{definition of g1}, we get   for all $k\in\{0,1\}$,
\begin{equation}\label{control of g by f}
{\| \partial_\theta^{(k)} g_m\|_{s}^{\textnormal{Lip}(\lambda)}\leqslant C \lambda^{-1}\min\left(\|f_m\|_{s+2\tau+k}^{\textnormal{Lip}(\lambda)}, N_m^{2\tau+k}\|f_m\|_{s}^{\textnormal{Lip}(\lambda)} \right)}+{C}{\epsilon_1}^{-1}\| f_m\|_{s+k-1}^{\textnormal{Lip}(\lambda)}.
				\end{equation}
In particular, from \eqref{small-C2} and  \eqref{assumptions KAM iterations}, we infer
\begin{align}\label{est gms0}
\nonumber \epsilon_2\|g_m\|_{s_{0}}^{\textnormal{Lip}(\lambda)}\leqslant &\, C\epsilon_2\big(\lambda^{-1}N_m^{2\tau}+{\epsilon_1}^{-1}\big)\|f_m\|_{s_{0}}^{\textnormal{Lip}(\lambda)}\\ \nonumber&\leqslant C\lambda^{-1}\epsilon_2N_m^{2\tau}\|f_m\|_{s_{0}}^{\textnormal{Lip}(\lambda)}\\ &\quad \leqslant CN_0^{-1}{\epsilon}_0.
\end{align}
Moreover, according to \eqref{control of g by f},
\eqref{hypothesis of induction deltam}, \eqref{boundedness of V0}, \eqref{Conv-Trans2}  and \eqref{small-C2}, we infer 
\begin{align*}
\epsilon_2\|g_m\|_{2s_0}^{\textnormal{Lip}(\lambda)}\leqslant  &\, C\lambda^{-1}\epsilon_2 N_m^{2\tau}\|f_m\|_{2s_0}^{\textnormal{Lip}(\lambda)} +{\epsilon_1}^{-1}\epsilon_2\| f_m\|_{2s_0}^{\textnormal{Lip}(\lambda)}\\ &\leqslant C\lambda^{-1} \epsilon_2N_m^{2\tau}\|f_m\|_{s_l}^{\textnormal{Lip}(\lambda)}\\ &\quad  \leqslant C\lambda^{-1}\epsilon_2 N_{0}^{\mu_{2}}N_m^{2\tau-\mu_{2}}\|f_0\|_{s_h}^{\textnormal{Lip}(\lambda)}
\\ &\qquad  \leqslant C \epsilon_2\lambda^{-1}N_0^{-1} N_{0}^{\mu_{2}}\big(1+\|\rho\|_{s_h+\sigma_0}^{\lambda}\big) \\ &\quad\qquad\leqslant 2CN_0^{-1}\epsilon_0.
\end{align*}
				By choosing $N_0$ large enough, one gets 
				\begin{align*}
					\epsilon_2\|g_m\|_{2s_0}^{\textnormal{Lip}(\lambda)}&\leqslant\epsilon_0.
				\end{align*}
Hence, taking ${\epsilon}_0$ small enough we may guarantee the smallness condition in Lemma \ref{Compos1-lemm}  and get  that the linear operator $\mathscr{B}_m$ is  invertible. From  \eqref{transformation KAM step transport} we impose   the identity \eqref{Conjug-ope-tr} with
$\mathtt{c}_{m+1}(\xi_0)$ and $f_{m+1}(\xi_0)$  defined by
\begin{equation}\label{definition Vm+1 and fm+1}
\left\lbrace\begin{array}{l}
\mathtt{c}_{m+1}(\xi_0)=\mathtt{c}_{m}(\xi_0)+\langle f_{m}(\xi_0)\rangle_{\varphi,\theta}\\
f_{m+1}(\xi_0)={\mathcal B}_{m}^{-1}\Big(P_{N_{m}}^{\perp}f_{m}(\xi_0)+{\epsilon_2} f_{m}(\xi_0)\partial_{\theta}g_{m}(\xi_0)\Big).
\end{array}\right.
\end{equation}
 We now set 
$$u_m(\xi_0)=P_{N_m}^{\perp}f_m(\xi_0)+{\epsilon_2} f_m(\xi_0)\partial_{\theta}g_m(\xi_0).$$
Using the product law of Lemma \ref{lem funct prop} and the estimate  \eqref{control of g by f}   we get, for all $s\in[s_{0},S]$, 
$$\begin{array}{rcl}
\|u_m\|_{s}^{\textnormal{Lip}(\lambda)} 
& \leqslant & \|P_{N_m}^{\perp}f_m\|_{s}^{\textnormal{Lip}(\lambda)}+C{\epsilon_2}\Big(\|f_m\|_{s_0}^{\textnormal{Lip}(\lambda)}\|\partial_{\theta}g_m\|_{s}^{\textnormal{Lip}(\lambda)}+\|f_m\|_{s}^{\textnormal{Lip}(\lambda)}\|\partial_{\theta}g_m\|_{s_0}^{\textnormal{Lip}(\lambda)}\Big)\\
& \leqslant & \|P_{N_m}^{\perp}f_m\|_{s}^{\textnormal{Lip}(\lambda)}+C{\epsilon_2}\big(\lambda^{-1}N_m^{2\tau+{1}}+{\epsilon_1}^{-1}\big)\|f_m\|_{s_0}^{\textnormal{Lip}(\lambda)}\|f_m\|_{s}^{\textnormal{Lip}(\lambda)}.
\end{array}$$
Then, by Lemma \ref{Compos1-lemm}-$1$, Lemma \ref{lem funct prop}-$1$ and the estimates \eqref{control of g by f}, \eqref{assumptions KAM iterations}, \eqref{est gms0} \eqref{small-C2} we obtain, for all $ s_{0} \leqslant s\leqslant \overline{s}\leqslant S$,
\begin{align}
\nonumber\|f_{m+1}\|_{s}^{\textnormal{Lip}(\lambda)} & =  \|{\mathcal B}_m^{-1}(u_m)\|_{s}^{\textnormal{Lip}(\lambda)}
\\
\nonumber & \leqslant  \|u_m\|_{s}^{\textnormal{Lip}(\lambda)}+C{\epsilon_2}\Big(\|u_m\|_{s}^{\textnormal{Lip}(\lambda)}\|g_m\|_{s_0}^{\textnormal{Lip}(\lambda)}+\|g_m\|_{s}^{\textnormal{Lip}(\lambda)}\|u_m\|_{s_0}^{\textnormal{Lip}(\lambda)}\Big)
\\
 \nonumber&\leqslant  N^{s-\overline{s}}\|f_m\|_{\overline{s}}^{\textnormal{Lip}(\lambda)}+C{\epsilon_2}\big(\lambda^{-1}N_m^{2\tau+{1}}+{\epsilon_1}^{-1}\big)\|f_m\|_{s_0}^{\textnormal{Lip}(\lambda)}\|f_m\|_{s}^{\textnormal{Lip}(\lambda)}
 \\
 &\leqslant  N^{s-\overline{s}}\|f_m\|_{\overline{s}}^{\textnormal{Lip}(\lambda)}+C\lambda^{-1}{\epsilon_2}N_m^{2\tau+1}\|f_m\|_{s_0}^{\textnormal{Lip}(\lambda)}\|f_m\|_{s}^{\textnormal{Lip}(\lambda)}.
  \label{estimate KAM step transport}
\end{align}
\smallskip
Next, we shall check \eqref{hypothesis of induction deltam} with $m$ replaced by $m+1$. Using \eqref{estimate KAM step transport}, we deduce the following  formulae, true for
any $s_0\leqslant s\leqslant \overline{s}\leqslant S$,
\begin{equation}\label{recurrence estimate deltam}
\delta_{m+1}(s)\leqslant N_{m}^{s-\overline{s}}\delta_{m}(\overline{s})+C N_m^{2\tau+{1}}\delta_{m}(s)\delta_{m}(s_{0}). 
\end{equation}
Applying \eqref{recurrence estimate deltam} with $s=s_l$ and $\overline{s}=s_{h}$,  we get 
$$\delta_{m+1}(s_l)\leqslant N_{m}^{s_l-s_{h}}\delta_{m}(s_{h})+C N_m^{2\tau+{1}}\delta_{m}(s_l)\delta_{m}(s_{0}).
$$
Using Sobolev embeddings, \eqref{hypothesis of induction deltam}, \eqref{assumptions KAM iterations} and \eqref{initial smallness condition in sh norm}  yields
\begin{align*}
\delta_{m+1}(s_l)&\leqslant N_{m}^{s_l-s_{h}}\delta_{m}(s_{h})+C N_m^{2\tau+{1}}\big(\delta_{m}(s_l)\big)^{2}
\\
&\leqslant\big(2-\tfrac{1}{m+1}\big)N_{m}^{s_l-s_{h}}\delta_{0}(s_{h})+CN_{0}^{2\mu_{2}}N_{m}^{2\tau+1-2\mu_{2}}\big(\delta_{0}(s_{h})\big)^{2}
\\
&\leqslant 2N_{m}^{s_l-s_{h}}\delta_{0}(s_{h})+C\epsilon_0 N_{0}^{\mu_{2}}N_{m}^{2\tau+1-2\mu_{2}}\delta_{0}(s_{h}).
\end{align*} 
If we select our parameter $\mu_{2}$ such that
\begin{equation}\label{conv-t1}
\begin{aligned}
N_{m}^{s_l-s_{h}}&\leqslant\tfrac{1}{6}N_{0}^{\mu_{2}}N_{m+1}^{-\mu_{2}},\\ C\epsilon_0 N_{0}^{\mu_{2}}N_{m}^{2\tau+1-2\mu_{2}}&\leqslant\tfrac{1}{3}N_{0}^{\mu_{2}}N_{m+1}^{-\mu_{2}},\\ C\epsilon_0 N_{0}^{\mu_{2}}N_{m}^{-\frac32\mu_{2}}  &\leqslant\tfrac{1}{3}N_{0}^{\mu_{2}}N_{m+1}^{-\mu_{2}},
\end{aligned}
\end{equation}
then 
$$\delta_{m+1}(s_{l})\leqslant \delta_{0}(s_{h})N_{0}^{\mu_{2}}N_{m+1}^{-\mu_{2}}.$$
Notice that under the assumption  \eqref{Conv-Trans2} and from the expressions \eqref{Conv-Trans2}, \eqref{def sl}, \eqref{definition of Nm}, the condition \eqref{conv-t1} is fulfilled provided that
$$6N_{0}^{-\mu_{2}}\leqslant 1\quad\textnormal{and}\quad 3C\epsilon_0\leqslant 1.$$
These conditions are automatically satisfied by taking $N_{0}$ sufficiently large and  ${\epsilon}_0$  small enough. This proves the first statement of the induction in \eqref{hypothesis of induction deltam}. Now, we turn to the proof of the second statement. Using  \eqref{recurrence estimate deltam} with $s=\overline{s}=s_h$  combined with the  induction assumption \eqref{hypothesis of induction deltam}, we get
\begin{align*}
\delta_{m+1}(s_{h})&\leqslant\delta_{m}(s_{h})\left(1+CN_{m}^{2\tau+1}\delta_{m}(s_{0})\right)\\
&\leqslant\big(2-\tfrac{1}{m+1}\big)\delta_{0}(s_{h})\big(1+CN_{0}^{\mu_{2}}N_{m}^{2\tau+1-\mu_{2}}\delta_{0}(s_{h})\big).
\end{align*}
If we impose the condition
\begin{equation}\label{conv-t2}
\big(2-\tfrac{1}{m+1}\big)\left(1+CN_{0}^{\mu_{2}}N_{m}^{2\tau+1-\mu_{2}}\delta_{0}(s_{h})\right)\leqslant 2-\tfrac{1}{m+2},
\end{equation}
then we find
$$
\delta_{m+1}(s_{h})\leqslant\big(2-\tfrac{1}{m+2}\big)\delta_{0}(s_{h}),
$$
achieving the induction argument of \eqref{hypothesis of induction deltam}. 
 Now, observe that \eqref{conv-t2} is equivalent to
$$\big(2-\tfrac{1}{m+1}\big)CN_{0}^{\mu_{2}}N_{m}^{2\tau+1-\mu_{2}}\delta_{0}(s_{h})\leqslant\tfrac{1}{(m+1)(m+2)}.$$
Since  $\mu_{2}\geqslant 2\tau+2$,  the preceding condition is satisfied if
\begin{equation}\label{conv-t3}
CN_{0}^{\mu_{2}}N_{m}^{-1}\delta_{0}(s_{h})\leqslant\tfrac{1}{(m+1)(m+2)}\cdot
\end{equation}
By vertue of \eqref{definition of Nm} and since $N_{0}\geqslant 2,$ we can find a constant $C_{0}>0$ small enough such that
$$
\forall m\in\mathbb{N},\quad C_{0}N_{m}^{-1}\leqslant\tfrac{1}{(m+1)(m+2)}\cdot
$$
Thus, \eqref{conv-t3} holds true provided that
\begin{equation*}
CN_{0}^{\mu_{2}}\delta_{0}(s_{h})\leqslant C_{0}.
\end{equation*}
Taking ${\epsilon_0}$ small enough, we can ensure from \eqref{initial smallness condition in sh norm} that
\begin{align*}
CN_{0}^{\mu_{2}}\delta_{0}(s_{h})&\leqslant C{\epsilon}_0\\
&\leqslant C_{0}.
\end{align*}
As an immediate consequence of \eqref{hypothesis of induction deltam} and \eqref{def sl} together with \eqref{boundedness of V0} and \eqref{small-C2} we infer
\begin{align}\label{Prat-pram}
\nonumber \| f_{m}\|_{s_l}^{\textnormal{Lip}(\lambda)}&\leqslant \| f_{0}\|_{s_h}^{\textnormal{Lip}(\lambda)}N_{0}^{\mu_{2}}N_{m}^{-\mu_{2}}\\
&\leqslant CN_{0}^{\mu_{2}}N_{m}^{-\mu_{2}}.
\end{align}
From  Sobolev embeddings, \eqref{definition Vm+1 and fm+1} and \eqref{Prat-pram}, we have
\begin{align}\label{Cauchy Vm}
\| \mathtt{c}_{m+1}-\mathtt{c}_{m}\|^{\textnormal{Lip}(\lambda)}&=\|\langle f_{m}\rangle_{\varphi,\theta}\|^{\textnormal{Lip}(\lambda)}\nonumber\\
&\leqslant \| f_{m}\|_{s_0}^{\textnormal{Lip}(\lambda)}\nonumber\nonumber\\
& \leqslant CN_{0}^{\mu_{2}}N_{m}^{-\mu_{2}}.
\end{align}
This implies in view of the  triangle inequality  combined with \eqref{initial smallness condition in sh norm}  
\begin{align*}
\sup_{m\in\N}\| \mathtt{c}_{m}\|^{\lambda}  & \leqslant  \displaystyle\| \mathtt{c}_{0}\|^{\lambda}+CN_{0}^{\mu_{2}}\sum_{k=0}^{\infty}N_{k}^{-\mu_{2}}\\
&\leqslant C.
\end{align*}
According to \eqref{Cauchy Vm}  one has,
\begin{align*}\sum_{m=0}^{\infty}\| \mathtt{c}_{m+1}-\mathtt{c}_{m}\|^{\textnormal{Lip}(\lambda)}&\leqslant CN_{0}^{\mu_{2}}\sum_{m=0}^{\infty}N_{m}^{-\mu_{2}}\\
&\leqslant C.
\end{align*}
where we have used the following  inequality: There exists $C_0>0$ such that 
\begin{align}\label{DEcay-seq-mercredi}
\forall N_0\geqslant 2, \forall k\in\N,  \sum_{m=k}^{\infty}N_{m}^{-\mu_{2}}\leqslant C_0 N_k^{-\mu_{2}},
\end{align}
which follows easily from the proof of  \cite[Lemma A.2]{HHM21}.
Therefore, the sequence $(\mathtt{c}_{m})_{m\in\mathbb{N}}$ converges towards an element $\xi_0\mapsto \mathtt{c}(\xi_0,\rho)\in \textnormal{Lip}_\lambda(\mathcal{O},\RR).$\\
 As $\mathtt{c}_0=0$  and  by \eqref{definition Vm+1 and fm+1} we may write 
\begin{align*}
\mathtt{c}&=\mathtt{c}_1+\sum_{m\geqslant 1} (\mathtt{c}_{m+1}-\mathtt{c}_m)\\
&=\langle f_0\rangle_{\varphi,\theta}+\sum_{m\geqslant 1} (\mathtt{c}_{m+1}-\mathtt{c}_m).
\end{align*}
Hence by  \eqref{definition Vm+1 and fm+1}, \eqref{DEcay-seq-mercredi} and \eqref{Cauchy Vm} we find 
\begin{align*}
\| \mathtt{c}-\langle f_0\rangle_{\varphi,\theta}\|^{\textnormal{Lip}(\lambda)}&\leqslant\sum_{m=1}^{\infty}\| \mathtt{c}_{m+1}-\mathtt{c}_{m}\|^{\textnormal{Lip}(\lambda)}\\
&\lesssim N_{0}^{\mu_{2}}N_1^{-\mu_2}\\
&\lesssim N_{0}^{-\frac12\mu_{2}}.
\end{align*}
On the other hand, by 
\eqref{estimate KAM step transport},  \eqref{assumptions KAM iterations} ,  \eqref{hypothesis of induction deltam}, \eqref{definition of Nm}, \eqref{initial smallness condition in sh norm}, \eqref{Est-N0} and since  $\mu_{2}\geqslant {3\tau+3}$, we deduce, 
\begin{align*}
\delta_{m+1}(s_{0})N_{m+1}^{2\tau+1}\leqslant& \Big(1 +CN_m^{2\tau+{1}}\delta_{m}(s_{0})\Big)N_{m+1}^{2\tau+1}\delta_{m}(s_{0})\nonumber\\ &\leqslant  C N_{m+1}^{2\tau+1}\delta_{m}(s_{0})\nonumber\\ &\quad \leqslant C\delta_{0}(s_{h})N_{0}^{\mu_{2}}N_{m}^{\frac32(2\tau+1)-\mu_{2}}\nonumber\\ 
& \qquad\leqslant C\epsilon_0 N_{0}^{-1}\nonumber\\
&\quad\qquad \leqslant{\epsilon_0}.
\end{align*}
{\textit{Regularity persistence.}} 
Making use of \eqref{recurrence estimate deltam} with $\overline{s}=s\in [s_{0},S]$, \eqref{hypothesis of induction deltam}, \eqref{initial smallness condition in sh norm} and using the fact that $\mu_2\geqslant 2\tau+2$, we obtain
\begin{align*}
\delta_{m+1}(s)  \leqslant &\,\delta_{m}(s)\left(1+CN_{m}^{2\tau+1}\delta_{m}(s_{0})\right)
\\
& \leqslant\delta_{m}(s)\left(1+C\delta_{0}(s_{h})N_{0}^{\mu_{2}}N_{m}^{2\tau+1-\mu_{2}}\right)\\
&\quad\leqslant\delta_{m}(s)\left(1+CN_{m}^{-1}\right).
\end{align*}
By a trivial induction on the previous estimate and using \eqref{estimate delta0 and I02}, we find
\begin{align*}
\delta_{m}(s)\leqslant&\, \delta_{0}(s)\prod_{k=0}^{+\infty}\left(1+CN_{k}^{-1}\right)\nonumber\\
&\leqslant C\delta_{0}(s)\\
&\quad \leqslant C\epsilon_2\lambda^{-1}\left(1+\|\rho\|_{s+\sigma_0}^{\textnormal{Lip}(\lambda)}\right).\nonumber
\end{align*}
The construction of the transformation $\mathscr{B}$ and the remaining results of Proposition \ref{Thm transport} can  be done in a similar way to  \cite[Prop. 6.2]{HR21}.
\end{proof}

{\bf Acknowledgements.}  The work of Z. Hassainia and N. Masmoudi
is supported by Tamkeen under the NYU Abu Dhabi Research Institute grant of the center SITE. The work of N. Masmoudi is  partially supported by NSF grant DMS-1716466. T. Hmidi has been supported by Tamkeen under the NYU Abu Dhabi Research Institute grant.

\vspace{0.3cm}
	
	\begin{tabular}{l}
		\textbf{Zineb Hassainia} \\
		{\small Department of Mathematics}\\
		{\small New York University in Abu Dhabi} \\ 
		{\small Saadiyat Island, P.O. Box 129188, Abu Dhabi, UAE} \\ 
		{\small Email: zh14@nyu.edu}
	\end{tabular}
	
	\vspace{0.2cm}
	
		\begin{tabular}{l}
		\textbf{Taoufik Hmidi} \\
		{\small Department of Mathematics}\\
		{\small New York University in Abu Dhabi} \\ 
		{\small Saadiyat Island, P.O. Box 129188, Abu Dhabi, UAE} \\ 
		{\small Email: th2644@nyu.edu}
	\end{tabular}
	
	\vspace{0.2cm}
	
	\begin{tabular}{l}
		\textbf{Nader Masmoudi} \\
		{\small Department of Mathematics}\\
		{\small New York University in Abu Dhabi} \\ 
		{\small Saadiyat Island, P.O. Box 129188, Abu Dhabi, UAE} \\ 
		{\small Email: nm30@nyu.edu}\\
		~\\
		{\small Courant Institute of Mathematical Sciences}\\
		{\small New York University} \\ 
		{\small 251Mercer Street, NewYork, NY10012, USA} \\ 
		{\small Email: masmoudi@cims.nyu.edu}
	\end{tabular}

\end{document}